\documentclass{amsart}

\usepackage{etex}

\usepackage{hyperref}
\usepackage{txfonts, amsmath,amstext,amsthm,amscd,amsopn,verbatim,amssymb, amsfonts, amsrefs}
\usepackage{fullpage}

\usepackage[bbgreekl]{mathbbol}

\usepackage{tikz}

\usetikzlibrary{matrix}
\usetikzlibrary{shapes}
\usetikzlibrary{arrows}
\usetikzlibrary{calc,3d}
\usetikzlibrary{decorations,decorations.pathmorphing}
\usetikzlibrary{through}
\tikzset{ext/.style={circle, draw,inner sep=1pt},int/.style={circle,draw,fill,inner sep=1pt},nil/.style={inner sep=1pt}}
\tikzset{exte/.style={circle, draw,inner sep=3pt},inte/.style={circle,draw,fill,inner sep=3pt}}
\tikzset{diagram/.style={matrix of math nodes, row sep=3em, column sep=2.5em, text height=1.5ex, text depth=0.25ex}}
\tikzset{diagram2/.style={matrix of math nodes, row sep=0.5em, column sep=0.5em, text height=1.5ex, text depth=0.25ex}}

\usepackage{tikz-cd}

\theoremstyle{plain}

\newtheorem{thm}{Theorem}[section]

\newtheorem{defn}[thm]{Definition}
\newtheorem{prop}[thm]{Proposition}
\newtheorem{cor}[thm]{Corollary}
\newtheorem{lemm}[thm]{Lemma}

\theoremstyle{definition}

\newtheorem{ex}[thm]{Example}
\newtheorem{rem}[thm]{Remark}
\newtheorem{const}[thm]{Construction}

\DeclareMathOperator{\mymod}{mod}

\newcommand{\alg}[1]{\mathfrak{{#1}}}

\newcommand{\eref}[1]{\eqref{#1}} 

\newcommand{\ad}{{\text{ad}}}

\newcommand{\Hom}{\mathop{Hom}}

\newcommand{\R}{{\mathbb{R}}}
\newcommand{\N}{{\mathbb{N}}}
\newcommand{\Z}{{\mathbb{Z}}}
\newcommand{\K}{{\mathbb{K}}}
\newcommand{\Q}{{\mathbb{Q}}}
\newcommand{\D}{{\mathbb{D}}}

\newcommand{\HGC}{{\mathrm{HGC}}}

\newcommand{\Graphs}{{\mathsf{Graphs}}}
\newcommand{\fGraphs}{{\mathsf{fGraphs}}}

 \newcommand{\KSi}{\mathsf{K}}

\newcommand{\fHGC}{{\mathrm{fHGC}}}
\newcommand{\mF}{\mathcal{F}}

\newcommand{\stG}{{}^*\Graphs}
\newcommand{\hoPoiss}{\mathsf{hoPois}}

\newcommand{\Gra}{{\mathsf{Gra}}}

\newcommand{\Tw}{\mathit{Tw}}
\newcommand{\Def}{\mathrm{Def}}

\newcommand{\Poiss}{\mathsf{Pois}}

\newcommand{\op}{\mathcal}

\newcommand{\Lie}{\mathsf{Lie}}

\newcommand{\hoAss}{\mathsf{hoAss}}
\newcommand{\hoLie}{\mathsf{hoLie}}

\newcommand{\CoDer}{\mathrm{CoDer}}

\newcommand{\Ass}{\mathsf{Assoc}}
\newcommand{\Com}{\mathsf{Com}}

\newcommand{\FM}{\mathsf{FM}}

\newcommand{\Der}{\mathrm{Der}}
\newcommand{\DerL}{\mathit{L}}

\newcommand{\bpm}{\begin{pmatrix}}
\newcommand{\epm}{\end{pmatrix}}

\newcommand{\GC}{\mathrm{GC}}

\newcommand{\fGC}{\mathrm{fGC}}

\renewcommand{\Hom}{\mathrm{Hom}}
\newcommand{\MC}{\mathsf{MC}}

\newcommand{\grt}{\alg {grt}}

\newcommand{\e}{\mathsf{e}}
\newcommand{\Op}{\mathrm{Op}}
\newcommand{\La}{\Lambda}
\newcommand{\Hopf}{\mathrm{Hopf}}

\DeclareMathOperator{\Emb}{Emb}
\DeclareMathOperator{\Imm}{Imm}
\DeclareMathOperator{\Embbar}{\overline{Emb}}
\DeclareMathOperator{\Inj}{Inj}
\DeclareMathOperator{\Diff}{Diff}

\newcommand{\Seq}{\mathrm{Seq}}

\newcommand{\lo}{\longrightarrow}
\newcommand{\lD}{\mathsf{D}}

\newcommand{\mO}{\mathcal{O}}
\newcommand{\Map}{\mathrm{Map}}
\newcommand{\Mor}{\mathrm{Mor}}

\newcommand{\gra}{\mathrm{gra}}
\newcommand{\gr}{\mathrm{gr}}
\newcommand{\dg}{\mathrm{dg}}

\newcommand{\BiDer}{\mathrm{BiDer}}
\newcommand{\GRT}{\mathrm{GRT}}

\newcommand{\bS}{\mathsf{N}}
\newcommand{\bN}{\mathsf{N}}
\newcommand{\stp}{{{}^*\mathfrak{p}}}

\newcommand{\xCoDer}{\widehat{\CoDer}}
\newcommand{\xDer}{\widehat{\Der}}
\newcommand{\Vect}{{\mathrm{Vect}}}

\DeclareMathOperator{\TopCat}{Top}
\DeclareMathOperator{\SetCat}{Set}

\newcommand{\CoDef}{\mathsf{CoDef}}

\newcommand{\bo}{{\mathbf{1}}}

\newcommand{\oW}{\mathring{W}} 

\newcommand{\beq}[1]{\begin{equation}\label{#1}}
\newcommand{\eeq}{\end{equation}}

\newcommand{\mT}{{\mathbf T}}
\newcommand{\mTc}{{\mathcal T}}

\newcommand{\E}{\mathsf{E}}

\newcommand{\Linfty}{\texorpdfstring{$L_{\infty}$}{L-infinity}}
\newcommand{\dgcAlg}{\dg\mathrm{Com}}
\newcommand{\ev}{\mathit{ev}}

\newcommand{\APL}{A_{\rm PL}}

\newcommand{\mG}{{\mathcal G}}

\newcommand{\mE}{{\mathcal E}}
\DeclareMathOperator*{\holim}{holim}

\newcommand{\ar}{\mathrm{ar}}
\DeclareMathOperator*{\colim}{\mathrm{colim}}

\begin{document}
\title{The rational homotopy of mapping spaces of E${}_n$ operads}

\author{Benoit Fresse}
\address{Universit\' e Lille 1 - Sciences et Technologies \\
Laboratoire Painlev\' e\\
Cit\' e Scientifique - B\^ atiment M2\\
F-59655 Villeneuve d'Ascq Cedex, France}
\email{Benoit.Fresse@math.univ-lille1.fr}

\author{Victor Turchin}
\address{Department of Mathematics\\
Kansas State University\\
138 Cardwell Hall\\
Manhattan, KS 66506, USA}
\email{turchin@ksu.edu}

\author{Thomas Willwacher}
\address{Department of Mathematics \\ ETH Zurich \\
R\"amistrasse 101 \\
8092 Zurich, Switzerland}
\email{thomas.willwacher@math.ethz.ch}

\date{January 30, 2017}

\thanks{B.F. acknowledges support by grant ANR-11-BS01-002 \lq\lq{}HOGT\rq\rq{} and by Labex ANR-11-LABX-0007-01 \lq\lq{}CEMPI\rq\rq{}.
V.T. acknowledges the University of Lille, the MPIM, and the IHES where he started to work on this project during his sabbatical.
T.W. has been partially supported by the Swiss National Science foundation, grant 200021-150012, the NCCR SwissMAP funded by the Swiss National Science foundation, and the ERC starting grant GRAPHCPX.}


\begin{abstract}
We express the rational homotopy type of the  mapping spaces $\Map^h(\lD_m,\lD_n^{\Q})$ of the little discs operads in terms of graph complexes.
Using known facts about the graph homology this allows us to compute the rational homotopy groups in low degrees, and construct infinite series of non-trivial homotopy classes in higher degrees.
Furthermore we show that for $n-m>2$, the spaces $\Map^h(\lD_m,\lD_n^{\Q})$ and $\Map^h(\lD_m,\lD_n)$ are simply connected and rationally equivalent. As application we determine the rational homotopy type of the deloopings of spaces of long embeddings. Some of the results hold also for mapping spaces $\Map_{\leq k}^h(\lD_m,\lD_n^{\Q})$, $\Map_{\leq k}^h(\lD_m,\lD_n)$, $n-m\geq 2$, of the truncated little discs operads, which
allows one to determine  rationally the delooping of the Goodwillie-Weiss tower for the spaces of long embeddings.
\end{abstract}

\maketitle

\setcounter{tocdepth}{1}
\tableofcontents

\sloppy

\part*{Introduction}

\renewcommand{\thethm}{\arabic{thm}}

Recall that an $E_n$-operad denotes an operad in topological spaces which is (weakly) homotopy equivalent to the operad
of little $n$-discs $\lD_n$ (equivalently, to the operad of little $n$-cubes).
In what follows, we also consider $E_n$-operads in chain complexes, which are defined analogously as operads
that are quasi-isomorphic to the chain operad of little $n$-discs $C_*(\lD_n)$.

The $E_n$-operads have played a growing role in various problems of algebra, topology and mathematical physics during the last decades.
To cite one significant application, we have a second generation of proofs of the existence of deformation-quantizations
that relies on the formality of the chain operad of little $2$-discs
over the rationals.
This approach has hinted the existence of an action of the Grothendieck-Teichm\"uller group
on the moduli spaces of deformation-quantizations of Poisson manifolds (see~\cite{K2} for a survey).
In fact, one can prove that the Grothendieck-Teichm\"uller group represents the group
of rational homotopy automorphisms
of $E_2$-operads in topological spaces (see~\cite{Fr})
and an analogous result holds in the category of chain complexes (see~\cite{Will}).

In another domain, the embedding calculus developed by Goodwillie and Weiss \cite{GW,Weiss}
allowed to express the homotopy of spaces of smooth embeddings of manifolds
in terms of mapping spaces associated to modules
over the little discs operads (see \cite{Turchin2,WBdB,Turchin4}).
Let $\Emb_c(\R^m,\R^n)$ denote the space of embeddings $f: \R^m\hookrightarrow\R^n$
which agree with the inclusion $\R^m\times\{0\}\subset\R^n$
outside a compact domain of $\R^m$ (the space of embeddings with compact support).
Let $\Imm_c(\R^m,\R^n)$ denote the analogously defined space of immersions with compact support $f: \R^m\looparrowright\R^n$.
By pushing the connection with the Goodwillie calculus further, it has been established
that the homotopy fiber of the map $\Emb_c(\R^m,\R^n)\rightarrow\Imm_c(\R^m,\R^n)$
is homotopy equivalent to the $(m+1)$-fold iterated loop space
of the space of operad maps $\Map^h(\lD_m,\lD_n)$
associated to the little discs operads $\lD^m$ and $\lD^n$
as soon as $n-m>2$ (see~\cite{DwyerHess0,Turchin5,DwyerHess,WBdB2,Weiss2,DucT}).
Let us mention, to be more precise, we consider the space of operad maps $\Map^h(\lD_m,\lD_n)$
in the derived (homotopical) sense
in this statement.
This space can be defined properly by using that the category of topological operads is equipped with a model structure
and by using a space of operad maps (in the ordinary sense)
associated to appropriate resolutions
of the little discs operads. We go back to this subject later on.

The goal of this paper is to compute the homotopy of mapping spaces $\Map^h(\lD_m,\lD_n^{\Q})$,
where we consider a rationalization of the operad of little discs as a target object $\lD_n^{\Q}$
instead of the ordinary operad $\lD_n$.
We more precisely establish that these mapping spaces $\Map^h(\lD_m,\lD_n^{\Q})$
are equivalent in the homotopy category to combinatorial objects
which we define in terms of the nerve of Lie algebras
of (hairy) graphs.
We give a more detailed description of our statement in the next section of this introduction.
We also prove that the space $\Map^h(\lD_m,\lD_n^{\Q})$
is rationally equivalent to $\Map^h(\lD_m,\lD_n)$
when $n-m>2$.
Thus, our results provide a solution to the problem of computing the rational homotopy type of the delooping
of the spaces of embeddings with compact support.
In the case $m=n$, we also get a description of the space of homotopy automorphisms
of the rational little $n$-discs operad $\lD_n^{\Q}$.

\subsection*{Main results}
In our constructions, we deal with mapping spaces with values in the category of simplicial sets
rather than with values in the category of topological spaces.
Therefore, we more generally adopt the terminologies of the homotopy theory of simplicial sets in what follows.
In particular, we call weak equivalence a map of simplicial sets whose geometric realization defines a homotopy equivalence
in the category of topological spaces.
Besides the category of simplicial sets, we consider objects defined in the category of (unbounded) chain complexes over a fixed ground field $\K$,
but we  use the language of differential graded algebra
rather than the chain complex terminology. We generally call ``dg vector space'' (where the prefix ``dg'' stands for ``differential graded'')
the objects of our base category of chain complexes,
and we reserve the chain complex terminology for specific objects which we may form
in this category, like the complexes of hairy graphs.
We similarly use the expressions ``commutative dg algebra'', ``dg Lie algebra'', {\dots}
for the standard generalizations of the classical categories of algebras
which we may form in the category of chain complexes.
We also take the field of rational numbers as a ground field $\K = \Q$
all through this introduction for simplicity.

The hairy graph complex $\HGC_{m,n}$, which we consider in our computations, is a dg Lie algebra spanned by combinatorial graphs
with external legs (or hairs), as the examples depicted in the following figure:
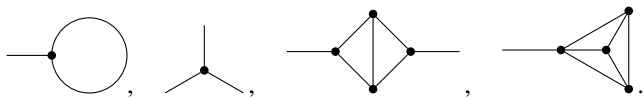
\[
\begin{tikzpicture}[scale=.5]
\draw (0,0) circle (1);
\draw (-180:1) node[int]{} -- +(-1.2,0);
\end{tikzpicture},
\quad\begin{tikzpicture}[scale=.6]
\node[int] (v) at (0,0){};
\draw (v) -- +(90:1) (v) -- ++(210:1) (v) -- ++(-30:1);
\end{tikzpicture}
{\,},
\quad\begin{tikzpicture}[scale=.5]
\node[int] (v1) at (-1,0){};\node[int] (v2) at (0,1){};\node[int] (v3) at (1,0){};\node[int] (v4) at (0,-1){};
\draw (v1)  edge (v2) edge (v4) -- +(-1.3,0) (v2) edge (v4) (v3) edge (v2) edge (v4) -- +(1.3,0);
\end{tikzpicture}
{\,},
\quad\begin{tikzpicture}[scale=.6]
\node[int] (v1) at (0,0){};\node[int] (v2) at (180:1){};\node[int] (v3) at (60:1){};\node[int] (v4) at (-60:1){};
\draw (v1) edge (v2) edge (v3) edge (v4) (v2)edge (v3) edge (v4)  -- +(180:1.3) (v3)edge (v4);
\end{tikzpicture}
{\,}.
\]
The degree of a graph is determined by a rule which depends on the dimension indices $m,n\in\N$
and the differential is defined through vertex splitting.
The Lie bracket of two graphs is computed by connecting a hair of one graph to vertices of the other as indicated in the following picture:
\[
\left[\begin{tikzpicture}[baseline=-.8ex]
\node[draw,circle] (v) at (0,.3) {$\alpha$};
\draw (v) edge +(-.5,-.7) edge +(0,-.7) edge +(.5,-.7);
\end{tikzpicture},
\begin{tikzpicture}[baseline=-.65ex]
\node[draw,circle] (v) at (0,.3) {$\beta$};
\draw (v) edge +(-.5,-.7) edge +(0,-.7) edge +(.5,-.7);
\end{tikzpicture}\right]
= \sum\begin{tikzpicture}[baseline=-.8ex]
\node[draw,circle] (v) at (0,1) {$\alpha$};
\node[draw,circle] (w) at (.8,.3) {$\beta$};
\draw (v) edge +(-.5,-.7) edge +(0,-.7) edge (w);
\draw (w) edge +(-.5,-.7) edge +(0,-.7) edge +(.5,-.7);
\end{tikzpicture}
\pm\sum\begin{tikzpicture}[baseline=-.8ex]
\node[draw,circle] (v) at (0,1) {$\beta$};
\node[draw,circle] (w) at (.8,.3) {$\alpha$};
\draw (v) edge +(-.5,-.7) edge +(0,-.7) edge (w);
\draw (w) edge +(-.5,-.7) edge +(0,-.7) edge +(.5,-.7);
\end{tikzpicture}.
\]
We review the definition of the hairy graph complexes and of this dg Lie algebra structure with more details in Section \ref{subsec:graphs:HGC}.

We consider the nerve $\MC_\bullet(\DerL)$ of the dg Lie algebra $\DerL = \HGC_{m,n}$, which is the simplicial set formed by the sets
of flat $\DerL$-valued polynomial connections on the simplices $\Delta^n$, $n\in\N$.
The main result of this paper reads as follows.

\begin{thm}\label{thm:main topological statement}
The mapping space $\Map^h(\lD_m,\lD_n^{\Q})$ is weakly equivalent to the nerve of the hairy graph complex $\HGC_{m,n}$
when $n\geq m\geq 2$.
\end{thm}

This result admits an extension for $m=1$.
In this case, we have to consider an $L_{\infty}$~structure on $\HGC_{1,n}$,
the Shoikhet $L_{\infty}$~structure,
which occurs as a deformation of the natural dg Lie structure
of the hairy graph complex.

\begin{thm}\label{thm:main topological statement:Shoikhet case}
The space $\Map^h(\lD_1,\lD_n^{\Q})$ is weakly equivalent to the nerve of the graph complex $\HGC_{1,n}$
equipped with the Shoikhet $L_{\infty}$~structure
for any $n\geq 2$.
\end{thm}

The Shoikhet $L_{\infty}$~structure is described with full details by the third author in the spin-off paper~\cite{WillDefQ}.
We give a brief reminder on the definition of this $L_{\infty}$~structure in Section \ref{subsec:biderivations to HGC:morphism Shoikhet case}.

The above theorems admit the following corollary:

\begin{cor}\label{cor:mapping spaces homotopy}
If $n-2\geq m\geq 1$, then the space $\Map^h(\lD_m,\lD_n^{\Q})$ is simply connected and its homotopy groups are given by:
\[
\pi_k(\Map^h(\lD_m,\lD_n^{\Q}))\cong H_{k-1}(\HGC_{m,n}),
\]
for all $k\geq 1$.
\end{cor}

The simply-connectedness of the space $\Map^h(\lD_m,\lD_n^{\Q})$ follows from the vanishing of the hairy graph complex in degree $\leq 1$,
and this result can also be deduced from the spectral sequence methods
of \cite{FW}.
We have no full description of the homology of the hairy graph complex, but a significant amount of information is now available on this homology \cite{KWZ2,TW2}.
We can for instance establish that the part of the hairy graph homology spanned by graphs of loop order $g$
is concentrated in degrees $\geq g(n-3)+1$
when $n-m\geq 2$.
We have the following explicit description of the part of loop order $g=0$:
\begin{align}
\label{eq:HGC0loop}
H(\HGC_{m,n}^{0-loop}) &=
\begin{cases}
\Q L, & \text{if $m\equiv n\mymod{2}$}, \\
\Q Y, & \text{if $m\not\equiv n\mymod{2}$},
\end{cases}
\intertext{where $L$ is the line graph $L := \begin{tikzpicture}[scale=.5, baseline=-.65ex]
\draw (0,0)--(1,0);
\end{tikzpicture}$ in degree $n-m-1$, and $Y$ is the tripod graph $Y := \begin{tikzpicture}[scale=.3, baseline=-.65ex]
\node[int] (v) at (0,0){};
\draw (v) -- +(90:1) (v) -- ++(210:1) (v) -- ++(-30:1);
\end{tikzpicture}$ in degree $2(n-m)-3$,
while we have the following description for the part of loop order $g=1$:}
\label{eq:HGC1loop}
H(\HGC_{m,n}^{1-loop}) & =
\begin{cases}
\prod_{k=1,3,5,\dots} \Q H_k, & \text{if $m$ and $n$ are even}, \\
\prod_{k=2,4,6,\dots} \Q H_k, & \text{if $m$ and $n$ are odd}, \\
\prod_{k=1,5,9,\dots} \Q H_k, & \text{if $m$ is odd and $n$ is even}, \\
\prod_{k=3,7,11,\dots} \Q H_k, & \text{if $m$ is even and $n$ is odd},
\end{cases}
\end{align}
where $H_k := \begin{tikzpicture}[baseline=-.65ex, scale=.6]
\node[int] (v1) at (0:1) {};
\node[int] (v2) at (72:1) {};
\node[int] (v3) at (144:1) {};
\node[int] (v4) at (216:1) {};
\node (v5) at (-72:1) {$\cdots$};
\draw (v1) edge (v2) edge (v5) (v3) edge (v2) edge (v4) (v4) edge (v5);
\draw (v1)  edge +(0:.6) ;
\draw (v2) edge +(72:.6) ;
\draw (v3) edge +(144:.6) ;
\draw (v4)  edge +(226:.6) ;
\end{tikzpicture}$
($k$ hairs) lies degree $k(n-m-2)+m$.
We also have the following formula for the generating function of the homology in loop order $g=2$:
\begin{align}\label{eq:HGC2loop}
\sum_{j,k} s^j t^k \text{dim} H^j(\HGC_{m,n}^{2-loop, k-hair}) & =
\begin{cases}
\frac{s^{n-3+m}T^6+s^{n-2+m}T^7}{(1-T^2)(1-T^6)}, & \text{if $m$ and $n$ are even}, \\
s^{n-3+m}\left( \frac{1}{(1-T^2)(1-T^6)}-1\right)+\frac{s^{n-2+m}T}{(1-T^2)(1-T^6)}, & \text{if $m$ and $n$ are odd}, \\
\frac{s^{n-3+m}(T^3+T^{11}+T^{14}-T^{15}) + s^{n-2+m}(T+T^{16})}{(1-T^4)(1-T^{12})}, & \text{if $m$ is even and $n$ is odd}, \\
\frac{s^{n-3+m}(T^2+T^{11}) + s^{n-2+m}(T^4+T^{13})}{(1-T^4)(1-T^{12})}, & \text{if $m$ odd and $n$ is even}, \\
\end{cases}
\end{align}
where $T:= s^{n-m-2}t$. (We refer to \cite{CCT} for a more explicit result.)

Let us record that the above computations are enough to determine the homotopy of our operadic mapping spaces in low degrees,
when the homology of the hairy graph complex reduces to this part
of loop order $\leq 2$.
To be explicit, we have the following statement.

\begin{cor}\label{cor:mapping spaces homotopy:low degrees}
Let $n-2\geq m\geq 1$. We have
\[
\pi_k(\Map^h(\lD_m,\lD_n^{\Q}))= H_{k-1}(\HGC_{m,n}^{\leq 2-loop})
\]
for $k\leq 3n-8$.
\end{cor}

In the case $n-m<2$, we get a different picture for the homotopy of our operadic mapping spaces.
Recall that the homology of the little $n$-discs operad $H(\lD_n)$
is identified with the operad governing $n$-Poisson algebras $\Poiss_n$,
for any $n\geq 2$,
where an $n$-Poisson algebra consists of a graded commutative algebra
equipped with a Poisson bracket operation
of degree $n-1$.
In the case $m=n$, we consider the locally constant function $F: \Map^h(\lD_n,\lD_n^{\Q})\rightarrow\Q$
which, to any homotopy class of maps $[\phi]\in\pi_0\Map^h(\lD_n,\lD_n^{\Q})$,
associates the constant $\lambda\in\Q$
such that we have the relation $\phi_*([x_1,x_2]) = \lambda [x_1,x_2]$
in $H(\lD_n(2),\Q)$, where $[x_1,x_2]\in\Poiss_n(2)$ is the generating operation of the operad $\Poiss_n$
that represents the Poisson bracket operation
in our category of algebras. We then have the following statement.

\begin{cor}\label{cor:codimension zero case}
Let $n\geq 2$. For $\lambda\neq 0$, we have a weak equivalence
\[
\Map^h(\lD_n,\lD_n^{\Q})\supset F^{-1}(\lambda)\simeq\MC_\bullet(H(\GC_n^2)[1]),
\]
where we consider the homology of the non-hairy graph complex $H(\GC_n^2)$, we use the notation $V[1]$ to mark a degree shift $V[1]_* = V_{*-1}$,
and we regard this graded vector space $H(\GC_n^2)[1]$ as an abelian graded Lie algebra (a Lie algebra with a trivial Lie bracket).
In particular the following properties hold:
\begin{itemize}
\item If $n>2$ and $n+1$ is not divisible by $4$, then the space $F^{-1}(\lambda)\subset\Map^h(\lD_n,\lD_n^{\Q})$ is connected.
\item If $n+1$ is divisible by $4$, then $F^{-1}(\lambda)$ contains $\Q$-many connected components.
\item If $n=2$, then the connected components of $F^{-1}(\lambda)$ are in bijection with elements
of the Grothendieck-Teichm\"uller group $\GRT_1(\Q)$ (see also Fresse \cite{Fr}).
\end{itemize}
In all cases, we have the identity:
\[
\pi_k(\Map^h(\lD_n,\lD_n^{\Q}),p_{\lambda}) = H_k(\GC_n^2),
\]
where $p_\lambda$ is any basepoint in $F^{-1}(\lambda)$ and $k\geq 1$.
\end{cor}

The graph complex $\GC_n^2$ was first defined by Kontsevich \cite{Kformal}, and consists of formal series of isomorphism classes of undirected graphs
whose vertices are at least bivalent,
like:
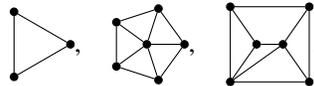
\[
\begin{tikzpicture}[baseline=-.65ex, scale=.5]
\node[int] (v1) at (0:1) {};
\node[int] (v2) at (120:1) {};
\node[int] (v3) at (240:1) {};
\draw (v1) edge (v2) edge (v3) (v3) edge (v2) ;
\end{tikzpicture},
\quad\begin{tikzpicture}[baseline=-.65ex, scale=.5]
\node[int] (c) at (0,0){};
\node[int] (v1) at (0:1) {};
\node[int] (v2) at (72:1) {};
\node[int] (v3) at (144:1) {};
\node[int] (v4) at (216:1) {};
\node[int] (v5) at (-72:1) {};
\draw (v1) edge (v2) edge (v5) (v3) edge (v2) edge (v4) (v4) edge (v5)
      (c) edge (v1) edge (v2) edge (v3) edge (v4) (c) edge (v5);
\end{tikzpicture},
\quad\begin{tikzpicture}[baseline=-.65ex, scale=.5]
\node[int] (c) at (0.7,0){};
\node[int] (v1) at (0,-1) {};
\node[int] (v2) at (0,1) {};
\node[int] (v3) at (2.1,-1) {};
\node[int] (v4) at (2.1,1) {};
\node[int] (d) at (1.4,0) {};
\draw (v1) edge (v2) edge (v3)  edge (d) edge (c) (v2) edge (v4) edge (c) (v4) edge (d) edge (v3) (v3) edge (d) (c) edge (d);
\end{tikzpicture}.
\]
The degree of a graph in $\GC_n$ is determined by a rule which depends on the dimension index $n\in\N$
and the differential is defined through vertex splitting.
as in the case of the hairy graph complex. (The definition of the Kontsevich graph complex is also reviewed with full details
later on in this article.)

We currently do not know the full graph homology $H(\GC_n^2)$, nevertheless big families of non-trivial classes,
which appear in the homotopy groups of the mapping space $\Map^h(\lD_n,\lD_n^{\Q})$
by the above corollary, are known.
We review a selection of known facts about the graph homology from the literature with no claim of completeness (we refer to \cite{KWZ} for a more detailed overview).
In what follows, we use the notation $V[m]$ for the $m$-fold suspension of a graded vector space $V$,
which we define by the shift $V[m]_* = V_{*-m}$ in the grading of $V$,
for any $m\in\Z$.
In the case $V = \K$, where we identify the ground field $\K$ with a graded vector space concentrated in degree $0$,
we get that $\K[m]$ represents the graded vector space of rank one
concentrated in (homological) degree $m$. (Recall that we take $\K = \Q$ in this introduction.)

\begin{itemize}
\item
The graph complexes split into subcomplexes according to loop order. The part of loop order $g=1$ is infinite dimensional, but well understood.
We explicitly have:
\[
H(\GC_n^{2,1-loop}) = \bigoplus_{\substack{k\equiv 2n+1\mymod{4}\\
k\geq 1}}\Q[n-k],
\]
with homology classes $L_k$ represented by loop graphs with $k$ vertices, for $k\geq 1$.
We moreover have (see \cite[Proposition 3.4]{Will}):
\[
H(\GC_n^2)\cong H(\GC_n^{2,1-loop})\oplus H(\GC_n)
\]
where $\GC_n\subset\GC_n^2$ is the complex formed by graphs whose vertices are at least trivalent, and the factors $\GC_n^{g-loop}$ of loop order $g\geq 2$
in $\GC_n$ form finite dimensional complexes.
\item
The graph complexes $\GC_n$ associated to even dimension indices $n$ (respectively, to odd dimension indices $n$)
are isomorphic modulo degree shifts:
\[
\GC_n^{g-loop}\cong\GC_{n+2}^{g-loop}[2g],
\]
so that the homology of our complexes have non-trivial periodicity properties.
\item
For $n=2$, we have $H_0(\GC_2)\cong\grt_1$, where $\grt_1$ denotes the graded Grothendieck-Teichm\"uller Lie algebra again (see \cite{Will}).
By a result of Francis Brown \cite{Br-2}, this Lie algebra contains a free Lie algebra $L(\sigma_3,\sigma_5,\dots)\subset\grt_1$.
The generators of this Lie algebra $\sigma_{2j+1}$ corresponds to graphs of loop order $2j+1$ in $H(\GC_2)$,
which, by periodicity, give non-trivial classes in the homology of the graph complexes $H(\GC_n)$,
for all even $n$.
\item
For $n=3$, the bottom homology group $H_3(\GC_3)$ is the space of trivalent graphs modulo IHX relations,
and we have a map
\[
\Q[t,\omega_0,\omega_1,\ldots]/(\omega_p\omega_q-\omega_0\omega_{p+q},P)\to H_3(\GC_3),
\]
for a certain polynomial $P$ (see \cite{vogel, kneissler1,kneissler2,kneissler3}),
which is conjecturally an isomorphism up to one class represented by the ``$\Theta$-graph'' $\begin{tikzpicture}[baseline=-.65ex, scale=1]
\node[int](v) at (0,0) {};
\node[int](w) at (0.5,0) {};
\draw (v) edge (w) edge [bend left](w) edge[bend right] (w);
\end{tikzpicture}$.
(The map is known to be an isomorphism in low loop orders and the image is known to form an infinite dimensional vector space.)
\end{itemize}
We do not know a topological construction of any of the families of classes of this list.

We can also use the result of Corollary \ref{cor:codimension zero case} to compute the homotopy groups of the mapping spaces $\Map^h(\lD_n,\lD_n)$ in low degrees.
We use that the piece of the graph homology $H(\GC_n^{g-loop})$ of loop order $g$
is concentrated in homological degrees $\geq g(n-3)+2$.
We already mentioned that the complex $\GC_n^{g-loop})$ is finite dimensional.
We can therefore compute the homology of this complex
numerically, for low values of $g$
at least.
For $n$ even and $g\leq 10$ (respectively, for $n$ odd and $g\leq 8$), we explicitly get:
\begin{equation}\label{eq:GC_lowloop}
\begin{aligned}[t]
H(\GC_n^{2,\leq 10-loop}) = & \underbrace{\oplus_{k\equiv 1\mymod{4}}\Q[n-k]}_{1\text{ loop}}
\oplus\underbrace{\Q[3n-6]}_{3\text{ loop}}
\oplus\underbrace{\Q[5n-10]}_{5\text{ loop}}
\oplus\underbrace{\Q[6n-15]}_{6\text{ loop}}
\\
\oplus & \underbrace{\Q[7n-14]}_{7\text{ loop}}
\oplus\underbrace{\Q[8n-16]\oplus\Q[8n-19]}_{8\text{ loop}}
\oplus\underbrace{\Q[9n-18]\oplus\Q[9n-21]}_{9\text{ loop}}
\\
\oplus & \underbrace{\Q[10n-14]\oplus\Q^2[10n-23]\oplus\Q[10n-27]}_{10 \text{ loop}}
\quad\text{for $n$ even} 
\\
H(\GC_n^{2,g \leq 8-loop}) = & \underbrace{\oplus_{k\equiv 3\mymod{4}}\Q[n-k]}_{1\text{ loop}}
\oplus\underbrace{\Q[2n-3]}_{2\text{ loop}}
\oplus\underbrace{\Q[3n-6]}_{3\text{ loop}}
\oplus\underbrace{\Q[4n-9]}_{4\text{ loop}}
\\
\oplus & \underbrace{\Q^2[5n-12]}_{5\text{ loop}}
\oplus\underbrace{\Q^2[6n-15]\oplus\Q[6n-12]}_{6\text{ loop}}
\oplus\underbrace{\Q^3[7n-18]\oplus\Q[7n-15]}_{7\text{ loop}}
\\
\oplus & \underbrace{\Q^4[8n-21]\oplus\Q^2[8n-18]}_{8\text{ loop}}
\quad\text{for $n$ odd}.
\end{aligned}
\end{equation}
We then get the following result for the low degree homotopy groups of our (rational) mapping space.

\begin{cor}\label{cor:codimension zero case:low degrees}
In the case $n=2$, we have $\pi_1(\Map^h(\lD_2,\lD_2^{\Q}),p_\lambda) = \Q$ and $\pi_k(\Map^h(\lD_2,\lD_2^{\Q}),p_\lambda) = 0$ when $k\geq 2$.
In the case where $n$ is even and $n\geq 4$, we have
\[
\pi_k(\Map^h(\lD_n,\lD_n^{\Q}),p_\lambda) = H_k(\GC_n^{2,\leq 10-loop})
\]
for $k\leq 11(n-3)+2$.
In the case where $n$ is odd and $n\geq 3$, we have
\[
\pi_k(\Map^h(\lD_n,\lD_n^{\Q}),p_\lambda)=H_k(\GC_n^{2,\leq 8-loop})
\]
for $k\leq 9(n-3)+2$.
In each case, we consider the homology of the complex computed in \eqref{eq:GC_lowloop},
and $p_\lambda$ is an arbitrary base point in $F^{-1}(\lambda)$
with $\lambda\neq 0$.
\end{cor}

We have little information about $F^{-1}(0)$. We know that there is one connected component
corresponding to the map $\lD_n\stackrel{*}{\lo}\lD_n^{\Q}$
that factors through the trivial topological operad
with one point in each arity $*(r)\equiv *$. We can use the hairy graph complex to compute the homotopy groups of this particular connected component at least.
To be explicit, we can record the following statement.

\begin{cor}\label{cor:codimension zero null component}
Let $n\geq 2$. We have
\[
\pi_k(\Map^h(\lD_n,\lD_n^{\Q}),*) = H_{k-1}(\HGC_{n,n}),
\]
for all $k\geq 1$, and the homotopy groups of any other connected components of the space $F^{-1}(0)$
are subquotients of this vector space.\footnote{However, we do not know whether there are any other connected component and which subquotients to take.}
\end{cor}

We now consider the case $n=m+1$. We then have the same results as in the case $n=m$:

\begin{cor}\label{cor:codimension one case}
Let $n\geq 2$. We have a locally constant function $J: \Map^h(\lD_{n-1},\lD_n^{\Q})\rightarrow\Q$
such that, for each $\lambda\neq 0$, we have the relation:
\[
J^{-1}(\lambda)\simeq\MC_\bullet(H(\GC_n^2)[1]),
\]
where we consider the homology of the non-hairy graph complex $H(\GC_n^2)$ together with a degree shift $V[1]_* = V_{*-1}$,
and we regard this object $H(\GC_n^2)[1]$ as an abelian graded Lie algebra again.
\end{cor}

\begin{rem}\label{rem:1/4}
Let us mention that the locally constant function $J$ takes the value $J(i) = 1/4$ on the connected component
of the canonical map $i: \lD_{n-1}\hookrightarrow\lD_n$ (with $n\geq 2$).
\end{rem}

We still have little information about $J^{-1}(0)$. We know that there is one connected component
corresponding to the trivial map $\lD_{n-1}\stackrel{*}{\lo}\lD_n^{\Q}$
and we can also compute the homotopy groups of this connected component
in terms of the hairy graph complex again.
We record this observation in the following statement.

\begin{cor}\label{cor:codimension one null component}
Let $n\geq 2$. We have
\[
\pi_k(\Map^h(\lD_{n-1},\lD_n^{\Q}), *) = H_{k-1}(\HGC_{n-1,n}),
\]
for all $k\geq 1$, and the homotopy groups of any other connected components of the space $J^{-1}(0)$
are subquotients of this vector space.\footnote{However, we do not know whether there are any other connected component and which subquotients to take in this case yet.}
\end{cor}

\subsection*{The algebraic results}
To establish our results we use the rational homotopy theory of operads developed by the first author in \cite[Part II]{Fr}.

Let us observe, in a preliminary step, that we can reduce our computations
about mapping spaces of topological operads
to computations about mapping spaces of simplicial operads,
because these categories of operads
inherit Quillen equivalent
model structures.

For the applications of rational homotopy methods, we consider the model category of simplicial $\La$-operads of \cite[Section II.8.4]{Fr}.
Briefly recall that this category is isomorphic to the subcategory formed by the simplicial operads
whose component of arity zero is reduced to a one-point set (the category of unitary operads
in the terminology of \cite{Fr}).
The idea of a $\La$-operad is to use restriction operators, which we may associate to any injective map between finite ordinals,
in order to model operadic composites with an element
of arity zero.
The category of simplicial $\La$-operads is precisely defined as the category of operads which have no element in arity zero,
but whose underlying collection is equipped with such restriction operators
that we can use to recover the composition operations
with an element of arity zero.
To apply the results of \cite{Fr}, we have to deal with simplicial $\La$-operads
whose component of arity one reduces to a point.
Let us mention that we have models of $E_n$-operads
which do satisfy this condition
though this is not the case of the operad of little $n$-discs.

In the previous statement of results, we consider mapping spaces $\Map^h(\lD_m,\lD_n)$
that we may form in the category of all operads
by using that we can equip this category with a model structure.
But, once we pick models of the operads $\lD_m$ and $\lD_n$ in the category of simplicial $\La$-operads,
we can actually reduce these mapping spaces to mapping spaces
of simplicial $\La$-operads (see~\cite{FNote}).

The category of simplicial $\La$-operads is used in \cite[Chapter II.12]{Fr} to define an operadic upgrading
of the Sullivan functor of piecewise linear differential forms. We precisely have a functor of the form:
\[
\Omega_{\sharp} : \{\text{simplicial $\La$-operads}\}\to\{\text{dg Hopf $\La$-cooperads}\},
\]
where we use the phrase dg Hopf $\La$-cooperad to refer to structures which are dual (in the categorical sense) to $\La$-operads
and which we form in the category of commutative dg algebras.
The category of dg Hopf $\La$-cooperads inherits a model structure by the results \cite[Section II.11.4]{Fr}.
Furthermore, one can use mapping spaces associated to dg Hopf $\La$-cooperads
to compute the mapping spaces associated to the rationalization
of objects in the category of simplicial operads. We use this correspondence to reduce the proof of our main results,
about the topological mapping spaces associated to the operad of little discs,
to a computation of mapping spaces
in the category of dg Hopf $\La$-cooperads.

To be explicit, we assume that $\Omega_{\sharp}(\lD_n)$ denotes, by abuse of notation, the dg Hopf $\La$-cooperad
which we obtain by applying the above functor to a good (cofibrant) model of the operad
of little $n$ discs
in the category of simplicial $\La$-operads.
In what follows, we also denote by $\E_n^c$ a dg Hopf cooperad quasi-isomorphic to $\Omega_{\sharp}(\lD_n)$.
In fact, we will use a dg Lie algebra model of the mapping space
associated to these objects
in the category of dg Hopf $\La$-cooperads.
To be precise, we are going to consider a dg Lie algebra of biderivations
of dg Hopf $\La$-cooperads:
\[
\Def(\E_n^c,\E_m^c) := \BiDer_{\dg\La}(\check{\E}_n^c,\hat{\E}_m^c),
\]
for a suitably chosen fibrant replacement $\hat{\E}_m^c$ of $\E_m^c$,
a cofibrant replacement $\check{\E}_n^c$ of $\E_n^c$,
and we will prove that we have the following statement:

\begin{thm}\label{thm:nerve to mapping spaces}
We have a weak equivalence of simplicial sets
\[
\Map^h(\lD_m,\lD_n^{\Q})\simeq\MC_\bullet(\Def(\E_n^c,\E_m^c)).
\]
\end{thm}

Then we establish the following algebraic counterpart of our main Theorem \ref{thm:main topological statement}\footnote{This algebraic statement
is in fact stronger than our result about mapping spaces because we also deal with negative degree components
of the deformation complex which vanish when we pass to the nerve of our dg Lie algebras.}:

\begin{thm}\label{thm:main algebraic statement}
For $n\geq m\geq 2$, the dg Lie algebra $\Def(\E_n^c,\E_m^c)$ is $L_{\infty}$~quasi-isomorphic to the dg Lie algebra $\HGC_{m,n}$.
For $n\geq 2$, the dg Lie algebra $\Def(\E_n^c,\E_1^c)$ is $L_{\infty}$~quasi-isomorphic to the hairy graph complex $\HGC_{1,n}$
equipped with the Shoikhet $L_{\infty}$~structure.
\end{thm}

The quasi-isomorphisms of this theorem will be realized by relatively explicit zigzags.

Theorem \ref{thm:main topological statement} follows from Theorem \ref{thm:nerve to mapping spaces} and Theorem~\ref{thm:main algebraic statement}.
We are similarly going to derive Corollaries \ref{cor:mapping spaces homotopy}-\ref{cor:codimension one null component},
from algebraic counterparts
of these results.

Thus, we mostly work in the setting of dg Hopf $\La$-cooperads in the rest of this paper.
We focus on the proof of our algebraic statement, Theorem \ref{thm:main algebraic statement},
and we only tackle the topological applications of our constructions
in a second step.

\begin{rem}\label{rem:Shoikhet structure explanations}
The algebraic counterpart of the trivial morphism $\lD_m\stackrel{*}{\lo}\lD_n^{\Q}$ is a canonical morphism
of dg Hopf $\La$-cooperads $\E_n^c\stackrel{*}{\lo}\E_m^c$
that factors through the unit object $\K = \Q$
in each arity $r$.
In the definition of the complex $\BiDer_{\dg\La}(\check{\E}_n^c,\hat{\E}_m^c)$, we actually consider biderivations
of this morphism $\check{\E}_n^c\stackrel{*}{\lo}\hat{\E}_m^c$,
for the fibrant and cofibrant replacements of our objects $\E_n^c$ and $\E_m^c$.
The formality of the little discs operads~\cite{LVformal,FW} implies that the Hopf $\La$-cooperads $\E_n^c$ and $\E_m^c$
are quasi-isomorphic to their cohomology $\e_n^c$ and $\e_m^c$,
as dg Hopf $\La$-cooperads.
The complex $\Def(\E_n^c,\E_m^c)$ can therefore be understood as the biderivation complex (in the derived sense)
of the morphism of Hopf $\La$-cooperads $\e_n^c\stackrel{*}{\lo}\e_m^c$
associated to these cohomology cooperads.

This identity explains why the case $m=1$ is special in the result of Theorem \ref{thm:main algebraic statement}.
Indeed, we have $\e_m^c=H(\E_m^c)$ is the associative cooperad in the special case  $m=1$, and is the
dual to the operad governing graded Poisson algebras in the case $m\geq 2$.
\end{rem}

\subsection*{Connection with the study of embedding spaces}
We mentioned at the beginning of this introduction that one of the main motivations of our work is given by the occurrence
of the mapping spaces associated to the little discs operads
in the study of embedding spaces.
We describe these connections with more details in order to complete the account of this introduction.

We consider smooth embeddings $f: \D^m\hookrightarrow\D^n$ of the $m$-discs $\D^m$ into the $n$-discs $\D^n$ from now on (rather than embeddings
of Euclidean space $f: \R^m\hookrightarrow\R^n$).
We then form the space $\Emb_{\partial}(\D^m,\D^n)$ of the smooth embeddings $f: \D^m\hookrightarrow\D^n$
that coincide with the ``equatorial'' inclusion $\D^m = \D^m\times\{0\}\subset\D^n$
near $\partial\D^m\subset\D^m$.
Let $\Embbar_{\partial}(\D^m,\D^n)$ denote the homotopy fiber of the Smale-Hirsch map
\begin{equation}\label{eq:SmHirsch}
\Emb_{\partial}(\D^m,\D^n)\to\Omega^m\Inj(\R^m,\R^n),
\end{equation}
where $\Inj(\R^m,\R^n)$ is the space of linear injections $\R^m\hookrightarrow\R^n$ based at the canonical inclusion $\R^m = \R^m\times\{0\}\subset\R^n$.
Note that $ \Inj(\R^m,\R^n)$ is homotopy equivalent to the Stiefel manifold $V_m(\R^n)$.
In the case $n>m$, the space $\Omega^m\Inj(\R^m,\R^n)$ is weakly equivalent to the space $\Imm_{\partial}(\D^m,\D^n)$
formed by the smooth immersions $f: \D^m\looparrowright\D^n$
which coincide with the equatorial inclusion $\D^m = \D^m\times\{0\}\subset\D^n$
near $\partial\D^m\subset\D^m$
like our embeddings.
For this reason, this space $\Embbar_{\partial}(\D^m,\D^n)$ is often called the space of smooth embeddings (with compact support) modulo immersions
in the literature.

Let $\mO_{\partial}(\D^m)$ denote the poset of open subsets of the $m$-discs $U\subset\D^m$ satisfying $\partial\D^m\subset U$.
The definition of the space $\Embbar_{\partial}(\D^m,\D^n)$ admits an obvious generalization $\Embbar_{\partial}(U,\D^n)$
for the open sets $U\in\mO_{\partial}(\D^m)$,
so that the mapping $\Embbar_{\partial}(-,\D^n): U\mapsto\Embbar_{\partial}(U,\D^n)$
defines a presheaf with values in the topology of topological spaces:
\begin{equation}\label{eq:emb_presh}
\Embbar_{\partial}(-,\D^n): \mO_{\partial}(\D^m)\to\TopCat.
\end{equation}
The Goodwillie-Weiss manifold calculus~\cite{GW} produces a series of polynomial (or Taylor) approximations $T_k\Embbar_{\partial}(-,\D^n)$, $k\geq 0$,
of this presheaf $\Embbar_{\partial}(-,\D^n)$.
In the case $n-m>2$, the maps $T_k\Embbar_{\partial}(-,\D^n)\to T_{k-1}\Embbar_{\partial}(-,\D^n)$
that link the stages of the Goodwillie-Weiss Taylor tower become higher and higher connected when $k$ increases,
so that the limit $T_{\infty}\Embbar_{\partial}(-,\D^n) = \holim_k T_k\Embbar_{\partial}(-,\D^n)$
defines a presheaf which is homotopy equivalent to our initial object
pointwise~\eqref{eq:emb_presh}.
In particular, we have:
\begin{equation}\label{eq:converge}
T_{\infty}\Embbar_{\partial}(\D^m,\D^n)\simeq\Embbar_{\partial}(\D^m,\D^n)
\end{equation}
as long as we assume $n-m>2$.

In fact, we have a comparison between this Goodwillie-Weiss Taylor tower $T_k\Embbar_{\partial}(\D^m,\D^n)$
and a tower of mapping spaces $\Map_{\leq k}^h(\lD_m,\lD_n)$, $k\geq 0$,
which we associated to our operadic mapping space $\Map^h(\lD_m,\lD_n)$.
In short, we set $\Map_{\leq k}^h(\lD_m,\lD_n) = \Map^h(\lD_m|_{\leq k},\lD_n|_{\leq k})$, where ${\op P}|_{\leq k} = \lD_m|_{\leq k},\lD_n|_{\leq k}$
denote truncated operads which we merely obtain by forgetting the components ${\op P}(r)$ of arity $r>K$
in our object ${\op P} = \lD_m,\lD_n$.
Then we have the following comparison statement:

\begin{thm}[see \cite{DwyerHess0,Turchin5,DwyerHess,WBdB2,Weiss2,DucT}]\label{thm:tower delooping}
For any $n\geq m\geq 1$, we have a homotopy equivalence of towers:
\begin{equation}\label{eq:towers}
\begin{tikzcd}
*\ar{d}{\simeq} &
\Omega^{m+1}\Map_{\leq 1}^h(\lD_m,\lD_n)\ar{d}{\simeq}\ar{l} &
\cdots\ar{l} &
\Omega^{m+1}\Map_{\leq k}^h(\lD_m,\lD_n)\ar{d}{\simeq}\ar{l} &
\cdots\ar{l}\\
T_0\Embbar_{\partial}(\D^m,\D^n) &
T_1\Embbar_{\partial}(\D^m,\D^n)\ar{l} &
\cdots\ar{l} &
T_k\Embbar_{\partial}(\D^m,\D^n)\ar{l} &
\cdots\ar{l}
\end{tikzcd},
\end{equation}
where $\Map_{\leq k}^h(-,-)$ denotes the (derived) mapping space associated to the $k$-truncated operads ${\op P}|_{\leq k} = \lD_m|_{\leq k},\lD_n|_{\leq k}$,
and we have the weak homotopy equivalence:
\begin{equation}\label{eq:delooping}
T_{\infty}\Embbar_{\partial}(\D^m,\D^n)\simeq\Omega^{m+1}\Map^h(\lD_m,\lD_n)
\end{equation}
when we pass to the limit.
\end{thm}

This comparison statement was first established in the case $m=1$, at the level of the limit \eqref{eq:delooping} in~\cite{DwyerHess0},
and at the level of the towers in~\cite{Turchin5}.
The generalization to the case of arbitrary $m\geq 1$ is given in \cite{DwyerHess,WBdB2,Weiss2,DucT} (in \cite{WBdB2,Weiss2,DucT} for the case of the towers).
The proofs of the references~\cite{DwyerHess0,Turchin5,DwyerHess,DucT} go through a description of the (limit of the) tower in terms of mapping spaces
of (truncated) infinitesimal bimodules obtained in~\cite{Sinha,Turchin2,Turchin4}.
This approach can be adapted to other similar situations (for instance, to check rational analogues of these comparison statements,
such as the ones which we formulate soon).
The approach of the article~\cite{WBdB2} is more focused on the case depicted in the above theorem,
but implies that the equivalence of this statement is compatible with a natural action
of the $(m+1)$-discs operad on our objects.

The Goodwillie-Weiss calculus can also be applied to the presheaf
\[
\Embbar_{\partial}(-,\D^n)^{\Q}: \mO_{\partial}(\D^m)\to\TopCat
\]
which we get by taking the rationalization of the spaces $\Embbar_{\partial}(U,\D^n)$, for $U\in\mO_{\partial}(\D^m)$.
The results of~\cite{DwyerHess0,Turchin5,DwyerHess,DucT} imply that we have a weak homotopy equivalence
\begin{equation}\label{eq:rational delooping}
T_{\infty}\Embbar_{\partial}(\D^m,\D^n)^{\Q}\simeq\Omega^{m+1}\Map^h(\lD_m,\lD_n^{\Q}),
\end{equation}
for any $m\geq 1$ (with the case $m=1$ addressed in~\cite{DwyerHess0,Turchin5}, the case $m\geq 1$ covered by \cite{DwyerHess,DucT}),
and  we have a counterpart of this homotopy equivalence
at the tower level~\cite{DucT}:
\begin{equation}\label{eq:rational towers}
\begin{tikzcd}
*\ar{d}{\simeq} &
\Omega^{m+1}\Map_{\leq 1}^h(\lD_m,\lD_n^{\Q})\ar{d}{\simeq}\ar{l} &
\cdots\ar{l} &
\Omega^{m+1}\Map_{\leq k}^h(\lD_m,\lD_n^{\Q})\ar{d}{\simeq}\ar{l} &
\cdots\ar{l}\\
T_0\Embbar_{\partial}(\D^m,\D^n)^{\Q} &
T_1\Embbar_{\partial}(\D^m,\D^n)^{\Q}\ar{l} &
\cdots\ar{l} &
T_k\Embbar_{\partial}(\D^m,\D^n)^{\Q}\ar{l} &
\cdots\ar{l}
\end{tikzcd}.
\end{equation}

We establish the following result to compare the mapping spaces $\Map_{\leq k}(\lD_m,\lD_n)$ related to the Goodwillie-Weiss Taylor tower for $\Embbar_{\partial}(\D^m,\D^n)$
to the mapping spaces $\Map_{\leq k}(\lD_m,\lD_n^{\Q})$ which we can compute by our methods:

\begin{thm}\label{thm:rational mapping space towers}
The spaces $\Map^h(\lD_m,\lD_n)$ and $\Map_{\leq k}^h(\lD_m,\lD_n)$ are $(n-m-1)$-connected (and, in particular, simply connected)
like the spaces $\Map^h(\lD_m,\lD_n^{\Q})$ and $\Map_{\leq k}^h(\lD_m,\lD_n^{\Q})$, for any $n,m\geq 1$,
as soon as $n-m\geq 2$.
Moreover, the map
\[
\Map^h(\lD_m,\lD_n)\to\Map^h(\lD_m,\lD_n^{\Q}),
\]
induced by the rationalization $\lD_n\to\lD_n^{\Q}$, is a rational homotopy equivalence when $n-m\geq 3$,
whereas the maps
\[
\Map_{\leq k}^h(\lD_m,\lD_n)\to\Map_{\leq k}^h(\lD_m,\lD_n^{\Q})
\]
are rational homotopy equivalences as soon as $n-m\geq 2$.
\end{thm}

We then get as a corollary of the results of Theorem~\ref{thm:nerve to mapping spaces} and Theorem~\ref{thm:main algebraic statement}:

\begin{thm}\label{thm:nerve to rational mapping spaces}
The rational homotopy type of the space $\Map^h(\lD_m,\lD_n)$ is described by the $L_{\infty}$~algebra of homotopy biderivations
of the morphism of dg Hopf $\La$-cooperads $\e_m^c\stackrel{*}{\lo}\e_n^c$
when $n-m\geq 3$.
Furthermore, the hairy graph complex $\HGC_{m,n}$, which we equip with its standard dg Lie algebra structure in the case $m\geq 2$,
with the Shoikhet $L_{\infty}$~structure in the case $m=1$,
defines an explicit model of this $L_{\infty}$~algebra.
\end{thm}

We can also establish the following analogue of the first claim of this theorem for the truncated mapping spaces $\Map_{\leq k}^h(\lD_m,\lD_n)$
which appear in our tower:

\begin{thm}\label{thm:nerve to truncated mapping spaces1}
The rational homotopy type of the spaces $\Map_{\leq k}^h(\lD_m,\lD_n)$ is described by the $L_{\infty}$~algebras of homotopy biderivations
of the morphism of truncated dg Hopf $\La$-cooperads $\e_m^c|_{\leq k}\stackrel{*}{\lo}\e_n^c|_{\leq k}$
as soon as $n-m\geq 2$.
\end{thm}

\begin{rem}\label{rem:nerve towers}
We can not obtain an explicit $L_{\infty}$~algebra modeling the homotopy invariant biderivations of $\e_m|_{\leq k}\stackrel{*}{\lo}\e_n|_{\leq k}$
by any kind of truncation from the hairy graph complex $\HGC_{m,n}$.
We just take the biderivation dg Lie algebras $\BiDer_{\dg\La}(\check{\E}_n^c|_{\leq k},\hat{\E}_m^c|_{\leq k})$
associated to the $k$-truncations $\check{\E}_n^c|_{\leq k}$ and $\hat{\E}_m^c|_{\leq k}$
of the cofibrant and fibrant replacements
of our dg Hopf cooperads $\E_n^c = \e_n^c$ and $\E_m^c$
to get such a model.
We can still produce a small complex quasi-isomorphic to this dg Lie algebra $\BiDer_{\dg\La}(\check{\E}_n^c|_{\leq k},\hat{\E}_m^c|_{\leq k})$ (see Theorem~\ref{thm:truncated deformation complex}).
In the case $m=1$, we just retrieve the $k$-th (normalized) partial totalization $(\prod_{i=0}^k\bS{\alg p}_n(i),\partial)$
of the cosimplicial chain complex ${\mathfrak p}_n(-)$
defined in~\cite{ScannellSinha}
(up to a shift in degrees due to delooping).
We accordingly have a small complex that can be used to compute the rational homotopy groups of the spaces $T_k\Embbar_{\partial}(\D^m,\D^n)$
and $\Map_{\leq k}^h(\lD_m,\lD_n)$ when $n-m\geq 2$ (see Corollary~\ref{cor:embedding spaces tower}),
but we do not have any obvious way to define the $L_{\infty}$~structure
associated to this small complex.
\end{rem}

\begin{rem}\label{rem:previous results}
In fact, it was proved in the article~\cite{Turchin3} that the hairy graph complex $\HGC_{m,n}$
computes the rational homotopy of the space $\Embbar_{\partial}(\D^m,\D^n)$
when $n\geq 2m+2$,
and that this statement is actually equivalent to the rational collapse at the second term of the Goodwillie-Weiss homotopy spectral sequence.
The results of Theorem~\ref{thm:tower delooping} and Theorem~\ref{thm:nerve to rational mapping spaces} imply that these results hold in the range $n\geq m+3$,
and hence, in the whole range of convergence of the Goodwillie-Weiss Taylor tower for the spaces $\Embbar_{\partial}(\D^m,\D^n)$.
\end{rem}

After the delloping, the hairy graph cocycles with $k$ hairs and of loop order $g$ contribute to the
rational homotopy of $\Embbar_{\partial}(\D^m,\D^n)$, $n-m>2$, only in degrees $\geq k(n-m-2)
+(g-1)(n-3)$. Therefore  only loop order zero cocycles~\eqref{eq:HGC0loop} can contribute
to $\Q\otimes \pi_0 \Embbar_{\partial}(\D^m,\D^n)$, $n-m>2$.

\begin{cor}\label{cor:pi0_embbar}
For $n-m>2$, $\pi_0\Embbar_{\partial}(\D^m,\D^n)$ is a finitely generated abelelian group of rank $\leq 1$. It is infinite if either $m=2k+1$, $n=4k+3$, $k\geq 1$, or $m=4k-1$, $n=6k$, $k\geq 1$. \footnote{This result should not be difficult to obtain using Haefliger\rq{}s approach~\cite{Haefliger2}.
For a similar statement about
$\pi_0\Emb_{\partial}(\D^m,\D^n)=\pi_0\Emb(S^m,S^n)$, $n-m>2$, see~\cite[Corollary~6.7]{Haefliger2}.}
\end{cor}

This corollary follows from the explicit description of the loop order~0 hairy graph-homology~\eqref{eq:HGC0loop}. The case $m=2k+1$, $n=4k+3$, corresponds to the line graph
$L = \begin{tikzpicture}[scale=.5, baseline=-.65ex]
\draw (0,0)--(1,0);
\end{tikzpicture}$. Geometrically it appears as the image of the $SO(n-m)$ Euler class
under the map $\Omega^{m+1} \Inj(\R^m,\R^n)\to \Embbar_{\partial}(\D^m,\D^n)$.\footnote{To
recall $\Inj(\R^m,\R^n)\simeq V_m(\R^n)=SO(n)/SO(n-m)$.} The second case $m=4k-1$, $n=6k$,
corresponds to the tripod $Y = \begin{tikzpicture}[scale=.3, baseline=-.65ex]
\node[int] (v) at (0,0){};
\draw (v) -- +(90:1) (v) -- ++(210:1) (v) -- ++(-30:1);
\end{tikzpicture}$, which geometrically is the Haefliger invariant~\cite{Haefliger1}.

\begin{rem}\label{rem:cerf lemma}
Our results  also provide interesting information about the embedding calculus for codimension~$\leq 1$.
The results of Corollary~\ref{cor:codimension zero case} and Corollary~\ref{cor:codimension one case} imply that the inclusion $i: \lD_{n-1}\to\lD_n$
induces a weak-equivalence
\begin{equation}\label{eq:algebraic cerf lemma}
\Map^h(\lD_n,\lD_n^{\Q})_{id}\stackrel{\simeq}{\to}\Map^h(\lD_{n-1},\lD_n^{\Q})_{i},
\end{equation}
where we consider the connected components of our mapping spaces associated to the obvious elongations of the identity map ${id}: \lD_n\to\lD_n$
and of this inclusion $i: \lD_{n-1}\to\lD_n$ to the rationalization $\lD_n^{\Q}$.
We call this result the \emph{Algebraic Cerf Lemma}.

The original result of Cerf asserts that the natural scanning map
\begin{equation}\label{eq:cerf}
\Diff_{\partial}(\D^n)\stackrel{\simeq}{\lo}\Omega\Emb_{\partial}(\D^{n-1},\D^n)
\end{equation}
is a weak homotopy equivalence (see \cite[Appendix, Section~5, Proposition~5]{Cerf}, \cite[Proposition~5.3]{Budney}).
By the above equivalence~\eqref{eq:algebraic cerf lemma} and the result of Theorem~\ref{thm:rational mapping space towers},
a similar statement holds for the limit of the Goodwillie-Weiss tower, at least rationally:
\begin{equation}\label{eq:goodwillie-weiss cerf lemma}
T_{\infty}\Embbar_{\partial}(\D^n,\D^n)^{\Q}\stackrel{\simeq}{\lo}\Omega T_{\infty}\Embbar_{\partial}(\D^{n-1},\D^n)^{\Q}.
\end{equation}
Thus, comparing~\eqref{eq:cerf} with~\eqref{eq:goodwillie-weiss cerf lemma},
we get that the embedding calculus still reflects the codimension one versus codimension zero rigidity of embeddings
though the Goodwillie-Weiss Taylor towers do not converge in codimensions $n-m\leq 1$.\footnote{Let us mention that a previous result
of the second and third authors \cite[Theorem~2]{TW}
implies that a similar connection holds at the chain complex level,
when we take the limit of the Goodwillie-Weiss Taylor tower associated to the functors $C_*(\Embbar_{\partial}(-,\D^n),\Q): U\mapsto C_*(\Embbar_{\partial}(U,\D^n),\Q)$
for $U\in\mO_{\partial}(\lD^{n-1})$ (respectively, $U\in\mO_{\partial}(\lD^{n})$).
The homology of the complexes $T_{\infty}C_*(\Embbar_{\partial}(\D^n,\D^n),\Q)$
and $T_{\infty}C_*(\Embbar_{\partial}(\D^{n-1},\D^n),\Q)$
form graded symmetric algebras
on graded vector spaces
that are isomorphic up to a shift of degree (see~\cite[Section~9]{TW}).}
\end{rem}

\subsection*{Organization of the paper}
This paper is organized as follows. In a preliminary step, we give a short account of our general conventions
on chain complexes, operads and cooperads. Then we divide our study in four main parts:
\begin{itemize}
\item
In Part~\ref{part:biderivation complexes}, we explain the definition of $L_{\infty}$~algebra structures on the complexes of biderivations
associated to dg Hopf cooperads and we prove that the solutions of the Maurer--Cartan equation
in these $L_{\infty}$~algebras are equivalent
to morphisms of dg Hopf cooperads.
\item
In Part~\ref{part:mapping spaces}, we examine the definition of mapping spaces on the category of dg Hopf cooperads.
We mainly prove that these mapping spaces of dg Hopf cooperads are weakly equivalent, as simplicial sets,
to the nerve of the $L_{\infty}$~algebras of biderivations
of dg Hopf cooperads.
We apply this result to $E_n$-operads in order to get the claim of Theorem~\ref{thm:nerve to mapping spaces},
the equivalence between our mapping spaces of $E_n$-operads
and the nerve of the $L_{\infty}$~algebras of biderivations
associated to our dg Hopf cooperad models
of $E_n$-operads.
\item
In Part~\ref{part:graph complexes}, we check that the $L_{\infty}$~algebras of biderivations associated to $E_n$-operads
are equivalent to the Lie (or $L_{\infty}$) algebras of hairy graphs.
Then we can use the result of the previous part in order to establish the claims of Theorem \ref{thm:main algebraic statement}
and of Theorem \ref{thm:main topological statement},
the equivalence between the mapping spaces of $E_n$-operads
and the nerve of the Lie (respectively, $L_{\infty}$) algebras of hairy graphs.
To complete this result, we explain the proof of the other statements on the mapping spaces of $E_n$-operads
stated in the first part of this introduction.
\item
In Part~\ref{part:rationalization}, we prove the claims of Theorem~\ref{thm:rational mapping space towers}, Theorem~\ref{thm:nerve to rational mapping spaces}, Theorem~\ref{thm:nerve to truncated mapping spaces1},
which we use in the applications of our work
to the study of embedding spaces.
\end{itemize}

\subsection*{Acknowledgement}
The second author is grateful to Paolo Salvatore for discussions in particular for asking the question about the rational homotopy type of the deloopings of spaces of long embeddings.

\setcounter{part}{-1}
\renewcommand{\thethm}{\thesection.\arabic{thm}}

\part{Background}\label{part:background}
We give a brief summary of our conventions on operads and cooperads in this preliminary part.
We mainly work in the category of dg vector spaces in order to form our model
for the rational homotopy of operads.
Therefore, we explain our conventions on dg vector spaces in the next paragraph.
We review the definition of an operad and of a cooperad afterwards.

\subsection{Conventions on vector spaces, graded vector spaces and dg vector spaces}\label{subsec:dg vector spaces}
We take a field of \emph{characteristic zero} $\K$ as a ground ring all through this article and we use the notation $\Vect$
for the category of vector spaces over $\K$.
We also use the notation $\dg\Vect$ for the category of differential graded vector spaces (the category of dg vector spaces for short),
where a dg vector space consists, for us, of a vector space $V$
equipped with a lower $\Z$-grading $V = V_*$
together with a differential,
usually denoted by $d: V\rightarrow V$,
which lowers degrees by one $d: V_*\rightarrow V_{*-1}$.

Note that our dg vector spaces are equivalent to unbounded chain complexes, but we prefer to use the terminology of dg vector space
in general. We just use the phrase ``chain complex'' in the case of specific examples
or to refer to specific constructions of objects
in the category of dg vector spaces (for instance, we may use the phrase ``complex of biderivations''
for the ``dg vector space of biderivations'' associated to certain morphisms
of dg Hopf cooperads).
Note also that, in some of our bibliographical references, authors assume that a dg vector space is an object,
equivalent to an unbounded cochain complex,
which is equipped with an upper grading (rather than a lower grading) and with a differential
that increases degrees (instead of decreasing degrees).
In what follows, we use the standard rule $V_n = V^{-n}$ to identify an upper grading with a lower grading
and to identify an object of this category of upper graded dg vector spaces
with a dg vector space in our sense.
In general, we prefer to deal with lower gradings. Therefore, we tacitely adopt this convention to consider lower gradings
when there is no indication of the contrary in the context.
But, this rule $V_n = V^{-n}$ enables us to use lower and upper grading interchangeably, and we can therefore apply the rule the other way around,
in order to convert a lower grading into an upper grading,
when upper gradings are more natural.

To be specific, we use upper gradings when we deal with upper graded dg vector spaces that are concentrated in non-negative degrees.
We adopt the notation $\dg^*\Vect$ for this category of non-negatively upper graded dg vector spaces.
We also adopt the convention to use the expression ``cochain dg vector space''
in order to refer to an object
of this category $V\in dg^*\Vect$.
We adopt similar notations and conventions when we deal with categories of algebras in dg vector spaces.
We can identify this category of cochain dg vector spaces $\dg^*\Vect$
with the full subcategory of our category of lower graded dg vector spaces $\dg\Vect$
generated by the objects $V$ which satisfy $V_n = V^{-n} = 0$ for $n>0$
according to our rule.

In what follows, we also consider the category of graded vector spaces which we can identify with the full subcategory
of the category of dg vector spaces generated by the dg vector spaces
equipped with a trivial differential $d = 0$.
The other way around, we have the obvious forgetful functor from the category of dg vector spaces
to the category of graded vector spaces.
We denote the degree of a homogeneous element $v$ of a graded vector space by $|v|\in\Z$.
Note that we will (ab)use the same notation $|S|$ for the number of elements in a finite set.
However, in practice this will not cause confusion.

In our constructions, we use the standard symmetric monoidal structure of the category of dg vector spaces,
with the tensor product $\otimes: dg\Vect\times dg\Vect\rightarrow dg\Vect$
inherited from the category of vector spaces over $\K$
and with the symmetry operator $c: V\otimes W\rightarrow W\otimes V$
that reflects the (Koszul) sign rule of differential graded algebra.
Recall that the category of dg vector spaces is also enriched over itself with hom-objects $\Hom_{\dg}(V,W)$
defined, for every $V,W\in dg\Vect$, by the vector spaces
of linear maps $f: V\rightarrow W$ satisfying $f(V_*)\subset W_{*+|f|}$
for some degree $|f|\in\Z$.
This dg vector space is equipped with the differential such that $d f = d_W f - \pm f d_V$,
where $d_V$ (respectively, $d_W$) denotes the internal differential
of the dg vector space $V$ (respectively, $W$)
and $\pm = (-1)^{|f|}$ is the sign determined by the sign rule of differential graded algebra.
Note that the tensor product $\otimes: dg\Vect\times dg\Vect\rightarrow dg\Vect$
preserves the category of cochain dg vector spaces $\dg^*\Vect$
inside $\dg\Vect$, which inherits a symmetric monoidal structure as well.
But the dg vector spaces $\Hom_{\dg}(V,W)$
may have components in all degrees
even when we assume $V,W\in dg^*\Vect$.
In what follows, we also consider the tensor product operation
on graded vector spaces, and an internal hom-bifunctor with values in this category $\Hom_{gr}(-,-)$
which we define by forgetting about the differentials
in our constructions.



Recall that we use the notation $V[m]$ for the $m$-fold suspension $V[m]_n = V_{n-m}$
of a graded vector space $V$, for any $m\in\Z$.
In the case $V = \K$, where we identify the ground field $\K$ with a graded vector space of rank one
concentrated in degree $0$,
we get that $\K[m]$ represents the graded vector space of rank one
concentrated in lower degree $m$,
as in \cite{Fr}.
Note that we get the rule $V[m]^n = V^{n+m}$ when we use upper gradings
instead of lower gradings
and this convention is also used in the reference~\cite{Will}.
If $V\in dg\Vect$, then we have $V[m]\in dg\Vect$.
To make this definition more precise, we can use the obvious identity $V[m] = V\otimes\K[m]$,
where we regard $\K[m]$ as a dg vector space equipped with a trivial differential,
and we consider the tensor product of our symmetric monoidal structure
on the category of dg vector spaces.

We also use the notation $V^*$ for the obvious extension of the duality functor of vector spaces to the category of dg vector spaces.
We actually have $V^* = \Hom_{\dg}(V,\K)$, for any object $V\in dg\Vect$, where we use that the ground field $\K$
is identified with a dg vector space of rank one concentrated in degree $0$,
and we consider our internal hom-bifunctor on the category of dg vector spaces.
If each graded component of a graded (respectively, dg) vector space $V$ is finite dimensional, then we say that $V$ is of finite type.

\subsection{Operads}\label{subsec:operads}
We refer to the textbook of Loday and Vallette \cite{LV} or to the upcoming book \cite{Fr} of the first author
for an introduction to the theory of operads.
The operads of little discs are defined in the category of topological spaces,
but we mainly deal with operads (and cooperads)
defined in the category of dg vector spaces (or in the subcategory of cochain dg vector spaces)
in what follows.
Therefore, we briefly review the definition of an operad (and of a cooperad) in the category of dg vector spaces
in order to make our conventions explicit
and in order to fix notation.
In what follows, we also use the expression ``dg operad'' when we consider operads
defined in the base category of dg vector spaces.

We have several equivalent definitions of the notion of an operad. Let $\N$ denote the set of non-negative
integers.
We always consider operads equipped with a symmetric structure in this article.
Thus, we generally assume that the components of a dg operad ${\op O}$
are dg vector spaces ${\op O}(r)$
equipped with an action of the symmetric groups $\Sigma_r$
for $r\in\N$.
The composition structure of our operad is determined by composition products $\circ_i: {\op O}(m)\otimes{\op O}(n)\rightarrow{\op O}(m+n-1)$,
which are defined for all $m,n\in\N$, for any $i\in\{1,\dots,m\}$,
and which satisfy natural equivariance, unit and associativity relations.
The unit of our operad is defined by a morphism $\eta: \K\to{\op O}(1)$ or, equivalently, by a distinguished element
which we denote by $1\in{\op O}(1)$.

The symmetric collection underlying an operad is also equivalent to a collection ${\op O}(S)$
indexed by the finite sets $S$
such that the mapping ${\op O}: S\mapsto{\op O}(S)$
defines a functor on the category formed by the finite sets as objects
together with the bijections of finite sets
as morphisms.
Indeed, we can regard a symmetric collection ${\op O}(r)$, $r\in\N$, as the restriction
of such a functor ${\op O}: S\mapsto{\op O}(S)$
to the category generated by the finite ordinals $\underline{r} = \{1<\dots<r\}$
inside the category of finite sets
and bijections. To be more explicit, we just set ${\op O}(r) = {\op O}(\{1<\dots<r\})$ to get our correspondence
between the structure of an operad with components indexed by the natural numbers
and the structure of an operad with components
indexed by finite sets.
In the context where we consider components indexed by arbitrary finite sets,
the partial composition products of an operad are equivalent to operations
of the form
\[
\circ_* : {\op O}(S')\otimes{\op O}(S'')\to{\op O}((S'\setminus\{*\})\sqcup S''),
\]
defined for any pair of finite sets $S',S''$, and where $*$
is a formal composition mark.
Informally, we regard the elements of the object ${\op O}(S)$ as operations with $|S|$ inputs
labelled by elements of the set $S$,
where we use the notation $|S|$ for the cardinal of any finite set $S$.
In the context of operads, we also use the classical terminology of ``arity'' to refer to this number $|S|$.

Recall also that the partial composition products of an operad can be used to determine
treewise composition operations
\[
\nabla_T: \bigotimes_{v\in VT} {\op O}(star(v))\to{\op O}(S),
\]
which we form by taking a tensor product of components of our operad over the vertex set $VT$ of a rooted tree $T$
whose leafs are indexed by $S$. We then use the notation $star(v)$
for the set of ingoing edges of any vertex $v$
in such a tree $T$.
In what follows, we also use the notation $\mT_S$ for the set of rooted trees with leafs indexed by $S$
which we use in this construction.
Intuitively, the structure of these trees $T\in\mT_S$ materialize general composition schemes
in the operad ${\op O}$.

The unit morphism $\eta: \K\to{\op O}(1)$ is equivalent to a morphism of operads $\eta: {\op I}\to{\op O}$
where $\op I$ denotes the collection such that
\[
\op I(r) = \begin{cases} \K, & \text{if $r=1$}, \\
0, & \text{otherwise},
\end{cases}
\]
which inherits an obvious operad structure.
This collection $\op I$ accordingly represents the initial object of the category of operads.
(In what follows, we also use the name ``unit operad'' for this operad $\op I$.)
We say that an operad ${\op O}$ is augmented (over this initial object) when we have an operad morphism $\epsilon: {\op O}\to{\op I}$
in the converse direction as the operadic unit $\eta: {\op I}\to{\op O}$.
We necessarily have $\epsilon\eta = \mathit{id}$.
We can also determine such a morphism by giving a morphism of dg vector spaces $\epsilon: {\op O}(1)\to\K$
satisfying $\epsilon(1) = 1$ when we consider the unit element of our operad $1 = \eta(1)\in\op O(1)$
and $\epsilon(p\circ_1 q) = \epsilon(p)\epsilon(q)$, for any composite of operations
of arity one $p,q\in\op O(1)$.
We then define the augmentation ideal of our operad $\op O$ by $\overline{\op O}(1) = \ker(\epsilon: {\op O}(1)\to\K)$
and ${\overline{\op O}}(r) = {\op O}(r)$ for $r\not=1$.
We mainly use augmented operads in the bar duality of operads. We also assume $\op O(0) = 0$
when we use this construction.

\subsection{Cooperads}\label{subsec:cooperads}
We dualize the definitions of the previous subsection in the context of cooperads. We still mainly work in the category of dg vector spaces when we deal with cooperads
and we also use the expression ``dg cooperad'' in this context.

In short, we define a dg cooperad $\op C$ as a symmetric collection of dg vector spaces $\op C(r)$, $r\in\N$,
together with coproducts $\Delta_i: \op C(m+n-1)\rightarrow\op C(m)\otimes\op C(n)$,
which are defined for all $m,n\in\N$, for any $i\in\{1,\dots,m\}$,
and which satisfy some natural equivariance, counit and coassociativity relations.
We can equivalently consider coproducts $\Delta_*: \op C(S)\rightarrow\op C(S')\otimes\op C(S'')$,
defined for all partitions $S = (S'\setminus\{*\})\sqcup S''$
of a finite set $S$,
when we use that $\op C$ is given by a collection $\op C(S)$
indexed by arbitrary finite sets $S$.
The counit of our cooperad is determined by a morphism $\epsilon: \op C(1)\rightarrow\K$.
In what follows, we also consider general coproduct operations
\[
\Delta_T : \op C(S)\to\bigotimes_{v\in VT}\op C(star(v)),
\]
shaped on rooted trees $T\in\mT_S$.

Recall that the dual dg vector spaces ${\op O}(r)^*$ of the components ${\op O}(r)$ of an operad ${\op O}$
form a cooperad as soon as each of these components ${\op O}(r)$
forms a dg vector space of finite type. We denote this cooperad by~${\op O}^c$.

We often need some conilpotence assumptions when we deal with cooperads. (For instance, we use cofree objects
of the category of conilpotent cooperads
in the expression of the bar duality of operads.)
We make such a conilpotence condition explicit in order to complete the account of this paragraph.
We go back to this subject later on, when we explain the definition
of the category of the category of $\La$-cooperads
which we use in the subsequent constructions
of this article.

In all cases, we restrict our attention to cooperads such that $\op C(0) = 0$.
We also need to assume that our cooperads $\op C$
are equipped with a coaugmentation
in order to formalize our conilpotence condition.
To be explicit, let us observe that the counit morphism of our cooperad $\epsilon: \op C(1)\to\K$
is equivalent to a morphism of cooperads $\epsilon: \op C\to\op I$,
where we again consider the collection $\op I$ such that $\op I(1) = \K$ and $\op I(r) = 0$ for $r\not=1$,
which inherits a natural cooperad structure (in addition to an operad structure).
In what follows, we also use the name ``unit cooperad'' when we regard this collection $\op I$
as an object of the category of cooperads.
Then we say that a cooperad $\op C$ is coaugmented (over the unit cooperad $\op I$)
precisely when the counit morphism $\epsilon: \op C\to\op I$
admits a section $\eta: \op I\to\op C$
in the category of cooperads.
We equivalently assume that we have a unit element $1 = \eta(1)\in\op C(1)$
which satisfies the identity $\Delta_1(1) = 1\otimes 1$
in our cooperad $\op C$ (and which forms a cycle of degree zero in the context of dg vector spaces).
We then have a natural splitting formula $\op C(1) = \K 1\oplus\overline{\op C}(1)$, where we set $\overline{\op C}(1) = \ker(\epsilon: \op C(1)\to\K)$.
We can also consider the collection $\overline{\op C}$ such that $\overline{\op C}(1) = \ker(\epsilon: \op C(1)\to\K)$
and $\overline{\op C}(r) = \op C(r)$
for $r>1$.
We just use the splitting $\op C(1) = \K 1\oplus\overline{\op C}(1)$ to identify this collection $\overline{\op C}$
with the cokernel of the coaugmentation morphism of our cooperad $\eta: \op I\to\op C$. (We also say that $\overline{\op C}$
represents the coaugmentation coideal of $\op C$ in this context.)

We now consider the reduced treewise coproducts
\[
\overline{\Delta}_T : {\op C}(S)\to\bigotimes_{v\in VT}\overline{\op C}(star(v)),
\]
which we obtain by taking the composite of the above treewise coproducts with our projection ${\op C}(star(v))\rightarrow{\overline C}(star(v))$,
for all $v\in VT$.
The counit relation of cooperads implies that the full treewise coproducts $\Delta_T$
are determined by these reduced treewise coproducts.
Furthermore, the assumption that $\eta: \op I\to\op C$ is a morphism of cooperads implies that the reduced treewise coproducts
vanish over the unit element $1\in{\op C}(S)$
when $|S| = 1$.
We then say that our coaugmented cooperad $\op C$ is conilpotent if, for any element $c\in{\op C}$, we have $\overline{\Delta}_T(c) = 0$
for all but a finite number of trees $T$.

We mainly deal with reduced treewise composition coproducts rather than with the full treewise composition coproducts in what follows.
In passing, let us observe that the condition $\op C(0) = 0$
implies that the treewise composition coproducts $\Delta_T : \op C(S)\to\bigotimes_{v\in VT}\op C(star(v))$
are trivial
when $T$ contains vertices with no ingoing edges.
Thus, we restrict ourselves to the subcategory of trees $T$ which satisfy the condition $|star(v)|\geq 1$ for all $v\in VT$
when we consider the reduced treewise composition coproducts
associated to a coaugmented cooperad.

Let us also observe that if the cooperad $\op C$ satisfies the relation $\op C(1) = \K$ in addition to this condition $\op C(0) = 0$
(we then say that the cooperad ${\op C}$ is reduced),
then we have $\overline{\op C}(1) = 0$,
and hence, we only have to consider reduced treewise coproducts over trees
whose vertices have at least two ingoing edges
each.
The number of trees $T$ which satisfy this condition $|star(v)|\geq 2$ for all $v\in VT$
and have a prescribed set of leafs $S$
is finite.
From this observation, we conclude that any cooperad $\op C$
which satisfies the relation $\op C(1) = \K$ (in addition to the condition $\op C(1) = 0$)
is automatically conilpotent in our sense.

\subsection{On $\La$-operads}\label{subsec:Lambda-operads}
The notion of a $\La$-operad has been introduced by the first author in~\cite[Chapter I.2]{Fr}.
This category of $\La$-operads is actually isomorphic to the full subcategory of the category of operads generated by the objects ${\op O}$
whose component of arity zero is reduced to the ground field ${\op O}(0) = \K$ (when we take the category of vector spaces
or the category of dg vector spaces
as a base category).
The idea of the definition of a $\La$-operad is that we can forget about the component of arity zero in this case ${\op O}(0) = \K$
and consider, instead, that our operads are equipped with restriction operators which reflect the composition operations
with this component ${\op O}(0) = \K$ in the definition
of the operad $\op O$.

The notation $\Lambda$ the definition of the notion of a $\La$-operad refers to the category formed by the finite ordinals $\underline{r}=\{1<\cdots<r\}$
as objects together with all injective maps on these ordinals (possibly non-monotonous)
as morphisms.
The main observation of~\cite{Fr} is that the composition operations with a distinguished arity zero element in an operad ${\op O}$
such that ${\op O}(0) = \K$ are equivalent to a contravariant action of this category $\La$
on the collection of dg vector spaces
underlying our object.
The restriction operator $u^*: {\op O}(s)\rightarrow{\op O}(r)$,
which we associate to any injective map $u: \{1<\dots<r\}\rightarrow\{1<\dots<s\}$
in this action of the category $\La$,
corresponds to the composition products with the distinguished arity zero element of our operad
at the inputs such that $j\not\in\{u(1),\dots,u(r)\}$
together with the re-indexing $i = u^{-1}(j)$ of the remaining inputs $j\in\{u(1),\dots,u(r)\}$.
In the context where we consider operads with components indexed by arbitrary finite sets,
we may also consider restriction operators of the form
\[
\epsilon_S: {\op O}(S\sqcup\{*\})\rightarrow{\op O}(S)
\]
which represent the composites with our distinguished arity zero element
at the composition mark $*$
in the component ${\op O}(S\sqcup\{*\})$
of our operad ${\op O}$.
In addition to these restriction operators, we have augmentation morphisms $\epsilon: {\op O}(r)\rightarrow\K$,
which correspond to full composition products with an arity zero factor $\op O(0) = \K$
at all positions in our operad.

In order to define the $\La$-operad equivalent to an operad ${\op O}$ with ${\op O}(0) = \K$,
we just forget about this component of arity zero,
we take a zero object as a component of arity zero instead,
and we assume that our operad is equipped with restriction operators $u^*: {\op O}(s)\rightarrow{\op O}(r)$,
which give a contravariant action of the category of $\La$
on our object,
together with the augmentation morphisms $\epsilon: {\op O}(r)\rightarrow\K$
that correspond to the full composition products
with arity zero operations.
In fact, one can observe that the collection of these augmentation morphisms $\epsilon: {\op O}(r)\rightarrow\K$, $r>0$,
is equivalent to an operad morphism with values the commutative operad $\Com$,
the operad, associated to the category of commutative algebras (without unit),
which satisfies $\Com(r) = \K$ for any $r>0$.

The category of $\La$-operads is precisely defined as the category formed by these operads ${\op O}$,
satisfying ${\op O}(0) = 0$,
which are equipped with a contravariant action of the category of $\La$
as above,
together with an augmentation morphism
with values in the commutative operad $\epsilon: {\op O}\rightarrow\Com$.
Let us mention that we also assume that the composition products and the restriction operators of a $\La$-operad
satisfy some compatibility relation which mimes the associativity relations of the composition products
of an operad with respect to an arity zero operation.
In~\cite{Fr}, the more precise expression ``augmented $\La$-operad'', where the adjective ``augmented'' refers to this structure augmentation $\epsilon: {\op O}\rightarrow\Com$
which we attach to our objects, is adopted for this category of $\La$-operads. We just forget about the adjective ``augmented''
in order to simplify our terminology in what follows.

Let us mention that the commutative operad $\Com$ is equipped with an augmentation over the initial operad $\op I$,
and as a consequence, any $\La$-operad $\op O$ canonically forms an augmented operad
in the sense defined in Section~\ref{subsec:operads}
when we forget about $\La$-structures.

\subsection{On $\La$-cooperads}\label{subsec:Lambda-cooperads}
We dualize the definitions of the previous subsection in order to define the notion of a $\La$-cooperad.
Thus, a dg $\La$-cooperad consists of a cooperad $\op C$
satisfying $\op C(0) = 0$
together with a covariant action of the category $\La$, which associates a corestriction operator $u_*: \op C(r)\rightarrow\op C(s)$
to any injective map between finite ordinals $u: \{1<\dots<r\}\rightarrow\{1<\dots<s\}$,
and a coaugmentation morphism $\eta: \Com^c\rightarrow\op C$,
where $\Com^c$ is the dual cooperad in dg vector spaces
of the operad of commutative algebras (without unit).
In the context where we consider cooperads with components indexed by arbitrary finite sets,
we also consider corestriction operators
of the form
\[
\eta_S: \op C(S)\rightarrow\op C(S\sqcup\{*\})
\]
which are dual to the partial composition products with a distinguished arity zero element
in an operad.
We obviously assume that the composition coproducts and the corestriction operators satisfy some compatibility relations when we define a $\La$-cooperad.
We then see that our category of dg $\La$-cooperads is isomorphic to the full subcategory of the category of plain dg cooperads
generated by the objects $\op C$ such that $\op C(0) = \K$ (dualize the observations of the previous subsection).

In~\cite{Fr}, the more precise expression ``coaugmented dg $\La$-cooperad'',
where the adjective ``coaugmented'' refers to the coaugmentation $\eta: \Com^c\rightarrow{\op C}$
which we attach to our objects,
is used for the objects in the category of dg $\La$-cooperads. We again forget about the adjective ``coaugmented''
in order to simplify our terminology in what follows (as in the context of operads).

Let us observe that any dg $\La$-cooperad $\op C$ inherits a canonical coaugmentation
over the unit cooperad $\op I$ (when we forget about the $\La$-structure
attached to our object)
since the commutative cooperad $\Com^c$ comes itself equipped with a canonical coaugmentation morphism $\eta: {\op I}\to\Com^c$
that is right inverse to the counit morphism $\epsilon: \Com^c\to{\op I}$.
Thus, we can give a sense to the reduced treewise composition coproducts
of Section~\ref{subsec:cooperads}
in the context of dg $\La$-cooperads. We can also consider a straightforward extension
of the conilpotence condition
of Section~\ref{subsec:cooperads}
in the context of dg $\La$-cooperads. We precisely say that a dg $\La$-cooperad
is conilpotent
if the coaugmented dg cooperad underlying our object is conilpotent in the sense defined in Section~\ref{subsec:cooperads}
when we forget about $\La$-structures.
To be explicit, we get that $\op C$ is conilpotent if, for each element $c\in\op C$, all but a finite number of the reduced treewise composition coproducts $\overline{\Delta}_T(c)$
vanish, where $T$ runs over the category of trees satisfying $|star(v)|\geq 1$
for all vertices $v\in VT$.

In this article, we always consider dg $\La$-cooperads that are conilpotent in this sense.
Therefore, in what follows, we adopt the convention to add this conilpotence condition
to the definition of the objects of our category of dg $\La$-cooperads.
Let us mention that the conilpotence assumption is automatically satisfied when we assume $\op C(1) = \K$
since this is the case in the category of plain cooperads.
In~\cite{Fr}, the convention is to consider the subcategory of dg $\La$-cooperads that satisfy this condition $\op C(1) = \K$
in addition to $\op C(0) = 0$,
but this is not the case of certain cooperads that we consider in this article.
Therefore, we do not adopt the convention to assume this connectedness condition $\op C(1) = \K$
in the definition of a dg $\La$-cooperad
in general.
We simply say that a dg $\La$-cooperad ${\op C}$ is reduced when we have $\op C(1) = \K$ (as in the case of plain cooperads).

Note that we allow non trivial corestriction operators $u^*(c)\not=0$ for an infinite number of maps $u: \{1<\dots<r\}\to\{1<\dots<s\}$
in this definition of a conilpotent $\La$-cooperad, though these operators
reflect non trivial composition coproducts
in the cooperad ${\op C}_+$ satisfying ${\op C}_+(0) = \K$
which we associate to our object.
Thus, this cooperad ${\op C}_+$ is not conilpotent in the sense specified in Section~\ref{subsec:cooperads}.
In fact, this cooperad ${\op C}_+$ is not even equipped with a coaugmentation
over the unit cooperad ${\op I}$
in general.
so that the conilpotence condition of Section~\ref{subsec:cooperads}
does not even make sense for this cooperad.

For our purpose, we make explicit the compatibility between the corestriction operators
and the reduced treewise composition coproducts.
Let $T\in \mT_{S\sqcup\{*\}}$ be a tree. Let $v$ be the vertex of $T$
connected to the leaf indexed by the mark $*$.
Let $T'$ be the tree obtained by removing this leaf indexed by $*$.
If $v$ has exactly $2$ ingoing edges in $T$, then we also consider the tree $T''$
which we obtain by deleting both the vertex $v$ and the ingoing edge
indexed by $*$,
and by merging the remaining ingoing edge with the outgoing edge of $v$.
Let $star'(v)$ be the set of outgoing edges of $v$ in $T'$.
We fix an ordering of the vertices of $T$ with the vertex $v$ in the first position.
We then have the following relation, for any element $x\in\overline{\op C}(S)$:
\beq{eq:Lacompatibility}
\overline{\Delta}_T(\eta_S x) = (\eta_{star'(v)}\otimes\mathit{id}\otimes\cdots\otimes\mathit{id})(\overline{\Delta}_{T'}x)
+
\begin{cases}
b\otimes\overline{\Delta}_{T''}(x), & \text{if $v$ has two ingoing edges in $T$}, \\
0, & \text{otherwise},
\end{cases}
\eeq
where $b\in\overline{\op C}(2)$ is the image of the generator of $\Com^c(2)$ under the coaugmentation $\eta: \Com^c\rightarrow\op C$.

We mentioned in the introduction of this paper that we deal with dg Hopf cooperads (cooperads in the category of commutative dg algebras)
when we define the Sullivan model of an operad in topological spaces. We go back to this subject
later on. Let us simply observe, for the moment, that in the case of a dg Hopf $\La$-cooperad,
the coaugmentation $\eta: \Com^c\to\op C$, which we associate to our object $\op C$,
is actually given by the unit morphism $\eta: \K\rightarrow\op C(r)$
which we associate to the commutative dg algebra $\op C(r)$
in each arity $r>0$.
We therefore forget about the coaugmentation when we specify the structure of a dg Hopf $\La$-cooperad.

\subsection{On $\Sigma$-collections and $\La$-collections}\label{subsec:collections}
Recall that  $\Lambda$ denotes the category with the finite ordinals $\underline{r} = \{1<\dots<r\}$
as objects and all injective maps (not necessarily monotonous)
between such ordinals as morphisms.
In what follows, we also consider the category $\Sigma\subset\La$
with the same collection of objects
as the category $\Lambda$
but where we only consider the bijective maps as morphisms.
We equivalently have $\Sigma = \coprod_{r\in\N}\Sigma_r$, where we regard the symmetric groups $\Sigma_r$
as categories with a single object of which we take the coproduct
in the category of categories.

We now call $\Sigma$-collection in dg vector spaces the structure defined by a contravariant functor on this category $\Sigma$
and with values in the category of vector spaces ${\op M}: \Sigma\to\dg{\Vect}^{op}$.
We similarly call $\La$-collection the structure defined by a contravariant functor on the whole category of finite ordinals
and injections ${\op M}: \La\to\dg\Vect^{op}$.
We denote the category of $\Sigma$-collections in dg vector spaces by $\dg\Sigma\Seq$ and the category of $\La$-collections by $\dg\La\Seq$.
We also consider the full subcategories of these categories $\dg\Sigma\Seq_{>0}$ and $\dg\La\Seq_{>0}$
generated by the objects which vanish in arity zero. We adopt similar conventions
when we deal with other instances of base categories
than the category of dg vector spaces.
We obviously get a functor from the category of $\La$-operads in dg vector spaces to the overcategory $\dg\La\Seq_{>0}/\Com$
when we forget about the composition products of $\La$-operads.

We dually consider a category of covariant $\La$-collections in dg vector spaces, where we use the phrase `covariant $\La$-collection'
to refer to a covariant functor ${\op M}: \La\to\dg\Vect$
on the category $\La$.
We adopt the notation $\dg\La\Seq^c$ for this category, and we similarly use the notation $\dg\La\Seq^c_{>0}$
for the subcategory of covariant $\La$-collections
which vanish in arity zero.
We immediately get again that we have a functor from the category of $\La$-cooperads in dg vector spaces
to the undercategory $\Com^c/\dg\La\Seq^c_{>0}$
when we forget about the composition coproducts of $\La$-cooperads.
In what follows, we also consider the functor with values in the undercategory $\overline{\Com}{}^c/\dg\La\Seq^c_{>0}$
which we obtain by taking the coaugmentation coideal
of dg $\La$-cooperads. In our study of the rational homotopy of operads,
we moreover consider the full subcategory $\dg\La\Seq^c_{>1}$
of the category of covariant $\La$-collections $\dg\La\Seq^c$
formed by the objects such that ${\op M}(0) = {\op M}(1) = 0$,
and the associated category of coaugmented objects $\overline{\Com}{}^c/\dg\La\Seq^c_{>1}$.
We get a functor with values in $\overline{\Com}{}^c/\dg\La\Seq^c_{>1}$
when we take the restriction of the coaugmentation coideal
functor to the subcategory of reduced dg $\La$-cooperads (see Section~\ref{subsec:Lambda-cooperads}).
Note that $\Com^c$ represents the constant object such that $\Com^c(r) = \K$
in the category of covariant $\La$-collections $\dg\La\Seq^c_{>0}$,
while the covariant $\La$-collections $\overline{\Com}{}^c$
satisfies $\overline{\Com}{}^c(1) = 0$ and $\overline{\Com}{}^c(r) = 0$ for $r>1$.

In what follows, we usually omit the word ``covariant'' when we deal with this category of $\La$-collections.
In fact, we most often deal with covariant $\La$-structures in our constructions,
and otherwise, the context makes clear whether we consider a contravariant $\La$-structure
or a covariant one.

The categories of $\Sigma$-collections and $\La$-collections in dg vector spaces are enriched over the category of dg vector spaces.
To be explicit, let $\Xi = \Sigma,\Lambda$ be any of our indexing categories, and let $\op M$, $\op N$
be a pair of $\Xi$-collections in dg vector spaces.
Then we set:
\[
\Hom_{\dg\Xi}(\op M,\op N) = \int_{\underline{r}\in\Xi}\Hom_{\dg}(\op M(r),\op N(r)),
\]
where we take the end of the hom-objects of dg vector spaces $\Hom_{\dg}(\op M(r),\op N(r))$
over the category $\Xi$.
Thus, an element of homogeneous degree of $\Hom_{\dg\Xi}(\op M,\op N)$ consists of a collection of linear maps $f: \op M(r)\rightarrow\op N(r)$
such that $f(\op M(r)_*)\subset\op N(r)_{*+|f|}$ and which preserve the action
of the morphisms of the category $\Xi = \Sigma,\La$
on our objects.
In the case of the category $\Xi = \Sigma$, we may equivalently set:
\[
\Hom_{\dg\Sigma}(\op M,\op N) = \prod_{r\in\N}\Hom_{\dg\Sigma_r}(\op M(r),\op N(r)),
\]
where $\Hom_{\dg\Sigma_r}(\op M(r),\op N(r))$ denotes the dg vector space of $\Sigma_r$-equivariant linear maps $f: \op M(r)\rightarrow\op N(r)$
inside $\Hom_{\dg}(\op M(r),\op N(r))$.
In the context of $\Xi$-collections in cochain dg vector spaces, where we again set $\Xi = \Sigma,\La$,
we consider the same hom-object by using that the category of the cochain dg vector spaces $\dg^*\Vect$
is identified with the subcategory of the category of dg vector spaces
concentrated in negative degrees. Note that this hom-object $\Hom_{\dg\Xi}(\op M,\op N)$
may still have components in all degrees,
even if this is not the case of the collections $\op M$ and $\op N$.

In the sequel, we also use the notation $\Hom_{\gr\Sigma}(\op M,\op N)$ to denote the graded vector space,
underlying $\Hom_{\dg\Sigma}(\op M,\op N)$, which we define
by forgetting about the differential
in our construction,
whereas we use the notation $\Mor_{\dg\Sigma\Seq}(\op M,\op N)$ to denote the actual morphism sets
of the category $\dg\Sigma\Seq$.
We adopt similar conventions in the context of $\La$-collections.
Let us observe that we can identify the set of morphisms $\Mor_{\dg\Sigma\Seq}(\op M,\op N)$
associated to any pair of $\Sigma$-collections in dg vector spaces $\op M$ and $\op N$
with the vector space formed by the cycles of degree zero in the dg vector space $\Hom_{\dg\Sigma}(\op M,\op N)$,
and similarly in the context of $\La$-collections.


\subsection{The main examples of operads considered in the paper and the Koszul duality}\label{subsec:Koszul duality}
We already mentioned that we use the notation $\Com$ for the operad of commutative algebras (the commutative operad).
We use the notation $\Ass$ for the associative operad and the notation $\Lie$ for the Lie operad.
We consider, besides, the $n$-Poisson operad $\Poiss_n$, which is generated by a commutative product operation $\cdot\wedge\cdot\in\Poiss_n(2)$
of degree $0$, and a Lie bracket operation $[-,-]\in\Poiss_n(2)$ of degree $n-1$
which is anti-symmetric for $n$ odd, symmetric for $n$ even,
and which satisfies an obvious graded generalization
of the usual Poisson distribution relation
with respect to the product.

We also use the notation $\Lie_n$ for the graded variant of the Lie operad where the Lie bracket operation $[-,-]$ has degree $n-1$
so that we have the relation $\Lie_n\subset\Poiss_n$
when $n\geq 2$.
We obviously have $\Lie = \Lie_1$, and we still get an operad embedding $\Lie_1 = \Lie\hookrightarrow\Ass$
in this case. (We consider the standard map which carries the Lie bracket $[-,-]\in\Lie(2)$
to the symmetrization of the generating product operation of the associative operad $\Ass$.)
We have in general $\Lie_n = \Lie\{1-n\}$, where we use the notation ${\op O}\{\ell\}$ for the $\ell$th operadic suspension
of any cooperad ${\op O}$ (see~\cite{GetzJones}).
Briefly recall that this operad ${\op O}\{\ell\}$ is defined by the identity ${\op O}\{\ell\}(r) = {\op O}(r)[(1-r)\ell]$, for each arity $r\in\N$,
with the convention that the action of $\Sigma_r$ on ${\op O}\{\ell\}(r)$
is twisted by the signature
when $\ell$ is odd.

We use the bar duality of operads and the Koszul duality in our computation of operadic mapping spaces.
We adopt the notation $B({\op O})$ for the bar construction of an operad $\op O$
equipped with an augmentation over the initial operad ${\op I}$.
We dually use the notation $\Omega(\op C)$ for the cobar construction of a cooperad $\op C$
equipped with a coaugmentation over the final cooperad $\op I$.
We denote the Koszul dual cooperad of a quadratic operad ${\op O}$ by ${\op O}^\vee$
and we use the same notation for the Koszul dual operad
of a quadratic cooperad.
We have $\Ass^{\vee} = \Ass^c\{-1\}$ and $\Poiss_n^{\vee} = \Poiss_n^c\{-n\}$, for $n\geq 2$,
where $\Ass^c$ (respectively, $\Poiss_n^c$) denotes the dual cooperad
in graded vector spaces of the associative (respectively, of the $n$-Poisson) operad.
We deduce from the Koszul duality that the operads $\hoAss = \Omega(\Ass^c\{-1\})$
and $\hoPoiss_n = \Omega(\Poiss_n^c\{-n\})$
form resolutions of the associative operad $\Ass$
and of the $n$-Poisson operad $\Poiss_n$
respectively. We similarly have $\Lie^{\vee} = \Com^c\{-1\}$, where $\Com^c$
denotes the dual cooperad in graded vector spaces
of the commutative operad,
and $\Lie_n^{\vee} = \Com^c\{-n\}$, for $n\geq 1$, so that $\hoLie_n = \Omega(\Com^c\{-n\})$
defines a resolution of the operad $\Lie_n$
by the Koszul duality.

In what follows, we use the notation $\e_n$ for the homology of the little discs operad $\lD_n$
with coefficients in our ground field $\K$.
We have:
\begin{equation}\label{eq:e_n}
\e_n = H_*(\lD_n,\K) =
\begin{cases}
\Ass, & \text{for $n=1$}, \\
\Poiss_n, & \text{for $n\geq 2$}.
\end{cases}
\end{equation}
The dual cooperad $\e_n^c$ of this operad $\e_n$ also represents the cohomology cooperad of the operad of little discs
since we trivially have $\e_n^c(r) = \e_n(r)^* = H^*(\lD_n(r),\K)$
for each arity $r\in\N$.
We deduce from the relations $\Ass^{\vee} = \Ass^c\{-1\}$ and $\Poiss_n^{\vee} = \Poiss_n^c\{-n\}$
that we have the identity $\e_n^{\vee} = \e_n^c\{-n\}$,
for every $n\geq 1$.
We dually have $(\e_n^c)^\vee = \e_n\{n\}$, when we consider the Koszul dual operad of the cooperad $\e_n^c$.

\part{The \Linfty-structure of biderivation complexes}\label{part:biderivation complexes}
Recall that we denote by dg Hopf cooperad the structure of a cooperad
in the category of commutative dg algebras.
The goal of this part is to prove that the sets of morphisms associated to certain dg Hopf cooperads
of a particular shape are identified with the sets of solutions of the Maurer--Cartan equation
in an $L_{\infty}$~algebra of biderivations
associated to our dg Hopf cooperads.
To be more precise, we address an analogue of this question for dg Hopf $\La$-cooperads,
since we deal with Hopf cooperads endowed with a $\La$-structure throughout this work.

In a preliminary section, we revisit a definition of $L_{\infty}$~structures on dg vector spaces
from a geometric viewpoint.
In a second step, we explain the definition of an $L_{\infty}$~structure
on dg vector spaces of coderivations
associated to dg cooperads. We explain the definition of an extension of this $L_{\infty}$~structure
to dg vector spaces of biderivations associated to dg Hopf $\La$-cooperads afterwards,
and then we tackle the correspondence with morphisms
between dg Hopf $\La$-cooperads.
This dg vector space of biderivations represents, up to extra term, the deformation complex associated to our objects.

\section{Preliminaries: the geometric definition of \Linfty~structures and of \Linfty~maps}\label{sec:geometric Linfty}
The goal of this part, as we just explained, is to define $L_{\infty}$~structures on the dg vector spaces of coderivations
associated to dg $\La$-cooperads and on the dg vector spaces of biderivations
associated to dg Hopf $\La$-cooperads.
For our purpose, we also examine the definition of $L_{\infty}$~morphisms connecting such $L_{\infty}$~algebras.
We address these constructions in the next sections.

The complexity of the structure of a dg Hopf $\La$-cooperad makes a direct definition of such $L_{\infty}$~structures hard to handle.
Therefore, we explain a geometrical approach which eases the definition of these $L_{\infty}$~structures in this section
before tackling the applications to dg Hopf $\La$-cooperads.

\subsection{A way of defining \Linfty~structures}\label{subsec:geometric Linfty:structures}
Recall that an $L_{\infty}$~structure on a graded vector space $\DerL$
is defined by giving a differential (a square zero coderivation of degree $-1$) $D: S^+(W)\to S^+(W)$
on the cofree cocommutative coalgebra without counit $S^+(W)$
cogenerated by the shifted graded vector space $W = \DerL[1]$.
Equivalently, we can define an $L_{\infty}$~structure on the graded vector space $\DerL$ such that $W = \DerL[1]$
as a differential $D$
on the cofree counital cocommutative coalgebra $S(W)$
whose composite with the coaugmentation $\K\to S(W)$
vanishes. We say that $D$ forms a coaugmented differential on $S(W)$ when this condition holds. (Note that the composite of $D$
with the counit $S(W)\to\K$
tautologically vanishes too by the coderivation property of differentials.)
The following elementary proposition will provide us with a way of defining such a structure.

\begin{prop}\label{prop:geometric Linfty:structures}
Suppose that $(V,d)$ is dg vector space and $W$ is a graded vector space.
Suppose we have morphisms of graded vector spaces $F: S(W)\to V$, $\lambda: V\to W$, with $d F(1) = 0$,
and let $D: S(W)\to S(W)$ be the unique coaugmented coderivation on $S(W)$
such that $\pi\circ D = \lambda\circ d\circ F$,
where $\pi: S(W)\to W$ is the canonical projection.
Suppose that the following diagrams commute:
\begin{itemize}
\item (Projection condition)
\[
\begin{tikzcd}
S(W)\ar{d}{F}\ar{dr}{\pi} & \\
V\ar{r}{\lambda} & W \\
\end{tikzcd},
\]
\item (Tangentiality condition)
\[
\begin{tikzcd}
S(W)\ar{d}{F}\ar{r}{D} & S(W)\ar{d}{F} \\
V\ar{r}{d} & V \\
\end{tikzcd}.
\]
\end{itemize}
Then we have $D^2 = 0$ and hence the map $D$ defines an $L_{\infty}$~algebra structure on the graded vector space $\DerL$
such that $W = \DerL[1]$.
\end{prop}

\begin{proof}
The map $D^2 = \frac{1}{2} [D,D]$ is a coderivation.
Hence, it suffices to check that we have the relation $\pi D^2=0$ in order get our conclusion,
and this result follows from the sequence of identities
\[
\pi D^2 = \lambda F D^2 = \lambda d F D = \lambda d^2 F = 0,
\]
where we just use the projection and tangentiality conditions
of the proposition.
\end{proof}

\begin{rem}\label{rem:geometric Linfty:structures:interpretation}
We want to remark on the geometric meaning underlying the above proposition.
(The content of this remark is somewhat vague, and will not be used in any mathematical arguments below.)
We may think of the graded vector spaces $W$ and $V$ as graded varieties and the map $F$ as a non-linear map between these graded varieties.
More concretely, for a nilpotent graded ring $R$,
the $R$-points of $W$ are the degree 0 elements of $R\otimes W$,
and the image of such a point $w$ is $F(\exp(w))$.
The projection condition implies that the map $F$ is an embedding of graded varieties.
The differential $d$ on $V$ should be thought of as a linear homological vector field on $V$.
The tangentiality condition states that this vector field is tangent to the image of $W$,
and hence can be pulled back to a homological vector field on $W$,
which is given by $D$.

The special element $F(1)=F(\exp(0))$ may be called the \emph{basepoint}. It must necessarily be $d$-closed and killed by $\lambda$.
The choice of a basepoint does not affect the definition of the $L_{\infty}$~structure,
and in fact the basepoint may be removed by using $S^+(W)$ instead of $S(W)$ in the above Proposition,
and also in Proposition \ref{prop:geometric Linfty:morphisms} below.
However, for later notational convenience we keep it here.
\end{rem}

\subsection{A way of defining $L_{\infty}$~morphisms}\label{sec:geometric Linfty:morphisms}
Recall that an $L_{\infty}$~morphism between $L_{\infty}$~algebras with $\DerL$ (respectively, $\DerL'$) as underlying graded vector space
and $D$ (respectively, $D'$) as associated differential
is defined by a morphism of dg coalgebras $\phi: (S^+(W),D)\to(S^+(W'),D')$,
where we set $W = \DerL[1]$ (respectively, $W' = \DerL'[1]$)
and we consider the commutative dg coalgebra $(S^+(W),D)$ (respectively, $(S^+(W'),D')$)
determined by our $L_{\infty}$~structure on $\DerL$ (respectively, $\DerL'$).
Equivalently, we can define an $L_{\infty}$~morphism by giving a morphism of dg coalgebras $\phi: (S(W),D)\to(S(W'),D')$
that intertwines the counit and the coaugmentation
on $S(W)$ and $S(W')$.
The following proposition will provide us with a way of defining such morphisms.

\begin{prop}\label{prop:geometric Linfty:morphisms}
Suppose that $(V,d)$ is a differential graded vector space and let $W,W'$ be graded vector spaces.
Suppose that $F: S(W)\to V$ (respectively, $F': S(W')\to V$) and $\lambda: V\to W$ (respectively, $\lambda': V\to W'$)
are linear maps such that the conditions of Proposition \ref{prop:geometric Linfty:structures}
hold
and let $D: S(W)\rightarrow S(W)$ (respectively, $D': S(W')\rightarrow S(W')$)
denote the differential
returned by the result of this proposition.
Suppose further that we have a map of coaugmented counital cocommutative coalgebras $G: S(W)\to S(W')$
such that the following diagram commutes:
\[
\begin{tikzcd}
S(W)\ar{r}{\pi}\ar{d}{F}\ar[bend right,swap]{dd}{G} & W \\
V \ar{ur}{\lambda}\ar{dr}{\lambda'} & \\
S(W')\ar{r}{\pi'}\ar[swap]{u}{F'} & W'
\end{tikzcd}.
\]
Then we have $D' G = G D$ and, hence, the map $G$ defines an $L_{\infty}$~morphism $G: (S(W),D)\to(S(W'),D')$
between the $L_{\infty}$~algebras $\DerL$ and $\DerL'$ defined such that $W = \DerL[1]$
and $W' = \DerL'[1]$.
\end{prop}

\begin{proof}
The maps $D'G$ and $G D$ are both coderivations.
Hence, it suffices to check that we have the relation $\pi' D' G = \pi' G D$ in order get our conclusion,
and this result follows from the sequence of identities
\[
\pi' D' G = \lambda' F' D' G = \lambda' d F' G = \lambda' d F = \lambda' F D = \lambda' F' G D = \pi' G D,
\]
where we use the assumption of our proposition together with the projection and tangentiality conditions
of Proposition \ref{prop:geometric Linfty:structures}.
\end{proof}

\subsection{Remark:
}\label{subsec:linear maps remark}
Suppose that $V$ and $W$ are graded vector spaces and let $F: S(W)\to V$ be a map of graded vector spaces.
Then, for any augmented graded commutative algebra $R$ whose augmentation ideal $\overline{R}\subset R$
only contains nilpotent elements (in what follows, we just say that $R$ is a nilpotent graded ring for short),
the map $F$
determines a nonlinear map
\begin{equation}\label{eq:expmap}
\begin{aligned}[t]
\Phi_{F,R}: (W\otimes\overline{R})^0\to V\otimes\overline{R} \\
w\mapsto F(\exp(w))
\end{aligned},
\end{equation}
where $(W\otimes\overline{R})^0\subset W\otimes R$ is the space of degree 0 elements in $W\otimes\overline{R}$,
and we set
\[
\exp(w) = \sum_{n\geq 0}\frac{1}{n!} w^n.
\]
This construction is functorial in $R$.

Let us observe that the knowledge of the above map \eqref{eq:expmap} for any nilpotent graded ring $R$
uniquely determines $F$.
Indeed, suppose we want to recover $F$ on $S^n(W)$ from $\phi_{F,R}$,
where $S^n(W)$ denotes the component of weight $n$
of the cofree cocommutative coalgebra $S(W)$.
Let $w_1,\dots,w_n\in W$ be homogeneous elements of degrees $d_1,\dots,d_n$.
We set
\[
w = \epsilon_1 w_1 + \cdots + \epsilon_n w_n,
\]
where $\epsilon_1,\dots,\epsilon_n$ denote formal variables of degrees $-d_1,\dots,-d_n$,
and we take the nilpotent graded ring such that $R = \K[\epsilon_1,\dots,\epsilon_n]/(\epsilon_1^2,\dots,\epsilon_n^2)$.
We readily see that the coefficient of the monomial $\epsilon_1\cdots\epsilon_n$
in the image of this element
under our map
\[
\Phi_{F,R}(w) = \sum_{n\geq 0} \frac{1}{n!} F(w^n)
\]
is given by:
\[
\frac{1}{n!}\sum_{\sigma\in\Sigma_n}\pm F(w_{\sigma(1)}\cdots w_{\sigma(n)})
= (-1)^{\sum_{i<j} d_i d_j} F(w_1\cdots w_n).
\]
The conclusion follows. Note that we can restrict ourselves to the case where the nilpotent graded ring $R$
is finite dimensional over the ground field
to get this correspondence, since this is the case of the ring $R = \K[\epsilon_1,\dots,\epsilon_n]/(\epsilon_1^2,\dots,\epsilon_n^2)$
which we use in the above construction.

In what follows, we often provide a description of the map $\Phi_{F,R}$ in order to define our maps $F$
on the cofree cocommutative coalgebra, with the implicit understanding
that we use the polarization process explained in this paragraph
to recover $F$ from $\Phi_{F,R}$.
We mainly use this correspondence to avoid tedious sign calculations.
Indeed, the signs occurring in $\Phi_{F,R}$ are generally simpler than those in $F$ itself since the elements
which we consider in the definition of this map $\Phi_{F,R}$
have degree zero by construction.

\subsection{Remark: A way of verifying the tangentiality condition}\label{subsec:tangentiality condition}
In practice, the hardest part in the verification of the conditions of Proposition \ref{prop:geometric Linfty:structures}
is the tangentiality condition $d F = F D$.
We record a small trick which we may use to simplify this verification.
We first see, by using the polarization argument of the previous subsection,
that the verification of our tangential condition is equivalent
to the verification of the identity:
\[
d F(\exp(w)) = F D(\exp(w)),
\]
for any $w\in(W\otimes\overline{R})^0$, and for any nilpotent graded ring $R$.
We now have the following observation:

\begin{lemm}\label{lemm:tangentiality trick}
If $d F(\exp(w))$ and $F D(\exp(w))$ both lie in a subspace of $V$ that is isomorphic to $W$
under the map $\lambda: V\to W$,
then the tangentiality condition $d F(\exp(w)) = F D(\exp(w))$ holds.
\end{lemm}

\begin{proof}
By the assumption of the lemma, we just have to check that we have the relation $\lambda d F(\exp(w)) = \lambda F D(\exp(w))$ in $W$.
Recall that we have the identity $\lambda d F = \pi D$ by definition of our differential $D$.
Now, this relation, together with the projection condition of Proposition~\ref{prop:geometric Linfty:structures},
implies that we have the identities
\[
\lambda d F(\exp(w)) = \pi D(\exp(w)) = \lambda F D(\exp(w))
\]
and hence we are done.
\end{proof}

\begin{rem}\label{rem:tangentiality trick}
Note that, in general, we have the identity
\[
D\exp(w) = \frac{d}{d\epsilon}\exp(w + \epsilon\eta)
\]
where $\epsilon$ is a formal variable of degree $1$ such that $\epsilon^2 = 0$, and we set
\[
\eta = \pi D(\exp(w)).
\]
This observation will be useful when we have to check the condition of Lemma \ref{lemm:tangentiality trick}.
\end{rem}

\section{Coderivations and the deformation complex of dg cooperads }\label{sec:coderivation complexes}
We study the complexes (the dg vector spaces) of coderivations associated to dg cooperads in this section.
We also consider deformation complexes which correspond to the dg vector spaces
of coderivations with values in certain coresolutions
of our cooperads.
We consider dg cooperads satisfying $\op C(0) = 0$. We always assume that our cooperads are conilpotent in the sense explained in the preliminary part
of this article (but we do not assume that the component of arity one
of our cooperad is reduced to the ground field
in general).
We also deal with dg cooperads equipped with a $\La$-structure.

We recall the definition of the dg vector spaces of coderivations
associated to a cooperad in the first subsection
of this section.
We then study the dg vector spaces of coderivations
with values in the bar construction
of an operad.
We first check that the dg vector space of coderivations
associated to the usual bar construction
of an operad
inherits a dg Lie algebra structure.
We explain this construction in the second subsection of this section.
We study an extension of this construction in the context of $\La$-cooperads afterwards.
We explain the definition of $\La$-structures on the bar construction in the third subsection of the section
and we prove in the fourth subsection that the dg space of coderivations
associated to this bar construction in the category of $\La$-cooperads
inherits the structure of a dg Lie algebra
as well.

The deformation complexes of cooperads, which we consider at the beginning of this introduction,
actually consist of the dg vector spaces of coderivations
with values in certain coresolution of cooperads
given by the bar construction
associated to an operad.
We study an extension of these deformation complexes
to the case where we take a bicomodule
as coefficients
in the fifth subsection
of this section.

\subsection{Coderivations of cooperads}\label{subsec:coderivation complexes:coderivations}
Let $\op B$ and $\op C$ be dg cooperads and let $\phi: \op B\to\op C$ be a morphism of dg cooperads.
For our purpose, we assume that $\op B$ and $\op C$ are both coaugmented over the unit cooperad, conilpotent, and that $\phi: \op B\to\op C$
preserves this coaugmentation.
Let ${\op B}^{\circ}$ (respectively, ${\op C}^{\circ}$) be the graded cooperad underlying ${\op B}$ (respectively, $\op C$)
which we obtain by forgetting about the differential of our object.
For a formal variable $\epsilon = \epsilon_{-k}$ of degree $-k$, we consider the algebra of dual numbers $\K_{\epsilon} := \K[\epsilon]/(\epsilon^2)$.
We then define a degree $k$ coderivation of the cooperad morphism $\phi$
as a map of graded $\Sigma$-collections $\theta: {\op B}^{\circ}\to{\op C}^{\circ}$
of degree $k$ such that the map
\[
\phi+\epsilon\theta: {\op B}^{\circ}\otimes\K_{\epsilon}\to{\op C}^{\circ}\otimes\K_{\epsilon}
\]
defines a morphism of graded cooperads over $\K_\epsilon$. We also assume that this morphism
preserves the coaugmentation which we attach to our objects.
We equivalently have the vanishing relation $\theta(1) = 0$ when we consider the unit element $1 = \eta(1)\in{\op B}(1)$
which we associated to the coaugmentation morphism $\eta: \K\to{\op B}(1)$
of the cooperad ${\op B}$.
We equip this graded vector space of coderivations with the differential
defined by the usual commutator formula $d\theta = d_{\op C}\circ\theta-(-1)^k\theta\circ d_{\op B}$,
where we consider the internal differential of our dg cooperads.
We adopt the notation
\[
\CoDer_{\dg\Sigma}({\op B}\stackrel{\phi}{\lo}{\op C})
\]
for this dg vector space of coderivations associated to a morphism of conilpotent dg cooperads $\phi: {\op B}\to{\op C}$.

If the cooperads $\op B$ and $\op C$ are equipped with a coaugmentation over $\Com^c$ and carry a $\La$-structure
so that the map $\phi: {\op B}\to \op C$ defines a morphism of $\La$-cooperads,
then we also consider the dg vector space
\[
\CoDer_{\dg\La}({\op B}\stackrel{\phi}{\lo}{\op C})\subset\CoDer_{\dg\Sigma}({\op B}\stackrel{\phi}{\lo}{\op C})
\]
formed by the coderivations that intertwine the action of the category $\Lambda$ and vanish
over the coaugmentation $\eta: \Com^c\to{\op B}$.
Note that the map $\phi+\epsilon\theta: {\op B}^{\circ}\otimes\K_{\epsilon}\to{\op C}^{\circ}\otimes\K_{\epsilon}$
defines a morphism of graded $\La$-cooperads over $\K_\epsilon$.
when $\theta$ is a coderivation of $\La$-cooperads.


In what follows, we also consider an extended version of this dg vector space of coderivations
which we associate to a certain extension ${\op C}^+$
of the cooperad $\op C$.

For the sake of clarity, we explain the dual construction of an extension ${\op P}^+$
of an object of the category of dg operads $\op P$
first.
In short, this operad ${\op P}^+$ is defined by freely adjoining a unary operation $D$ of degree $1$
to our operad $\op P$,
and by providing this operad ${\op P}^+$
with the differential such that
\begin{align*}
d_{{\op P}^+}(D) & := -D^2 \\
\text{and}\quad
d_{{\op P}^+}(p) & := d_{\op P}(p) + D\circ p - (-1)^{|p|}\sum_{i=1}^r p\circ_i D,
\end{align*}
for any $p\in{\op P}(r)$. Formally, this operad $\op P^+$ is defined by taking the coproduct of the object $\op P$ with the free graded operad $\op F(D)$
generated by the unary operation $D$ in the category of operads,
and by providing this graded operad with the differential determined by the above formulas
on the summand $\op P$ and on the extra generating element $D\in\op F(D)$
of our object.
We refer to \cite[Section A.5]{Fr} for an explicit description, in terms of semi-alternate two-colored trees, of a coproduct of this form $\op P\vee\op F(D)$
in the category of operads.
Recall simply that these semi-alternate two-colored trees
represent formal composites of elements of the operad $\op P$
with the operation $D$ inside the coproduct $\op P\vee\op F(D)$.
The differential $d_{{\op P}^+}(p)$ of an element $p\in{\op P}(r)$ in $\op P\vee\op F(D)$
has the following description
\[
d_{{\op P}^+}(p)\;=\;\begin{tikzpicture}[baseline=-.65ex]
 \node (a) at (0,.25) {$d_{\op P}(p)$};
 \draw (a) edge +(0,.5) edge+(-0.25,-.5) edge+(0.25,-.5) edge +(-0.75,-.5) edge +(0.75,-.5);
\end{tikzpicture}\;-\;\pm \begin{tikzpicture}[baseline=-.65ex]
 \node (a) at (0,.25) {$p$};
 \node (b) at (0.25,-.5) {$D$};
 \draw (a) edge +(0,.5) edge (b) edge +(-0.25,-.5) edge +(-0.75,-.5) edge +(0.75,-.5);
 \draw (b) edge +(0,-.5);
\end{tikzpicture}
\;+\;\begin{tikzpicture}[baseline=-.65ex]
 \node (a) at (0,.25) {$D$};
 \node (b) at (0,-.5) {$p$};
 \draw (a) edge (b) edge +(0,.5);
 \draw (b) edge +(-0.75,-.5) edge +(-0.25,-.5) edge +(0.25,-.5) edge +(0.75,-.5);
\end{tikzpicture}
\]
when we use this tree-wise representation of our coproduct of cooperads.

If $\op P$ is equipped with the structure of a $\La$-operad, then so is the extended operad ${\op P}^+$. Indeed, we can provide the free operad $\op F(D)$
with a trivial $\La$-structure,
which we determine by the relation $\epsilon(D) = 0$ when we take the image of the generating operation $D$
under the augmentation $\epsilon: \op F(D)\to\Com$.
Then we merely use that the coproduct $\op P\vee\op F(D)$ of the objects $\op P$ and $\op F(D)$ in the category of ordinary operads
also represents the coproduct of these objects $\op P$ and $\op F(D)$
in the category of $\La$-operads (see \cite[Remark A.5.11]{Fr})
and hence inherits a natural $\La$-operad structure
to get our result.

If $(V,d)$ is a dg vector space equipped with an action of the operad ${\op P}^+$,
then we readily see that the object $(V,d+D)$, defined by adding the action of the operation $D$ on $V$
to the internal differential $d$, forms a dg vector space
and that this dg vector space inherits an action of the operad ${\op P}$.
Thus, we can identify an algebra over the operad ${\op P}^+$ with the structure formed by an algebra over the operad $\op P$
together with a deformation of the differential.


We proceed as follows to dualize the above construction in the context where we have a dg cooperad $\op C$.
We first consider the cofree (conilpotent) graded cooperad ${\op F}^c(D)$
cogenerated by the operation $D$ in arity one.
We can also identify this cooperad given by the cofree (conilpotent) coalgebra cogenerated by $D$
in arity one
and which vanishes otherwise.
We explicitly have ${\op F}^c(D)(1) = T^+(D)$ and ${\op F}^c(D)(r) = 0$ for $r>1$,
where we use the notation $T^+(D) = \bigoplus_{n\geq 1} \K D^{\otimes n}$
for the graded tensor coalgebra without counit cogenerated by $D$.
We then define the cooperad ${\op C}^+$ by taking the cartesian product ${\op C}\times{\op F}^c(D)$ of the objects ${\op C}$ and ${\op F}^c(D)$
in the category of cooperads
and by providing this graded cooperad with the differential $d_{{\op C}^+}: {\op C}\times{\op F}^c(D)\to{\op C}\times{\op F}^c(D)$
whose projection onto $\K D$ is induced by the map $T^+(D)\rightarrow\K D$
such that
\[
D^{\otimes m}\mapsto\begin{cases} -D, & \text{if $m=2$}, \\
0, & \text{otherwise},
\end{cases}
\]
and whose projection onto ${\op C}$ is given by the following mapping
\[
\begin{tikzpicture}[baseline=-.65ex]
 \node (a) at (0,.25) {$c$};
 \draw (a) edge +(0,.5) edge+(-0.25,-.5) edge+(0.25,-.5) edge +(-0.75,-.5) edge +(0.75,-.5);
\end{tikzpicture}
\mapsto d_{\op C}(c),
\qquad\begin{tikzpicture}[baseline=-.65ex]
 \node (a) at (0,.25) {$c$};
 \node (b) at (0.25,-.5) {$D$};
 \draw (a) edge +(0,.5) edge (b) edge +(-0.25,-.5) edge +(-0.75,-.5) edge +(0.75,-.5);
 \draw (b) edge +(0,-.5);
\end{tikzpicture}
\mapsto - \pm c,
\qquad\begin{tikzpicture}[baseline=-.65ex]
 \node (a) at (0,.25) {$D$};
 \node (b) at (0,-.5) {$c$};
 \draw (a) edge (b) edge +(0,.5);
 \draw (b) edge +(-0.75,-.5) edge +(-0.25,-.5) edge +(0.25,-.5) edge +(0.75,-.5);
\end{tikzpicture}
\mapsto c,
\]
for any $c\in\op C(r)$, $r>0$, where we use a tree-wise description of the product ${\op C}\times{\op F}^c(D)$
which is the analogue, for cooperads, of the tree-wise description of coproducts
which we consider in the context of operads (we also refer to \cite[Paragraph C.1.16]{Fr} for this construction of products in the category of cooperads).
To be more explicit, we use that the trees in our picture correspond to certain summands of the expansion of the product ${\op C}\times{\op F}^c(D)$
in the base category of dg vector spaces.
Then we just assume that the composite of the differential $d_{{\op C}^+}: {\op C}\times{\op F}^c(D)\to{\op C}\times{\op F}^c(D)$
with the canonical projection ${\op C}\times{\op F}^c(D)\to{\op C}$
is given by the maps depicted in our figure on these summands
and vanishes otherwise.
Let us note that we use the coderivation relation to determine the value of the differential $d_{{\op C}^+}$
on the whole cartesian product ${\op C}\times{\op F}^c(D)$
from the above assignments.

If $\op C$ is equipped with the structure of a $\La$-cooperad, then so is ${\op C}^+$ with the coaugmentation $\eta: \Com\to{\op C}^+$
and the action of the category $\Lambda$ determined by the coaugmentation $\eta: \Com\to\op C$
and the $\La$-structure of our coaugmented $\La$-cooperad $\op C$.

To any morphism of cooperads (respectively, of coaugmented $\La$-cooperads) $\phi: {\op B}\to{\op C}$,
we associate the extended morphism $\phi^+: {\op B}\to{\op C}^+$
given by $\phi$ on the factor $\op C$ of the cartesian product ${\op C}\times{\op F}^c(D)$
and by the null map on the cofree cooperad ${\op F}^c(D)$.
Then we define our extended coderivation complexes by:
\[
\xCoDer_{\dg\Sigma}({\op B}\stackrel{\phi}{\to}{\op C}) = \CoDer_{\dg\Sigma}({\op B}\stackrel{\phi^+}{\to}{\op C}^+)
\quad\text{and}
\quad\xCoDer_{\dg\La}({\op B}\stackrel{\phi}{\to}{\op C}) = \CoDer_{\dg\La}({\op B}\stackrel{\phi^+}{\to}{\op C}^+).
\]

\subsection{The dg Lie algebra structure of the coderivation complexes of dg cooperads}\label{subsec:coderivation complexes:plain cooperads}
We now assume that the dg cooperad ${\op C}$ in the construction of the previous subsection
is given by the bar construction of an augmented dg operad ${\op P}$:
\[
{\op C} = B({\op P}).
\]
Recall that the bar construction of an augmented operad ${\op P}$
is a dg cooperad $B({\op P})$
which is cofree when we forget about differentials.
We explicitly have $B({\op P})^{\circ} = {\op F}^c(\overline{\op P}{}^{\circ}[1])$, where $\overline{\op P}$
denotes the augmentation ideal of our operad, and where we again use the notation ${\op F}^c(-)$
for the conilpotent cofree cooperad functor on the category of symmetric collections
such that $\op M(0) = 0$.
Recall also that the differential of the bar construction
is the coderivation
$d: {\op F}^c(\overline{\op P}{}^{\circ}[1])\to{\op F}^c(\overline{\op P}{}^{\circ}[1])$
given,
on the cogenerating collection $\overline{\op P}{}^{\circ}[1]$ of this cofree cooperad ${\op F}^c(\overline{\op P}{}^{\circ}[1])$,
by the internal differential of our operad $d_{\op P}: {\op P}\to{\op P}$
and the restriction of the partial composition products $\circ_i: {\op P}(k)\otimes{\op P}(l)\to{\op P}(k+l-1)$
to the augmentation ideal $\overline{\op P}\subset{\op P}$,
after observing that the domains of these operations
correspond to summands of weight two
in the cofree cooperad.

We have the following well-known statement:

\begin{prop}\label{prop:coderivation lie algebra}
If we assume ${\op C} = B({\op P})$ as above, then the morphisms of conilpotent coaugmented graded cooperads $\phi: {\op B}^{\circ}\to{\op C}^{\circ}$
are in bijection with the morphisms of graded $\Sigma$-collections
\[
\alpha_{\phi}: \overline{\op B}{}^{\circ}\to\overline{\op P}{}^{\circ}[1]
\]
or, equivalently, with elements of degree $-1$ in the graded hom-object
of $\Sigma$-collections
\[
\DerL = \Hom_{\gr\Sigma}(\overline{\op B},\overline{\op P})
\]
such as defined in Section~\ref{subsec:collections}.
This map $\alpha_{\phi}$, which we associate to a morphism of conilpotent graded cooperads $\phi: {\op B}^{\circ}\to{\op C}^{\circ}$,
is explicitly defined by taking by the projection of our morphism $\phi: {\op B}^{\circ}\to{\op C}^{\circ}$
onto the cogenerating collection of the cofree cooperad $B({\op P})^{\circ} = {\op F}^c(\overline{\op P}{}^{\circ}[1])$
on the target and its restriction to the coaugmentation coideal of the cooperad ${\op B}$
on the source.

The hom-object $\DerL = \Hom_{\gr\Sigma}(\overline{\op B},\overline{\op P})$ moreover inherits a dg Lie algebra structure
such that the above correspondence restricts to a bijection between the set
of dg cooperad morphisms $\phi: {\op B}\to{\op C}$
and the set of Maurer-Cartan elements
in $\DerL$.

Besides, we have an isomorphism
\[
\CoDer_{\dg\Sigma}({\op B}\stackrel{\phi}{\to}{\op C})[-1]\cong\DerL^{\alpha_{\phi}},
\]
where we consider the (degree shift of the) coderivation complex associated to our morphism of dg cooperads $\phi: {\op B}\to{\op C}$
on the left-hand side and the twisted dg Lie algebra associated to the Maurer-Cartan element $\alpha_{\phi}$
corresponding to $\phi$ on the right-hand side.
This shifted coderivation complex accordingly inherits a dg Lie algebra structure by transport of structure
from the twisted dg Lie algebra $\DerL^{\alpha_{\phi}}$.
\end{prop}

This proposition can be found in \cite{LV} for instance. Nevertheless, we prefer to revisit the proof of this statement.
The arguments which we give in our proof will later on serve as a blueprint
for analogous results about Hopf cooperads.

\begin{proof}
The first statement of the Proposition is an immediate consequence of the cofree structure of the bar construction
in the category of graded cooperads.
The same argument implies that the space of coderivations $\CoDer_{\dg\Sigma}({\op B}\stackrel{\phi}{\to}{\op C})$
is isomorphic to the hom-object $\Hom_{\gr\Sigma}(\overline{\op B},\overline{\op P}[1])$
as a graded vector space.
We revisit the definition of this bijection later on in this proof.
We focus on the construction of our dg Lie algebra structure on the graded vector space $\DerL = \Hom_{\gr\Sigma}(\overline{\op B},\overline{\op P})$
for the moment.
We first apply the construction of Proposition \ref{prop:geometric Linfty:structures} to produce an $L_{\infty}$~structure on this object.
We will check afterwards that this $L_{\infty}$~structure reduces to an ordinary dg Lie algebra structure.

We keep the notation of Proposition \ref{prop:geometric Linfty:structures}.
We first set
\[
V := \Hom_{\dg\Sigma}(\overline{\op B},\overline{\op C}).
\]
We provide this object with its natural differential, which is defined by the commutator formula
\[
d f = d_{\op C}\circ f - (-1)^{|f|}f\circ d_{\op B},
\]
for any homogeneous map $f\in\Hom_{\dg\Sigma}(\overline{\op B},\overline{\op C})$.
We then set
\[
W := \underbrace{\Hom_{\gr\Sigma}(\overline{\op B},\overline{\op P})}_{= \DerL}[1] = \Hom_{\gr\Sigma}(\overline{\op B},\overline{\op P}[1])
\]
and we consider the map $\lambda : V\to W$ induced by the projection onto the cogenerators ${\op C}^{\circ}\to\overline{\op P}{}^{\circ}[1]$
in the cofree cooperad ${\op C}^{\circ} = {\op F}^c(\overline{\op P}{}^{\circ}[1])$.

We use the trick of Section \ref{subsec:linear maps remark} to define the map $F: S(W)\to V$.
We proceed as follows.
Let $R$ be a nilpotent graded ring of finite dimension over $\K$.
In the context of the present proposition, we can interpret a degree zero element $w\in(W\otimes\overline{R})^0$
as an $R$-linear morphism of graded $\Sigma$-collections $w: \overline{\op B}{}^{\circ}\otimes R\to\overline{\op P}{}^{\circ}[1]\otimes R$
such that $w(\overline{\op B}{}^{\circ}\otimes R)\subset\overline{\op P}{}^{\circ}[1]\otimes\overline{R}$.
This morphism determines a morphism of conilpotent graded cooperads (defined over $R$)
\[
\widetilde{w}: {\op B}^{\circ}\otimes R\to{\op C}^{\circ}\otimes R,
\]
such that $\widetilde{w}(\overline{\op B}{}^{\circ}\otimes R)\subset\overline{\op C}^{\circ}\otimes\overline{R}$,
and which we can regard as a degree zero element of the dg vector space $V\otimes\overline{R}$
when we forget about the counit (and the composition coproducts) of our cooperads.
We then set $\Phi_{F,R}(w) = \widetilde{w}$ and we consider the map $F$, such that $F(\exp(w)) = \Phi_{F,R}(w)$,
which we determine from this relation.

We can use the polarization construction of Section \ref{subsec:linear maps remark} to retrieve the explicit formula of this map $F$
on the symmetric algebra.
To be explicit, let $w_1,\dots,w_n\in W$ and $b\in{\op B}(r)$.
If we use the standard construction of the cofree cooperad ${\op C}^{\circ} = {\op F}^c(\overline{\op P}{}^{\circ}[1])$
in terms of tree-wise tensor products,
then we get that the components of the element $F(w_1\cdots w_n)(b)\in{\op C}(r)$
can be obtained by applying the symmetrized sum
$\sum_{\sigma\in\Sigma_n}\pm w_{\sigma(1)}\otimes\cdots\otimes w_{\sigma(n)}\in W^{\otimes n}$
to the treewise coproducts $\Delta_T(b)\in\bigotimes_{v\in VT}{\op B}(star(v))$,
for all trees $T$ such that $n = |VT|$.
In particular, the linear component of $F$ is just the inclusion $W\subset V$. We readily deduce from this observation that our map $F$
fulfills the projection condition of Proposition \ref{prop:geometric Linfty:structures}.

We use the trick of Lemma \ref{lemm:tangentiality trick} to check the tangentiality condition of this proposition.
Recall that $F(\exp(w)) = \widetilde{w}$ is a morphism of graded cooperads
from ${\op B}^{\circ}\otimes R$ to ${\op C}^{\circ}\otimes R$
by definition of our map $F$.
We take the space of coderivations associated to this morphism $F(\exp(w)): {\op B}^{\circ}\otimes R\to{\op C}^{\circ}\otimes R$
as the subspace of $V$, isomorphic to $W$ under the map $\lambda$,
which we need for the construction of Lemma \ref{lemm:tangentiality trick}.
We just use that any such coderivation is uniquely determined by its projection onto the cogenerators
of the cofree cooperad ${\op C}^{\circ} = {\op F}^c(\overline{\op P}{}^{\circ}[1])$
to check the requested isomorphism condition.
We readily see that $d F(\exp(w))$ automatically belongs to this space of coderivations of $F(\exp(w))$ too,
and so does $F(D\exp(w))$ since this map is the derivative of a morphism
of graded cooperads
by Remark \ref{rem:tangentiality trick}.
Hence, the requirements of Lemma \ref{lemm:tangentiality trick}
are fully satisfied, and we can therefore apply this lemma to conclude that our map $F$
satisfies the tangentiality projection condition of Proposition \ref{prop:geometric Linfty:structures}
in addition to the projection condition.

We can now use the result of Proposition \ref{prop:geometric Linfty:structures} to finish the construction of the $L_{\infty}$~structure on $W[-1]$.
We can use the explicit formula of the map $F$ to get an explicit description of the coderivation $D$
which determines this $L_{\infty}$~structure.
Recall that this coderivation is characterized by the relation $\pi D = \lambda d F$,
where $d$ is the differential of the dg vector space $V = \Hom_{\dg\Sigma}(\overline{\op B},\overline{\op C})$
and $\pi: S(W)\to W$ denotes the obvious projection onto the graded vector space $W$
in the symmetric algebra $S(W)$ (see Proposition \ref{prop:geometric Linfty:structures}).
To go further, we have to go back to the definition of the differential of the bar construction ${\op C} = B({\op P})$
in the expression of the differential $df = d_{\op C} f - \pm f d_{\op B}$
of the dg vector space $V$.
Recall simply that the composite of this differential $d_{\op C}$
with the projection ${\op C}^{\circ} = {\op F}^c(\overline{\op P}{}^{\circ}[1])\to\overline{\op P}{}^{\circ}[1]$,
which determines our map $\lambda: V\to W$,
reduces to a linear component,
which is induced by the internal differential of the operad ${\op P}$,
plus a quadratic component,
which we associate to the partial composition products $\circ_i: \overline{\op P}(m)\otimes\overline{\op P}(n)\to\overline{\op P}(m+n-1)$.
From this observation, we readily deduce that all components of our $L_{\infty}$~structure on the graded vector space $\DerL$
vanish except the linear and the quadratic one,
so that this $L_{\infty}$~structure actually reduces to a dg Lie algebra structure
as stated in the proposition. We also get, by the way, that the linear part of this $L_{\infty}$~structure is identified with the differential
that the dg hom-object $\Hom_{\dg\Sigma}(\overline{\op B},\overline{\op P})$
inherits from the collections $\overline{\op B}$ and $\overline{\op P}$.

If $w\in W$ is a Maurer-Cartan element in this dg Lie algebra,
then we have $\pi D\exp(w) = 0$ (by definition of the notion of a Maurer-Cartan element),
and $\lambda d F(\exp(w)) = 0$ by definition of our coderivation $D$,
where we again consider the element $F(\exp(w))\in V\otimes R$
associated to $w$, for some choice of nilpotent graded ring $R$ of finite dimension over $\K$.
We equivalently have $d F(\exp(w)) = 0$
since we observed that the map $\lambda$
restricts to an isomorphism
on the subspace of coderivations which contains this element $d F(\exp(w)$
inside $V$.
This relation $d F(\exp(w)) = 0$ is equivalent to the assumption that $\widetilde{w} = F(\exp(w))\in W\otimes R$
correspond to a differential preserving morphism $\widetilde{w}: \overline{\op B}\otimes R\to\overline{\op C}\otimes R$
when we use the identity $W\otimes R = \Hom_{\gr\Sigma}(\overline{\op B}\otimes R,\overline{\op C}\otimes R)$.
Note that, if we identify $w\in W$ with a morphism of graded collections $f: \overline{\op B}{}^{\circ}\to\overline{\op P}{}^{\circ}[1]$,
then $\widetilde{w} = F(\exp(w))$ is identified with the (restriction to coaugmentation coideals
of the) obvious scalar extension of the morphism
of conilpotent graded cooperads $\phi: {\op B}^{\circ}\to{\op C}^{\circ}$
associated to $f$
by definition of our map $F$.
Hence, the assumption that $\widetilde{w}: \overline{\op B}\otimes R\to\overline{\op C}\otimes R$
preserves differentials is equivalent to the assumption that this map
defines a morphism of conilpotent dg cooperads $\phi: {\op B}\to{\op C}$.

Recall that a coderivation $\theta\in\CoDer_{\dg\Sigma}({\op B}\stackrel{\phi}{\to}{\op C})$
of a morphism of graded cooperads $\phi$
is equivalent to a morphism
of the form $\phi+\epsilon\theta: {\op B}\otimes\K_{\epsilon}\to{\op C}\otimes\K_{\epsilon}$,
where we consider the graded ring $\K_{\epsilon} = \K[\epsilon]/(\epsilon^2)$ (see Section \ref{subsec:coderivation complexes:coderivations}).
The map of graded $\Sigma$-collections which we associate to such a morphism satisfies $\alpha_{\phi+\epsilon\theta} - \alpha_{\phi} = \epsilon f$,
where $f: \overline{\op B}\to\overline{\op P}[1]$ is a map,
associated to our coderivation $\theta$,
and which we can also explicitly obtain by taking the composite of $\theta$
with the projection onto the cogenerating collection
of the cofree cooperad ${\op C}^{\circ} = {\op F}^c(\overline{\op P}{}^{\circ}[1])$.
The bijectivity of our general correspondence between the morphisms of graded cooperads with values in the bar construction ${\op C} = B({\op P})$
and the morphisms of graded $\Sigma$-collections with values in $\overline{\op P}$
now implies that this mapping gives an isomorphism
of graded vector spaces:
\[
\CoDer_{\dg\Sigma}({\op B}\stackrel{\phi}{\to}{\op C})\cong\Hom_{\gr\Sigma}(\overline{\op B},\overline{\op P}[1]).
\]
To complete our verifications, we readily check from the definition of our dg Lie algebra structure
on $\DerL = \Hom_{\gr\Sigma}(\overline{\op B},\overline{\op P})$
that the differential of $\theta$
in our dg vector space of coderivations
corresponds to the differential
of the element $\epsilon f = \alpha_{\phi+\epsilon\theta} - \alpha_{\phi}$
in the twisted dg Lie algebra $\DerL^{\alpha_{\phi}}\otimes\K_{\epsilon}$.
The isomorphism of the proposition between our dg vector space of coderivation
and the underlying dg vector space twisted dg Lie algebra
follows.
\end{proof}

\begin{rem}\label{rem:extended coderivation Lie algebra}
Note that, for ${\op C}$ as in the above proposition, we have an identity:
\[
{\op C}^+ = B({\op P}_{\bo}),
\]
where ${\op P}_{\bo}$ is the augmented operad obtained by freely adjoining a (second) unit to ${\op P}$.

If we replace ${\op C}$ by ${\op C}^+$ in the proposition, then we obtain that the graded vector space
\[
\xCoDer_{\dg\Sigma}({\op B}\stackrel{\phi}{\to}{\op C})[-1] = \Hom_{\gr\Sigma}(\overline{\op B},{\op P})^{\alpha_f}
\]
also carries a dg Lie algebra structure.
\end{rem}

\subsection{Bar and cobar constructions for $\La$-cooperads}\label{subsec:coderivation complexes:BLambda-structures}
The purpose of this subsection is to explain the definition of $\La$-structures on the bar construction $B({\op P})$
from structures that we associate to the augmented dg operad ${\op P}$.
To be explicit, we use the following notion:

\begin{defn}\label{defn:BLambda-structures}
We assume that:
\begin{itemize}
\item
the augmentation ideal of our operad $\overline{\op P}$ is equipped with a covariant $\La$-diagram structure,
which we can determine by giving generating corestriction operators $\eta_S : \overline{\op P}(S)\to\overline{\op P}(S\sqcup\{*\})$,
for every finite set $S$,
as we explained in the case of cooperads in Section \ref{subsec:Lambda-cooperads};
\item
and we have a map of $\La$-collections $u\in\Hom_{\gr\La}(\overline{\Com}{}^c,\overline{\op P})$, of degree $-1$,
where we regard the coaugmentation coideal of the commutative cooperad $\overline{\Com}{}^c$
as a constant object in the category of $\La$-collections.
\end{itemize}
We also assume that the following properties hold:
\begin{itemize}
\item
The operators $\eta_S$ are derivations with respect to the composition operations of our operad
in the sense that we have a commutative diagram:
\[
\begin{tikzcd}[column sep=3cm]
\overline{\op P}(S\sqcup\{\diamond\})\otimes\overline{\op P}(T)
\ar{d}{\circ_{\diamond}}\ar{r}{\eta_{S\sqcup\{\diamond\}}\otimes\mathit{id}\pm\mathit{id}\otimes\eta_T} &
\overline{\op P}(S\sqcup\{\diamond,*\})\otimes\overline{\op P}(T)\oplus\overline{\op P}(S\sqcup\{\diamond\})\otimes\overline{\op P}(T\sqcup\{*\})
\ar{d}{\circ_{\diamond}} \\
\overline{\op P}(S\sqcup T)\ar{r}{\eta_{S\sqcup T}} &
\overline{\op P}(S\sqcup T\sqcup\{*\}),
\end{tikzcd}
\]
for each pair of finite sets $S,T$.
\item
The operators $\eta_S$ satisfy the following relation:
\[
d_{\op P}\circ\eta_S - \eta_S\circ d_{\op P} = \ad_b,
\]
where $b\in\overline{\op P}(2)$ is the image of the generator of $\Com^c(2)$ under our map $u: \overline{\Com}{}^c\to\overline{\op P}[1]$,
and $\ad_b: \overline{\op P}(S)\to\overline{\op P}(S\sqcup\{*\}$
is given by the following composition operations
\[
\ad_b x =
\begin{tikzpicture}[baseline=-.65ex]
\node (v) at (0,0) {$x$};
\node (b) at (0.5,.5) {$b$};
\node (s) at (1,0) {$*$};
\draw (v) edge (b) edge +(-.5,-.5) edge +(0,-.5) edge (.5,-.5)
      (b) edge +(0,.5) edge (s);
\end{tikzpicture}
-
\sum\begin{tikzpicture}[baseline=-.65ex]
\node (v) at (0,0) {$x$};
\node (b) at (-0,-.5) {$b$};
\node (s) at (.5,-1) {$*$};
\draw (v) edge (b) edge +(0,.5) edge +(-0.5,-.5) edge (.5,-.5)
      (b) edge +(-0.5,-.5) edge (s);
\end{tikzpicture}
\]
in the augmentation ideal of our operad $\op P$.
\item
The map $u: \overline{\Com}{}^c\to\overline{\op P}$ satisfies the Maurer-Cartan equation
in the dg Lie algebra $\DerL = \Hom_{\dg\Sigma}(\overline{\Com}{}^c,\overline{\op P})$
of Proposition~\ref{prop:coderivation lie algebra}
when we forget about the $\La$-diagram structures attached to our objects.
Hence, this map is associated to a morphism of dg cooperads $\phi_u: \Com^c\to B({\op P})$
in the correspondence of Proposition~\ref{prop:coderivation lie algebra}.
\end{itemize}
We then say that ${\op P}$ forms a $B\La$ operad.
\end{defn}

We have the following statement:

\begin{prop}\label{prop:BLambda-operad bar construction}
The bar construction $B({\op P})$ of a $B\La$ operad ${\op P}$ inherits the structure of a dg $\La$-cooperad
such that $B({\op P})^{\circ} = {\op F}^c(\overline{\op P}{}^{\circ}[1])$
is identified with the cofree object cogenerated by the coaugmented $\La$-collection $\overline{\op P}[1]$
in the category of graded $\La$-cooperads
when we forget about differentials.
\end{prop}

\begin{proof}[Explanations and proof]
In this proposition, we use the general observation that the cofree cooperad ${\op F}^c({\op M})$
inherits the structure of a $\La$-cooperad when ${\op M}$ is a covariant $\La$-collection
equipped with a coaugmentation over $\overline{\Com}{}^c$.
To be more precise, one can observe that the usual cofree cooperad functor with values in the category of ordinary cooperads
lifts to a cofree object functor from this category of coaugmented covariant $\La$-collections
to the category of $\La$-cooperads.
We refer to~\cite[Proposition C.2.18]{Fr} for the proof of an analogous result in the context of the cobar-bar resolution
of operads satisfying ${\op P}(1) = \K$.

We just give an explicit description of the corestriction operators $\eta_S$
which we associate to the $\La$-structure of the bar construction,
in the case ${\op M} = \overline{\op P}[1]$.
We then use the expansion of the cofree cooperad ${\op F}^c(\overline{\op P}[1])$
as a direct sum of treewise tensor products ${\op F}^c_T(\overline{\op P}[1]) = \bigotimes_{v\in VT}\overline{\op P}(star(v))[1]$,
where we again use the notation $star(v)$ for the set of ingoing edges
which we attach to any vertex $v\in VT$
in a rooted tree $T$ (see Section~\ref{subsec:operads}).
Recall that the expansion of a component of the cofree cooperad ${\op F}^c(\overline{\op P}[1])(S)$, where $S$ is any finite set,
ranges over the set of rooted trees $T$ with leafs indexed by $S$,
and that we also adopt the notation $\mT_S$
for this set of trees.

Let $x\in B({\op P})(S)$. We use the notation $x_T$ for the component of this element $x$
in the treewise tensor product ${\op F}^c_T(\overline{\op P}[1])$,
for any tree $T\in\mT_S$.
Let $T\in\mT_{S\sqcup\{*\}}$. We aim to define the component $\eta_S(X)_T\in{\op F}^c_T(\overline{\op P}[1])$
of the image of our element $x\in{\op F}^c(\overline{\op P}[1])(S)$
under the corestriction operator $\eta_S$ on the cofree cooperad ${\op F}^c_T(\overline{\op P}[1])$.
We assume that $v$ is the vertex of $T$ connected to the leaf indexed by the mark $*$.
We consider the tree $T'$ which we obtain removing this leaf indexed by $*$
in the tree $T$
as in the expression of the compatibility relation between the treewise coproducts
and the corestriction operators of a cooperad
in Section~\ref{subsec:Lambda-cooperads}.
If $v$ has exactly $2$ ingoing edges in $T$, then we also consider the tree $T''$
which we obtain by deleting both the vertex $v$ and the ingoing edge
indexed by $*$,
and by merging the remaining ingoing edge with the outgoing edge of $v$.
Let $star'(v)$ be the set of outgoing edges of $v$ in $T'$.
We also fix an ordering of the vertices of $T$ with the vertex $v$ in the first position.
We then set:
\beq{eq:temp9}
\eta_S(x)_T = (\eta_{star'(v)}\otimes\mathit{id}\otimes\cdots\otimes\mathit{id})(x_{T'})
+
\begin{cases}
b\otimes x_{T''}, & \text{if $v$ has two ingoing edges in $T$}, \\
0, & \text{otherwise},
\end{cases}
\eeq
where $b\in\overline{\op P}(2)$ denotes the image of the generator of $\Com^c(2)$
under our map $u: \overline{\Com}{}^c\rightarrow{\op P}[1]$ (as in Definition \ref{defn:BLambda-structures}).
The following picture gives an example of application of this construction:
\[
\eta
\ \
\begin{tikzpicture}[baseline=-.65ex]
\node (a) at (0,.25) {$\alpha$};
\node (b) at (.5,-.25) {$\beta$};
\draw (a) edge (b) edge +(0,.5) edge +(-.5,-.5)
      (b) edge +(.5,-.5) edge +(-.5,-.5);
\end{tikzpicture}
=
\begin{tikzpicture}[baseline=-.65ex]
\node (a) at (0,.4) {$\eta(\alpha)$};
\node (b) at (.5,-.25) {$\beta$};
\node (s) at (0,-.25) {$*$};
\draw (a) edge (b) edge (s) edge +(0,.5) edge +(-.5,-.5)
      (b) edge +(.5,-.5) edge +(-.5,-.5);
\end{tikzpicture}
+
\begin{tikzpicture}[baseline=-.65ex]
\node (a) at (0,.4) {$\alpha$};
\node (b) at (.5,-.25) {$\eta(\beta)$};
\node (s) at (0.5,-.95) {$*$};
\draw (a) edge (b) edge +(0,.5) edge +(-.5,-.5)
      (b) edge +(.5,-.5) edge +(-.5,-.5) edge (s);
\end{tikzpicture}
+
\begin{tikzpicture}[baseline=-.65ex]
\node (a) at (0,.25) {$\alpha$};
\node (b) at (.5,-.25) {$\beta$};
\node (bb) at (-.5,-.25) {$b$};
\node (s) at (-.2,-.95) {$*$};
\draw (a) edge (b) edge +(0,.5) edge (bb)
      (b) edge +(.5,-.5) edge +(-.3,-.5)
      (bb) edge +(-.5,-.5) edge (s);
\end{tikzpicture}
+
(\text{4 similar terms}).
\]

We use the adjunction properties of the cofree cooperad ${\op F}^c(\overline{\op P}{}^{\circ}[1])$
to get a coaugmentation morphism $\eta = \phi_u: \Com^c\to{\op F}^c(\overline{\op P}{}^{\circ}[1])$
from the map $u: \overline{\Com}{}^c\to\overline{\op P}{}^{\circ}[1]$
given with our $B\La$ operad $\op P$. (In fact, we already implicitly used this construction in Definition \ref{defn:BLambda-structures}
when we form the coaugmentation with values in the bar construction $B({\op P})$
associated to our map $u: \overline{\Com}{}^c\to\overline{\op P}{}^{\circ}[1]$.)

We easily check that the above corestriction operators and this coaugmentation morphism define a valid $\La$-cooperad structure
on the cofree cooperad $B({\op P})^{\circ} = {\op F}^c(\overline{\op P}{}^{\circ}[1])$.
We easily deduce from the constraints of Definition \ref{defn:BLambda-structures}
that the corestriction operators preserve the differential
of the bar construction too, whereas we already observed that the morphism $\eta = \phi_u$
which we associate to our map $u: \overline{\Com}{}^c\to\overline{\op P}{}^{\circ}[1]$
defines a morphism of dg cooperads with values in the bar construction $\eta = \phi_u: \Com^c\to B({\op P})$
as soon as we assume that this map satisfies the Maurer-Cartan equation
in the dg Lie algebra $\DerL = \Hom_{\dg\Sigma}(\overline{\Com}{}^c,\overline{\op P})$
of Proposition~\ref{prop:coderivation lie algebra}.
Hence, our construction provides the cobar construction $B({\op P})$ with a $\La$-cooperad structure
which is compatible with the differential structure
of our object.
\end{proof}

We record a consequence of the cofree $\La$-cooperad structure of this proposition in the following statement:

\begin{prop}\label{prop:morphisms to bar Lambda-cooperads}
We equip the bar construction $B({\op P})$ associated to a $B\La$ operad $\op P$ with the $\La$-cooperad structure
of Proposition~\ref{prop:BLambda-operad bar construction}.
We assume that $\op B$ is a dg $\La$-cooperad and we consider the morphism of dg cooperads $\phi: {\op B}\to B({\op P})$,
which we associate to a Maurer-Cartan element $\alpha_{\phi}$
of the dg Lie algebra $\DerL = \Hom_{\dg\Sigma}(\overline{\op B},\overline{\op P})$
in the correspondence of Proposition~\ref{prop:coderivation lie algebra}.
Then this morphism $\phi$ preserves the $\La$-cooperad structures associated to our objects
and hence defines a morphism $\phi: {\op B}\to B({\op P})$
in the category of dg $\La$-cooperads
if and only if the corresponding map of $\Sigma$-collections $\alpha_{\phi}: \overline{\op B}\to\overline{\op P}[1]$
\begin{itemize}
\item
preserves the covariant $\La$-diagram structure
associated to these collections (thus, this map belongs to the graded vector space $\Hom_{\gr\Lambda}(\overline{\op B},\overline{\op P})$
inside $\Hom_{\gr\Sigma}(\overline{\op B},\overline{\op P})$),
\item
and makes the following diagram commute
\[
\begin{tikzcd}
\overline{\Com}{}^c\ar{r}{=}\ar{d} & \overline{\Com}{}^c\ar{d}{u} \\
\overline{\op B} \ar{r}{\alpha_{\phi}} & \overline{\op P}[1]
\end{tikzcd}
\]
where we consider the morphism $\eta: \overline{\Com}{}^c\to\overline{\op B}$ induced by the coaugmentation of our dg $\La$-cooperad ${\op B}$.
\end{itemize}
\end{prop}

\begin{proof}
We mentioned in the proof of Proposition~\ref{prop:BLambda-operad bar construction}
that the cofree cooperad functor lifts to a cofree object functor
from the category of coaugmented covariant $\La$-collections
to the category of $\La$-cooperads.
We just use this general observation in order to establish that our morphism of dg cooperads $\phi: {\op B}\to B({\op P})$
defines a morphism of $\La$-cooperads
as soon as this morphism is induced by a map of coaugmented covariant $\La$-collections $\alpha_{\phi}: \overline{\op B}\to\overline{\op P}{}^c[1]$
in the correspondence of Proposition~\ref{prop:coderivation lie algebra}.
\end{proof}

We review examples of applications of the constructions of these propositions in the next paragraphs.

\begin{ex}\label{ex:BLambda-operads}
First, we can see that the cobar construction $\Omega({\op C})$ of a conilpotent dg $\La$-cooperad ${\op C}$
is equipped with a natural $B\La$ structure.

Indeed, recall that the cobar construction forms a free operad when we forget about differentials.
We explicitly have $\Omega({\op C})^{\circ} = {\op F}(\overline{\op C}{}^{\circ}[-1])$,
where we consider the degree shift of the coaugmentation coideal
of our cooperad $\overline{\op C}[-1]$.
We consider the restriction of the corestriction operators of our cooperad $\eta_S: {\op C}(S)\to{\op C}(S\sqcup\{*\})$
to the coaugmentation coideal $\overline{\op C}$,
by using the identity $\overline{\op C}(S) = \ker(\epsilon: {\op C}(S)\to{\op I}(S))$, where ${\op I}$ is the unit cooperad,
and the relation $\overline{\op C}(T) = {\op C}(T)$,
which is valid as soon as $|T|\geq 2$ (see Section \ref{subsec:Lambda-cooperads}).
We apply the derivation relation in our definition of a $B\La$ operad
to extend these reduced corestriction operators $\eta_S: \overline{\op C}(S)\to\overline{\op C}(S\sqcup\{*\})$
to our free operad $\Omega({\op C})^{\circ} = {\op F}(\overline{\op C}{}^{\circ}[-1])$.
We use the adjunction between the bar and the cobar construction to get a morphism $\phi_u: \Com^c\to B(\Omega{\op C})$,
corresponding to a map of $\La$-collections $u: \overline{\Com}{}^c\to\overline{\Omega({\op C})}$
of degree $-1$,
from the coaugmentation $\eta: \Com^c\to{\op C}$ of our $\La$-cooperad ${\op C}$.
We can more explicitly define this map $u: \overline{\Com}{}^c\to\overline{\Omega({\op C})}$ as the composite
of the map of $\La$-collections $\overline{\Com}{}^c\to\overline{\op C}[-1]$
yielded by our coaugmentation $\eta: \Com^c\to{\op C}$
with the canonical embedding of the object $\overline{\op C}[-1]$
in the free operad $\Omega({\op C})^{\circ} = {\op F}(\overline{\op C}{}^{\circ}[-1])$.
We immediately deduce from this construction that our map $u: \overline{\Com}{}^c\to\overline{\Omega({\op C})}$
intertwines the action of the corestriction operators
since we provide the free operad $\Omega({\op C})^{\circ} = {\op F}(\overline{\op C}{}^{\circ}[-1])$ with a $\La$-diagram structure
which lifts the one that we associate to the collection $\overline{\op C}{}[-1]$.

Recall that the differential of the cobar construction is defined, on the generating collection $\overline{\op C}{}^{\circ}[-1]$
of our object $\Omega({\op C})^{\circ} = {\op F}(\overline{\op C}{}^{\circ}[-1])$,
by the internal differential of the cooperad ${\op C}$
plus the reduced composition coproducts $\overline{\Delta}_*: \overline{\op C}(S)\rightarrow\overline{\op C}(S')\otimes\overline{\op C}(S'')$,
by using that the tensor products $\overline{\op C}(S')\otimes\overline{\op C}(S'')$,
which represent the target of these operations,
are identified with quadratic terms in the expansion of the free operad ${\op F}(\overline{\op C}{}^{\circ}[-1])$.
The corestriction operators commute with the internal differential of our cooperad by construction,
and our relation $d_{\op P}\circ\eta_S - \eta_S\circ d_{\op P} = \ad_b$
in the definition of a $B\La$ structure
actually reduces, on $\overline{\op C}{}^{\circ}[-1]\subset\Omega({\op C})^{\circ}$,
to our compatibility relation between the corestriction operators
and the reduced composition coproducts
of our $\La$-cooperad ${\op C}$
(see Section~\ref{subsec:Lambda-cooperads}).
This observation finishes the verification of the validity of our $B\La$ operad structure
on the cobar construction.

We can then use the result of Proposition~\ref{prop:BLambda-operad bar construction} to provide the bar-cobar construction $B(\Omega({\op C})$
with the structure of a dg $\La$-cooperad.
We can use the result of Proposition~\ref{prop:morphisms to bar Lambda-cooperads} to establish further that the quasi-isomorphism
of the bar duality $\phi: {\op C}\stackrel{\sim}{\to}B(\Omega({\op C}))$
forms a quasi-isomorphism of dg $\La$-cooperads,
because the natural map $\alpha_{\phi}: \overline{\op C}\to\overline{\Omega({\op C})}[1]$, which we associate to this morphism
and which is identified with the canonical embedding of the collection $\overline{\op C}{}^{\circ}[-1]$
in the free operad $\Omega({\op C})^{\circ} = {\op F}(\overline{\op C}{}^{\circ}[-1])$ (up to an obvious shift of degrees),
preserves the $\La$-diagram structure and the coaugmentation associated to our objects
by definition of our $B\La$ structure on the cobar construction.

Let us mention that this dg $\La$-cooperad structure on the bar-cobar construction $B(\Omega({\op C}))$ of a cooperad
is dual to the dg $\La$-operad structure of the cobar-bar resolution of a $\La$-operad
given in \cite[Propositions C.2.18]{Fr}.
\end{ex}

\begin{ex}\label{ex:BLambda-en-operads}
Recall that we use the notation $\e_n^c$ for the dual cooperad of the operad
\[
\e_n = \begin{cases} \Ass, & \text{if $n=1$}, \\
\Poiss_n, & \text{otherwise},
\end{cases}
\]
which represents the homology of the operad of little $n$-discs operad, $\e_n = H(\lD_n,\K)$.
Recall also that this cooperad is Koszul, with the $n$-fold operadic suspension of the operad $\e_n$
as Koszul dual operad. We explicitly have $(\e_n^c)^{\vee} = \e_n\{n\}$, for any $n\geq 1$.

We can check that the natural $B\La$ structure of the cobar construction $\Omega(\e_n^c)$
descends to this operad ${\op P} = \e_n\{n\}$
through the quasi-isomorphism $\kappa: \Omega(\e_n^c)\stackrel{\sim}{\to}\e_n\{n\}$
of the Koszul duality equivalence $(\e_n^c)^{\vee} = \e_n\{n\}$.
To be more explicit, we may check that the corestriction operators of the cobar construction restrict to trivial maps $\eta_S = 0$
on our object $(\e_n^c)^{\vee} = \e_n\{n\}$
through this quasi-isomorphism $\kappa: \Omega(\e_n^c)\stackrel{\sim}{\to}\e_n\{n\}$.
Then we just take the obvious prolongment of the coaugmentation of the cobar construction $u: \overline{\Com}{}^c\to\overline{\Omega(\e_n^c)}$
to get a coaugmentation map $u: \overline{\Com}{}^c\to\overline{\e_n\{n\}}$
and to complete the definition of this $B\La$ structure.
In fact, we can see that this coaugmentation map $u: \overline{\Com}{}^c\to\overline{\e_n\{n\}}$
which we associate to the operad $(\e_n^c)^{\vee} = \e_n\{n\}$
is null in arity $r\not=2$,
and is given by the obvious embedding $\Lie\{1\}(2)\subset\e_n\{n\}(2)$ in arity $r=2$,
where we use the identity $\Lie\{1\}(2) = \K = \Com^c(2)$
for the $1$-fold suspension of the Lie operad $\Lie$.
Thus, the element $b\in\e_n\{n\}(2)$, in our previous definitions, corresponds to the bracket operation
of the Poisson operad $\e_n = \Poiss_n$
in the case $n\geq 2$
and to the antisymmetrization of the generating product operation of the associative operad $\e_1 = \Poiss_1$
in the case $n=1$ (the antisymmetrization becomes a symmetrization when we take the operadic suspension).

We can again use the result of Proposition~\ref{prop:BLambda-operad bar construction} to provide the bar construction $B(\e_n\{n\})$
associated to this operad $(\e_n^c)^{\vee} = \e_n\{n\}$
with the structure of a dg $\La$-cooperad.
We moreover get that the cooperadic counterpart of the quasi-isomorphism $\kappa: \e_n^c\stackrel{\sim}{\to} B(\e_n\{n\})$
of the Koszul duality equivalence $(\e_n^c)^{\vee} = \e_n\{n\}$
forms a quasi-isomorphism of dg $\La$-cooperads.
We can again deduce this claim from the result of Proposition~\ref{prop:morphisms to bar Lambda-cooperads}
or from the observation that this morphism $\kappa: \e_n^c\stackrel{\sim}{\to} B(\e_n\{n\})$
is given by an obvious prolongment of the quasi-isomorphism of the bar duality $\phi: \e_n^c\stackrel{\sim}{\to} B(\Omega(\e_n^c))$
to the bar Koszul construction $B((\e_n^c)^{\vee}) = B(\e_n\{n\})$
through the already considered operadic version
of the Koszul duality quasi-isomorphism $\kappa: \Omega(\e_n^c)\stackrel{\sim}{\to}\e_n\{n\}$.

Let us mention that this dg $\La$-cooperad structure on the bar-Koszul construction $B(\e_n\{n\})$ of the cooperad $\e_n^c$
is dual to the dg $\La$-operad structure considered in \cite[Paragraph III.4.1.6]{Fr}
for the cobar-Koszul construction of the operad $\e_n = \Poiss_n$
in the case $n\geq 2$ (see also \cite[Proposition C.3.5]{Fr} for a generalization of this construction
to all Koszul operads).
\end{ex}

\begin{ex}\label{ex:extended BLambda-operads}
If ${\op P}$ is equipped with a $B\La$ structure, then so is the augmented dg operad ${\op P}_{\bo}$,
which we obtain by freely adjoining an operadic unit to ${\op P}$.
In short, we just extend the corestriction operators on $\overline{\op P}$ by zero on the extra unit of ${\op P}_{\bo}$
to provide $\overline{\op P}_{\bo}$ with a covariant $\La$-diagram structure
and we take the obvious composite $\overline{\Com}{}^c\to\overline{\op P}\to\overline{\op P}_{\bo}$
to prolong the coaugmentation $u: \overline{\Com}{}^c\to\overline{\op P}$ associated to our operad ${\op P}$
to this object ${\op P}_{\bo}$.

The morphism ${\op P}\to{\op P}_{\bo}$ trivially induces a morphism of dg $\La$-cooperads $B({\op P})\to B({\op P}_{\bo})$
when we apply the result of Proposition~\ref{prop:morphisms to bar Lambda-cooperads}
to this $B\La$ operad structure on ${\op P}_{\bo}$.
\end{ex}

\subsection{The dg Lie algebra structure of the coderivation complexes of dg $\La$-cooperads}\label{subsec:coderivation complexes:Lambda cooperads}
We now examine an extension of the results of Proposition~\ref{prop:coderivation lie algebra}
in the context of dg $\La$-cooperads.
In our applications, we actually deal with source objects $\op B$ and target objects $\op C$
of particular shapes in the category of dg $\La$-cooperads.
To be explicit, we use the following definitions.

\begin{defn}\label{defn:good Lambda-cooperads}
We first say that the dg $\La$-cooperad ${\op B}$ is a \emph{good source $\La$-cooperad} when the following properties hold.
\begin{itemize}
\item
The coaugmentation $\eta: \Com^c\to{\op B}$ admits a retraction in the category of dg $\La$-cooperads $\epsilon: {\op B}\to\Com^c$.
Thus, we have the relation
\[
{\op B} = \Com^c\oplus I{\op B}
\]
in the category of $\La$-collections, where we set $I{\op B} = \ker(\epsilon: {\op B}\to\Com^c)$.
\footnote{In the sequel, the dg coperad ${\op B}$ consists of a collection of augmented commutative algebras,
and $I{\op B}(r)$ represents the augmentation ideal of the dg algebra ${\op B}(r)$,
for each $r>0$.}
Note also that we trivially have $I{\op B}\subset\overline{\op B}$, and that the above decomposition relation
admits the following obvious restriction
\[
\overline{\op B} = \overline{\Com}{}^c\oplus I{\op B}
\]
when we consider the coaugmentation coideal of our cooperad $\overline{\op B}\subset{\op B}$.
\item
The $\La$-collection $I{\op B}$ is freely generated by a $\Sigma$-collection $\bS{\op B}$.
To be more explicit, we have the identity $I{\op B} = \La\otimes_{\Sigma}\bS{\op B}$,
where we use notation $\La\otimes_{\Sigma}-$
for the usual Kan extension functor from the category of $\Sigma$-collections
to the category of $\La$-collections.\footnote{In our applications, our free $\La$-collections have also a natural cosimplicial structure
and their generating $\Sigma$-collections are isomorphic to their conormalized parts.
This relationship motivates our notation $\bS$ for these objects.}
\end{itemize}
We secondly say that the dg $\La$-cooperad ${\op C}$ is a \emph{good target $\La$-cooperad}
when we have ${\op C} = B({\op P})$ for an augmented dg operad ${\op P}$
equipped with a $B\La$ structure.
\end{defn}

Note that for dg $\La$-cooperads ${\op B}$ and ${\op C}$ as in this definition we have a canonical map $*: {\op B}\to{\op C}$,
which is given by the composite
\[
{\op B}\to\Com^c\to{\op C},
\]
where we take the augmentation of the good source $\La$-cooperad ${\op B}$, followed by the coaugmentation of the cooperad ${\op C}$.

We can now establish our extension of the result of Proposition \ref{prop:coderivation lie algebra}.

\begin{prop}\label{prop:Lambda-coderivation Lie algebra}
If we assume that ${\op B}$ is a good source $\La$-cooperad and that ${\op C} = B({\op P})$ is a good target $\La$-cooperad in the sense
of Definition \ref{defn:good Lambda-cooperads},
then the morphisms of graded $\La$-cooperads $\phi: {\op B}^{\circ}\to{\op C}^{\circ}$
are in bijection with the morphisms of graded $\Sigma$-collections
\[
\alpha_{\phi}': \bS{\op B}^{\circ}\to\overline{\op P}{}^{\circ}[1]
\]
or, equivalently, with elements of degree $-1$ in the graded hom-object
of $\Sigma$-collections
\[
\DerL' = \Hom_{\gr\Sigma}(\bS{\op B},\overline{\op P})
\]
such as defined in Section~\ref{subsec:collections}.
This map $\alpha_{\phi}$, which we associate to a morphism of graded $\La$-cooperads $\phi: {\op B}^{\circ}\to{\op C}^{\circ}$,
is explicitly defined by taking by the projection of our morphism $\phi: {\op B}^{\circ}\to{\op C}^{\circ}$
onto the cogenerating collection of the cofree cooperad $B({\op P})^{\circ} = {\op F}^c(\overline{\op P}{}^{\circ}[1])$
on the target and its restriction to the generating collection $\bS{\op B}$
of the $\La$-collection $I{\op B}\subset{\op B}$
on the source.

The hom-object $\DerL' = \Hom_{\gr\Sigma}(\bS{\op B},\overline{\op P})$ moreover inherits a dg Lie algebra structure
such that the above correspondence restricts to a bijection between the set
of dg $\La$-cooperad morphisms $\phi: {\op B}\to{\op C}$
and the set of Maurer-Cartan elements
in $\DerL'$.
This dg Lie algebra structure on the hom-object $\DerL' = \Hom_{\gr\Sigma}(\bS{\op B},\overline{\op P})$
is actually identified with a restriction
of the twisted dg Lie algebra structure $\DerL^{\alpha_*}$
on the hom-object $\DerL = \Hom_{\gr\Sigma}(\overline{\op B},\overline{\op P})$
of Proposition \ref{prop:coderivation lie algebra},
where $\alpha_*\in\DerL$ is the Maurer-Cartan element of this hom-object which we associate to our canonical morphism
of dg $\La$-cooperads $*: {\op B}\to{\op C}$.
We then use the relations
\[
\Hom_{\gr\Sigma}(\bS{\op B},\overline{\op P})\cong\Hom_{\gr\Lambda}(I{\op B},\overline{\op P})
\subset\Hom_{\gr\Sigma}(I{\op B},\overline{\op P})
\subset\Hom_{\gr\Sigma}(\overline{\op B},\overline{\op P})
\]
to exhibit $\DerL'$ as a graded vector subspace of the latter graded vector space $\DerL$. (The first isomorphism of this sequence of relations
follows from the universal property of the free object $I{\op B} = \La\otimes_{\Sigma}\bS{\op B}$,
whereas we use the decomposition $\overline{\op B} = \overline{\Com}{}^c\oplus I{\op B}$
and the composition of maps $f: I{\op B}\to\overline{\op P}$
with the projection $\overline{\op B}\to I{\op B}$
to get the last inclusion relation.)

Besides, we have an isomorphism of dg vector spaces
\[
\CoDer_{\dg\La}({\op B}\stackrel{*}{\to}{\op C})[-1]\cong\DerL',
\]
where we consider the (degree shift of the) coderivation complex associated to our canonical morphism of dg $\La$-cooperads $*: {\op B}\to{\op C}$
on the left-hand side, and the underlying dg vector space of our dg Lie algebra $\DerL'$,
on the right hand side.
This isomorphism can be obtained as a restriction,
through the obvious inclusion $\CoDer_{\dg\La}({\op B}\stackrel{*}{\to}{\op C})\subset\CoDer_{\dg\Sigma}({\op B}\stackrel{*}{\to}{\op C})$
and the already considered inclusion of dg Lie algebras $\DerL'\subset\DerL^{\alpha_*}$,
of the isomorphism
\[
\CoDer_{\dg\Sigma}({\op B}\stackrel{*}{\to}{\op C})[-1]\cong\Hom_{\gr\Sigma}(\overline{\op B},\overline{\op P})^{\alpha_*},
\]
which is given by the result of Proposition~\ref{prop:coderivation lie algebra}
in the case $\phi = *$.
The shifted coderivation complex $\CoDer_{\dg\La}({\op B}\stackrel{*}{\to}{\op C})[-1]$ accordingly forms a dg Lie subalgebra
of the dg Lie algebra $\CoDer_{\dg\Sigma}({\op B}\stackrel{*}{\to}{\op C})[-1]$
considered in Proposition~\ref{prop:coderivation lie algebra}.
\end{prop}

\begin{proof}
By Proposition \ref{prop:coderivation lie algebra}, the morphisms of graded cooperads $\phi: {\op B}^{\circ}\to{\op C}^{\circ}$
are in bijection with the morphisms
of graded $\Sigma$-collections $\alpha_{\phi}: \overline{\op B}{}^{\circ}\to\overline{\op P}{}^{\circ}[1]$.
Proposition \ref{prop:morphisms to bar Lambda-cooperads} implies that the morphisms of graded $\La$-cooperads
correspond,
under this bijection, to the extension $\alpha_{\phi} = \eta + \alpha_{\phi}'$
of the morphisms of graded $\La$-collections $\alpha_{\phi}': I{\op B}{}^{\circ}\to\overline{\op P}{}^{\circ}[1]$,
where we consider the coaugmentation morphism $\eta: \overline{\Com}{}^c\to\overline{\op C}$
associated to the coaugmentation coideal of the dg $\La$-cooperad ${\op C}$.
Then we just use the free structure $I{\op B} = \La\otimes_{\Sigma}\bS{\op B}$
to get that any such morphism $\alpha_{\phi}': I{\op B}{}^{\circ}\to\overline{\op P}{}^{\circ}[1]$
in the category of graded $\La$-collections
is equivalent to a morphism $\alpha_{\phi}': \bS{\op B}{}^{\circ}\to\overline{\op P}{}^{\circ}[1]$
in the category of graded $\Sigma$-collections,
Equivalently, we use the isomorphism
of graded hom-objects $\Hom_{\gr\Sigma}(\bS{\op B},\overline{\op P})\cong\Hom_{gr\Lambda}(I{\op B},\overline{\op P})$,
already considered in our statement,
to get this map of $\Sigma$-collections equivalent to our morphism $\alpha_{\phi}': I{\op B}{}^{\circ}\to\overline{\op P}{}^{\circ}[1]$.
This proves the first claim of the proposition.

Note that the mapping $\alpha_{\phi}'\mapsto\alpha_{\phi} = \eta + \alpha_{\phi}'$, which we use in this correspondence,
is an affine map, as opposed to the inclusion of graded vector spaces $\DerL\subset\DerL'$,
considered in the statement of our proposition.
In short, this affine map differs from the latter inclusion $\DerL\subset\DerL'$ by the consideration
of the coaugmentation morphism $\eta: \overline{\Com}{}^c\to\overline{\op C}$
instead of the zero map on the summand $\overline{\Com}{}^c$
of the object $\overline{\op B} = \overline{\Com}{}^c\oplus I{\op B}$.
In what follows, we use that we can associate this coaugmentation morphism $\eta: \overline{\Com}{}^c\to\overline{\op C}$
to the map $\alpha_*$ that corresponds to our canonical morphism of dg $\La$-cooperads $\phi = *$
in the dg Lie algebra $\DerL$. We equivalently have the identities $\alpha_* = \eta\Leftrightarrow\alpha_*' = 0$
when we apply our correspondence to this morphism $\phi = *$.


We now examine the definition of the dg Lie algebra structure on the graded vector space $\DerL' = \Hom_{\gr\Sigma}(\bS{\op B},\overline{\op P})$.
We could proceed as in the proof of Proposition \ref{prop:coderivation lie algebra}
and deduce this construction from the general result of Proposition \ref{prop:geometric Linfty:structures}.
We then consider the dg hom-object of maps of $\La$-collections
\[
V' = \Hom_{\dg\La}(I{\op B},\overline{\op C})
\]
and the graded hom-object
\[
W' = \Hom_{\gr\La}(I{\op B},\overline{\op P}[1])
\]
in analogy with the dg hom-object $V = \Hom_{\dg\Sigma}(\overline{\op B},\overline{\op C})$
and the graded hom-object $W = \Hom_{\dg\Sigma}(\overline{\op B},\overline{\op P}[1])$
considered in the proof of Proposition \ref{prop:coderivation lie algebra}.
We also have
\[
W'\cong\underbrace{\Hom_{\gr\Sigma}(\bS{\op B},\overline{\op P})}_{= \DerL'}[1]
\]
by the already considered isomorphism
of hom-objects $\Hom_{\gr\Sigma}(\bS{\op B},\overline{\op P})\cong\Hom_{gr\Lambda}(I{\op B},\overline{\op P})$.

We still have an obvious morphism of graded vector spaces $\lambda: V'\to W'$, analogous to the map $\lambda: V\to W$
considered in the proof of Proposition \ref{prop:coderivation lie algebra},
which is induced by the projection onto the cogenerators ${\op C}^{\circ}\to\overline{\op P}{}^{\circ}[1]$
in the cofree cooperad ${\op C}^{\circ} = {\op F}^c(\overline{\op P}{}^{\circ}[1])$,
but we follow another route from this step on.
We are actually going to give an indirect definition of the other map $F': S(W')\to V'$,
which we need to complete the construction of Proposition \ref{prop:geometric Linfty:structures}.

Indeed, we claim in our statement that the outcome of our construction will make the object $\DerL'$
a dg Lie subalgebra of the twisted dg Lie algebra $\DerL^{\alpha_*}$
which we obtain by applying the results of Proposition \ref{prop:coderivation lie algebra}
to the morphism $\phi = *$.
Therefore, we construct our map $F': S(W')\to V'$ by using a restriction process from a twisted counterpart of the map $F: S(W)\to V$
which we use to determine the dg Lie algebra structure on $\DerL$.

To be explicit, we use the obvious embedding $V'\subset V$ (respectively, $W'\subset W$),
which identifies the elements of $V'$ (respectively, $W'$)
with the maps that preserve $\La$-diagram structures and vanish over the summand $\overline{\Com}{}^c$
of the object $\overline{\op B} = \overline{\Com}{}^c\oplus I{\op B}$
inside $V = \Hom_{\dg\Sigma}(\overline{\op B},\overline{\op C})$ (respectively, $W = \Hom_{\gr\Sigma}(\overline{\op B},\overline{\op P}[1])$).
Then we consider the map $F^{\alpha_*}: S(W)\to V$ such that $\Phi_{F^{\alpha_*},R}(\exp(w)) = \widetilde{\alpha_*+w} - \widetilde{\alpha_*}$,
for any element $w\in(W\otimes\overline{R})^0$, for any nilpotent graded ring $R$,
and where, with the notation of the proof of Proposition \ref{prop:coderivation lie algebra},
we consider the morphisms of cooperads $\widetilde{\alpha_*+w},\widetilde{\alpha_*}: {\op B}\to{\op C}$
associated to $\alpha_*+w\in(W\otimes R)^0$
and to $\alpha_*\in(W\otimes R)^0$. (We identify the element $\alpha_*\in W$ with $\alpha_*\otimes 1\in(W\otimes R)^0$
when we take the extension of scalars to the nilpotent ring $R$.)
We easily see that we have the implication $w\in(W'\otimes\overline{R})^0\Rightarrow\Phi_{F^{\alpha_*},R}(\exp(w))\in(V'\otimes\overline{R})^0$
from which we deduce that this map $F^{\alpha_*}: S(W)\to V$
admits a restriction through our embeddings $W'\subset W$
and $V'\subset V$.
We accordingly have a commutative diagram:
\[
\begin{tikzcd}
S(W')\ar{r}{\subset}\ar{d}{\exists F'} & S(W)\ar{d}{F^{\alpha_*}} \\
V'\ar{r}{\subset}\ar{d}{d} & V\ar{d}{d} \\
V'\ar{r}{\subset}\ar{d}{\lambda'} & V\ar{d}{\lambda} \\
W'\ar{r}{\subset} & W
\end{tikzcd},
\]
where $F'$ denotes this map that we obtain by restriction from $F^{\alpha_*}$.
We easily check that the vertical composite on the right hand side of this diagram is the map of graded vector spaces $\pi D^{\alpha}: S(W)\to W$
associated to the coderivation $D^{\alpha}: S(W)\to S(W)$
that reflects the structure of the twisted dg Lie algebra $\DerL^{\alpha_*}$.
We conclude that this twisted dg Lie algebra structure on $\DerL = W[-1]$
admits a restriction to the graded vector space $\DerL' = W'[-1]$,
which therefore forms a dg Lie subalgebra of $\DerL^{\alpha_*}$,
as expected.

To complete the verifications of the claims of our proposition, we can argue as in the proof of Proposition \ref{prop:coderivation lie algebra}
in order to establish that we have an isomorphism
of graded vector spaces
\[
\CoDer_{\dg\La}({\op B}\stackrel{*}{\to}{\op C})\cong\Hom_{\gr\La}(I{\op B},\overline{\op P}[1])
\]
which carries any coderivation $\theta\in\CoDer_{\dg\La}({\op B}\stackrel{*}{\to}{\op C})$
to its projection onto the generating collection of the dg cooperad ${\op C}$
and its restriction to the object $I{\op B}$.
We immediately see that our inclusion of graded vector spaces $W'\subset W$
corresponds to the inclusion of the graded vector spaces
of coderivations $\CoDer_{\dg\La}({\op B}\stackrel{*}{\to}{\op C})\subset\CoDer_{\dg\Sigma}({\op B}\stackrel{*}{\to}{\op C})$
when we forget about differentials.
We can now deduce the last assertions of our proposition
from this observation,
from the definition of our dg Lie algebra structure on $\DerL' = W'[-1]$
by restriction of the twisted dg Lie algebra structure $\DerL^{\alpha_*}$ on $\DerL = W[-1]$,
and from the identity of dg vector spaces $\CoDer_{\dg\Sigma}({\op B}\stackrel{*}{\to}{\op C})[-1] = \DerL^{\alpha_*}$
given by the results of Proposition \ref{prop:coderivation lie algebra}.
\end{proof}


\begin{rem}\label{rem:extended Lambda-coderivation Lie algebra}
The observations of Example \ref{ex:extended BLambda-operads} imply that ${\op C}^+ = B({\op P}_{\bo})$
forms a good target $\La$-cooperad
as soon as ${\op C} = B({\op P})$
is so.
Hence, we can apply the above proposition to this cooperad ${\op C}^+$.
In this case, our results imply that the graded hom-object $\DerL' = \Hom_{\gr\Sigma}(\bS{\op B},{\op P})$ inherits a dg Lie algebra structure
and that the dg vector space of extended coderivations $\xCoDer_{\dg\La}({\op B}\stackrel{*}{\to}{\op C})[-1]$
(with the usual degree shift) is isomorphic to the underlying dg vector space
of this dg Lie algebra structure.
\end{rem}

We record the following comparison result between the extended coderivation dg Lie algebra
and the ordinary coderivation dg Lie algebra.

\begin{lemm}\label{lemm:extended Lambda-coderivation complex:acyclicity}
We assume that ${\op B}$ is a good source $\La$-cooperad and that ${\op C}$
is a good target $\La$-cooperad as in Proposition~\ref{prop:Lambda-coderivation Lie algebra}.
We assume further that $\bS{\op B}(1)$ is acyclic as a dg vector space.
The canonical morphism
\[
\CoDer_{\dg\La}({\op B}\stackrel{*}{\to}{\op C})\to\CoDer_{\dg\La}({\op B}\stackrel{*}{\to}{\op C}^+)
= \xCoDer_{\dg\La}({\op B}\stackrel{*}{\to}{\op C}),
\]
induced by the composition of coderivations with the morphism $\mathit{id}^+: {\op C}\to{\op C}^+$,
is a quasi-isomorphism.
\end{lemm}

\begin{proof}
We endow $\CoDer_{\dg\La}({\op B}\stackrel{*}{\to}{\op C})$ and $\xCoDer_{\dg\La}({\op B}\stackrel{*}{\to}{\op C})$
with the descending complete filtration
that we deduce from the isomorphisms
of graded vector spaces $\CoDer_{\dg\La}({\op B}\stackrel{*}{\to}{\op C})^{\circ}\cong\Hom_{\gr\Sigma}(\bS{\op B},\overline{\op P})$
and $\xCoDer_{\dg\La}({\op B}\stackrel{*}{\to}{\op C})^{\circ}\cong\Hom_{\gr\Sigma}(\bS{\op B},{\op P})$
and from the natural arity filtration $F^m\Hom_{\gr\Sigma}({\op M},{\op N}) = \prod_{r\geq m}\Hom_{gr\Sigma_r}({\op M}(r),{\op N}(r))$
of hom-objects
in the category of graded $\Sigma$-collections
$\Hom_{\gr\Sigma}(\op M,\op N) = \prod_{r\in\N}\Hom_{\gr\Sigma_r}({\op M}(r),{\op N}(r))$.
We easily check that the differential of $\CoDer_{\dg\La}({\op B}\stackrel{*}{\to}{\op C})$
(respectively, $\xCoDer_{\dg\La}({\op B}\stackrel{*}{\to}{\op C})$)
preserves this filtration, and that the $E^0_{m *}$ term of the corresponding spectral sequence is identified with the dg hom-object
of $\Sigma_m$-equivariant maps $\Hom_{\dg\Sigma_m}(\bS{\op B}(m),\overline{\op P}(m))$
(respectively, $\Hom_{\dg\Sigma_m}(\bS{\op B}(m),{\op P}(m))$
with the notation of Section \ref{subsec:collections}.
We immediately deduce from the acyclicity of the dg vector space $\bS{\op B}(1)$
that the morphism $\mathit{id}^+: {\op C}\to{\op C}^+$
induces an isomorphism on the $E^1$ page
of this spectral sequence. The conclusion follows.
\end{proof}

For later reference, we want to give a name to the dg Lie algebra constructed in this section.

\begin{defn}\label{defn:canonical Lambda-coderivation complex}
For ${\op B}$ a good source $\La$-cooperad and ${\op C}$ a good target $\La$-cooperad, we adopt the notation
\[
\CoDer_{\dg\La}({\op B},{\op C}) := \CoDer_{\dg\La}({\op B}\stackrel{*}{\to}{\op C})[-1]
\]
for the shifted coderivation complex associated to the canonical map $*: {\op B}\to{\op C}$,
equipped with the dg Lie algebra structure
that we deduce from the result of Proposition~\ref{prop:Lambda-coderivation Lie algebra}.
We similarly set
\[
\xCoDer_{\dg\La}({\op B},{\op C}) := \xCoDer_{\dg\La}({\op B}\stackrel{*}{\to}{\op C})[-1],
\]
when we deal with the complex of extended coderivations.
\end{defn}

Note that we drop the degree shift in the notation of this definition.

We may note that $\CoDer_{\dg\La}({\op B},{\op C})$ satisfies obvious bifunctoriality properties
and so does the extended coderivation complex $\xCoDer_{\dg\La}({\op B},{\op C})$.
We record the following homotopy invariance properties of these bifunctors:

\begin{prop}\label{prop:Lambda-coderivation complex homotopy invariance}
If $\phi: {\op B}\stackrel{\simeq}{\to}{\op B}'$ is a quasi-isomorphism of good source $\La$-cooperads,
and ${\op C}$ is a good target $\La$-cooperad,
then $\phi$ induces a quasi-isomorphism at the coderivation complex level:
\[
\phi^*: \CoDer_{\dg\La}({\op B}',{\op C})\stackrel{\simeq}{\to}\CoDer_{\dg\La}({\op B},{\op C}),
\]
and similarly when we pass to the extended coderivation complex.
If we symmetrically assume that ${\op B}$ is a good source $\La$-cooperad, and $\psi: {\op C}\to{\op C}'$
is a morphism of good target $\La$-cooperads ${\op C} = B({\op P})$, ${\op C}' = B({\op P}')$
induced by a structure preserving quasi-isomorphism of $B\La$-operads $\psi: {\op P}\stackrel{\simeq}{\to}{\op P}'$,
then $\psi$ also induces a quasi-isomorphism at the coderivation complex level:
\[
\psi_*: \CoDer_{\dg\La}({\op B},{\op C})\stackrel{\simeq}{\to}\CoDer_{\dg\La}({\op B},{\op C}'),
\]
and similarly when we pass to the extended coderivation complex.
%
\end{prop}

We can actually establish the homotopy invariance properties of the coderivation complex for morphisms
of good target $\La$-cooperads $\psi: {\op C}\to{\op C}'$
of a more general form,
but we only use the case stated in this proposition.

\begin{proof}
This proposition follows from spectral sequence arguments. To be explicit, we can provide
the coderivation complex $\CoDer_{\dg\La}({\op B},{\op C})\cong\DerL'$ associated to any pair $({\op B},{\op C})$,
where ${\op B}$ is a good source $\La$-cooperad and ${\op C}$ is a good target $\La$-cooperad,
with the complete descending filtration
such that
\[
\mF^k\DerL' = \prod_{r\geq k}\Hom_{\gr\Sigma_r}(\bS{\op B}(r),\overline{\op P}(r)),
\]
where we use the identity $\DerL' = \Hom_{\gr\Sigma}(\bS{\op B},\overline{\op P})$
Then we may check that the differential of the dg Lie algebra $\DerL'$ reduces to the map induced by the internal differential
of the collection $\bS{\op B}$ and of the operad ${\op P}$
on the hom-object $E^0_k\DerL' = \Hom_{\gr\Sigma_k}(\bS{\op B}(k),\overline{\op P}(k))$,
which defines the $E^0$-page of the spectral sequence
associated to this filtration.

If $\phi: {\op B}\to{\op B}'$ is a quasi-isomorphism of good source $\La$-cooperads, and ${\op C} = B({\op P})$ is a good target $\La$-cooperad,
then we immediately get that $\phi$ preserves our filtration, and induces an isomorphism on the $E^1$-page
of the associated spectral sequence. Thus, we conclude that $\phi$
induces a quasi-isomorphism on our coderivation complexes which represent the abutment
of this spectral sequence. We use similar arguments when we have a morphism $\psi: {\op C}\to{\op C}'$
of good target $\La$-cooperads ${\op C} = B({\op P})$, ${\op C}' = B({\op P}')$
induced by a structure preserving quasi-isomorphism of $B\La$-operads $\psi: {\op P}\stackrel{\simeq}{\to}{\op P}'$.
%
\end{proof}

\subsection{Deformation complexes of cooperads with coefficients in bicomodules}\label{subsec:coderivation complexes:bicomodules}
In \cite[Sections 1.3.2, 1.4.2]{FW}, we consider deformation complexes with coefficients in bi-comodules over cooperads,
which give generalizations of the coderivation complexes considered in the previous subsections.
We use this construction later on in this article.

Let ${\op C}$ be a dg $\La$-cooperad. Briefly recall that dg $\La$-bicomodule over ${\op C}$
is a $\La$-collection in dg vector spaces ${\op M}$
equipped with coproduct operations $\Delta_i: {\op M}(k+l-1)\to{\op M}(k)\otimes{\op C}(l)$
and $\Delta_i: {\op M}(k+l-1)\to{\op C}(k)\otimes{\op M}(l)$
which satisfy a natural extension of the equivariance, counit and coassociativity properties
attached to the composition coproducts of a cooperad (see \cite{FW} for details).

The deformation complex of ${\op C}$ with coefficients in ${\op M}$
is defined by
\[
\CoDef_{\dg\La}({\op M},{\op C}) = (\Hom_{\dg\La}({\op M},\overline\Omega({\op C})),\partial),
\]
where $\overline{\Omega}({\op C})$ is the augmentation ideal of the cobar construction of our dg cooperad,
whereas $\partial$ denotes a differential,
determined by the coaction of the cooperad ${\op C}$ on ${\op M}$,
which we add to the natural differential of the dg hom-object $\Hom_{\dg\La}({\op M},\overline{\Omega}({\op C})$.

In the case where ${\op C} = B({\op P})$ is the bar construction of an augmented dg operad ${\op P}$
equipped with a $B\La$ structure,
then we can replace the cobar construction $\Omega({\op C})$
in the above definition by the object ${\op P}$.
We then use the canonical quasi-isomorphism of dg operads $\Omega({\op C})\to{\op P}$
to extend the twisting differential of our deformation complex $\partial$
to the hom-object $\Hom_{\dg\La}({\op M},{\op P})$.
We adopt the following notation for this complex
\[
K({\op M},{\op P}) = (\Hom_{\dg\La}({\op M},\overline{\op P}),\partial).
\]
We have in particular:
\[
\CoDef_{\dg\La}({\op M},{\op C})\cong K({\op M},\Omega({\op C})).
\]
We note that these objects are functorial in both arguments in some natural sense.

We can identify the cooperadic coderivation complexes of the previous section
with particular examples of such deformation complexes.
We can consider the general case where we have a good source $\La$-cooperad ${\op B}$
together with a morphism of dg $\La$-cooperads $\phi: {\op B}\to{\op C}$ (not necessarily the canonical one $\phi = *$).
We use this morphism to provide the dg $\La$-collection $I{\op B}$
with the structure of a dg $\La$-bicomodule over ${\op C}$.
We still assume that ${\op C}$ is a good target $\La$-cooperad, so that we have ${\op C} = B({\op P})$
for some augmented dg operad equipped with a $B\La$ structure ${\op P}$.
We can generalize the isomorphism construction of Proposition~\ref{prop:Lambda-coderivation Lie algebra}
to identify the coderivation complex $\CoDer_{\dg\La}({\op B}\stackrel{\phi}{\to}B({\op P}))$
with a twisted version of the dg hom-object $\Hom_{\dg\La}(I{\op B},\overline{\op P})$.
We actually have:
\[
\CoDer_{\dg\La}({\op B}\stackrel{\phi}{\to} B({\op P}))[-1]\cong K(I{\op B},{\op P}).
\]
We similarly get:
\[
\xCoDer_{\dg\La}({\op B}\stackrel{\phi}{\to}B({\op P}))[-1]\cong K(I{\op B},{\op P}_{\bo})
\]
when we consider the extended coderivation complex.

%
%

\section{Hopf cooperads and biderivations}\label{sec:biderivation complexes}
Recall that a Hopf operad is an operad in the category of cocommutative counital coalgebras.
We dually define a Hopf cooperad as a cooperad in the category of unital commutative algebras.
To be precise, we still work in the base category of dg vector spaces in what follows
and therefore we mainly consider dg Hopf cooperads which are Hopf cooperads
in the category of unital commutative dg algebras.
For our purpose, we also deal with dg Hopf $\La$-cooperads,
which are dg Hopf cooperads equipped with a compatible $\La$-structure.

The goal of this section is to extend the results of the previous section about the complexes of coderivations
to the Hopf cooperad setting. We then consider complexes of biderivations formed by maps
which are both coderivations with respect to the cooperad structure
and derivations with respect to the commutative algebra
structure.
We review the definition of this notion in the first subsection of this section.
We then prove that the complexes of biderivations associated to good dg Hopf cooperads
and to good dg Hopf $\La$-cooperads inherit a dg Lie algebra structure
just like the complexes of coderivations associated to good dg cooperads
and to good dg $\La$-cooperads
which we considered in the previous section.
We devote the second and the third subsections of this section to this subject.
We then examine the definition of deformation complexes with coefficients in bicomodules
that generalize the deformation complexes
of dg cooperads studied in the previous section.
We tackle this subject in the fourth subsection.

\subsection{Biderivations of Hopf cooperads}\label{subsec:biderivation complexes:biderivations}
Let $\op B$ and $\op C$ be dg Hopf cooperads and let $\phi: \op B\to\op C$ be a morphism of dg Hopf cooperads.
Let ${\op B}^{\circ}$ (respectively, ${\op C}^{\circ}$) denote the graded Hopf cooperad underlying ${\op B}$ (respectively, $\op C$)
which we obtain by forgetting about the differential of our object (as in Section \ref{subsec:coderivation complexes:coderivations}).
We again consider the algebra of dual numbers $\K_{\epsilon} := \K[\epsilon]/(\epsilon^2)$,
where $\epsilon = \epsilon_{-k}$ is a formal variable of degree $-k$,
and we define a degree $k$ biderivation of the morphism $\phi$
as map of graded $\Sigma$-collections $\theta: {\op B}^{\circ}\to{\op C}^{\circ}$
of degree $k$ such that the map
\[
\phi+\epsilon\theta: {\op B}^{\circ}\otimes\K_{\epsilon}\to{\op C}^{\circ}\otimes\K_{\epsilon}
\]
defines a morphism of graded Hopf cooperads over $\K_\epsilon$.
We readily see that $\theta$ is a biderivation if $\theta$ forms a cooperad coderivation in the sense
of Section \ref{subsec:coderivation complexes:coderivations}
when we forget about the commutative algebra structure of our Hopf cooperads
and each map $\theta: {\op B}^{\circ}(r)\to{\op C}^{\circ}(r)$
defines a derivation of the commutative algebra morphism $\phi: {\op B}(r)\to{\op C}(r)$
when we forget about cooperad structures.

Recall that the unit morphisms of the commutative algebras ${\op B}(r)$ (respectively, ${\op C}(r)$)
are equivalent to a cooperad morphism $\eta: \Com^c\to{\op B}$ (respectively, $\eta: \Com^c\to{\op C}$)
whose restriction to the unit cooperad ${\op I}$
defines the canonical coaugmentation $\eta: {\op I}\to{\op B}$ (respectively, $\eta: {\op I}\to{\op C}$)
associated to our object.
We may observe that the derivation relation with respect to the commutative algebra products
implies that our biderivation satisfies $\theta(1) = 0$ (as usual for a derivation),
and hence automatically preserves this cooperad coaugmentation,
as we require for a coderivation of cooperads
in general (see Section \ref{subsec:coderivation complexes:coderivations}).

The graded vector space of biderivations inherits a differential such that $d\theta = d_{\op C}\circ\theta-(-1)^k\theta\circ d_{\op B}$
as usual,
and we adopt the notation
\[
\BiDer_{\dg\Sigma}({\op B}\stackrel{\phi}{\lo}{\op C})
\]
for this dg vector space of biderivations which we associate to a morphism of Hopf dg cooperads $\phi: {\op B}\to{\op C}$.
%
%
%
If the Hopf cooperads $\op B$ and $\op C$ carry a $\La$-structure and the map $\phi: {\op B}\to{\op C}$
defines a morphism of Hopf $\La$-cooperads, then we also consider the dg vector space
\[
\BiDer_{\dg\La}({\op B}\stackrel{\phi}{\lo}{\op C})\subset\BiDer_{\dg\Sigma}({\op B}\stackrel{\phi}{\lo}{\op C})
\]
formed by the biderivations that intertwine the action of the category $\Lambda$ on our objects.
These biderivations obviously form coderivations of $\La$-cooperads
in the sense of Section \ref{subsec:Lambda-cooperads}. Note simply, again, that the derivation relation
with respect to the products implies that a biderivation
of Hopf $\La$-cooperads
automatically vanishes over the coaugmentation morphism $\eta: \Com^c\to{\op B}$
that we associate to the Hopf cooperad ${\op B}$.
Note also that the map $\phi+\epsilon\theta: {\op B}^{\circ}\otimes\K_{\epsilon}\to{\op C}^{\circ}\otimes\K_{\epsilon}$
defines a morphism of graded Hopf $\La$-cooperads over $\K_\epsilon$.
when $\theta$ is a biderivation of Hopf $\La$-cooperads.

\subsection{The dg Lie algebra structure of the biderivation complexes of dg Hopf cooperads}\label{subsec:biderivation complexes:plain Hopf cooperadsHopf cooperads}
We now assume that the dg Hopf cooperad ${\op C}$
in the construction of the previous subsection
is given, as a dg cooperad,
by the bar construction
\[
{\op C} = B({\op P})
\]
of an augmented dg operad ${\op P}$. We also restrict ourselves to the case where the dg Hopf cooperad ${\op B}$
is given by a Chevalley-Eilenberg cochain complex
\[
{\op B} = C({\alg g})
\]
of a cooperad in the category of graded Lie coalgebras ${\alg g}$.

To be explicit, we assume that ${\alg g}$
is a $\Sigma$-collection of graded Lie coalgebras ${\alg g}(r)$, $r>0$,
equipped with graded Lie algebra morphisms $\Delta_i: {\alg g}(m+n-1)\to{\alg g}(m)\oplus{\alg g}(n)$
which fulfill an additive analogue of the equivariance, counit and coassociativity properties
of the composition coproducts of cooperads.
If we forget about the Lie coalgebra structure of our objects, then we can identify these additive
composition coproducts
with the sum of the left $\Delta_i: {\alg g}(m+n-1)\to\Com^c(m)\otimes{\alg g}(n)$
and right coproduct operation $\Delta_i: {\alg g}(m+n-1)\to{\alg g}(m)\otimes\Com^c(n)$
of a $\Com^c$-bicomodule structure on the $\Sigma$-collection
underlying ${\alg g}$.
To each graded Lie coalgebra ${\alg g}(r)$, we then associate the Chevalley-Eilenberg complex $C({\alg g}(r))$,
which is the graded symmetric algebra $S({\alg g}(r)[-1])$,
generated by the desuspension of the graded vector space underlying ${\alg g}(r)$,
together with the differential induced by the Lie cobracket
on this graded vector space of generators.
The collection $C({\alg g})$ of these commutative dg algebras $C({\alg g}(r))$, $r>0$,
inherits a Hopf cooperad structure
with the composition coproducts
\[
C({\alg g}(m+n-1))\stackrel{(\Delta_i)_*}{\to}C({\alg g}(m)\oplus{\alg g}(n))\cong C({\alg g}(m))\otimes C({\alg g}(n))
\]
induced by the cocomposition coproduct operations of the graded Lie coalgebra cooperad ${\alg g}$.
This Hopf cooperad ${\op B} = C({\alg g})$
is equipped with a canonical augmentation over the commutative cooperad $\epsilon: C({\alg g})\to\Com^c$
that is induced by the zero map ${\alg g}\to 0$
at the Lie coalgebra level. If we take the composite of this morphism with the unit $\eta: \Com^c\to{\op C}$
of our dg Hopf cooperad, then we get a canonical morphism $*: C({\alg g})\to{\op C}$
that connects our dg Hopf cooperads ${\op B} = C({\alg g})$
and ${\op C}$.

We then have the following extension of the results of Proposition \ref{prop:coderivation lie algebra}.

\begin{prop}\label{prop:biderivation Lie algebra}
If we assume ${\op C} = B({\op P})$ and ${\op B} = C({\alg g})$ as above,
then the morphisms of graded Hopf cooperads $\phi: {\op B}^{\circ}\to{\op C}^{\circ}$
are in bijection with the morphisms of graded $\Sigma$-collections
\[
\alpha_{\phi}: {\alg g}[-1]\to\overline{\op P}{}^{\circ}[1]
\]
or, equivalently, with elements of degree $-1$ in the graded hom-object
of $\Sigma$-collections
\[
\DerL = \Hom_{\gr\Sigma}({\alg g}[-1],\overline{\op P}).
\]
This map $\alpha_{\phi}$, which we associate to a morphism of graded Hopf cooperads $\phi: {\op B}^{\circ}\to{\op C}^{\circ}$,
is explicitly defined by taking by the projection of our morphism $\phi: {\op B}^{\circ}\to{\op C}^{\circ}$
onto the cogenerating collection of the cofree cooperad $B({\op P})^{\circ} = {\op F}^c(\overline{\op P}{}^{\circ}[1])$
on the target and its restriction to the subobject ${\alg g}[-1]\subset C({\alg g})$
on the source.

This hom-object $\DerL = \Hom_{\gr\Sigma}(\overline{\op B},\overline{\op P})$ moreover inherits a dg Lie algebra structure
such that the above correspondence restricts to a bijection between the set
of morphisms of dg Hopf cooperads $\phi: {\op B}\to{\op C}$
and the set of Maurer-Cartan elements
in $\DerL$.

Besides, we have an isomorphism of dg vector spaces
\[
\BiDer_{\dg\Sigma}({\op B}\stackrel{*}{\to}{\op C})[-1]\cong\DerL,
\]
where we consider the (degree shift of the) biderivation complex associated to the canonical morphism $*: C({\alg g})\to{\op C}$
on the left-hand side and the underlying dg vector space
of our dg Lie algebra $\DerL$
on the right-hand side.
We consequently have a dg Lie algebra structure on this shifted biderivation complex.
\end{prop}

\begin{proof}
First note that the $\Sigma$-collection $C^{\leq 1}(\alg g)\subset C({\alg g})^{\circ}$
formed by the components of weight $\leq 1$
of the symmetric algebras $C(\alg g(r))^{\circ} = S({\alg g}(r)[-1])$
forms a subobject of $C({\alg g})^{\circ}$
in the category of graded cooperads. We actually have an identity $C^{\leq 1}(\alg g) = \Com^c\oplus{\alg g}[-1]$,
where we take the left $\Delta_i: {\alg g}(m+n-1)\to\Com^c(m)\otimes{\alg g}(n)$
and right coproduct operation $\Delta_i: {\alg g}(m+n-1)\to{\alg g}(m)\otimes\Com^c(n)$
equivalent to the additive cooperad composition coproducts
of our Lie coalgebra cooperad ${\alg g}$
to prolong the natural composition coproducts of the commutative cooperad $\Com^c$
to this object $C^{\leq 1}(\alg g)$.

Let $\alpha: {\alg g}[-1]\to\overline{\op P}{}^{\circ}[1]$ be a morphism of graded $\Sigma$-collections.
We use the morphism $\overline{\Com}{}^c\to\overline{\op P}{}^{\circ}[1]$
induced by the canonical coaugmentation of the dg Hopf cooperad ${\op C}$
to extend $\alpha$ to the object $C^{\leq 1}(\alg g)$.
We then use the cofree structure of the graded cooperad ${\op C}^{\circ} = {\op F}^c(\overline{\op P}{}^{\circ}[1])$
to get a morphism of graded cooperads $\psi: C^{\leq 1}(\alg g)\to{\op C}^\circ$
extending $\alpha$. This morphism preserves the coaugmentation
over the commutative cooperad $\Com^c$
by construction.
We can now extend $\psi$ to a map $\phi: C({\alg g})^\circ\to{\op C}^\circ$
by using the formula
\[
\phi(x_1\cdots x_n) := \psi(x_1)\cdots\psi(x_n)\in{\op C}(r),
\]
for any product of generators $x_1,\dots,x_n\in{\alg g}(r)$, so that $\phi$ defines a morphism of commutative algebras
in each arity $r>0$.
We claim that this map $\phi$ also preserves the (augmentation and) the composition coproducts of our Hopf cooperads
and hence forms a morphism in the category of graded Hopf cooperads.
Indeed, this is the case when we restrict ourselves to the subobject $C^{\leq 1}({\alg g})\subset C({\alg g})^{\circ}$
by construction of $\phi$.
Now, let $\Delta_T$ be a $k$-fold composition coproduct shaped on a tree $T$
with $r$ leafs and $k$ vertices,
and let again $x_1,\dots,x_n$ denote a collection of elements
of the Lie coalgebra $\alg g(r)$.
By the compatibility between the composition coproducts and the products
in a Hopf cooperad, we get the relations:
\begin{align*}
\phi^{\otimes k}(\Delta_T(x_1\cdots x_n)) & = \phi^{\otimes k}(\Delta_T(x_1)\cdots\Delta_T(x_n)) \\
& = \psi^{\otimes k}(\Delta_T(x_1))\cdots\psi^{\otimes k}(\Delta_T(x_n)) \\
& = \Delta_T(\psi(x_1))\cdots\Delta_T(\psi(x_n)) \\
& = \Delta_T(\psi(x_1)\cdots\psi(x_n)) = \Delta_T(\phi(x_1\cdots x_n)),
\end{align*}
and this verification proves the validity of our claim for any element $x_1\cdots x_n\in C({\alg g})$
of our Hopf cooperad $C({\alg g})$.
We easily check that this construction enables use to retrieve every graded Hopf operad morphism $\phi: C({\alg g})^\circ\to{\op C}^\circ$
from its restriction to the Lie algebra cooperad ${\alg g}$ and its composite with the projection onto the cogenerating collection
of the bar cooperad $B({\op P})^{\circ} = {\op F}^c(\overline{\op P}[1])$.

Note that we can carry out this construction of a morphism of graded Hopf cooperads $\phi$
from a morphism of graded collections $\alpha$
over any graded ground ring $R$.
In particular, we can work over the ground ring $\K_\epsilon$ and argue as in the proof of Proposition~\ref{prop:coderivation lie algebra}
to establish that we have an isomorphism of graded vector spaces
at the level of the space of biderivations
associated to our morphism $\phi$:
\[
\BiDer_{\dg\Sigma}({\op B}\stackrel{\phi}{\to}{\op C})\cong\Hom_{\gr\Sigma}({\alg g}[-1],\overline{\op P}[1]).
\]

We now examine the definition of a dg Lie algebra structure on the graded vector space $\DerL = \Hom_{\gr\Sigma}({\alg g}[-1],\overline{\op P})$.
We rely on the result of Proposition~\ref{prop:geometric Linfty:structures} as in the proof of Proposition \ref{prop:coderivation lie algebra}.
We set again $V := \Hom_{\dg\Sigma}(\overline{\op B},\overline{\op C})$,
and
\[
W := \underbrace{\Hom_{\gr\Sigma}({\alg g}[-1],\overline{\op P})}_{= \DerL}[1] = \Hom_{\gr\Sigma}({\alg g}[-1],\overline{\op P}[1]).
\]
We now consider the map $\lambda: V\to W$ defined by taking the restriction of maps to the collection ${\alg g}[-1]$
on the source ${\op B} = C({\alg g})$
and the projection onto the cogenerating collection of our object $\overline{\op P}[1]$
on the target ${\op C} = B({\op P})$.
We define the map $F: S(W)\to V$ by using the polarization trick of Section \ref{subsec:linear maps remark}.
We just adapt the construction of Proposition \ref{prop:coderivation lie algebra}.
We now regard an element $w\in(W\otimes\overline{R})^0$, where $R$ is any nilpotent graded ring,
as a morphism of graded $\Sigma$-collections $w: {\alg g}[-1]\otimes R\to\overline{\op P}{}^{\circ}[1]\otimes R$
such that $w({\alg g}[-1]\otimes\overline{R})\subset\overline{\op P}{}^{\circ}[1]\otimes\overline{R}$.
We then set $\Phi_{F,R}(\exp(w)) = \widetilde{w}$, where $\widetilde{w}: {\op B}^\circ\otimes R\to{\op C}^\circ\otimes R$
is the morphism of graded Hopf cooperads
associated to $w$.

We immediately check that our maps fulfill the projection condition of Proposition \ref{prop:geometric Linfty:structures}
and we can use the same arguments as in the proof of Proposition \ref{prop:coderivation lie algebra}
to check the tangentiality condition.
We can therefore use the construction of Proposition \ref{prop:geometric Linfty:structures}
to produce an $L_{\infty}$~structure on $\DerL = W[-1]$.

We claim again that this $L_{\infty}$~structure reduces to a dg Lie algebra structure.
To extract the $n$-ary $L_{\infty}$~product of the elements $w_1,\dots,w_n\in W$ in $\DerL = W[-1]$,
we have to polarize the expression
\beq{eq:temp5}
\pi\circ d F(\exp w)\circ\iota = \pi\circ d_{\op C}\circ F(\exp(w))\circ\iota - \pi\circ F(\exp(w))\circ d_{\op B})\circ\iota
\eeq
where we set $w = \sum_i\epsilon_i w_i$ (see Section \ref{subsec:linear maps remark}),
we again regard this element as a map $w: {\alg g}[-1]\otimes R\to{\overline P}{}^{\circ}[1]\otimes R$,
and we use the notation $\pi$ for the projection onto the collection $\overline{\op P}{}^{\circ}[1]$
in the bar construction $B({\op P})^{\circ} = {\op F}^c(\overline{\op P}{}^{\circ}[1])$,
while $\iota$ is the inclusion of the collection ${\alg g}[-1]$
in the Chevalley-Eilenberg cochain complex $C({\alg g})^{\circ} = S({\alg g}[-1])$.

We deduce from our definition of the map $\phi(\exp(w)) = \widetilde{w}$ that the terms of our expression \eqref{eq:temp5}
involving $k$ factors $w_i$
occur exactly when we apply $F(\exp w)$ to a $k$ fold tensor products
of elements of ${\alg g}$.
We deduce from this observation that the first term $\pi\circ d_{\op C}\circ F(\exp(w))\circ\iota$
of our expression only produces a linear expression
of these maps $w_i$,
whereas the second term $\pi\circ F(\exp(w))\circ d_{\op B})\circ\iota$
produces a quadratic expression,
since the Chevalley-Eilenberg differential produces $2$-fold tensors
from ${\alg g}$. (Recall that ${\alg g}$ is defined in the category of graded Lie coalgebras by assumption,
and hence has a trivial differential.)
We therefore conclude that the $n$-ary $L_{\infty}$~product of the elements $w_1,\dots,w_n\in W$
vanish for $n>2$, as expected.

We can argue as in the proof of Proposition \ref{prop:coderivation lie algebra}
to check that the Maurer-Cartan elements
of this dg Lie algebra $\DerL$
correspond to the morphisms of dg Hopf cooperads $\phi: {\op B}\to{\op C}$.

We already observed that the vector space of biderivations associated to a morphism $\BiDer_{\dg\Sigma}({\op B}\stackrel{\phi}{\to}{\op C})$
is isomorphic to $W = \DerL[1]$ as a graded vector space.
We now assume $\phi = *$.
We immediately get from the above analysis of the $L_{\infty}$~structure on $\DerL$
that the differential of an element in this dg Lie algebra $w\in\DerL$,
which is given by the linear part of this $L_{\infty}$~structure, is yielded by the differential $d_{\op C}$
of the cooperad ${\op C} = B({\op P})$
after taking the cooperad morphism $\psi: C^{\leq 1}({\alg g})\to B({\op P})^{\circ}$
associated to $\alpha = w$.
We easily check that this differential agrees with the differential
of the biderivation associated to $w$
in the biderivation complex $\BiDer_{\dg\Sigma}({\op B}\stackrel{*}{\to}{\op C})$ (just observe
that the morphism $\psi: C^{\leq 1}({\alg g})\to B({\op P})^{\circ}$
is given by the same formula as this biderivation
on ${\alg g}[-1]\subset{\op B}$).
\end{proof}

\begin{rem}\label{rem:biderivation dg Lie structure formulas}
For the sake of concreteness, we extract an explicit expression for the differential and bracket
on the dg Lie algebra of this proposition (modulo signs).

Let $\alpha\in\DerL$. We observed in the proof of our proposition that the linear part of our $L_{\infty}$~algebra
structure, which determines the differential $d\alpha$ of the element $\alpha$ in $\DerL$,
is given by the composite $\pi\circ d_{\op C}\circ\psi\circ\iota$,
where $\psi: C^{\leq 1}(\alg g)\to{\op C}^\circ$
is the cooperad morphism
associated to our map $\alpha: {\alg g}[-1]\to\overline{\op P}[1]$,
and $\pi\circ d_{\op C}$ denotes the composite of the differential of the bar construction ${\op C} = B({\op P})$
with the canonical projection $\pi: {\op F}^c(\overline{\op P}^{\circ}[1])\to\overline{\op P}^{\circ}[1]$,
while $\iota$ denotes the obvious embedding of the collection ${\alg g}[-1]$
into $C^{\leq 1}({\alg g})$.
We therefore have to identify this cooperad morphism $\psi: C^{\leq 1}(\alg g)\to{\op C}^\circ$
first.

Let $x\in{\alg g}(r)[-1]$. We again use the expansion of the cofree cooperad $B({\op P})^{\circ} = {\op F}^c(\overline{\op P}^{\circ}[1])$
as a direct sum of treewise tensor products ${\op F}^c_T(\overline{\op P}^{\circ}[1])$
as in the proof of Proposition~\ref{prop:BLambda-operad bar construction}.
We still use the notation $y_T$ for the component of an element $y\in{\op F}^c(\overline{\op P}^{\circ}[1])(S)$
in this treewise tensor product ${\op F}^c_T(\overline{\op P}^{\circ}[1])$
for a tree $T\in\mT_S$, where $S$ is any fixed finite set.
We assume that the tree $T$ has $k$ vertices, numbered from $1$ to $k$.
We then have
\beq{eq:temp6}
\psi(x)_T = \left(\sum_{j=1}^k\pi_*^{\otimes j-1}\otimes\alpha\otimes\pi_*^{\otimes k-j}\right)\overline{\Delta}_T(x),
\eeq
where we consider the reduced treewise coproduct of $x$ in the cooperad $C^{\leq 1}({\alg g})$,
and $\pi_*$ denotes the composite $C^{\leq 1}({\alg g})\to\Com^c\to{\op C}\to\overline{\op P}[1]$.
We may note that the treewise composition coproduct $\overline{\Delta}_T(x)$
is a sum of tensors with exactly one factor in ${\alg g}$,
whereas the remaining factors belong to $\Com^c$.
We apply our map $\alpha$ to this distinguished factor, whereas we apply the canonical map $\Com^c\to{\op C}\to\overline{\op P}[1]$
to the other factors of our treewise tensor product.

Recall that the projection onto $\overline{\op P}[1]$ of the differential $d_{\op C}$ of the bar construction ${\op C} = B({\op P})$
is induced by the internal differential of the operad ${\op P}$
on the linear part of the cofree cooperad $B({\op P})^c = {\op F}^c(\overline{\op P}[1])$,
which we associate to the trees with one vertex (the corollas),
and by composition products of ${\op P}$
on the quadratic part,
which we associate to the trees with two vertices.
To form this quadratic part of the differential, we use that the structure of trees with two vertices $T$
with leafs indexed by $S$
can be determined by giving a partition $S = (S'\setminus\{*\})\sqcup S''$
on which we shape the composition products
of our operad.
For such a tree $T$, we can also identify the value of the treewise composition coproduct $\Delta_T$
on the collection ${\alg g}$
with the sum of the coproduct operations $\Delta_*: {\alg g}(S)\to{\alg g}(S')\otimes\Com^c(S'')$
and $\Delta_*: {\alg g}(S)\to\Com^c(S')\otimes{\alg g}(S'')$
that make this collection ${\alg g}$
into a bicomodule over the commutative cooperad.
Eventually, we get the formula
\[
d\alpha = \pm d_{\op P}\circ\alpha + \sum_{S = S'\setminus\{*\}\amalg S''}(\pi_*\otimes\alpha + \alpha\otimes\pi_*)\Delta_*
\]
for the value of our differential on $\alpha\in\DerL$, where we consider these coproduct operations $\Delta_*$
associated to ${\alg g}$.

For the bracket $[\alpha_1,\alpha_2]$ of elements $\alpha_1,\alpha_2\in\DerL$, we get the formula
\[
[\alpha_1,\alpha_2] = \pi\circ\mu\circ(\alpha_1\otimes\alpha_2)\circ\Delta,
\]
where $\Delta$ is the cobracket on $\alg g$ and $\mu$ is the commutative algebra product in ${\op C}$.
\end{rem}



In analogy to Definition \ref{defn:canonical Lambda-coderivation complex}, we introduce the following notation.

\begin{defn}\label{defn:canonical biderivation complex}
For ${\op B}$, ${\op C}$ as in Proposition \ref{prop:biderivation Lie algebra}, we adopt the notation
\[
\BiDer_{\dg\Sigma}({\op B},{\op C}) := \BiDer_{\dg\Sigma}({\op B}\stackrel{*}{\to}{\op C})[-1]
\]
for the shifted biderivation complex associated to the canonical map $*: {\op B}\to{\op C}$, equipped with the dg Lie algebra structure
that we deduce from the result of Proposition \ref{prop:biderivation Lie algebra}.
\end{defn}

\begin{rem}
We may see that a variant of Proposition \ref{prop:biderivation Lie algebra} holds
if ${\op B}$ is just free as Hopf sequence
and ${\op C}$ is cofree as a cooperad,
but we get a general $L_{\infty}$~algebra structure on our object in this case
instead of a dg Lie algebra structure. (We shall not use this fact in the present paper.)
\end{rem}

\subsection{The dg Lie algebra structure of the biderivation complexes of dg Hopf $\La$-cooperads}\label{subsec:biderivation complexes:Hopf Lambda-cooperads}
We now address an extension of the results of Proposition~\ref{prop:biderivation Lie algebra}
in the context of dg Hopf cooperads equipped with a $\La$-structure.
We again restrict ourselves to source objects $\op B$ and target objects $\op C$
of particular shapes when we work in the category of dg Hopf $\La$-cooperads.
We record these conditions in the following definition.

\begin{defn}\label{defn:good Hopf Lambda-cooperads}
We first say that a dg Hopf $\La$-cooperad ${\op B}$ is a \emph{good source Hopf $\La$-cooperad} when the following properties hold.
\begin{itemize}
\item
We have ${\op B} = C({\alg g})$, for a $\La$-cooperad in graded Lie coalgebras $\alg g$.\footnote{We obviously define a $\La$-cooperad
in graded Lie coalgebras as a cooperad in graded Lie coalgebras
equipped with a $\La$-diagram structure,
such that the corestriction operators $\eta_S: {\alg g}(S)\to{\alg g}(S\sqcup\{*\})$
satisfy an obvious additive analogue of the compatibility relations of Section \ref{subsec:Lambda-cooperads}
with respect to the composition coproduct operations
of our object. We also take the zero map as a coaugmentation $\eta: 0\to{\alg g}(r)$
when we use this definition.
We may easily see that the category of $\La$-cooperads in graded Lie coalgebras
is isomorphic to the subcategory of the category of cooperads
in graded Lie coalgebras whose component of arity zero
is the zero object.}
We then require that the $\La$-diagram structure of the object ${\op B}$
is induced by the $\La$-diagram structure of the $\La$-cooperad $\alg g$
by the functoriality of the Chevalley-Eilenberg complex $C({\alg g})$.
\item
The $\La$-cooperad ${\alg g}$ is, as a $\La$-collection, freely generated by a $\Sigma$-collection $\bS{\alg g}$.
To be more explicit, we have the identity ${\alg g} = \La\otimes_{\Sigma}\bS{\alg g}$,
where we again consider the Kan extension functor $\La\otimes_{\Sigma}-$
from the category of $\Sigma$-collections
to the category of $\La$-collections.
\end{itemize}
We secondly say that a dg Hopf $\La$-cooperad ${\op C}$ is a \emph{good target Hopf $\La$-cooperad}
when ${\op C}$ is, as a dg $\La$-cooperad, given by the bar construction ${\op C} = B({\op P})$
of an augmented dg operad ${\op P}$ equipped with a $B\La$ structure.
\end{defn}

We can now establish our extension of the result of Proposition \ref{prop:biderivation Lie algebra}.

\begin{prop}\label{prop:Lambda-biderivation Lie algebra}
If we assume that ${\op B} = C(\alg g)$ is a good source Hopf $\La$-cooperad
and that ${\op C} = B({\op P})$ is a good target Hopf $\La$-cooperad
in the sense of the above definition,
then the morphisms of graded Hopf $\La$-cooperads $\phi: {\op B}^{\circ}\to{\op C}^{\circ}$
are in bijection with the morphisms of graded $\Sigma$-collections
\[
\alpha_{\phi}'': \bS{\alg g}[-1]\to\overline{\op P}{}^{\circ}[1]
\]
or, equivalently, with elements of degree $-1$ in the graded hom-object
of $\Sigma$-collections
\[
\DerL'' = \Hom_{\gr\Sigma}(\bS{\alg g}[-1],\overline{\op P}).
\]
This map $\alpha_{\phi}''$, which we associate to a morphism of graded Hopf $\La$-cooperads $\phi: {\op B}^{\circ}\to{\op C}^{\circ}$,
is explicitly defined by taking by the projection of our morphism $\phi: {\op B}^{\circ}\to{\op C}^{\circ}$
onto the cogenerating collection of the cofree cooperad $B({\op P})^{\circ} = {\op F}^c(\overline{\op P}{}^{\circ}[1])$
on the target side and its restriction to the generating collection $\bS{\alg g}[-1]$
of the $\La$-collection ${\alg g}[-1]\subset C({\alg g})$
on the source side.

The hom-object $\DerL'' = \Hom_{\gr\Sigma}(\bS{\alg g}[-1],\overline{\op P})$ moreover inherits a dg Lie algebra structure
such that the above correspondence restricts to a bijection between the set
of dg Hopf $\La$-cooperad morphisms $\phi: {\op B}\to{\op C}$
and the set of Maurer-Cartan elements
in $\DerL''$.

This dg Lie algebra structure on the hom-object $\DerL'' = \Hom_{\gr\Sigma}(\bS{\alg g}[-1],\overline{\op P})$
is actually identified with a restriction
of the dg Lie algebra structure $\DerL$
on the hom-object $\DerL = \Hom_{\gr\Sigma}({\alg g}[-1],\overline{\op P})$
of Proposition \ref{prop:biderivation Lie algebra}.
We then use the relations
\[
\Hom_{\gr\Sigma}(\bS{\alg g}[-1],\overline{\op P})\cong\Hom_{\gr\Lambda}({\alg g}[-1],\overline{\op P})
\subset\Hom_{\gr\Sigma}({\alg g}[-1],\overline{\op P})
\]
to exhibit $\DerL''$ as a graded vector subspace of the latter graded vector space $\DerL$. (The first isomorphism of this sequence of relations
follows from the universal property of the free object ${\alg g} = \La\otimes_{\Sigma}\bS{\alg g}$.)

Besides, we have an isomorphism of dg vector spaces
\[
\BiDer_{\dg\La}({\op B}\stackrel{*}{\to}{\op C})[-1]\cong\DerL'',
\]
where we consider the (degree shift of the) biderivation complex associated to the canonical morphism
of dg Hopf $\La$-cooperads $*: {\op B}\to{\op C}$
on the left-hand side, and the underlying dg vector space of our dg Lie algebra $\DerL''$,
on the right hand side.
This isomorphism can be obtained as a restriction,
through the obvious inclusion $\BiDer_{\dg\La}({\op B}\stackrel{*}{\to}{\op C})\subset\BiDer_{\dg\Sigma}({\op B}\stackrel{*}{\to}{\op C})$
and the already considered inclusion of dg Lie algebras $\DerL''\subset\DerL$,
of the isomorphism
\[
\BiDer_{\dg\Sigma}({\op B}\stackrel{*}{\to}{\op C})[-1]\cong\DerL
\]
given by the result of Proposition~\ref{prop:BLambda-operad bar construction}.
The shifted biderivation complex $\BiDer_{\dg\La}({\op B}\stackrel{*}{\to}{\op C})[-1]$ accordingly forms a dg Lie subalgebra
of the dg Lie algebra $\BiDer_{\dg\Sigma}({\op B}\stackrel{*}{\to}{\op C})[-1]$
considered in Proposition~\ref{prop:BLambda-operad bar construction}.
\end{prop}

\begin{proof}
We rely on the result of Proposition \ref{prop:biderivation Lie algebra} which gives a bijection between the set of morphisms
of graded Hopf cooperads $\phi: {\op B}^\circ\to{\op C}^\circ$
and the set of morphisms of graded $\Sigma$-collections $\alpha_{\phi}: {\alg g}[-1]\to\overline{\op P}{}^{\circ}[1]$.

We immediately see that $\alpha_{\phi}: {\alg g}[-1]\to\overline{\op P}{}^{\circ}[1]$
defines a morphism of graded $\La$-collections
if $\phi: {\op B}^\circ\to{\op C}^\circ$
is a morphism of graded Hopf $\La$-cooperads,
because we obtain this map $\alpha_{\phi}$ by taking the composite of our morphism $\phi$
with the canonical inclusion ${\alg g}[-1]\subset C({\alg g})$
on the source,
with the projection $\overline{\op C}{}^{\circ}\to\overline{\op P}{}^{\circ}[1]$
on the target,
and both maps define morphisms of graded $\La$-collections
by definition of our $\La$-structure on the cooperads ${\op B} = C({\alg g})$
and ${\op C}^{\circ} = {\op F}^c(\overline{\op P}{}^{\circ}[1])$.\footnote{Note in particular
that the map $\overline{\op C}{}^{\circ}\to\overline{\op P}{}^{\circ}[1]$
is identified with the augmentation of the adjunction between $\La$-cooperads and coaugmented $\La$-collections
that we use in the proof of Proposition \ref{prop:biderivation Lie algebra}.}

We assume, in the converse direction, that the map $\alpha_{\phi}: {\alg g}[-1]\to\overline{\op P}{}^{\circ}[1]$
is a morphism of graded $\La$-collections.
We revisit the construction of Proposition \ref{prop:biderivation Lie algebra} to check that $\phi: {\op B}^\circ\to{\op C}^\circ$
forms a morphism of graded Hopf $\La$-cooperads
in this case.
We immediately get that the extension $\alpha_{\phi}': \overline{\Com}{}^c\oplus{\alg g}[-1]\to\overline{\op P}{}^{\circ}[1]$
of our map $\alpha_{\phi}$ to the coaugmentation coideal of the cooperad $C^{\leq 1}({\alg g}) = \Com^c\oplus{\alg g}[-1]$
defines a morphism of $\La$-collections that preserves the coaugmentation of our objects
over the commutative cooperad.
We deduce from this observation and the result of Proposition \ref{prop:morphisms to bar Lambda-cooperads} that the morphism
of graded cooperads $\psi: C^{\leq 1}({\alg g})\to B({\op P})^{\circ}$
which we associate to this map in the construction of Proposition \ref{prop:biderivation Lie algebra}
preserves $\La$-cooperad structures
and then we readily check that this is also the case of the morphism $\phi: C({\alg g})^{\circ}\to{\op C}^{\circ}$
which we obtain by using the free commutative algebra structure
of the Chevalley-Eilenberg complex $C({\alg g})$.

We conclude from these verifications that the correspondence of Proposition \ref{prop:biderivation Lie algebra}
restricts to a bijection between the set of morphisms
of graded Hopf $\La$-cooperads $\phi: {\op B}^\circ\to{\op C}^\circ$
and the set of morphisms of graded $\La$-collections $\alpha_{\phi}: {\alg g}[-1]\to\overline{\op P}{}^{\circ}[1]$.
We just use the universal property of the $\La$-collection ${\alg g} = \La\otimes_{\Sigma}\bS{\alg g}$
to deduce the assertion of our proposition
from this result.

To establish the next assertions of our proposition, we can adapt the arguments given in the proof of Proposition \ref{prop:Lambda-coderivation Lie algebra},
where we establish parallel results about the hom-object $\DerL' = \Hom_{\gr\Sigma}(\bS{\op B},\overline{\op P})$
associated to a good source $\La$-cooperad ${\op B}$
and to a good target $\La$-cooperad ${\op C} = B({\op P})$.

To be explicit, in parallel to the dg hom-object $V  = \Hom_{\dg\Sigma}(\overline{\op B},\overline{\op C})$
considered in the proof of Proposition \ref{prop:biderivation Lie algebra},
we consider the dg hom-object of maps of $\La$-collections
\[
V'' = \Hom_{\dg\La}(I{\op B},\overline{\op C})
\]
We may note that the object $I{\op B}$ in the expression of this dg vector space $V''$
is given by the augmentation ideal of the Chevalley-Eilenberg cochain
complex ${\op B}(r) = C({\alg g}(r))$
in each arity $r>0$.
We explicitly have $I{\op B}^{\circ}(r) = S^+({\alg g}(r)[-1])$,
where $S^+(-)$ denotes the symmetric algebra without unit
as in Section \ref{subsec:geometric Linfty:structures}.
We again consider the embedding $V'\subset V$ which identifies the elements of $V'$
with the maps that preserve $\La$-diagram structures and vanish over the summand $\overline{\Com}{}^c$
of the object $\overline{\op B} = \overline{\Com}{}^c\oplus I{\op B}$
inside $V = \Hom_{\dg\Sigma}(\overline{\op B},\overline{\op C})$.

In the proof of Proposition \ref{prop:biderivation Lie algebra},
we also consider the graded hom-object such that $W = \Hom_{\gr\Sigma}({\alg g}[-1],\overline{\op P}[1])$.
We now set
\[
W'' = \Hom_{\gr\La}({\alg g}[-1],\overline{\op P}[1])
\]
and we consider the obvious embedding $W''\subset W$.
Note that we also have an isomorphism
\[
W''\cong\underbrace{\Hom_{\gr\Sigma}(\bS{\alg g},\overline{\op P})}_{= \DerL'}[1],
\]
where we consider the hom-object of our proposition,
by the already considered universal property of the $\La$-collection ${\alg g} = \La\otimes_{\Sigma}\bS{\alg g}$.

Then we consider the map $F^0: S(W)\to V$ such that $\Phi_{F^0,R}(\exp(w)) = \widetilde{w} - \widetilde{0}$,
for any element $w\in(W\otimes\overline{R})^0$, for any nilpotent graded ring $R$,
and where, with the notation of the proof of Proposition \ref{prop:biderivation Lie algebra},
we consider the morphisms of dg Hopf cooperads $\widetilde{w},\widetilde{0}: {\op B}\to{\op C}$
associated to $w\in(W\otimes R)^0$
and to the zero map $0\in(W\otimes R)^0$. Note that this morphism $\widetilde{0}: {\op B}\to{\op C}$
is actually identified with the canonical morphism $*: {\op B}\to\Com^c\to{\op C}$
which we associate to our objects.
Now, we deduce from our first verifications that the morphism $\widetilde{w}$
preserves the $\La$-cooperad structures
when $w$ is a morphism of graded $\La$-collections.
In particular, the restriction of $\widetilde{w}$ to the collection $\Com^c\subset{\op B}$
is given by the canonical unit morphism $\eta: \Com^c\to{\op C}$
of the Hopf cooperad ${\op C}$.
We deduce from these observations that we have the implication
$w\in(W''\otimes\overline{R})^0\Rightarrow\Phi_{F^0,R}(\exp(w))\in(V'\otimes\overline{R})^0$,
from which we obtain that our map $F^0: S(W)\to V$
admits a restriction $F'': S(W'')\to V''$.
We can now conclude, by the same arguments as in the proof of Proposition~\ref{prop:Lambda-coderivation Lie algebra},
that our dg Lie algebra structure on $\DerL = W[-1]$
admits a restriction to $\DerL'' = W''$. Note simply that the translation operation $\Phi_{F^0,R}(\exp(w)) = \widetilde{w} - \widetilde{0}$,
which is equivalent to a change of basepoint in the definition of the map $F: S(W)\to V$
of Proposition~\ref{prop:biderivation Lie algebra} (see Remark~\ref{rem:geometric Linfty:structures:interpretation}),
does not modify our dg Lie algebra structure on $W$.

We also use the same arguments as in the proof of Proposition~\ref{prop:Lambda-coderivation Lie algebra}
to check the rest of our assertions,
about the biderivation complex of the morphism $*: {\op B}\to{\op C}$
in the category of dg Hopf $\La$-cooperads.
\end{proof}


The correspondence of Proposition~\ref{prop:Lambda-biderivation Lie algebra}
between the dg Lie structure on $\DerL''$ and the dg Lie structure
of Proposition~\ref{prop:biderivation Lie algebra}
implies that we can still use the observations of Remark~\ref{rem:biderivation dg Lie structure formulas}
to get the explicit expression of the differential and of the Lie bracket on our dg Lie algebra.

\begin{rem}\label{rem:Lambda-biderivation Lie algebra filtration}
In what follows, we often deal with the situation where our $\La$-cooperad in graded Lie coalgebras ${\alg g}$
is equipped with a weight decomposition ${\alg g}(r) = \bigoplus_{m\geq 1}{\alg g}_m(r)$ in each arity $r>0$
which is preserved by the structure operations
attached to our object.
To be explicit, we assume that the Lie cobracket $\Delta: {\alg g}(r)\to{\alg g}(r)\otimes{\alg g}(r)$ carries ${\alg g}(r)_m$
into the sum $\bigoplus_{p+q=m}{\alg g}(r)_m$ for each weight $m\geq 1$. We similarly assume
that the action of permutations $\sigma: {\alg g}(r)\to{\alg g}(r)$, the corestriction operators $u_*: {\alg g}(k)\to{\alg g}(l)$
and the composition coproducts $\Delta_i: {\alg g}(k+l-1)\to{\alg g}(k)\oplus{\alg g}(l)$
define weight preserving morphisms on our objects.
We also require that the $\Sigma$-collection $\bS{\alg g}$ which generates our object ${\alg g}$
forms a subobject of homogeneous weight inside ${\alg g}$.
We explicitly have $\bS{\alg g}(r) = \bigoplus_{m\geq 1}\bS{\alg g}_m(r)$ for each arity $r>0$,
where we set $\bS{\alg g}_m(r) = \bS{\alg g}(r)\cap{\alg g}_m(r)$.

In this situation, we can equip the dg Lie algebra $\DerL''$ constructed in the previous proposition
with a complete weight grading $\DerL'' = \prod_{m\geq 1}\DerL''_m$
in the sense of Definition~\ref{defn:complete graded Linfty algebras}.
To be explicit, in order to define this structure, we can use the expression $\DerL'' = \Hom_{\gr\Sigma}(\bS{\alg g},\overline{\op P})$
of our dg Lie algebra $\DerL''$.
Then we merely define $\DerL''_m$ as the graded vector space formed by the homomorphisms $f\in\Hom_{\gr\Sigma}(\bS{\alg g},\overline{\op P})$
that vanish on the components of weight $n\not=m$
of the $\Sigma$-collection $\bS{\alg g}(r) = \bigoplus_{m\geq 1}\bS{\alg g}_m(r)$.
We easily deduce from the description of Remark~\ref{rem:biderivation dg Lie structure formulas}
that this complete grading $\DerL'' = \prod_{m\geq 1}\DerL''_m$
is preserved by the differential and the Lie bracket
of our dg Lie algebra structure
on $\DerL''$,
as we require in the definition of a complete weight graded dg Lie algebra.

Let us mention that we can extend the weight grading of our graded Lie coalgebras ${\alg g}(r)$
to the Chevalley-Eilenberg complex $C({\alg g}(r))$ by assuming that a product $x_1\cdots x_r$
of elements of homogeneous weight $x_1\in{\alg g}_{m_1}(r),\dots,x_1\in{\alg g}_{m_n}(r)$
defines an element of weight $m_1+\dots+m_n$
in $C({\alg g}(r))$.
The component of weight $m$ of our dg Lie algebra actually corresponds to the vector subspace
of the biderivation complex $\BiDer_{\dg\La}(C({\alg g})\stackrel{*}{\to}{\op C})$
spanned by the biderivations that vanish on the components of weight $n\not=m$
of the Chevalley-Eilenberg complex $C({\alg g}(r))$.
This observation follows again from a straightfoward inspection
of our constructions in Remark~\ref{rem:biderivation dg Lie structure formulas}.
\end{rem}

We again introduce a shorthand notation for the dg Lie algebra constructed in Proposition~\ref{prop:Lambda-biderivation Lie algebra}:

\begin{defn}\label{defn:canonical Lambda-biderivation complex}
For ${\op B}$, ${\op C}$ as in Proposition \ref{prop:Lambda-biderivation Lie algebra}, we adopt the notation
\[
\BiDer_{\dg\La}({\op B},{\op C}) := \BiDer_{\dg\La}({\op B}\stackrel{*}{\to}{\op C})[-1]
\]
for the shifted biderivation complex associated to the canonical map $*: {\op B}\to{\op C}$,
equipped with the dg Lie algebra structure
that we deduce from the result of Proposition \ref{prop:Lambda-biderivation Lie algebra}.
\end{defn}

We may note again that this biderivation dg Lie algebra $\BiDer_{\dg\La}({\op B},{\op C})$ satisfies obvious bifunctoriality properties.
We record the following homotopy invariance property of this bifunctor:

\begin{prop}\label{prop:Lambda-biderivation complex homotopy invariance}
If ${\op B} = C({\alg g})$ be a good source Hopf $\La$-cooperad, and $\psi: {\op C}\to{\op C}'$ is a morphism
of good target Hopf $\La$-cooperads ${\op C} = B({\op P})$, ${\op C}' = B({\op P}')$
induced by a structure preserving quasi-isomorphism of $B\La$-operads $\psi: {\op P}\stackrel{\simeq}{\to}{\op P}'$,
then $\psi$ induces a quasi-isomorphism at the biderivation complex level:
\[
\psi_*: \BiDer_{\dg\La}(C({\alg g}),{\op C})\stackrel{\simeq}{\to}\BiDer_{\dg\La}(C({\alg g}),{\op C}').
\]
In the situation of Remark~\ref{rem:Lambda-biderivation Lie algebra filtration},
where we have a complete weight graded structure on our biderivation dg Lie algebras,
we also get that our quasi-isomorphism preserves the weight grading, and hence
induces a weak-equivalence of simplicial sets at the nerve level
\[
\psi_*: \MC_{\bullet}(\BiDer_{\dg\La}({\op B},{\op C}))\stackrel{\simeq}{\to}\MC_{\bullet}(\BiDer_{\dg\La}({\op B},{\op C}'))
\]
by the observations of Theorem~\ref{thm:GoldmanMillson}.
\end{prop}

We may again establish the homotopy invariance properties of the biderivation complex for morphisms
of good target Hopf $\La$-cooperads $\psi: {\op C}\to{\op C}'$
of a more general form,
but we only use the case stated in this proposition.

\begin{proof}
This proposition follows from spectral sequence arguments, as in the case of the parallel result
of Proposition~\ref{prop:Lambda-coderivation complex homotopy invariance}.
To be explicit, we provide the biderivation complex $\BiDer_{\dg\La}({\op B},{\op C})\cong\DerL''$
associated to any pair $({\op B},{\op C})$,
where ${\op B}$ is a good source Hopf $\La$-cooperad ${\op B} = C({\alg g})$
and ${\op C}$ is a good target Hopf $\La$-cooperad ${\op C} = B({\op P})$,
with the complete descending filtration
such that
\[
\mF^k\DerL'' = \prod_{r\geq k}\Hom_{\gr\Sigma_r}(\bS{\alg g}(r),\overline{\op P}(r)),
\]
where we use the identity $\DerL'' = \Hom_{\gr\Sigma}(\bS{\alg g},\overline{\op P})$
Then we readily deduce from the explicit constructions of Remark~\ref{rem:biderivation dg Lie structure formulas}
that the differential of this dg Lie algebra $\DerL''$
reduces to the map induced by the internal differential of the operad ${\op P}$
on the hom-object $E^0_k\DerL'' = \Hom_{\gr\Sigma_k}(\bS{\alg g}(k),\overline{\op P}(k))$,
which defines the $E^0$-page of the spectral sequence
associated to this filtration.

We immediately see that a morphism $\psi: {\op C}\to{\op C}'$ of good target Hopf $\La$-cooperads ${\op C} = B({\op P})$, ${\op C}' = B({\op P}')$
induced by a structure preserving quasi-isomorphism of $B\La$-operads $\psi: {\op P}\stackrel{\simeq}{\to}{\op P}'$
preserve the filtration.
We immediately get that $\psi$ induces an isomorphism on the $E^1$-page of our spectral sequence too,
and hence induces a quasi-isomorphism at the biderivation complex level.

The last assertion of the proposition is immediate from the definition of our weight decomposition
in Remark~\ref{rem:Lambda-biderivation Lie algebra filtration}
in the case where ${\alg g}$ is equipped with a weight grading.
%
\end{proof}

We also want to be able to compare the complexes of biderivations and of coderivations.
To this end we show the following result.

\begin{prop}\label{prop:biderivations to coderivations forgetful morphism}
Let ${\op B} = C(\alg g)$ and ${\op C} = B({\op P})$ be as in the previous proposition.
We assume in particular that ${\op B}$ is a good source Hopf $\La$-cooperad. Then ${\op B}$ also forms
a good source $\La$-cooperad (when we forget about the Hopf structure),
and we have an $L_{\infty}$~morphism
\[
U: \BiDer_{\dg\La}({\op B}\stackrel{*}{\to}{\op C})\to\CoDer_{\dg\La}({\op B}\stackrel{*}{\to} {\op C})
\]
such that the induced map of Maurer-Cartan elements agrees with the obvious forgetful map
that re-interprets a morphism of dg Hopf $\La$-cooperads
as a morphism dg $\La$-cooperads.
\end{prop}

\begin{proof}
We refer to~\cite[Proposition A.4]{FW} for the proof that ${\op B}$ forms a good source $\La$-cooperad
when ${\op B}$ is a good source Hopf $\La$-cooperad.
We rely on the construction of Proposition \ref{prop:geometric Linfty:morphisms}
to define this $L_{\infty}$~morphism.
We set $V'' = V' = \Hom_{\dg\La}(I{\op B},\overline{\op C})$,
$W'' = \Hom_{\dg\La}({\alg g}[-1],\overline{\op P}[1])$,
and $W' = \Hom_{\dg\La}(I{\op B},\overline{\op C}[1])$,
as in the proof of Proposition \ref{prop:Lambda-coderivation Lie algebra}
and of Proposition \ref{prop:Lambda-biderivation Lie algebra}.
We also consider the maps $F'': S(W'')\to V''$, $F': S(W')\to V'$, $\lambda'': W''\to V''$ and $\lambda': W'\to V'$
defined in the proof of these Propositions.
We moreover use the notation $\pi'': S(W'')\to W''$ (respectively, $\pi': S(W')\to W'$) for the canonical projection
onto the cogenerating graded vector space of the symmetric coalgebra $S(W'')$ (respectively, $S(W')$).

Then let $G: S(W'')\to S(W')$ be the a unique morphism of coaugmented counital cocommutative coalgebras such that $\pi' G = \lambda' F''$.
This relation is actually equivalent to the identity $\lambda' F' G = \lambda' F'' G$
since we have $\lambda' F' = \pi'$ by the projection condition
of Proposition~\ref{prop:geometric Linfty:structures}.
We aim to check that the morphism $G$ further satisfies the condition $F' G = F''$ expressed by the commutative diagram
of Proposition \ref{prop:geometric Linfty:morphisms}.
By polarization, it is sufficient to check this relation for elements of the form $\exp(w'')$,
where $w''\in(W''\otimes\overline{R})^0$ and $R$ is a graded nilpotent ring.
We necessarily have $G(\exp(w'')) = \exp(w')$ for some $w'\in(W'\otimes\overline{R})^0$,
by preservation of the coalgebra structures.
Recall that $F''(\exp(w''))$ (respectively, $F'(\exp(w'))$) is defined by the restriction to $I{\op B}\otimes R$
of the morphism of graded Hopf $\La$-cooperads (respectively, of graded $\La$-cooperads) $\widetilde{w}: {\op B}^{\circ}\to{\op C}^{\circ}$
associated to the map $w = w''$ (respectively, $w = w'$,
and $\lambda'$ is defined by taking the composite of this morphism with the canonical projection ${\op C}^{\circ}\to\overline{\op P}{}^{\circ}[1]$
in the (coaugmentation coideal of the) cofree cooperad ${\op C}^{\circ} = {\op F}^c(\overline{\op P}{}^{\circ}[1])$.
The identity $\lambda' F' G = \lambda' F'' G$ which we use to define our coalgebra morphism $G$
implies that these projections of the morphisms $\widetilde{w'},\widetilde{w''}: {\op B}^{\circ}\to{\op C}^{\circ}$
coincide,
and we consequently have $\widetilde{w'} = \widetilde{w''}\Rightarrow F G(\exp(w'')) = F(\exp(w')) = F(\exp(w''))$
because this projection determines $\widetilde{w'}: {\op B}^{\circ}\to{\op C}^{\circ}$
(respectively, $\widetilde{w''}: {\op B}^{\circ}\to{\op C}^{\circ}$)
as a morphism of graded $\La$-cooperad.

We therefore have the requested identity $F' G = F''$. This result completes the verifications
required by the construction of Proposition \ref{prop:geometric Linfty:morphisms},
so that we do get an $L_{\infty}$~morphism $U: W''[-1]\to W'[-1]$
between our dg Lie algebras.
\end{proof}

\begin{rem}\label{rem:biderivations to coderivations forgetful morphism formulas}
We can also give an explicit formula for the $L_{\infty}$~morphism of the previous proposition.
We fix $\alpha_1,\dots,\alpha_n\in\Hom_{\gr\Sigma\Seq}(\bS{\alg g}[-1],\overline{\op P})$.
We can extract the $n$-ary part of this $L_{\infty}$~morphism and associate an element
of the graded vector space $\beta\in\Hom_{\gr\La}(I{\op B},\overline{\op P})\cong\Hom_{\gr\Sigma}(\bS{\op B},\overline{\op P})$
to these elements as follows.
We use the construction of the proof of Proposition ??
to produce morphisms of graded $\La$-cooperads
\[
\alpha_1',\dots,\alpha_n': C^{\leq 1}(\alg g)\to{\op C}^{\circ}
\]
extending $\alpha_1,\dots,\alpha_n: \bS{\alg g}[-1]\to\overline{\op P}{}^{\circ}[1]$.
We then define $\beta$ by a formula of the form
\[
\beta(x_1\dots x_n)\propto\sum_{\sigma\in\Sigma_n}\pm\pi(\alpha_1'(x_{\sigma(1)})\cdots\alpha_n'(x_{\sigma(n)})),
\]
for any monomial $x_1\dots x_n\in C({\alg g})$, where $\pi: {\op C}^{\circ}\to\overline{\op P}{}^{\circ}[1]$
denotes the canonical projection associated to the cogenerating collection
of the cofree cooperad ${\op C}^{\circ} = {\op F}^c(\overline{\op P}{}^{\circ}[1])$.
(We just assume that $\beta$ vanishes on the components of weight $k\not=n$ of the Chevalley-Eilenberg complex.)
\end{rem}

\begin{rem}
The map of Maurer-Cartan elements alluded to in Proposition \ref{prop:biderivations to coderivations forgetful morphism}
is defined by sending a Maurer-Cartan element $\alpha$
to the sum:
\[
U_*(\alpha) = \sum_{n\geq 1} \frac{1}{n!} U_n(\alpha,\dots,\alpha).
\]
We note that this sum converges since $U_n(\alpha,\dots,\alpha)$ is zero except possibly on the component of weight $n$
of the symmetric algebra $C({\alg g})^{\circ} = S({\alg g}[-1])$.
\end{rem}

\subsection{Biderivation complexes and deformation complexes with coefficients in bicomodules}\label{subsec:biderivation complexes:bicomodules}
We may relate the biderivation complex $\BiDer_{\dg\La}({\op B},{\op C})$ defined in the previous section
to the deformation complexes with coefficients in bicomodules
introduced in Section \ref{subsec:coderivation complexes:bicomodules}.
To be explicit, we have the following statement.

\begin{lemm}\label{lemm:biderivation bicomodule complex}
Let ${\op B}$ and ${\op C}$ be as in Proposition \ref{prop:Lambda-biderivation Lie algebra}. There is an isomorphism of dg vector spaces
\[
\BiDer_{\dg\La}({\op B},{\op C})\cong K(\alg g[-1],{\op P}),
\]
where, to form the complex $K(\alg g[-1],{\op P})$, we use that the $\La$-cooperad in graded Lie coalgebras $\alg g$
inherits the structure of a bicomodule over the commutative cooperad $\Com^c$,
and hence over the dg $\La$-cooperad ${\op C}$ by restriction of structure
through the canonical morphism $\eta: \Com^c\to{\op C}$.
\end{lemm}

\begin{proof}
Both sides are identified with the hom-object $\DerL'' = \Hom_{\gr\Sigma}(\bS{\alg g}[-1],\overline{\op P})$
as graded vector spaces.
To complete the proof of our claim, we just check that the differential of the biderivation complex,
which we identify with the differential of the dg Lie algebra $\DerL''$
on this hom-object,
agrees with the differential of the complex $K(\alg g[-1],{\op P})$.
\end{proof}

We may note that the complex $K({\alg g}[-1],{\op P})$ is purely defined in terms of cooperads, without reference to the Hopf structure.
In fact, the Hopf structure enters into the definition of the Lie bracket of $\DerL''$,
but not into the differential.

We record one more important statement that we use in our subsequent constructions.

\begin{lemm}\label{lemm:biderivation bicomodule complex forgetful morphism}
In the previous lemma, we have inclusions of dg vector spaces
\begin{align*}
K(\alg g[-1],{\op P}) & \hookrightarrow K(I{\op B},{\op P})\cong\CoDer_{\dg\La}({\op B},{\op C}) \\
K(\alg g[-1],{\op P}_{\bo}) & \hookrightarrow K(I{\op B},{\op P}_{\bo})\cong\xCoDer_{\dg\La}({\op B},{\op C})
 \end{align*}
induced by the canonical projection morphism $I{\op B}\to\alg g[-1]$
in the Chevalley-Eilenberg complex ${\op B} = C(\alg g)$.
\end{lemm}

\begin{proof}
We merely use that the canonical projection morphism $C(\alg g)\to\alg g[-1]$ defines a morphism of dg bicomodules
over the cooperad $\Com^c$ and the functoriality properties
of the construction of the deformation complexes
with coefficients in bicomodules.
\end{proof}

\part{The homotopy theory of Hopf cooperads and mapping spaces}\label{part:mapping spaces}
We prove in this part that the biderivation complexes of dg Hopf cooperads studied in the previous section
can be used to determine the homotopy of mapping spaces
of dg Hopf cooperads.
To be explicit, we establish that the nerve of these dg Lie algebras of biderivations
are weakly equivalent, as simplicial sets, to the mapping spaces
which we associate to our dg Hopf cooperads.
We apply this statement to certain models of $E_n$-operads in the category of dg Hopf cooperads
in order to get the result of Theorem~\ref{thm:nerve to mapping spaces}.

We mentioned in the introduction of this paper that the category of dg Hopf cooperads
defines a model for the rational homotopy of operads
in simplicial sets.
We can therefore use our constructions of mapping spaces on the category dg Hopf cooperads
to compute mapping spaces associated to the rationalization
of the little discs operads
in the category of topological spaces.
We give brief recollections on this subject in the first section of this part.
We also recall the definition of the model structure on the category of dg Hopf cooperads
that we use in our constructions.

To be precise, we still deal with dg Hopf $\La$-cooperads in order to get a model
for the rational homotopy of operads in topological spaces (simplicial sets)
satisfying ${\op P}(0) = *$.
We therefore review the definition of the model structure
that we associate to this category of dg Hopf cooperads.

In the second section of this part, we explain an explicit construction of fibrant resolutions
and of a simplicial framing functor on the category of dg Hopf $\La$-cooperads.
We use these constructions to get an explicit definition of mapping spaces of dg Hopf $\La$-cooperads
and, eventually, to define the weak-equivalence between these mapping spaces
and the nerve of the dg Lie algebras of biderivations.
We establish this result in the third section.

\section{Recollections on the rational homotopy theory of operads}\label{sec:rational homotopy}
The rational homotopy theory of operads has been studied by the first author in the book \cite{Fr}.
We will mostly follow the conventions of this reference unless otherwise noted.
For the convenience of the reader, we summarize the results of this paper that we use in the sequel.
To be specific, we recall the definition of the model structure on the category of $\La$-operads
in topological spaces (respectively, in simplicial sets),
we recall the definition of the model structure on the category of dg Hopf $\La$-cooperads,
and we review criteria for the (co)fibrancy of objects and morphisms
in these categories.

We want to emphasize that the homotopy theory of \cite{Fr} is developed for a subcategory of the category of Hopf $\La$-cooperads
considered in the previous part.
To be explicit, we first consider the category $\dg^*\Op^c_{01}$ (respectively, $\dg^*\Hopf\Op^c_{01}$
formed by the dg cooperads (respectively, the dg Hopf cooperads) ${\op C}$
\begin{itemize}
\item which are defined in the category of non-negatively cochain graded dg vector spaces $\dg^*\Vect$
(as opposed to the category of general dg vector spaces $\dg\Vect$ that we consider in the previous part)
\item and which satisfy the conditions ${\op C}(0) = 0$, ${\op C}(1) = \K$
(whereas we only assume ${\op C}(0) = 0$ in the previous part).
\end{itemize}
Then we consider the category $\dg^*\La\Op^c_{01}$ (respectively, $\dg^*\Hopf\La\Op^c_{01}$
formed by the dg $\La$-cooperads (respectively, the dg Hopf $\La$-cooperads)
that satisfy the same conditions.
Recall that any dg $\La$-cooperad is equipped with a coaugmentation
over the commutative cooperad,
but we do not mention this extra structure for simplicity. We therefore shorten the notation of the category of coaugmented dg $\La$-cooperads
by $\dg^*\La\Op^c_{01} = \Com^c/\dg^*\La\Op^c_{01}$. Recall also that, in the case of a Hopf $\La$-cooperad ${\op A}$,
this coaugmentation is defined by the unit morphism $\eta: \K\to{\op A}(r)$
of the commutative algebra ${\op A}(r)$
in each arity $r>0$.

In the sequel, we have to be careful that we only apply the statements of this section
to cooperads
that satisfy the above assumptions.
In fact, this technical complication makes a detour necessary when we prove that the dg Lie algebra of biderivations
associated to our dg Hopf cooperad models of $E_n$-operads reduce to the graph complexes
considered in the introduction.

In addition to the above categories of cooperads (respectively, of Hopf cooperads),
we consider the category $\dg^*\Sigma\Seq_{>1}^c$ (respectively, $\dg^*\Hopf\Sigma\Seq^c_{01}$)
formed by the $\Sigma$-collections (respectively, the Hopf $\Sigma$-collections) in $\dg^*\Vect$
such that ${\op M}(0) = {\op M}(1) = 0$,
and the category $\dg^*\La\Seq_{>1}^c$ (respectively, $\dg^*\Hopf\La\Seq^c_{01}$)
formed by the (covariant) $\La$-collections in $\dg^*\Vect$
that satisfy the same conditions ${\op M}(0) = {\op M}(1) = 0$.
Recall that we use the phrase ``Hopf $\Sigma$-collection''
to refer to a $\Sigma$-collection in a category of unital commutative algebras
in a base category, and we similarly use the phrase ``Hopf $\La$-collection''
for a $\La$-collection in unital commutative algebras.

For our purposes, we also consider the undercategory $\overline{\Com}{}^c/\dg^*\La\Seq_{>1}^c$
whose objects are the (covariant) $\La$-collections ${\op M}\in\dg^*\La\Seq_{>1}^c$
equipped with a coaugmentation $\eta: \overline{\Com}{}^c\to{\op M}$
over the coaugmentation coideal of the commutative cooperad $\overline{\Com}{}^c$.
In the case of Hopf $\La$-collections, we have an identity $\dg^*\Hopf\La\Seq_{>1}^c = \overline{\Com}{}^c/\dg^*\Hopf\La\Seq_{>1}^c$,
because any object ${\op A}\in\dg^*\Hopf\La\Seq_{>1}^c$ inherits a canonical coaugmentation $\eta: \overline{\Com}{}^c\to{\op A}$
which is defined by the unit morphism $\eta: \K\to{\op A}(r)$
of the commutative algebra ${\op A}(r)$
in each arity $r>0$ (as in the Hopf cooperad case).
The cofree cooperad functor ${\op F}^c: {\op M}\mapsto{\op F}^c({\op M})$ defines a right adjoint
of the coaugmentation coideal functor $\overline{U}: \dg^*\Op_{01}^c\to\dg^*\Sigma\Seq_{>1}^c$
which carries any cooperad ${\op C}$ in the category $\dg^*\Op_{01}^c$
to the collection $\overline{\op C}$ such that $\overline{\op C}(0) = \overline{\op C}(1) = 0$
and $\overline{\op C}(r) = {\op C}(r)$
for $r>1$.
This adjunction $\overline{U}: \dg^*\Op_{01}^c\leftrightarrows\dg^*\Seq_{>1}^c :{\op F}^c$
lifts to and adjunction $\overline{U}: \dg^*\Hopf\Op_{01}^c\leftrightarrows\dg^*\Hopf\Sigma\Seq_{>1}^c :{\op F}^c$
between the category of Hopf cooperad in $\dg^*\Vect$
and the category of Hopf $\Sigma$-collections.
The cofree cooperad functor ${\op F}^c: {\op M}\mapsto{\op F}^c({\op M})$
also lifts to the category of $\La$-cooperads
so that we get adjunction relations $\overline{U}: \dg^*\La\Op_{01}^c\leftrightarrows\overline{\Com}{}^c/\dg^*\La\Seq_{>1}^c :{\op F}^c$
and $\overline{U}: \dg^*\Hopf\La\Op_{01}^c\leftrightarrows\dg^*\Hopf\La\Seq_{>1}^c :{\op F}^c$
when we pass to this setting.

The following statement summarizes the model category constructions of \cite{Fr}
that we use in this part.

\begin{thm}[{\cite[Chapters II.9, II.11]{Fr}, see also \cite[section 0]{FW} for an overview}]\label{thm:model categories}
The category of $\Sigma$-collections $\dg^*\Sigma\Seq_{>1}^c$ (respectively, of $\La$-collections $\dg^*\La\Seq_{>1}^c$)
is equipped with a model structure
such that:
\begin{itemize}
\item
the weak-equivalences are the morphisms $\phi: {\op M}\stackrel{\sim}{\to}{\op N}$
that form a weak-equivalence of dg vector spaces (a quasi-isomorphism) in each arity $\phi: {\op M}(r)\stackrel{\sim}{\to}{\op N}(r)$,
\item
the fibrations are the morphisms that are surjective in all degrees,
\item
and the cofibrations
are characterized by the left lifting property with respect to the class
of acyclic fibrations of this model structure.
\end{itemize}

The category of cooperads $\dg^*\Sigma\Op_{01}^c$
is equipped with a model structure
such that:
\begin{itemize}
\item
the weak-equivalences are the morphisms $\phi: {\op C}\stackrel{\sim}{\to}{\op D}$
that form a weak-equivalence of dg vector spaces (a quasi-isomorphism) in each arity $\phi: {\op C}(r)\stackrel{\sim}{\to}{\op D}(r)$,
\item
the cofibrations are the morphisms whose image
under the coaugmentation coideal functor $\overline{U}: {\op C}\mapsto\overline{\op C}$
defines a cofibration in the category of $\Sigma$-collections $\dg^*\Sigma\Seq_{>1}^c$
\item
and the fibrations are characterized by the right lifting property with respect to the class
of acyclic cofibrations of this model structure.
\end{itemize}

The category of $\La$-cooperads $\dg^*\La\Op_{01}^c$
is equipped with a model structure
such that:
\begin{itemize}
\item
the weak-equivalences are the morphisms $\phi: {\op C}\stackrel{\sim}{\to}{\op D}$
that form a weak-equivalence of dg vector spaces (a quasi-isomorphism) in each arity $\phi: {\op C}(r)\stackrel{\sim}{\to}{\op D}(r)$,
\item
the fibrations are the morphisms that form a fibration in the undercategory $\Com^c/\dg^*\Op_{01}^c$
formed by the category of ordinary cooperads equipped with a coaugmentation
over the commutative cooperad,
\item
and the cofibrations are characterized by the right lifting property with respect to the class
of acyclic fibrations of this model structure.
\end{itemize}

The category of Hopf cooperads $\dg^*\Hopf\Sigma\Op_{01}^c$ (respectively, of Hopf $\La$-cooperads $\dg^*\Hopf\La\Op_{01}^c$)
is equipped with model structures
such that:
\begin{itemize}
\item
the weak-equivalences are the morphisms $\phi: {\op C}\stackrel{\sim}{\to}{\op D}$
that form a weak-equivalence of dg vector spaces (a quasi-isomorphism) in each arity $\phi: {\op C}(r)\stackrel{\sim}{\to}{\op D}(r)$,
\item
the fibrations are the morphisms that form a fibration
in the category of cooperads $\dg^*\Sigma\Op_{01}^c$ (respectively, of $\La$-cooperads $\dg^*\La\Op_{01}^c$)
when we forget about Hopf structures,
\item
and the cofibrations
are characterized by the left lifting property with respect to the class
of acyclic fibrations of this model structure.\qed
\end{itemize}
\end{thm}

\begin{rem}\label{rem:Quillen adjunctions}
In fact, we may see that a morphism defines a cofibration in the category of $\Sigma$-collections if and only if this morphism
is injective in positive degrees (as soon as we assume that the ground ring is a field of characteristic zero).

Recall also that we use the notation $\La\otimes_{\Sigma}-$
for the usual Kan extension functor from the category of $\Sigma$-collections
to the category of $\La$-collections.
From the definition of our model structures, we easily deduce that the obvious forgetful functor and this Kan extension
functor define a Quillen adjunction $\La\otimes_{\Sigma}-: \dg^*\Sigma\Seq_{>1}^c\leftrightarrows\dg^*\La\Seq_{>1}^c :\text{Forgetful}$
between the model category of dg $\Sigma$-collections $\dg^*\Sigma\Seq_{>1}^c$
and the model category of dg $\La$-collections $\dg^*\Sigma\Seq_{>1}^c$.
Furthermore, we can lift the Kan extension functor to the category of dg cooperads (respectively, of dg Hopf cooperads)
to get a counterpart of this Quillen adjunction at the dg cooperad (respectively, dg Hopf cooperad)
level.

The obvious forgetful functor from the category of dg Hopf cooperads
to the category of dg cooperads
also admits a left adjoint, defined by an appropriate application
on the classical symmetric algebra functor on dg vector spaces.
The definition of our model structure on dg Hopf cooperads implies that these functors define
a Quillen adjunction too, and we have a similar result
for the model category of dg Hopf $\La$-cooperads.
We refer to~\cite[Chapters II.9, II.11]{Fr} and to the overview of~\cite[section 0]{FW} cited in our theorem
for more details on these statements.
\end{rem}

In what follows, we also use the following fibrancy criteria:

\begin{prop}[{see \cite[Proposition II.9.2.9]{Fr}}]\label{prop:fibrant Hopf Lambda-cooperads}
If a dg Hopf $\La$-cooperad ${\op C}\in\dg^*\Hopf\La\Op_{01}^c$ is cofree
as an ordinary graded cooperad (explicitly, if we have ${\op C}^{\circ} = {\op F}^c({\op M})$
for some $\Sigma$-collection ${\op M}\in\gr\Sigma\Seq_{>1}^c$
when we forget about the Hopf structure, the $\La$-structure,
and the differential),
then ${\op C}$ is fibrant as an object of the model category of dg Hopf $\La$-cooperads $\dg^*\Hopf\La\Op_{01}^c$.
\end{prop}

\begin{proof}
This result is proved for objects of the category of dg cooperads $\dg^*\Op_{01}^c$ in the cited reference,
but ${\op C}$ is fibrant in the model category of dg Hopf $\La$-cooperads $\dg^*\Hopf\La\Op_{01}^c$
as soon as this object is fibrant in the model category of dg cooperads $\dg^*\Op_{01}^c$
by definition of our model structure.
The conclusion of the proposition follows.
\end{proof}

\begin{prop}[{see \cite[Proposition II.9.2.10]{Fr}}]\label{prop:fibrations of Hopf Lambda-cooperads}
Let $\psi: {\op C}\to{\op D}$ be a morphism of dg Hopf $\La$-cooperads.
We assume that we have an identity ${\op C}^{\circ} = {\op F}^c({\op M})$ (respectively, ${\op D}^{\circ} = {\op F}^c({\op N})$)
for some collection ${\op M}\in\gr\Sigma\Seq_{>1}^c$ (respectively, ${\op M}\in\gr\Sigma\Seq_{>1}^c$)
when we forget about the Hopf structure, the $\La$-cooperad structure
and the differential
as in Proposition~\ref{prop:fibrant Hopf Lambda-cooperads}.
We also assume that $\psi = \psi_f$ is given by the image of a morphism of $\Sigma$-collections $f: {\op M}\to{\op N}$
under the cofree cooperad functor ${\op F}^c(-)$
when we forget about the Hopf structure, the $\La$-cooperad structure
and the differential
and that this morphism $f$ is surjective in each arity and in every degree.
Then $\psi$ is a fibration in the model category of dg Hopf $\La$-cooperads $\dg^*\Hopf\La\Op_{01}^c$.
\end{prop}

\begin{proof}
This result is, like the previous statement, proved for morphisms of the category of dg cooperads $\dg^*\Op_{01}^c$ in the cited reference.
Nevertheless, we can again use that $\psi$ defines a fibration in the model category of dg Hopf $\La$-cooperads $\dg^*\Hopf\La\Op_{01}^c$
as soon as this morphism defines a fibration in the model category of dg cooperads $\dg^*\Op_{01}^c$
to get our conclusion.
\end{proof}

In parallel to the category of dg Hopf $\La$-cooperads,
we consider the category $\TopCat\La\Op_{\varnothing *}$ (respectively, $s\La\Op_{\varnothing *} = s\SetCat\La\Op_{\varnothing *}$)
formed by the $\La$-operads in topological spaces (respectively, in simplicial sets) ${\op P}$
such that ${\op P}(0) = \varnothing$, ${\op P}(1) = *$.
The general correspondence of Section~\ref{subsec:Lambda-operads} between the structure of a $\La$-operad
and the structure of an operad equipped with a distinguished element
in arity zero
gives an isomorphism between this category of reduced $\La$-operads $\TopCat\La\Op_{\varnothing *}$ (respectively, $s\La\Op_{\varnothing *}$)
and the category formed by the operads in the ordinary sense
that satisfy ${\op P}(0) = {\op P}(1) = *$.
The categories $\TopCat\La\Op_{\varnothing *}$ and $s\La\Op_{\varnothing *}$
inherit Quillen equivalent model structures,
with the Quillen equivalence $|-|: s\La\Op_{\varnothing *}\leftrightarrows\TopCat\La\Op_{\varnothing *} :S_{\bullet}$
induced by the geometric realization functor $|-|: K\mapsto|K|$
on the category of simplicial sets $s\mathrm{Set}$
and the singular complex functor $S_{\bullet}(X) = \Mor_{\TopCat}(\Delta^{\bullet},X)$
on the category of topological spaces $\TopCat$.
Hence, we can equivalently work in the category of topological spaces and in the category of simplicial sets
for the homotopy theory of operads.

We go back to the definition of this model category of $\La$-operads in topological spaces (respectively, in simplicial sets)
in Part~\ref{part:rationalization}.
We mainly deal with algebraic models of operads that we form in the category of dg Hopf $\La$-cooperads for the moment.
Recall simply that we equip our model category of $\La$-operads in topological spaces (respectively, in simplicial sets)
with the obvious class of weak-equivalences $\phi: {\op P}\stackrel{\sim}{\to}{\op Q}$,
consisting of the morphisms of operads that define a weak-equivalence
of topological spaces (respectively, of simplicial sets) $\phi: {\op P}(r)\stackrel{\sim}{\to}{\op Q}(r)$
in each arity $r>0$.

We then have the following statement:

\begin{thm}[{\cite[Chapter II.12]{Fr}}]\label{thm:rationalization adjunction}
We have a Quillen adjunction $G_\bullet: \dg^*\Hopf\La\Op_{01}^c\leftrightarrows s\La\Op_{\varnothing *}^{op} :\Omega_{\sharp}$
between the model category of dg Hopf $\La$-cooperads $\dg^*\Hopf\La\Op_{01}^c$
and the model category of $\La$-operads in simplicial sets $s\La\Op_{\varnothing *}$.
Let ${LG}_{\bullet}$ (respectively, ${R\Omega}_{\sharp}$)
denotes the left (respectively, right) derived functor of $LG_{\bullet}$ (respectively, ${R\Omega}_{\sharp}$).
Let ${\op P}\in s\La\Op_{\varnothing *}^{op}$.
The object ${\op P}^{\Q} := {LG}_\bullet{R\Omega}_{\sharp}{\op P}$ forms an operad in simplicial sets whose components ${\op P}^{\Q}(r)$
are equivalent to the Sullivan rationalization ${\op P}(r)^{\Q}$
of the spaces ${\op P}(r)$
in the homotopy category of simplicial sets
provided that each of these spaces has a rational cohomology $H({\op P}(r),\Q)$
which forms a finitely generated $\Q$ vector space
in each degree.\qed
\end{thm}

Recall that the little discs operads $\lD_n$ do not reduce to a point in arity $1$.
Nevertheless, we have models of $E_n$-operads in topological spaces
that fulfill this property (this is the case of the Fulton-MacPherson
operads for instance).
For our purpose, we consider a cofibrant model of $E_n$-operad in the category of $\La$-operads in simplicial sets $\mE_n$,
by using our Quillen equivalence between topological operads and simplicial operads.
Then we can take the image of this operad $\mE_n$
under the functor $\Omega_{\sharp} : s\La\Op_{\varnothing *}\rightarrow\dg^*\Hopf\La\Op_{01}^c$
to define our model of the class of $E_n$-operads
in the category of dg Hopf $\La$-cooperads.
To be explicit, we set $\E_n^c = \Omega_{\sharp}(\mE_n)$, for each $n\leq 1$.

Now, the formality of the $E_n$-operads implies that this object $\E_n^c$ is weakly-equivalent
to the dual cooperad of the associative operad $\e_1^c = \Ass^c$ in the case $n=1$,
and to the dual cooperad of the $n$-Poisson operad $\e_n^c = \Poiss_n^c$
when $n\geq 2$.
Hence, we can use these objects $\e_n^c$ instead of $\E_n^c = \Omega_{\sharp}(\mE_n)$
as our working models for the class of $E_n$-operads
in the category of dg Hopf $\La$-cooperads.

For our purpose, we also recall the following statement from \cite{Fr}.

\begin{thm}[{\cite[Theorem II.14.1.7 and Proposition II.14.1.15]{Fr}}]\label{thm:Drinfeld Kohno model}
Let $\stp_n$ be the cooperad in graded Lie coalgebras
formed by the dual Lie coalgebras
of the graded Drinfeld-Kohno Lie algebras
$$\mathfrak{p}_n(r) = L(t_{ij}=(-1)^nt_{ji},1\leq i<j\leq r)/\langle [t_{ij},t_{kl}],[t_{ij},t_{ik}+t_{jk}]\rangle,$$
where we have a generating element $t_{ij}$ of lower degree $n-2$,
associated to each pair $1\leq i<j\leq r$.
The Chevalley-Eilenberg cochain complex of this cooperad $C(\stp_n)$
defines a cofibrant resolution of the object $\e_n^c$
in the category of dg Hopf $\La$-cooperads,
for any $n\geq 2$.\qed
\end{thm}

\section{Fibrant replacements and simplicial frames for dg Hopf \texorpdfstring{$\La$}{Lambda} cooperads}\label{sec:fibrant resolutions}
In this section, we introduce a fibrant replacement functor in the category of dg Hopf cooperads and in the category of dg Hopf $\La$-cooperads.
Essentially, our construction is a variant (and a dual) of the construction of cofibrant replacements
of operads in monoidal categories
described by Berger and Moerdijk in \cite{BM}.
We revisit this construction which we need to adapt to the setting of cooperads in unital commutative algebras.
We also have to integrate $\La$-structures in the result of the construction
when we deal with dg Hopf $\La$-cooperads.
We devote a preliminary subsection to the definition of (the dual of) an interval object $I$
on which our construction depends.
We address the definition of the fibrant replacement functor on the category of dg Hopf cooperads afterwards.
We also explain the definition of a simplicial framing functor which extends
this fibrant replacement functor on the category of dg Hopf cooperads.
We devote Sections~\ref{subsec:fibrant resolutions:Hopf cooperads}-\ref{subsec:fibrant resolutions:simplicial frames}
to these constructions.
We eventually prove that the fibrant replacement functor and the simplicial framing functor lift to the category of dg Hopf $\La$-cooperads.
We address this extension of our constructions in subsection~\ref{sec:fibrant resolutions:Hopf Lambda-cooperads}.

\subsection{The interval object}\label{rem:counit}\label{subsec:fibrant resolutions:interval object}
Recall that the Berger-Moerdijk construction depends on a choice of a Hopf interval $I$
in the chosen monoidal category (see \cite[Definition 4.1]{BM}).
We dualize the requirements of this reference.
We take for $I$ the commutative dg algebra of polynomial forms on the unit interval $I = \K[t,dt]$
with the de Rham differential.
We have natural algebra maps $d_0,d_1: I\to\K$ by evaluation at the endpoints $t=0$ and $t=1$
and a coassociative coproduct
\[
m^*: I\to I\otimes I
\]
given by the pullback of the multiplication map
\begin{equation}\label{eq:Icoprod}
\begin{aligned}
m: [0,1]\times [0,1]\to [0,1] \\
(s,t)\mapsto 1-(1-s)(1-t)\,.
\end{aligned}
\end{equation}
The evaluation at the endpoint $t=0$ in $I$ defines a counit for this coproduct $m^*$.
The evaluation at the other endpoint $t=1$ in $I$ fits in the following diagram
\[
\begin{tikzcd}
I\otimes I\ar{ddr}[swap]{\mathit{id}\otimes ev_{t=1}} &
I \ar{d}{ev_{t=1}}\ar{r}{m^*}\ar{l}[swap]{m^*} &
I\otimes I\ar{ddl}{ev_{t=1}\otimes\mathit{id}} \\
& \K\ar{d} & \\
& I &
\end{tikzcd},
\]
so that the (dual of the) conditions of \cite[Definition 4.1]{BM}, required for the Berger-Moerdijk construction,
are satisfied.

\subsection{A fibrant resolution for dg Hopf cooperads}\label{subsec:fibrant resolutions:Hopf cooperads}
Let ${\op C}$ be a dg Hopf cooperad such that ${\op C}(0)=0$, ${\op C}(1)=\K$ \footnote{Recall that such a cooperad is automatically conilpotent.}.
We do not assume that ${\op C}$ is equipped with a $\La$-structure for the moment.
We construct a dg Hopf cooperad $W{\op C}$ together with a quasi-isomorphism ${\op C}\stackrel\simeq
\longrightarrow  W{\op C}$
such that $W{\op C}$ is cofree as a graded cooperad.
We actually prove that $W{\op C}$ is the bar construction of an augmented dg operad.
If we informally regard ${\op C}$ as the ``cooperad of differential forms'' on a topological operad ${\op P}$,
then $W{\op C}$ models the ``differential forms''
on the Boardman-Vogt W-construction of ${\op P}$, hence the notation $W{\op C}$.
We explain the definition of this dg Hopf cooperad $W{\op C}$ in the next paragraph and we establish its properties afterwards.

\begin{const}\label{const:W-construction}
Let $S$ be any finite set.
Recall that $\mT_S$ is the set of trees with $|S|$ leafs, labelled (uniquely) by the elements of $S$.

For such a tree $T$, and for a $\Sigma$-collection ${\op M}$,
we form the treewise tensor product
\[
{\op M}(T) := \bigotimes_{v\in VT} {\op M}(star(v)),
\]
where $v$ runs over the set of vertices $VT$ of the tree $T$ and $star(v)$ denotes the set of ingoing edges
of any such vertex $v$
in $T$.
These treewise tensor products are the same as the objects, formerly denoted by ${\op F}^c_T({\op M})$,
that occur in the expansion of the cofree cooperad ${\op F}^c({\op M})$.

For technical reasons, we also consider the set $\mT_S'\subset\mT_S$ formed by those trees
whose vertices have at least two ingoing edges.
We consider the category of trees $\mTc_S$ which has $ob\mTc_S = \mT_S'$ as object set
and whose morphisms are generated by the edge contractions
together with the isomorphisms of trees that respect the leaf labelling (see \cite[Sections B.0, C.0]{Fr}).

Recall that we use the notation $\overline{\op C}$ for the coaugmentation coideal of our dg cooperad ${\op C}$
which, under the assumption ${\op C}(0) = 0$, ${\op C}(1) = \K$,
is given by $\overline{\op C}(0) = \overline{\op C}(1) = 0$
and $\overline{\op C}(r) = {\op C}(r)$
for $r\geq 2$.
We have $\overline{\op C}(T) = {\op C}(T)$ when $T\in\mT_S'$.
Note that we have ${\op C}(T)\in\dgcAlg$,
where $\dgcAlg$ denotes the category of unital commutative dg algebras,
since we assume that ${\op C}$
is a dg Hopf cooperad.
The mapping $T\mapsto\overline{\op C}(T) = {\op C}(T)$ extends to a contravariant functor ${\op C}: \mTc_S^{op}\to\dgcAlg$.
The morphism ${\op C}(T/e)\to{\op C}(T)$
associated to an edge contraction $T\to T/e$
is defined by the performance of the composition coproduct $\Delta_e: {\op C}(star(v))\to{\op C}(star(v'))\otimes{\op C}(star(v''))$
where $v'$ (respectively, $v''$) is the source (respectively, the target) of the edge $e$ in $T$
and $v$ is the vertex of the tree $T/e$ which we obtain by merging these vertices
in our edge contraction process (see \cite[Paragraph C.1.6]{Fr}).

Besides this functor $T\mapsto\overline{\op C}(T)$, we consider the covariant functor $E: \mTc_S\to\dgcAlg$
such that
\[
E(T) := \otimes_{e\in ET} I_e
\]
for each tree $T\in\mT_S'$, where we take a tensor product over the set $ET$ of internal edges of $T$
of copies of the Hopf interval of the previous subsection $I_e = \K[t,dt]$.
(We say that an edge is internal when this edge connects two vertices.)
The isomorphisms of trees act by the obvious renaming operation of the edges on the indexing sets
of these tensor products.
The morphism $E(T)\to E(T/e)$ associated to an edge contraction $T\to T/e$
is defined by the performance of the evaluation morphism $ev_{t=0}: I\to\K$
on the factor $I_e = I$ associated to the edge $e\in ET$
in the tensor product $E(T)$.
Now, we define the dg Hopf $\Sigma$-collection underlying our (desired) dg Hopf cooperad $W{\op C}$
as the following end in the category of commutative dg algebras:
\[
W{\op C}(S) = \int_{T\in\mTc_S}\overline{\op C}(T)\otimes E(T) = \int_{T\in\mTc_S}{\op C}(T)\otimes E(T).
\]
Note that for $S=\{*\}$ a one-point set, the category $\mTc_S$ is empty and we obtain $W{\op C}(S) = \K$.
We may also note that the automorphism group of every object is trivial in this category $\mTc_S$.
We could accordingly replace $\mTc_S$ by an equivalent category with no isomorphism in this end construction,
by picking a representative of every isomorphism class of tree in $\mTc_S$,
but we do not really need this more rigid construction, and we therefore do not make such a choice.

The elements of $W{\op C}(S)$ can be identified with functions $\xi: T\mapsto\xi(T)$ that assign to each tree $T\in\mT'_S$
a ``decoration'' $\xi(T)$ in the commutative dg algebra ${\op C}(T)\otimes E(T)$.
Intuitively, we assume that each vertex is decorated by an element of ${\op C}$,
and each edge by a polynomial differential form on the unit interval.
Then -implicitly in the above end construction- we require that the following properties hold.
\begin{itemize}
\item (Equivariance Condition) The function $\xi$ is invariant under isomorphisms of trees in the obvious sense.
\item (Contraction Condition) Let $e\in ET$ be an internal edge in a tree $T$.
Let $v'$ (respectively, $v''$) be the source (respectively, the target) of this edge $e$ in $T$.
Let $v$ be the vertex of the tree $T/e$ which we obtain by merging these vertices
in our edge contraction process.
Then the values of $\xi$ on $T$ and $T/e$ are related by the formula
\[
\Delta_e\xi(T/e) = \ev_{t=0}^e\xi(T),
\]
where we use the notation $\Delta_e$ to denote the composition coproduct
applied to the decoration of the vertex $v$
in $\xi(T/e)\in{\op C}(T/e)$,
and $\ev_{t=0}^e$ is the evaluation at $t=0$,
applied to the decoration
of the edge $e$
in $\xi(T)\in{\op C}(T)$.
\end{itemize}
The differential on $W{\op C}$ is the one induced by the differentials on ${\op C}$ and $I$.
The commutative algebra structure is given by the pointwise multiplication of the functions $\xi: T\mapsto\xi(T)$
in the commutative dg algebras ${\op C}(T)\otimes E(T)$.

We now describe the cooperadic composition coproducts on our object $W{\op C}$.
We aim to define a map
\beq{eq:temp10}
\Delta_*: W{\op C}(S)\to W{\op C}(S'\sqcup\{*\})\otimes W{\op C}(S''),
\eeq
for each decomposition $S = S'\sqcup S''$ of a set $S$.
We use that the set of isomorphism classes of the category $\mT_S'$ is finite for each set $S$.
We deduce from this observation that the target of our composition coproduct \eqref{eq:temp10} is spanned by functions
defined on pairs of trees $(T',T'')\in\mT_{S'\sqcup\{*\}}'\times T''\in\mT_{S''}'$
and which satisfy our conditions with respect to both variables $T'$ and $T''$.
Let $\xi\in W{\op C}(S)$. We set
\[
(\Delta_*\xi)(T',T'') = \ev_{t=1}^{e_*}\xi(T'\circ_* T''),
\]
for each pair of trees $(T',T'')$, where $T = T'\circ_* T''$ is the tree obtained by grafting the root of $T''$
to the leaf of $T'$ indexed by the composition mark $*$,
and we take the evaluation at $t=1$ of the decoration of the internal edge $e_*\in ET$
produced by this grafting process in the tensor $\xi(T) = \xi(T'\circ_* T'')$.
We readily check that this composition coproduct is well-defined.
We immediately get that these composition coproducts fulfill the equivariance, counit and coassociativity relations of cooperads
too.
\end{const}

We have a canonical morphism of dg Hopf cooperads $\rho: {\op C}\to W{\op C}$ which carries any element $c\in\overline{\op C}(S)$
to the function such that $\rho(c)(T) = \overline{\Delta}_T(c)\otimes 1^{\otimes ET}$,
where, for each $T\in\mT_S'$, we consider the image of $c$ under the reduced treewise coproduct
of our cooperad $\overline{\Delta}_T: \overline{\op C}(S)\to\overline{\op C}(T)$
and we take the constant decoration $l_e\equiv 1$
of the edges of our tree $e\in ET$
in $I$.

\begin{prop}\label{prop:W-construction quasi-iso}
The above morphism $\rho: {\op C}\to W{\op C}$ is a quasi-isomorphism of dg Hopf cooperads.
\end{prop}

\begin{proof}
We use the splitting
\[
I = \K[t,dt]\cong\K 1\oplus I',
\]
where $I'\subset I$ is the acyclic dg vector space formed by the polynomial differential forms that vanish at $t=0$.
We have
\[
\otimes_{e\in ET} I_e\cong\oplus_{E'\subset ET}(\otimes_{e\in E'} I'_e),
\]
for each tree $T\in\mT_S'$, where the direct sum on the right-hand side runs over the subsets $E'\subset ET$
of the set of internal edges of our tree $ET$,
and we consider for anu such $E'$ a tensor product of copies
of the above acyclic dg vector space $I'_e = I'$.
We use this splitting to produce a decomposition $\xi(T) = \sum_{E'\subset ET}\tilde{\xi}(T)_{E'}$,
with $\tilde{\xi}(T)_{E'}\in{\op C}(T)\otimes(\otimes_{e\in E'} I'_e)$,
for any term $\xi(T)\in{\op C}(T)\otimes E(T)$
of a function $\xi: T\mapsto\xi(T)$
that represents an element of our end $W{\op C}$
in our construction.

We can associate any subset $E'\subset ET$ as above to a tree morphism $T\to T'$,
where $T'$ is formed by contracting the edges $\xi\in ET\setminus E'$
in the tree $T$. We then have $E' = ET'$.
We use that $\ev_{t=0}^e$ is the identity on $K1\subset I_e$ and vanishes on $I_e'$.
We deduce from this observation and the equalizer relations in our construction of $W{\op C}$
that the term $\tilde{\xi}(T)_{ET'}\in{\op C}(T)\otimes(\otimes_{e\in ET'} I'_e)$
in the decomposition of $\xi(T)$
is equal to the image of $\tilde{\xi}(T')_{ET'}\in{\op C}(T')\otimes(\otimes_{e\in ET'} I'_e)$
under the morphism ${\op C}(T')\otimes(\otimes_{e\in ET'} I'_e)\to{\op C}(T)\otimes(\otimes_{e\in ET'} I'_e)$
which we deduce from the functoriality of the tensor products ${\op C}(T)$
with respect to tree morphisms.
We easily check that we can retrieve a well-defined element of $W{\op C}(T)$
from a collection of elements $\tilde{\xi}(T)_{ET'}$, $T\in\mT_S'$,
by using this correspondence.
We accordingly have an isomorphism of dg vector spaces
\[
W{\op C}(S)\cong\prod_{[T]}{\op C}(T)\otimes(\otimes_{e\in ET} I_e'),
\]
where the product on the right-hand side runs over a set of representatives of isomorphism classes of trees $T\in\mT_S'$
(recall that the automorphism group of every object is trivial in our category of trees equipped with a fixed leaf labelling $\mTc_S$).

The acyclicity of the dg vector space $I_e'$ implies that the factors of this cartesian product are acyclic for all trees $T$
such that $ET\not=\emptyset$.
The relation $ET=\emptyset$ implies that $T$ is a corolla (a tree with a single vertex),
and we have in this case ${\op C}(T) = {\op C}(S)$.
The canonical map $\rho: {\op C}\to W{\op C}$ is given by the identity on this factor ${\op C}(T) = {\op C}(S)$.
The conclusion follows.
\end{proof}

\begin{lemm}\label{lemm:W-construction cofree structure}
The object $W{\op C}$ is cofree as a graded cooperad. We explicitly have $W{\op C}^{\circ}\cong{\op F}^c(\oW{\op C}[1])$
for some graded $\Sigma$-collection $\oW{\op C}[1]$, which explicitly consists of the elements $\xi\in W{\op C}(S)$
whose reduced composition coproducts vanish in $W{\op C}$.
\end{lemm}

\begin{proof}
By the construction of the composition coproducts of the cooperad, we have $\xi\in\oW{\op C}[1](S)$
if and only if for every tree $T\in\mT_S'$ and for every edge $e\in ET$,
we have $\ev_{t=1}^e\xi(T)=0$.
Hence, the decoration that $\xi$ assigns to a tree $T$ must be so that the edge decoration belong to the subspace $I''\subset I$
of differential forms that vanish at the endpoint $t=1$.
Let $E''(T) = \otimes_{e\in ET} I_e''$ be the tensor product of copies of this subspace $I''_e = I''$
inside $E(T) = \otimes_{e\in ET} I_e$.
We accordingly have the end formula
\[
\oW{\op C}(S) = \int_{T\in\mTc_S}{\op C}(T)\otimes E''(T)[1],
\]
for every indexing set such that $|S|\geq 2$, where we use that the graded vector spaces $E''(T)\subset E(T)$
are preserved by the action of the category $\mTc_S$
on the objects $E(T)$.

We consider the map $\pi'': I\to I''$ which carries any polynomial $p(t,dt)\in I = \K[t,dt]$
to the polynomial $\pi''(p) = \tilde{p}(t,dt)$
such that $\tilde{p}(t,dt) = p(t,dt) - t p(1,0)$.
We clearly have $\ev_{t=0}\tilde{p}(t,dt) = \ev_{t=0}p(t,dt)$.
We deduce from this relation that the morphism $\pi''_*: E(T)\to E''(T)$
induced by this projection map intertwines the action
of the edge contractions $T\to T/e$
on our objects.
We accordingly have a natural retraction
\[
\pi: W{\op C}\to\oW{\op C}[1]
\]
of the embedding $\oW{\op C}[1]\subset W{\op C}$ induced by these morphisms $\pi''_*: E(T)\to E''(T)$ on our end.

We form the morphism of graded cooperads $\psi: W{\op C}\to{\op F}^c(\oW{\op C}[1])$
induced by this morphism
of graded $\Sigma$-collections $\pi: W{\op C}\to\oW{\op C}[1]$.
The collection $\oW{\op C}[1]$ inside $W{\op C}$ is defined as a cooperadic analogue of the vector space of primitive elements in a coalgebra.
The injectivity of our morphism $\psi: W{\op C}\to{\op F}^c(\oW{\op C}[1])$ on these primitives
implies that $\psi$ is injective itself for general reasons.
We are therefore left to proving that $\psi$ is surjective.

We fix a tree $T'\in\mT_S'$ and an element $\xi$ in the summand ${\op F}^c_{T'}(\oW{\op C}[1])$ shaped on this tree $T'$
in the cofree cooperad ${\op F}^c(\oW{\op C}[1])$.
We can regard $\xi$ as a collection of tensors $\xi(T) = \sum_k(\otimes_{u\in VT'}\xi_u^k(T_u))$,
indexed by tree morphisms $f: T\to T'$,
where $T_u$ is the subtree of $T$ defined by the preimage of the vertex $u\in VT'$
under our map $f: T\to T'$ (see~\cite[Proposition B.0.3 and Paragraph C.0.8]{Fr}),
and we assume $\xi_u^k(T_u)\in{\op C}(T_u)\otimes E''(T_u)$.
We also require that $\xi(T)$ fulfills the defining relations of our end factor-wise
in this tensor product over $VT'$.

We use that, for any tree $T\in\mT_S'$, we have at most one morphism $f: T\to T'$ as above,
with the tree $T'$ as target (see~\cite[Theorem B.0.6]{Fr}).
In this case, for an edge $e\in ET$, we have either $e\in T_u$, for some vertex $u\in VT$,
or $e$ represents the preimage of an edge $e'\in ET'$
under our map $f: T\to T'$.
We consider the function $\tilde{\xi}: T\mapsto\tilde{\xi}(T)$ such that $\tilde{\xi}(T) = 0$ if a morphism $f: T\to T'$ does not exist,
and $\tilde{\xi}(T) = \xi(T)\otimes(\otimes_{e\in f^{-1}(ET')} t)$ otherwise,
where we complete the edge decoration of the tensors $\xi_u^k(T_u)\in{\op C}(T_u)\otimes E''(T_u)$
by taking the polynomials $p_e(t,dt) = t$, for these edges $e\in f^{-1}(ET')$
that correspond to the preimage of an edge $e'\in ET'$
under our map $f: T\to T'$.
We easily check that this function $\tilde{\xi}: T\mapsto\tilde{\xi}(T)$ gives a well-defined element of $W{\op C}$
and we moreover have the identity $\psi(\tilde{\xi}) = \xi$
in the cofree cooperad ${\op F}^c(\oW{\op C}[1])$.

We conclude that our morphism $\psi: W{\op C}\to{\op F}^c(\oW{\op C}[1])$ is also surjective
in addition to being injective, and hence, defines an isomorphism
as requested.
\end{proof}

We now claim that the collection $\oW{\op C}_{\bo}$, which we obtain by adding a unit term $\oW{\op C}_{\bo}(1) = \K 1$ to $\oW{\op C}$,
forms an augmented operad.
We focus on the definition of the reduced operadic composition operations
\[
\circ_* : \oW{\op C}(S\sqcup\{*\})\otimes\oW{\op C}(S')\to\oW{\op C}(S\sqcup S'),
\]
since the composition products with the unit term $\oW{\op C}_{\bo}(1) = \K$
are forced by the unit axioms.
We use that a tree $T\in\mT_{S\sqcup S'}$ admits at most one decomposition $T\cong T'\circ_* T''$,
where $T'\in\mT_{S'\sqcup\{*\}}$, $T''\in\mT_{S''}$,
and we again consider the operation of grafting the root of $T''$
to the leaf of $T'$ indexed by the composition mark $*$.
We say that the tree $T$ is decomposable if such a decomposition exists.
We fix $\xi'\in\oW{\op C}(S'\sqcup\{*\})$ and $\xi''\in\oW{\op C}(S'')$.
We define $\xi'\circ_*\xi''\in\oW{\op C}(S\sqcup S')$
by the formula
\[
(\xi'\circ_*\xi'')(T) = \begin{cases} (-1)^{|\xi'|+1}\xi'(T')\otimes(dt)_{e_*}\otimes\xi''(T''), & \text{if $T$ is decomposable}, \\
0, & \text{otherwise}, \end{cases}
\]
for each tree $T\in \mT_{S\sqcup S'}$, where we assign the polynomial $p_{e_*}(t,dt) = dt$
to the edge $e_*$ that we produce by grafting the root of $T''$
to the leaf of $T'$ indexed by $*$
in the tree $T$.
We just use that $\xi'(T')$ (respectively, $\xi''(T'')$) provides a decoration of the vertices and edges of $T'$ (respectively, $T''$).
We put the edge labelling and these decorations together to get $(\xi'\circ_*\xi'')(T)$.
We readily check that this construction returns a well-defined element of~$\oW{\op C}(S\sqcup S')$.
We also readily check that these composition products satisfy the axioms
of operads.

We then have the following strengthening of Lemma \ref{lemm:W-construction cofree structure}.

\begin{lemm}\label{lemm:W-construction bar structure}
We have an isomorphism of dg cooperads
\[
W{\op C}\cong B(\oW{\op C}_{\bo}),
\]
where we consider the bar construction of the operad $\oW{\op C}_{\bo}$ on the right hand side.
\end{lemm}

\begin{proof}
We use the isomorphism $\psi: W{\op C}^{\circ}\stackrel{\cong}{\to}{\op F}^c(\oW{\op C}[1])$ of the proof of Lemma \ref{lemm:W-construction cofree structure}.
We mainly check that this isomorphism carries the differential of the dg cooperad
to the differential of the bar construction $B(\oW{\op C}_{\bo})$.
We are left to proving this relation on the cogenerating collection $\oW{\op C}[1]$
of the cofree cooperad $B(\oW{\op C}_{\bo})^{\circ} = {\op F}^c(\oW{\op C}[1])$.
We then consider the morphism $\pi: W{\op C}^{\circ}\to\oW{\op C}[1]$
which we define by applying the map $\pi''(p(t,dt)) = p(t,dt) - t p(1,0)$
to the edge decoration of any element $\xi\in W{\op C}(S)$
in the proof of Lemma \ref{lemm:W-construction cofree structure}.
We get that, in terms of this map $\pi$, our identity of differentials
is equivalent to the relation
\beq{eq:temp11}
d\pi(\xi) = \pi(d\xi) + \pm\sum_{(\xi)}\pi(\xi')\circ_*\pi(\xi''),
\eeq
for any $\xi\in W{\op C}(S)$, where we use Sweedler's notation to depict the composition coproducts
of $\xi$ in $W{\op C}$.
We have
\beq{eq:temp12}
d\pi''(p) = d(p(t,dt) - t p(1,0)) = dp(t,dt) - p(1,0) dt\quad\text{and}\quad\pi''(dp(t,dt)) = dp(t,dt),
\eeq
for any $p = p(t,dt)\in I$.
We easily check that the term with an edge decoration $p(1,0) dt$ produced by such a differential $d\pi''(p)$
of an edge labeling in the expansion
of $d\pi(\xi)$
corresponds to a term of a composition product $\sum_{(\xi)}\pi(\xi')\circ_*\pi(\xi'')$, since the composition coproducts $\Delta_*(\xi)$
involves the substitution operation $\ev_{t=1}^e(p(t,dt)) = p(1,0)$
of such edge labelling, and the composition product $\circ_*$
insert a factor $dt$ at the same position.
We readily get that the remaining terms of the expansion of $d\pi(\xi)$
corresponds to terms of $\pi(d\xi)$ too. The conclusion follows.
\end{proof}

\begin{rem}\label{rem:W-construction bar structure}
We can deduce from our statements that we have a canonical quasi-isomorphism of dg operads
\[
\Omega({\op C})\stackrel{\sim}{\to}\Omega(W{\op C})\cong\Omega(B(\oW{\op C}_{\bo}))\stackrel{\sim}{\to}\oW{\op C}_{\bo},
\]
where we take the image of the quasi-isomorphism ${\op C}\stackrel{\sim}{\to}W{\op C}$ of Proposition \ref{prop:W-construction quasi-iso}
under the cobar functor to get the first arrow of this composition,
and the counit of the bar-cobar adjunction
to get the last arrow.
\end{rem}

\subsection{A simplicial frame for dg Hopf cooperads}\label{subsec:fibrant resolutions:simplicial frames}
We now explain the definition of a simplicial frame $W{\op C}^{\Delta^{\bullet}}$ of the object $W{\op C}$
in the category of dg Hopf cooperads.

\begin{const}\label{const:W-construction framing}
To carry out this construction, we consider an extension of the functor $E: T\mapsto E(T)$
of the previous section.
To be explicit, for each $n\in\N$, and for each finite set $S$, we consider the functor
\[
E^{\Delta^n}: \mTc_S\to\dgcAlg
\]
such that
\[
E^{\Delta^n}(T) = (\otimes_{e\in ET} I_e)\otimes(\otimes_{v\in VT}\APL(\Delta^n)_v),
\]
for each tree $T\in\mT_S'$, where $\APL(\Delta^n)_v = \APL(\Delta^n)$ is a copy of the commutative dg algebra of polynomial differential forms
on the $n$-simplex $\Delta^n$ which we assign to each vertex $v\in VT$.
We consider the obvious action of isomorphisms of trees on this collection.
We assume that the morphism $E^{\Delta^n}(T)\to E^{\Delta^n}(T/e)$ associated to an edge contraction $T\to T/e$
is induced by the evaluation map $\ev_{t=0}: I_e\to\K$
on the factor $I_e = I$ associated to our edge,
and by the multiplication operation $m_{v' v''}: \APL(\Delta^n)_{v'}\otimes\APL(\Delta^n)_{v''}\to\APL(\Delta^n)_{v}$
on the factors $\APL(\Delta^n)_{v'}$ and $\APL(\Delta^n)_{v''}$
associated to the target $v'$ and to the source $v''$ of this edge $e$ in $T$,
where $v$ is the vertex of $T/e$ produced by the merging of $v'$
and $v''$
in the result of the edge contraction operation.

We immediately see that these functors $E^{\Delta^n}: T\mapsto E^{\Delta^n}(T)$, $n\in\N$,
inherit the simplicial structure operators from the commutative dg algebras
of polynomial differential forms $\APL(\Delta^n)$, $n\in\N$.
We then define the simplicial dg Hopf collection $W{\op C}^{\Delta^{\bullet}}$
by the end
\[
W{\op C}^{\Delta^{\bullet}}(S) = \int_{T\in\mTc_S}{\op C}(T)\otimes E^{\Delta^{\bullet}}(T),
\]
for each finite set $S$.

The elements of this end can still be identified with functions $\xi: T\mapsto\xi(T)$
that assign to each tree $T\in\mT'_S$
a ``decoration'' $\xi(T)$ in the commutative dg algebra ${\op C}(T)\otimes E^{\Delta^{\bullet}}(T)$.
Intuitively, we assume that each vertex is decorated by an element of ${\op C}\otimes\APL(\Delta^n)$,
and each edge by a polynomial differential form on the unit interval.
Then we also require that the following properties hold.
\begin{itemize}
\item (Equivariance Condition) The function $\xi$ is invariant under isomorphisms of trees in the obvious sense.
\item (Contraction Condition) Let $e\in ET$ be an internal edge in a tree $T$.
Let $v'$ (respectively, $v''$) be the source (respectively, the target) of this edge $e$ in $T$.
Let $v$ be the vertex of the tree $T/e$ which we obtain by merging these vertices
in our edge contraction process.
Then the values of $\xi$ on $T$ and $T/e$ are related by the formula
\[
\Delta_e\xi(T/e) = (\ev_{t=0}^e\otimes m_{v' v''})\xi(T),
\]
where we use the notation $\Delta_e$ to denote the composition coproduct
applied to the decoration of the vertex $v$
in $\xi(T/e)\in{\op C}(T/e)$,
and $\ev_{t=0}^e$ is the evaluation at $t=0$,
applied to the decoration
of the edge $e$
in $\xi(T)\in{\op C}(T)$, while $m_{v' v''}: \APL(\Delta^n)_{v'}\otimes\APL(\Delta^n)_{v''}\to\APL(\Delta^n)_{v}$
is the multiplication on the factors $\APL(\Delta^n)_{v'}$ and $\APL(\Delta^n)_{v''}$
associated to the target $v'$
and to the source $v''$ of this edge $e$ in $T$,
as in our previous definition.
\end{itemize}
The differential on $W{\op C}^{\Delta^{\bullet}}$ is induced by the differentials of ${\op C}$, $I$, and $\APL(\Delta^{\bullet})$.
The commutative algebra structure is again given by the pointwise multiplication of the functions $\xi: T\mapsto\xi(T)$
in the commutative dg algebras ${\op C}(T)\otimes E^{\Delta^n}(T)$.

The cooperadic composition coproducts on $W{\op C}^{\Delta^{\bullet}}$
is the same as in the case of the dg Hopf cooperad $W{\op C}$
in the previous subsection. Note simply that the definition of these cooperadic composition coproducts
only involves the edge decorations.
\end{const}

We have the following analogues of the results of the previous subsection for $W{\op C}^{\Delta^{\bullet}}$.

\begin{lemm}\label{lemm:W-construction framing bar structure}
Each object $W{\op C}^{\Delta^n}$ is cofree as a graded cooperad.
We explicitly have $(W{\op C}^{\Delta^n})^{\circ}\cong{\op F}^c(\oW{\op C}^{\Delta^n}[1])$
for some graded $\Sigma$-collection $\oW{\op C}^{\Delta^n}[1]$,
which explicitly consists of the elements $\xi\in W{\op C}(S)$
whose reduced composition coproducts
vanish in $W{\op C}^{\Delta^n}$.
Furthermore, the dg collection $\oW{\op C}^{\Delta^n}_{\bo}$, which we obtain by adding a unit to $\oW{\op C}^{\Delta^n}$
in arity one,
carries an operad structure, and we have an identity in the category of dg cooperads
\[
W{\op C}^{\Delta^n} = B(\oW{\op C}^{\Delta^n}_{\bo}),
\]
where we consider the bar construction of $\oW{\op C}^{\Delta^n}_{\bo}$ on the right hand side.
\end{lemm}

\begin{proof}
The proof is the same as for $W{\op C}$.
\end{proof}

We may also observe that the simplicial structure operators of the simplicial dg Hopf cooperad $W{\op C}^{\Delta^{\bullet}}$
preserve the cogenerating collection $\oW{\op C}^{\Delta^{\bullet}}$
defined in this lemma, and are identified with morphisms
of cofree cooperads
when we forget about differentials.
We use this observation in the proof of the following lemma.

\begin{lemm}\label{lemm:W-construction framing fibrations}
The restriction morphism $i^*: W{\op C}^{\Delta^n}\to W{\op C}^{\partial\Delta^n}$, where we use the notation $W{\op C}^{\partial\Delta^n}$
for the matching object of the simplicial dg Hopf cooperad $W{\op C}^{\Delta^{\bullet}}$,
is a fibration for all $n>0$.
\end{lemm}

\begin{proof}
We refer to~\cite[Paragraph 3.1.15 and Paragraph 3.2.8]{Fr} for detailed recollections on the definition of the matching objects
associated to a simplicial object in a category and for a detailed analysis of the notion of a matching object
in the context of simplicial frames.
We can identify this matching object $W{\op C}^{\partial\Delta^n}$, which is given by an appropriate limit in the category of dg Hopf cooperads,
with the object that we obtain by replacing the dg algebras $\APL(\Delta^n)$
in the construction of $W{\op C}^{\Delta^n}$
by the dg algebras of polynomal forms $\APL(\partial\Delta^n)$
associated to the boundary $\partial\Delta^n$
of the simplex $\Delta^n$.
We have in particular $(W{\op C}^{\partial\Delta^n})^{\circ}\cong{\op F}^c(\oW{\op C}^{\partial\Delta^n}[1])$
for an analogous definition of a $\Sigma$-collection $\oW{\op C}^{\partial\Delta^n}$,
and the morphism $i^*: W{\op C}^{\Delta^n}\to W{\op C}^{\partial\Delta^n}$
is induced by the restriction map $i^*: \APL(\Delta^n)\to\APL(\partial\Delta^n)$
at the level of cogenerators.
We deduce from this observation that this morphism $i^*: W{\op C}^{\Delta^n}\to W{\op C}^{\partial\Delta^n}$
is defined by a morphism of graded cofree cooperads
which is surjective at the cogenerator level,
because the restriction map $i^*: \APL(\Delta^n)\to\APL(\partial\Delta^n)$
has this property. We then use the result of Proposition \ref{prop:fibrations of Hopf Lambda-cooperads}
to get the conclusion of this lemma.
\end{proof}

We have the following additional observation.

\begin{lemm}\label{lemm:W-construction framing quasi-isomorphism}
The morphism $W{\op C} = W{\op C}^{\Delta^0}\to W{\op C}^{\Delta^n}$ induced by the simplicial map $\epsilon: \underline{n}\to\underline{0}$
is a quasi-isomorphism.
\end{lemm}

\begin{proof}
The morphism of the lemma is identified with the map $W{\op C}\to W{\op C}^{\Delta^n}$
which assigns a constant decoration $1\in\APL(\Delta^n)$
to the vertices of trees
in the expression of the elements
of $W{\op C}$.
Pick some identification $\APL(\Delta^n)\cong\K[u_1,\dots,u_n,du_1,\dots,du_n]$
and define the total degree of an element of $\APL(\Delta^n)$
as the number of variables $u_i$ occurring in the expression of this element.
Let $\iota$ be the contraction with the Euler vector field, so that for any homogeneous element $\alpha\in\APL(\Delta^n)$
we have the relation
\[
(d\iota+\iota d)\alpha =(\text{total degree})\cdot\alpha.
\]
We extend this contraction $\iota$ to $W{\op C}^{\Delta^n}$ as a derivation of the decoration of the vertices $\otimes_{v\in VT}\APL(\Delta^n)_v$
in the definition of $W{\op C}^{\Delta^n}$ (simply note that $\iota$ map preserves the defining of our end).
We then have the relation
\[
(d\iota+\iota d)\xi =(\text{total degree of forms})\cdot\xi,
\]
for every element $\xi\in W{\op C}^{\Delta^n}$, where we consider the sum of the total degrees of the form
that decorate the vertices in the expression
of $\xi$.
We can identify our map $W{\op C}\to W{\op C}^{\Delta^n}$ with the obvious embedding that identifies $W{\op C}$
with the summand of the object $W{\op C}^{\Delta^n}$
spanned by the elements for which this total degree is zero.
We therefore deduce the result of the lemma from the above homotopy relation.
\end{proof}

We obtain the following concluding statement.

\begin{thm}\label{thm:W-construction framing}
The simplicial dg Hopf operad $W{\op C}^{\Delta^{\bullet}}$ is a simplicial frame for $W{\op C}$.
\end{thm}

\begin{proof}
We just checked that $W{\op C}^{\Delta^{\bullet}}$ fulfills the defining properties
of a simplicial frame in the previous lemmas.
\end{proof}

\subsection{Fibrant resolutions and simplicial frames for dg Hopf \texorpdfstring{$\La$}{Lambda}-cooperads}\label{sec:fibrant resolutions:Hopf Lambda-cooperads}
We now check that the constructions of the previous subsection extend to the category of dg Hopf $\La$-cooperads.
We first have the following statement.

\begin{prop}\label{prop:W-construction Lambda-structure}
If ${\op C}$ is a reduced dg Hopf $\La$-cooperad, then the dg Hopf cooperad $W{\op C}$ inherits a natural $\La$-structure
such that the morphism
\[
\rho: {\op C}\to W{\op C}
\]
defined in Proposition~\ref{prop:W-construction quasi-iso} is a quasi-isomorphism of dg Hopf $\La$-cooperads.
\end{prop}

\begin{proof}
We use that a $\La$-structure on $W{\op C}$ is determined by corestriction operators
\[
\eta_S : W{\op C}(S)\to W{\op C}(S\sqcup\{*\}),
\]
for any finite set $S$.
In the case $|S| = 1$, this corestriction operator is identified with the unit morphism $\eta_S: \K\to W{\op C}(S\sqcup\{*\})$ of the dg algebra $W{\op C}(S\sqcup\{*\})$.
We therefore assume $|S|\geq 2$ from now on.
We fix $\xi\in W{\op C}(S)$. We proceed as follows to define the collection $\eta_S(\xi)(T)$, $T\in\mT_{S\sqcup\{*\}}'$,
that represent the image of $\xi$ under this map $\eta_S$.

For each tree $T\in\mT_{S\sqcup\{*\}}$, we consider the vertex $v\in VT$ connected to the leaf indexed by $*$.
We distinguish the following cases and subcases.
\begin{itemize}
\item
If $v$ has at least three ingoing edges, then we consider the tree $T'$ which we obtain by removing this leaf indexed by $*$ from $T'$,
and we set
\[
\eta_S(\xi)(T) = (\eta_{star(v)}\otimes\mathit{id}\otimes\dots\otimes\mathit{id})(\xi(T')),
\]
where we apply the restriction operator $\eta_{star(v)}$ to the decoration of the vertex $v$ in $T'$,
assuming that we put this decoration in front of the tensor $\xi(T')$.
\item
If $v$ has two ingoing edges, one which is the leaf $e_*$ indexed by $*$, and the other one which we denote by $e_1$,
then we consider the tree $T''$ which we obtain by withdrawing both the leaf $e_*$ and the vertex $v$ from $T$,
and by gluing $e_1$ onto the outgoing edge $e_0$ of $v$ in $T$.
Let $e$ be the edge of $T''$ that we deduce from this merging operation.
\begin{itemize}
\item
If both $e_0$ and $e_1$ form internal edges of $T$, then we set:
\[
\eta_S(\xi)(T) = b_v\otimes m_e^*\xi(T''),
\]
where we take the element $b_v = b = \eta(1)\in{\op C}(2)$ associated to the unit morphism of our Hopf cooperad $\eta: \Com^c\to{\op C}$
as a decoration of the vertex $v$ in $T$,
and $m_e^*$ denotes the performance of the coproduct operation $m_e^* = m^*: I_e\to I_{e_0}\otimes I_{e_1}$
on the factor $I_e = I$ that give the decoration of the edge $e$ in $\xi(T'')$.
The latter operation gives the decoration of the edges $e_0$ and $e_1$ in $\eta_S(\xi)(T)$.
\item
If $e_0$ is an internal edge of $T$, but $e_1$ is a leaf like $e_*$, then we set:
\[
\eta_S(\xi)(T) = b_v\otimes 1_{e_0}\otimes\xi(T''),
\]
where we again take the element $b_v = \eta(1)\in{\op C}(2)$ as a decoration of the vertex $v$ in $T$,
while we insert an algebra unit $1_{e_0} = 1\in I$ to get the decoration
of the edge $e_0$ associated to our element.
\item
If $e_1$ is an internal edge, but $e_0$ is the root of $T$, then we set:
\[
\eta_S(\xi)(T) = b_v\otimes 1_{e_1}\otimes\xi(T''),
\]
where we still take the element $b_v = \eta(1)\in{\op C}(2)$ as a decoration of the vertex $v$ in $T$,
but we now take the algebra unit $1_{e_1} = 1\in I$ to decorate the edge $e_1$.
\end{itemize}
\end{itemize}
Note that the case where both $e_0$ is the root and $e_1$ is an internal edge
is excluded since we assume $|S|\geq 2$.
In fact, we can reduces the subcases of the second part of this definition to the first subcase by adopting the convention
that the leafs and the root of a tree are decorated by unit elements $1\in I$
in our description of $W{\op C}$ (just observe that our coproduct satisfies $m^*(1) = 1\otimes 1$
for the algebra unit $1\in I$).
Besides we can apply this simplification to extend our construction to the degenerate case $|S| = 1$. We retrieve that our corestriction operator is identified
with the unit morphism of our algebra in this case.

We check that this construction returns a well-defined element $\eta_S(\xi)$ of the end $W{\op C}(S\sqcup\{*\})$.
We fix an internal edge $e'$ of the tree $T\in\mT_{S\sqcup\{*\}}$. Let $w'$ (respectively, $w''$)
be the source (respectively, the target) of this edge $e'$
in $T$. Let $w$ be the vertex of $T/e'$ produced by the merging of $w'$ and $w''$
when we perform the contraction of the edge $e'$ in $T$.
We have to check that the performance of the composition coproduct $\Delta_*: {\op C}(star(w))\to{\op C}(star(w'))\otimes{\op C}(star(w''))$
on the decoration of the vertex $w$ in $\eta_S(\xi(T/e'))$
agrees with the result of the evaluation $t_{e'} = 0$ of the decoration of the edge $e'$
in $\eta_S(\xi(T))$.

We use the same notation as in our definition of $\eta_S(\xi)(T)$ and we still denote by $v$ the target of the leaf indexed by $*$ in the tree $T$.
If $e'$ is not incident to $v$, then $\eta_S(\xi)$ satisfies the required compatibility relation for this choice of $T$ and $e'$, since so does $\xi$.
If $e'$ is incident to $v$ in $T$, so that we have either $v = w'$ or $v = w''$,
then the requested compatibility relations reduce to the compatibility relations between the composition coproduct
and the corestriction operators of the $\La$-cooperad ${\op C}$.
In the degenerate case where $v$ has two ingoing edges, among which one is the leaf $*$,
we additionally use that the evaluation operation $\ev_{t=0}$
defines a counit of the coproduct $m_e^* = m^*: I_e\to I_{e_0}\otimes I_{e_1}$ (see Section \ref{rem:counit})
to check the coherence of the edge decorations.

We conclude that our mapping $\eta_S: \xi\mapsto\eta_S(\xi)$ gives a well-defined corestriction operator $\eta_S : W{\op C}(S)\to W{\op C}(S\sqcup\{*\})$,
which is obviously a morphism of dg algebras by construction.
We easily check that this $\La$-diagram structure on $W{\op C}$ is also compatible with our cooperad composition coproducts
by using the compatibility relations between the evaluation operations $\ev_{t=1}$
and the coproduct $m^*: I\to I\otimes I$ in our commutative dg algebra (see Section \ref{rem:counit}).

Recall that the morphism $\rho: {\op C}\to W{\op C}$ carries any element $c\in{\op C}(S)$ to the function
such that $\rho(c)(T) = \overline{\Delta}_T(c)\otimes 1^{\otimes ET}$,
where, for each $T\in\mT_S'$, we consider the image of $c$ under the reduced treewise coproduct
of our cooperad $\overline{\Delta}_T: \overline{\op C}(S)\to\overline{\op C}(T)$.
We use the compatibility relation between these reduced treewise coproducts and the corestriction operators,
and that the unit $1\in I$ is group-like $m^*(1) = 1\otimes 1$
in $I$
to establish that this map $\rho: c\mapsto\rho(c)$
intertwines the action of the corestriction operators
on our objects.
This verification finishes the proof of this proposition.
\end{proof}

We can extend the construction of the previous proposition to the simplicial frame $W{\op C}^{\Delta^{\bullet}}$
of Section~\ref{subsec:fibrant resolutions:simplicial frames}.
We record this result in the following statement.

\begin{prop}\label{prop:W-construction framing Lambda-structure}
If ${\op C}$ is a reduced dg Hopf $\La$-cooperad, then we also have a natural $\La$-structure on the dg Hopf cooperads $W{\op C}^{\Delta^n}$
of Section~\ref{subsec:fibrant resolutions:simplicial frames}
so that $W{\op C}^{\Delta^{\bullet}}$ defines a simplicial frame for the object $W{\op C}$
in the category of dg Hopf $\La$-cooperads.
\end{prop}

\begin{proof}
We define the corestriction operators $\eta_S: W{\op C}^{\Delta^n}(S)\to W{\op C}^{\Delta^n}(S\sqcup\{*\})$
by the same construction as in the case of the dg Hopf cooperad $W{\op C}$
in the proof of Proposition~\ref{prop:W-construction Lambda-structure}.
We just take the constant $1\in\APL(\Delta^n)$ to define the decoration of the vertex $v\in VT$ connected to the leaf indexed by the mark $*$
in the degenerate case where this vertex $v$ has only two ingoing edges
in the tree $T$.
We can carry over the verifications of the proof of Proposition \ref{prop:W-construction Lambda-structure}
without change so that this construction
returns a well-defined $\La$-structure on the dg Hopf cooperads $W{\op C}^{\Delta^n}$.
We use that a morphism is a weak-equivalence (respectively, a fibration) in the category of dg Hopf $\La$-cooperads if and only if this morphism
defines a weak-equivalence (respectively, a fibration) in the category of dg Hopf cooperads (see Proposition~\ref{prop:fibrations of Hopf Lambda-cooperads})
and that this forgetful functor creates limits
to conclude that $W{\op C}^{\Delta^{\bullet}}$ still forms a simplicial frame
of the object $W{\op C}$
in the category of dg Hopf $\La$-cooperads.
\end{proof}

\begin{lemm}\label{lemm:W-construction BLambda-structure}
The $\La$-structure on the dg Hopf cooperad $W{\op C}$ (respectively, $\oW{\op C}^{\Delta^n}$) is identified with the result of the construction of a $\La$-structure
on the bar construction $W{\op C} = B(\oW{\op C}_{\bo})$ (respectively, $W{\op C}^{\Delta^n} = B(\oW{\op C}^{\Delta^n}_{\bo})$)
from a $B\La$ structure associated to the operad $\oW{\op C}_{\bo}$ (respectively, $\oW{\op C}^{\Delta^n}_{\bo}$).
\end{lemm}

\begin{proof}
We equip the collection $\oW{\op C}$ with the covariant $\La$-diagram structure
such that we have a commutative diagram
\[
\begin{tikzcd}
\oW{\op C}(S)[1]\ar[hookrightarrow]{d}\ar[dashed]{r} & \oW{\op C}(S\sqcup\{*\})[1] \\
W{\op C}(S)\ar{r}{\eta_S} & W{\op C}(S\sqcup\{*\})\ar{u}
\end{tikzcd},
\]
for each corestriction operator $\eta_S$, where we consider the canonical embedding $\oW{\op C}[1]\subset W{\op C}$
and the projection map $W{\op C}\to\oW{\op C}[1]$
constructed in the proof of Lemma~\ref{lemm:W-construction cofree structure}.
We similarly take the composite
\[
\K = \Com^c(S)\stackrel{\eta}{\to} W{\op C}(S)\to\oW{\op C}(S)[1],
\]
where we consider the unit morphism of dg algebra $W{\op C}(S)$ for each set of arity $|S|\geq 2$,
to get a map of $\La$-collections $u\in\Hom_{\gr\La}(\overline{\Com}{}^c,\oW{\op C})$
of degree $-1$.
We just check that the corestriction operators defined in the proof of Proposition~\ref{prop:BLambda-operad bar construction}
agree with the result of the construction of (the proof of) Proposition~\ref{prop:W-construction Lambda-structure},
when we use the identity $W{\op C}^{\circ} = {\op F}^c(\oW{\op C}[1])$, and we provide the collection $\oW{\op C}[1]$
with this coaugmented $\La$-diagram structure.
We trivially get that the coaugmentation $\eta: \Com^c\to W{\op C}$ is identified with the cooperad morphism
that lifts our coaugmentation morphism on the cogenerating collection $\oW{\op C}[1]$
too.
We can use that the coherence constraints in the definition of a $B\La$ structure are equivalent to the compatibility of the $\La$-structure
of the graded cooperad $W{\op C}^{\circ} = {\op F}^c(\oW{\op C}[1])$
with the differential of the bar construction
to complete our verifications.

We argue similarly when we consider the simplicial dg Hopf cooperad $W{\op C}^{\Delta^{\bullet}}$.
We may note that our $B\La$ structure is preserved by the simplicial structure
of our object in this case.
\end{proof}

\section{Mapping spaces through biderivations and the proof of Theorem \ref{thm:nerve to mapping spaces}}\label{sec:HGC nerve to mapping spaces}
We now prove the claim of Theorem \ref{thm:nerve to mapping spaces}, the weak-equivalence between the nerve of the dg Lie algebra of biderivations
associated to our dg Hopf cooperad models
of $E_n$-operads
and the mapping spaces associated to (the rationalization of) these $E_n$-operads
in the category of simplicial sets.
We pick a model $\mE_n$ of the operad of little $n$-discs $\lD_n$
in the category of $\La$-operads
in simplicial sets (thus, a cofibrant model of $E_n$-operads in simplicial sets),
as we explain in Section~\ref{sec:rational homotopy}.
Recall that we define the rationalization of this operad by $\mE_n^{\Q} = LG_{\bullet} R\Omega_{\sharp}\mE_n$,
where we consider the derived functors
of the adjoint Quillen functors $G_\bullet: \dg^*\Hopf\La\Op_{01}^c\leftrightarrows s\La\Op_{\varnothing *}^{op} :\Omega_{\sharp}$
between the model category of dg Hopf $\La$-cooperads $\dg^*\Hopf\La\Op_{01}^c$
and the model category of $\La$-operads in simplicial sets $s\La\Op_{\varnothing *}$.
Recall also that we have a weak-equivalence in the category of dg Hopf cooperads $R\Omega_{\sharp}\mE_n\simeq\e_n^c$,
where we set $\e_1^c = \Ass^c$ and $\e_n^c = \Poiss_n^c$ for $n\geq 2$ (see Section~\ref{sec:rational homotopy}).
We then have the following statement.

\begin{prop}\label{prop:Hopf En-cooperad mapping spaces}
We have a weak-equivalence of simplicial sets
\[
\Map_{s\La\Op_0}^h(\mE_m,\mE_n^{\Q})\simeq\Map_{\dg^*\Hopf\La\Op_{01}^c}^h(\e_n^c,\e_m^c),
\]
where we consider the mapping space associated to the objects $\e_m^c$ and $\e_n^c$
in the (homotopy) category of dg Hopf $\La$-cooperads
on the right-hand side.

We have an identity of simplicial sets
\[
\Map_{\dg^*\Hopf\La\Op_{01}^c}^h(\e_n^c,{\op C}) = \Mor_{\dg^*\Hopf\La\Op_{01}^c}(C(\stp_n),W{\op C}^{\Delta^{\bullet}}),
\]
for any $n\geq 2$ and for every dg Hopf $\La$-cooperad ${\op C}$, where we consider the Chevalley-Eilenberg complex $C(\stp_n)$
on the dual cooperad of the graded Drinfeld-Kohno Lie algebra operad $\stp_n$,
and we take the simplicial Hopf $\La$-cooperad $W{\op C}^{\Delta^{\bullet}}$
of Section~\ref{subsec:fibrant resolutions:simplicial frames}.
\end{prop}

\begin{proof}[Explanations]
The first assertion of this proposition follows from the relations
\[
\Map_{s\La\Op_{\varnothing *}}^h(\mE_m,\mE_n^{\Q}) = \Map_{s\La\Op_{\varnothing *}}^h(\mE_m,{LG}_{\bullet}{R\Omega}_{\sharp}\mE_n)
\simeq\Map_{\dg^*\Hopf\La\Op_{01}^c}^h(\Omega_{\sharp}\mE_n,\Omega_{\sharp}\mE_m)
\simeq\Map_{\dg^*\Hopf\La\Op_{01}^c}^h(\e_n^c,\e_m^c)
\]
which we deduce from our Quillen adjunction relation and the equivalence $R\Omega_{\sharp}\mE_n\simeq\e_n^c$
in the homotopy category of dg Hopf $\La$-cooperads.

The second assertion of the proposition follows from the definition of mapping spaces $\Map_{\dg^*\Hopf\La\Op_{01}^c}({\op B},{\op C})$
as morphisms sets $\Mor_{\dg^*\Hopf\La\Op_{01}^c}(\hat{\op B},\hat{\op C}{}^{\Delta^{\bullet}})$,
where we consider a cofibrant resolution $\hat{\op B}$ of the object ${\op B}$ on the source
and a simplicial frame $\hat{\op C}{}^{\Delta^{\bullet}}$
of a fibrant resolution $\hat{\op C}$
of the object ${\op C}$
on the target.
In our case, we just take $\hat{\op C} = W{\op C}$ and this simplicial frame $\hat{\op C}{}^{\Delta^{\bullet}} = W{\op C}^{\Delta^{\bullet}}$
constructed in the previous section. We also use that the Chevalley-Eilenberg complex $C(\stp_n)$
forms a cofibrant resolution of the dg Hopf $\La$-cooperad $\e_n^c$
when $n\geq 2$ (see Theorem~\ref{thm:Drinfeld Kohno model}).
\end{proof}

We now study mapping spaces of this general form
\[
\Map_{\dg^*\Hopf\La\Op_{01}^c}^h(C({\alg g}),{\op C}) = \Mor_{\dg^*\Hopf\La\Op_{01}^c}(C({\alg g}),W{\op C}^{\Delta^{\bullet}}),
\]
where ${\alg g}$ is a $\La$-cooperad in graded Lie coalgebras which is free as a $\La$-collection in graded vector spaces.
We explicitly have ${\alg g} = \La\otimes_{\Sigma}\bS{\alg g}$,
for some $\Sigma$-collection $\bS{\alg g}$.
We also assume that the collection of dg Lie coalgebras ${\alg g}(r)$ underlying our $\La$-cooperad ${\alg g}$
is equipped with a weight grading which is preserved by the structure operations
of our object, as in Remark~\ref{rem:Lambda-biderivation Lie algebra filtration}. (Recall that the cooperad $\stp_n$
satisfies these assumptions.) We have the following statement:

\begin{lemm}\label{lemm:simplicial Hopf cooperad biderivations}
We have an isomorphism of simplicial sets
\[
\Mor_{\dg^*\Hopf\La\Op_{01}^c}(C({\alg g}),W{\op C}^{\Delta^{\bullet}})\cong\MC(\BiDer_{\dg\La}(C({\alg g}),W{\op C}^{\Delta^{\bullet}})),
\]
where we consider the set of Maurer-Cartan elements in the simplicial dg Lie algebra of biderivations
associated to the dg Hopf $\La$-cooperad ${\op B} = C({\alg g})$
and to the simplicial dg Hopf $\La$-cooperad $W{\op C}^{\Delta^{\bullet}}$.
\end{lemm}

\begin{proof}
This lemma is a formal corollary of the result of Proposition~\ref{prop:Lambda-biderivation Lie algebra}.
Note simply that the dg Hopf $\La$-cooperads $W{\op C}^{\Delta^n}$, $n\in\N$, satisfy the assumptions
of this proposition
by the results of Lemma~\ref{lemm:W-construction framing bar structure}
and of Lemma~\ref{lemm:W-construction BLambda-structure}.
\end{proof}

Recall that the complexes of biderivations associated to a source cooperad ${\op B} = C({\alg g})$ form complete graded dg Lie algebras
in the sense of Definition~\ref{defn:complete graded Linfty algebras}
when we assume that our $\La$-cooperad in graded Lie coalgebras ${\alg g}$
is equipped with a weight grading (see Remark~\ref{rem:Lambda-biderivation Lie algebra filtration}).
We may now form another simplicial object in the category of (complete graded) dg Lie algebras
by taking the (completed) tensor product of the dg Lie algebras
of biderivations associated to the pair $(C(\alg g),W{\op C})$
with the commutative dg algebras of polynomial forms on the simplices $\APL(\Delta^{\bullet})$:
\[
\DerL\hat{\otimes}\APL(\Delta^{\bullet}) = \BiDer_{\dg\La}(C(\alg g),W{\op C})\hat{\otimes}\APL(\Delta^{\bullet})
\]
(see Appendix~\ref{sec:Linfty algebras}).
We have a natural map
\[
\phi^{\bullet}: \BiDer_{\dg\La}(C({\alg g}),W{\op C}^{\Delta^{\bullet}})\to\BiDer_{\dg\La}(C({\alg g}),W{\op C})\hat{\otimes}\APL(\Delta^{\bullet})
\]
which we get by multiplying all decorations of vertices of the form $\APL(\Delta^{\bullet})$
in the target of a derivation with values in $W{\op C}^{\Delta^{\bullet}}$.
We have the following statement.

\begin{lemm}\label{lemm:Hopf cooperad biderivation framing}
The above map $\phi^{\bullet}$ is a morphism of simplicial complete graded dg Lie algebras
which is a quasi-isomorphism dimensionwise.
\end{lemm}

\begin{proof}
We easily check, from the construction of Section~\ref{sec:biderivation complexes}, that our map preserves the dg Lie algebra
structures attached to our objects.
We have a commutative diagram
\[
\begin{tikzcd}
\BiDer_{\dg\La}(C({\alg g}),W{\op C}^{\Delta^0})\ar{r}{\cong}\ar{d} &
\BiDer_{\dg\La}(C({\alg g}),W{\op C})\hat{\otimes}\APL(\Delta^0)\ar{d} \\
\BiDer_{\dg\La}(C({\alg g}),W{\op C}^{\Delta^n})\ar{r} &
\BiDer_{\dg\La}(C({\alg g}),W{\op C})\hat{\otimes}\APL(\Delta^n)
\end{tikzcd}
\]
by the functoriality of our construction, for each dimension $n\in\N$, where the vertical maps are given by the simplicial operators
associated to the map $\underline{n}\to\underline{0}$
in the simplicial category. We use that $W{\op C} = W{\op C}^{\Delta^0}\to W{\op C}^{\Delta^n}$
defines a quasi-isomorphism and the homotopy invariance of the dg Lie algebras of biderivations $\BiDer_{\dg\La}(C({\alg g}),-)$
for dg Hopf $\La$-cooperads of this form (see Proposition \ref{prop:Lambda-biderivation complex homotopy invariance})
to check that the vertical morphism on the left-hand side
of this diagram is a quasi-isomorphism.
We use that $\K = \APL(\Delta^0)\to\APL(\Delta^n)$ defines a quasi-isomorphism
as well to check that the vertical morphism on the right-hand side
of our diagram is a quasi-isomorphism too.
We conclude that the lower horizontal morphism is a quasi-isomorphism too,
which is the assertion of the lemma.
\end{proof}

We deduce the following statement from the results of these lemmas.

\begin{thm}\label{thm:Hopf cooperad biderivation nerve}
Let ${\alg g}$ be a $\La$-cooperad in graded Lie coalgebras which is free as a $\La$-collection in graded vector spaces.
Let ${\op C}$ be any dg Hopf $\La$-cooperad.
We have an equivalence in the homotopy category of simplicial sets
\[
\Map_{\dg^*\Hopf\La\Op_{01}^c}^h(C({\alg g}),{\op C})\simeq\MC_\bullet(\BiDer_{\dg\La}(C({\alg g}),W{\op C})),
\]
where we consider the nerve of the complete graded dg Lie algebra of biderivations $\BiDer_{\dg\La}(C({\alg g}),W{\op C})$
on the right hand side.
\end{thm}

\begin{proof}
We use the identity of Lemma~\ref{lemm:simplicial Hopf cooperad biderivations}
and that the quasi-isomorphism of Lemma~\ref{lemm:Hopf cooperad biderivation framing}
induces a weak-equivalence of simplicial sets
on the nerve of our complete graded dg Lie algebras (see Theorem~\ref{thm:GoldmanMillson}
and Remark~\ref{rem:graded quasi-isomorphisms}).
\end{proof}

We can now complete the:

\begin{proof}[Proof of Theorem \ref{thm:nerve to mapping spaces}]
We apply the above theorem to the dual cooperad of the graded Drinfeld-Kohno Lie algebra operad ${\alg g} = \stp_n$
and to the dg Hopf $\La$-cooperad ${\op C} = \e_m^c$.
Then we just use the result of Proposition~\ref{prop:Hopf En-cooperad mapping spaces}
to get the weak-equivalence of Theorem \ref{thm:nerve to mapping spaces}
\[
\Map^h(\lD_m,\lD_n^{\Q})\simeq\MC_\bullet(\Def(\E_n^c,\E_m^c)),
\]
where we take the deformation complex $\Def(\E_n^c,\E_m^c) = \BiDer_{\dg\La}(\check{\E}_n^c,\hat{\E}_m^c)$
determined by the choice $\check{\E}_n^c = C(\stp_n)$,
for the cofibrant model of the dg Hopf $\La$-cooperad $\E_n^c = R\Omega_{\sharp}{\op E}_n$,
and by the choice $\hat{\E}_m^c = W(\e_m^c)$,
for the fibrant model of the dg Hopf $\La$-cooperad $\E_m^c = R\Omega_{\sharp}{\op E}_m$.
(Recall also that we have the relation $\Map^h(\lD_m,\lD_n^{\Q})\simeq\Map_{s\La\Op_0}^h(\mE_m,\mE_n^{\Q})$
by the Quillen equivalence between topological $\La$-operads
and simplicial $\La$-operads.)
\end{proof}

\part{The applications of graph complexes}\label{part:graph complexes}
We complete the proof of the main results of this paper in this part.
To be explicit, we establish that the dg Lie algebra of biderivations $\Def(\E_n^c,\E_m^c) = \BiDer_{\dg\La}(C(\stp_n),W{\e_m})$,
considered in the previous section,
is quasi-isomorphic to the hairy graph complex $\HGC_{m,n}$
as an $L_{\infty}$~algebra.
Then we can use the result of Theorem~\ref{thm:nerve to mapping spaces}, established in the previous part,
and the homotopy invariance
of the nerve of $L_{\infty}$~algebras
under $L_{\infty}$~equivalences,
to get the weak-equivalence $\Map^h(\lD_m,\lD_n^{\Q})\simeq\MC_\bullet(\HGC_{m,n})$
asserted by the results of Theorem~\ref{thm:main topological statement}
and Theorem~\ref{thm:main topological statement:Shoikhet case}.
We address these verifications in the second section of this part, after preliminary recollections on the subject of graph complexes,
to which the first section of the part is devoted.

We then give the detailed proof of the corollaries of Theorem~\ref{thm:main topological statement} and Theorem~\ref{thm:main topological statement:Shoikhet case},
notably the results of Corollary~\ref{cor:codimension zero case}
and Corollary~\ref{cor:codimension one case}
about the homotopy of the mapping spaces $\Map^h(\lD_m,\lD_n^{\Q})$
in the special cases $m=n$ and $m=n-1$.
We devote the third section of this part to these questions.

\section{Recollections on graph complexes and graph operads}\label{sec:graphs}
In this section we recall the construction of the Kontsevich graph complexes $\GC_n$, of the hairy graph complexes $\HGC_{m,n}$,
and of Kontsevich's graph operad $\Graphs_n$.
We will be brief, referring the reader to \cite{K2} or \cite{Will} for more details.

We recall the definition of the operads $\Graphs_n$ and of the Kontsevich graph complexes $\GC_n$ in the first subsection of the section.
We review the definition of the hairy graph complex and we explain an interpretation
of hairy graphs in terms of coderivations of $\La$-cooperads afterwards,
in the second subsection of this section.
We also recall the definition of the Shoikhet $L_{\infty}$~structure on the complex of hairy graphs $\HGC_{1,n}$.
We devote the third subsection to this subject.
We review the connection between the Kontsevich graph complexes $\GC_n$ and the complexes of hairy graphs $\HGC_{n,n}$
to conclude these recollections. We treat this question in the fourth subsection.

\subsection{The Kontsevich graph complexes and the graph operads}\label{subsec:graphs:graph operads}
The operads $\Graphs_n$ were introduced by Kontsevich in his proof of the formality of the little discs operads.
In fact, the operad $\Graphs_n$ is quasi-isomorphic to the $n$-Poisson operad $\Poiss_n$ when $n\geq 2$.
The dual cooperad of this graph operad accordingly forms a dg cooperad that is quasi-isomorphic to the dg cooperad $\e_n^c = \Poiss_n^c$
considered in the previous section.
In the sequel, we will mainly use this graph operad model of the $n$-Poisson operad
to relate the biderivation dg Lie algebras considered in the previous part
to the hairy graph complex.

We provide a survey on the definition of the operads $\Graphs_n$ in this section.
We also recall the definition of the Kontsevich graph complex $\GC_n$.
We start with the preliminary construction of an operad of graphs
which we explain in the next paragraph.

\begin{const}[{The operad of graphs}]\label{const:operad of graphs}
We use the notation $\gra_{r,k}$ for the set of directed graphs with $r$ vertices, numbered from $1$ to $r$,
and $k$ edges, numbered from $1$ to $k$.\footnote{We allow graphs with multiple edges and short loops (edges that connect a vertex to itself).}
We equip this space with an action of the group $\Sigma_r\times\Sigma_k\ltimes\Sigma_2^k$
by permuting the vertex and edge labels and changing the edge directions.
We then consider the graded vector space such that
\[
\Gra_n(r) = \oplus_k(\K\langle\gra_{r,k}\rangle[(n-1)k])_{\Sigma_k\ltimes\Sigma_2^k},
\]
where the action of $\Sigma_k$ involves the signature of permutations when $n$ is even
and the action of $\Sigma_2^k$ involves the signature when $n$ is odd.
We have natural composition products $\circ_i: \Gra_n(k)\otimes\Gra_n(l)\to\Gra_n(k+l-1)$,
so that the collection of these graded vector spaces $\Gra_n$
inherits the structure of an operad. In short, the composite of graphs $\alpha\circ_i\beta$
is defined by plugging the graph $\beta$ in the $i$th vertex of the graph $\alpha$
and by taking the sum of all possible reconnections of the dandling edges
of this vertex in $\alpha$ to vertices of $\beta$.
\end{const}

\begin{rem}\label{rem:operad of graphs degrees}
We see from the above definition that the degree of an element $\alpha$ in the operad of graphs $\Gra_n(r)$
can be determined by assuming that each edge contributes to the degree by $n-1$ (whereas a null degree is assigned to the vertices).
We accordingly have $|\alpha| = (n-1)k$, where $k$ is the number of edges of our graph $\alpha$.
\end{rem}

We use the following observation:

\begin{prop}\label{prop:Poisson operad and graphs}
We have a morphism of operads $\phi: \Poiss_n\to\Gra_n$
such that:
\begin{align*}
\phi(x_1 x_2) & = \begin{tikzpicture}[baseline=-.25ex]
\node[ext] at (0,0) {1};
\node[ext] at (0.5,0) {2};
\end{tikzpicture},
\\
\phi([x_1,x_2]) & = \begin{tikzpicture}[baseline=-.25ex]
\node[ext] (v) at (0,0) {1};
\node[ext] (w) at (0.5,0) {2};
\draw (v) edge (w);
\end{tikzpicture},
\end{align*}
for each $n\geq 2$.\qed
\end{prop}

We are going to consider the following prolongments of the operad morphism
of this proposition:
\[
\hoLie_n\to\Lie_n\to\Poiss_n\to\Gra_n.
\]
Recall that $\Lie_n\subset\Poiss_n$ denotes the operad generated by the graded Lie bracket $[-,-]$ inside $\Poiss_n$,
and we use the notation $\hoLie_n = \Omega(\Com^c\{-n\})$ and $\hoPoiss_n = \Omega(\Poiss_n^c\{-n\})$
for the Koszul resolution of these operads (see Section~\ref{subsec:Koszul duality}).
We actually use the morphism $\hoLie_n\to\Gra_n$ to define the Kontsevich graph complex.
We review the definition of this object in the next paragraph.

\begin{const}[{The Kontsevich graph complexes}]\label{const:GC}
We first define the full graph complex $\fGC_n$ as the operadic deformation complex:
\[
\fGC_n = \Def_{\dg\Op}(\hoLie_n\to\Gra_n),
\]
which is an extension of the complex of operadic derivations $\Der_{\dg\Sigma}(\hoLie_n\to\Gra_n)$
associated to the morphism $\hoLie_n\to\Gra_n$.

In short, the latter dg vector space $\Der_{\dg\Sigma}(\hoLie_n\to\Gra_n)$
is dual to the dg vector spaces
of coderivations considered in Section~\ref{sec:coderivation complexes}.
Recall that we have $\hoLie_n = \Omega(\Com^c\{-n\})$
by construction of the operad $\hoLie_n$,
where we consider the cobar construction
on the $n$-fold operadic desuspension of the commutative cooperad $\Com^c$.
To go further, we can dualize the constructions of Proposition~\ref{prop:coderivation lie algebra}
in order to identify the shifted derivation complex $\Der_{\dg\Sigma}(\hoLie_n\to\Gra_n)[-1]$
with the twisted dg Lie algebra structure associated to the graded hom-object
of $\Sigma$-collections $\DerL = \Hom_{\gr\Sigma}(\overline{\Com}{}^c\{-n\},\Gra_n)$,
We formally have $\Der_{\dg\Sigma}(\hoLie_n\to\Gra_n) = \DerL^{\alpha_{\phi}}$,
for some Maurer-Cartan element $\alpha_{\phi}$
associated to our morphism $\phi: \hoLie_n\to\Gra_n$.
We should specify that the derivations are not required to preserve augmentation ideals in the context of operads,
because we do not necessarily assume that operads, in contrast to cooperads,
come equipped with an augmentation over the unit operad.
We therefore keep the full operad $\Gra_n$ in the above expression of the dg Lie algebra of derivations.
Note however that this subtlety is not significant at this stage,
because the collection $\overline{\Com}{}^c$
vanishes in arity $r\leq 1$.

We actually have
\[
\Def_{\dg\Op}(\hoLie_n\to\Gra_n) = \widehat{\DerL}{}^{\alpha_{\phi}}
\]
when we consider our deformation complex, where we consider a natural extension of the dg Lie algebra $\DerL$
such that
\[
\widehat{\DerL} = \Hom_{\gr\Sigma}(\Com^c\{-n\},\Gra_n).
\]
Thus, in comparison to the coderivation and deformation complexes of cooperads,
which we study in Section~\ref{sec:coderivation complexes},
we add a component $\K = \Hom(\Com^c\{-n\}(1),\Gra_n(1))$
to our graded vector space $\DerL$. We may actually identify this deformation complex $\widehat{\DerL}$
with the extended complex
of derivations $\xDer_{\dg\Sigma}(\hoLie_n\to\Gra_n) = \Der_{\dg\Sigma}(\hoLie_n^+,\Gra_n)$,
where we consider the extended operad $\hoLie_n^+$
of Section~\ref{subsec:coderivation complexes:coderivations}. (To check this identity, we crucially need the convention that operad derivations
are not required to preserve augmentation ideals.)

We can now use that $\Com^c\{-n\}(r)$ is identified with the trivial representation $\Com^c\{-n\}(r) = \K$
of the symmetric group $\Sigma_r$ when $n$ is even (respectively, with the signature representation
when $n$ is odd) to give an explicit description
of this dg Lie algebra $\fGC_n$.
We just get that $\fGC_n$ is identified with the graded vector space spanned by formal series of graphs
with indistinguishable vertices (usually colored in black in our pictures)
together with the differential defined by the blow-up operation
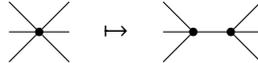
\[
\begin{tikzpicture}[scale=.5,baseline=-.5ex]
\node[int] (v1) at (0,0){};
\draw (v1) -- +(-0.8,0.8);
\draw (v1) -- +(-0.8,0);
\draw (v1) -- +(-0.8,-0.8);
\draw (v1) -- +(0.8,0.8);
\draw (v1) -- +(0.8,0);
\draw (v1) -- +(0.8,-0.8);
\end{tikzpicture}
\quad\mapsto
\quad\begin{tikzpicture}[scale=.5,baseline=-.5ex]
\node[int] (v1) at (-0.5,0){};
\node[int] (v2) at (0.5,0){};
\draw (v1) -- +(-0.8,0.8);
\draw (v1) -- +(-0.8,0);
\draw (v1) -- +(-0.8,-0.8);
\draw (v1) edge (v2);
\draw (v2) -- +(0.8,0.8);
\draw (v2) -- +(0.8,0);
\draw (v2) -- +(0.8,-0.8);
\end{tikzpicture}
\]
on the vertices of graphs.
We can determine the Lie bracket of $\fGC_n$ by the commutator
\[
[\alpha,\beta] = \alpha\bullet\beta -(-1)^{|\alpha||\beta|}\beta\bullet\alpha
\]
of a pre-Lie composition operation of the form
\beq{eq:prelie}
\alpha\bullet\beta = \sum_v\alpha\bullet_v\beta,
\eeq
where $\alpha\bullet_v\beta$ is the sum of the graphs that we obtain by plugging the graph $\beta$
into a vertex $v$ of $\alpha$ and by reconnecting the incident edges of this vertex $v$
in $\alpha$ to vertices of $\beta$ in all possible ways.
The differential is identified as the Lie bracket with the two-vertex graph
\[
\begin{tikzpicture}
\node[int](v) at (0,0) {};
\node[int](w) at (0.5,0) {};
\draw (v) edge (w);
\end{tikzpicture},
\]
which represents the Maurer-Cartan element that corresponds to our morphism $\hoLie_n\to\Gra_n$
in $\fGC_n$.

We immediately see that the graded vector space spanned by (formal series of) connected graphs with at least bivalent vertices
is preserved by the differential and the Lie bracket of this graph complex $\fGC_n$,
and so does the graded vector space spanned by (formal series of) connected graphs with at least trivalent vertices.
We use the notation $\GC_n^2$ for the former dg Lie algebra, which consists of connected graphs with at least bivalent vertices
inside the full graph complex $\fGC_n$, and we actually define the Kontsevich graph complex $\GC_n$
as the latter dg Lie algebra,
which consists of the connected graphs with at least trivalent vertices.
\end{const}

\begin{rem}\label{rem:GC degrees}
We can determine the degree of an element $\alpha$ in the graph complex $\GC_n$ (or in $\GC_n^2$)
by taking the assumption that each edge contributes to the degree by $n-1$
each vertex contributes by $-n$, and by adding $n$
to the obtained total degree.
We accordingly have $|\alpha| = (n-1)k - n r + n$, where $k$ is the number of edges of our graph
and $r$ is the number of vertices. (The term $(n-1)k$ corresponds to the grading of the operad of graphs $\Gra_n$,
and the terms $- n r + n$ come from the graded vector space $\Com^c\{-n\}(r)$,
which is concentrated in degree $n(r-1)$ by definition
of the operadic suspension.)
\end{rem}

We mainly deal with the complexes $\GC_n^2$ and $\GC_n$ in what follows, rather than with the full graph complex $\fGC_n$.
We have the following proposition:

\begin{prop}[{see \cite[Proposition 3.4]{Will}}]\label{prop:graph complex homology}
We have an identity:
\[
H(\GC_n^2) = H(\GC_n)\oplus\bigoplus_{r\equiv 2n-3\mymod{4}}\K L_r
\]
where $L_r$ denotes the (homology class of the) loop graph with $r$ vertices:
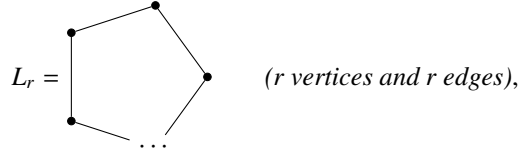
\[
L_r = \begin{tikzpicture}[baseline=-.65ex]
\node[int] (v1) at (0:1) {};
\node[int] (v2) at (72:1) {};
\node[int] (v3) at (144:1) {};
\node[int] (v4) at (216:1) {};
\node (v5) at (-72:1) {$\cdots$};
\draw (v1) edge (v2) edge (v5) (v3) edge (v2) edge (v4) (v4) edge (v5);
\end{tikzpicture}
\qquad\text{($r$ vertices and $r$ edges)},
\]
so that we have the degree formula $|L_r| = n-r$, for any $r\geq 1$.\qed
\end{prop}

We use constructions that parallel the definitions of the graph complexes $\GC_n\subset\GC_n^2\subset\fGC_n$ in the previous paragraph
in order to form a full graph operad $\fGraphs_n$ from the operad of graphs $\Gra_n$,
and suboperads $\Graphs_n\subset\Graphs_n^2\subset\fGraphs_n$,
among which the graph operad $\Graphs_n$
that we consider in the introduction of this subsection.
We briefly review the definitions of these operads
in the next paragraph.

\begin{const}[{The graph operads}]\label{const:graph operads}
We first define by the full graph operad as the twisted operad
\[
\fGraphs_n = \Tw\Gra_n,
\]
which governs the structure attached to a twisted dg vector spaces $A^{\alpha}$
such that $A$ is an algebra over the operad of graphs $\Gra_n$
and $\alpha$ is a Maurer-Cartan element in $A$ (see \cite{DolWill}).
We just use that any algebra over the operad of graphs $\Gra_n$ inherits a natural Lie algebra structure (up to suspension)
by restriction of structure through the operad morphism $\Lie_n\to\Gra_n$
when we consider Maurer-Cartan elements in $A$.

We can represent the elements of the operad $\fGraphs_n$ as formal series of graphs $\alpha\in\Gra_n$
where we informally fill vertices to mark the insertion of our Maurer-Cartan element
in the operations of the operad of graphs $\Gra_n$.
We accordingly get that the elements of the graded vector space $\fGraphs_n(r)$ are represented by (formal series of) graphs
with external vertices, numbered from $1$ to $r$, and indistinguishable internal
vertices, which we usually mark in black in what follows.
We equip this graded vector space $\fGraphs_n(r)$ with the differential defined by the same blow-up operation of graph vertices
as in the case of the full graph complex $\fGC_n$,
except that we may now consider the blow-up of external vertices into an external vertex
and an internal vertex, in addition to the blow-up of internal vertices:
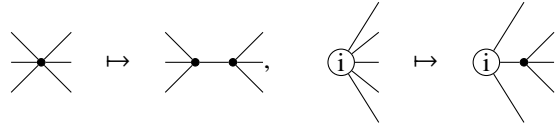
\[
\begin{tikzpicture}[scale=.5,baseline=-.5ex]
\node[int] (v1) at (0,0){};
\draw (v1) -- +(-0.8,0.8);
\draw (v1) -- +(-0.8,0);
\draw (v1) -- +(-0.8,-0.8);
\draw (v1) -- +(0.8,0.8);
\draw (v1) -- +(0.8,0);
\draw (v1) -- +(0.8,-0.8);
\end{tikzpicture}
\quad\mapsto
\quad\begin{tikzpicture}[scale=.5,baseline=-.5ex]
\node[int] (v1) at (-0.5,0){};
\node[int] (v2) at (0.5,0){};
\draw (v1) -- +(-0.8,0.8);
\draw (v1) -- +(-0.8,0);
\draw (v1) -- +(-0.8,-0.8);
\draw (v1) edge (v2);
\draw (v2) -- +(0.8,0.8);
\draw (v2) -- +(0.8,0);
\draw (v2) -- +(0.8,-0.8);
\end{tikzpicture},
\qquad\begin{tikzpicture}[scale=.5,baseline=-.5ex]
\node[ext] (v1) at (0,0){i};
\draw (v1) -- +(1,1.6);
\draw (v1) -- +(1,-1.6);
\draw (v1) -- +(1,0.8);
\draw (v1) -- +(1,0);
\draw (v1) -- +(1,-0.8);
\end{tikzpicture}
\quad\mapsto
\quad\begin{tikzpicture}[scale=.5,baseline=-.5ex]
\node[ext] (v1) at (-0.5,0){i};
\node[int] (v2) at (0.5,0){};
\draw (v1) -- +(1,1.6);
\draw (v1) -- +(1,-1.6);
\draw (v1) edge (v2);
\draw (v2) -- +(0.8,0.8);
\draw (v2) -- +(0.8,0);
\draw (v2) -- +(0.8,-0.8);
\end{tikzpicture}\, .
\]
We equip $\fGraphs_n$ with the operadic composition products $\circ_i: \fGraphs_n(k)\otimes\fGraphs_n(l)\to\fGraphs_n(k+l-1)$
inherited from the operad of graphs $\Gra_n$.
We have an obvious identity $\fGC_n = \fGraphs_n(0)$ when we consider the component of arity zero of this operad $\fGC_n$.

We then define $\Graphs_n^2$ as the suboperad of the full graph operad $\fGraphs_n$ spanned by (formal series of) graphs
with at least bivalent vertices and no connected component consisting entirely of internal vertices,
whereas we define the graph operad $\Graphs_n$ as the operad spanned by (formal series of) graphs with at least trivalent vertices
inside the latter operad $\Graphs_n^2$.
We just check that these subobjects $\Graphs_n\subset\Graphs_n^2\subset\fGraphs_n$
are preserved by the differential and the operadic composition
operations of the full graph operad.

We have an action of the dg Lie algebra $\fGC_n$ by operad derivations on $\fGraphs_n$
which we define by an obvious generalization of the definition of the Lie bracket
of graphs in $\fGC_n$.
We easily check that this action restricts to an action of the dg Lie algebra $\GC_n^2$ on the suboperad $\Graphs_n^2\subset\fGraphs_n$
and to an action of the dg Lie algebra $\GC_n$ on $\Graphs_n\subset\Graphs_n^2$.
We easily check, besides, that the multiplicative group $\K^{\times}$
acts by operad automorphisms
on $\Graphs_n^2$ (respectively, on $\Graphs_n$)
by the operation
\[
\alpha\mapsto\lambda^{\omega(\alpha)}\cdot\alpha,
\]
for any graph $\alpha\in\Graphs_n^2(r)$ (respectively, $\alpha\in\Graphs_n(r)$), for any scalar $\lambda\in\K^{\times}$,
and where set $\omega(\alpha) = (\text{number of vertices})-(\text{number of edges})$.
\end{const}

\begin{rem}\label{rem:graph operad degrees}
We can determine the degree of an element $\alpha$ in the graph operad $\Graphs_n$ (or $\Graphs_n^2$)
by taking the assumption that each edge contributes to the degree by $n-1$
and each internal vertex contributes by $-n$.
We accordingly have $|\alpha| = (n-1)k - n l$, where $k$ is the number of edges of our graph
and $l$ is the number of internal vertices. (The term $(n-1)k$ corresponds to the grading of the operad of graphs $\Gra_n$,
and the term $- n l$ comes from the twisting procedure.)
\end{rem}

We now have an operad morphisms $\Poiss_n\to\Graphs_n$ which is given by the same formulas
as the morphism $\Poiss_n\to\Gra_n$
of Proposition~\ref{prop:Poisson operad and graphs}.
We then have the following statement:

\begin{prop}[{see \cite{K2,LV,Will}}]\label{prop:Poisson to graph operad quasi-iso}
The morphisms
\[
\Poiss_n\to\Graphs_n\hookrightarrow\Graphs_n^2
\]
are quasi-isomorphisms.\qed
\end{prop}

We also consider dual cooperads of the graph operads in our constructions, just as we deal with the Poisson cooperad
rather than with the Poisson operad
when we take our model for the rational homotopy of the little discs operads.
We examine the definition of these objects in the next paragraph.

\begin{const}[The graph cooperads]\label{const:graph cooperads}
We use that each graded vector space $\Graphs_n(r)$ is equipped with a complete weight grading,
since we defined $\Graphs_n(r)$ as the graded vector space spanned by formal series
of graphs with $r$ external vertices
and an arbitrary number of internal vertices (which provide the weight grading).
We consider the natural filtration associated to this weight grading.
We immediately see that this filtration is preserved by the differential and the operad
structure of our objects.
We can therefore take the continuous dual of the dg vector spaces $\Graphs_n(r)$
to get a collection of (filtered) dg vector spaces $\stG_n(r)$
equipped with a cooperad structure
so that we have the identity $\Graphs_n = (\stG_n)^*$
when we take the dual of this cooperad in dg vector spaces to go back to dg operads.
We use the same construction in the case of the graph operad $\Graphs_n$
to get a variant of this graph cooperad $\stG_n^2$
such that $\Graphs_n^2 = (\stG_n^2)^*$.

We can describe $\stG_n(r)$ as the graded vector space spanned by graphs with $r$ external vertices
and an arbitrary number of internal vertices, with the same restrictions on our graphs
as in the case of the graph operad $\Graphs_n(r)$
(the vertices are at least trivalent and we have no connected component entirely made of internal vertices).
We just consider finite linear combinations of graphs, as opposed to the formal series of the graph operad,
when we deal with this cooperad $\stG_n$.
We define the differential of graphs in $\stG_n(r)$ by the adjoint maps of the blow-up operations
considered in Construction~\ref{const:graph operads} (thus, this differential is given by edge contraction operations
that either merge internal vertices together or merge an internal vertex
with an external vertex). 
We similarly define the composition coproducts of this cooperad $\stG_n$
by the adjoint maps of the insertion operations
considered in Construction~\ref{const:graph operads} (thus, these composition coproducts
are given by subgraph extraction operations).
We get a similar picture for the variant $\stG_n^2$
of this cooperad $\stG_n$.

We may observe that the dg vector spaces $\stG_n(r)$ inherit a commutative algebra structure
so that the collection of these objects $\stG_n$
actually form a dg Hopf cooperad.
We formally define the product of the commutative dg algebra $\stG_n(r)$
by the sum of graphs along external vertices.
We get by duality that the dg vector spaces $\Graphs_n(r)$ which form the graph operad
inherit a cocommutative dg coalgebra structure in the complete sense
so that the graph operad $\Graphs_n$
actually forms a complete dg Hopf operad (an operad in the category of complete cocommutative dg coalgebras).
We have a similarly defined dg Hopf cooperad structure on the variant $\stG_n^2$
of the dg cooperad $\stG_n$
and a corresponding complete dg Hopf operad structure on the variant $\Graphs_n^2$
of the graph operad $\Graphs_n$.

We have not been precise about the component of arity zero of the graphs operads $\Graphs_n$ and $\Graphs_n^2$
so far. We may actually see that our constructions make sense for the arity zero component
of these operads
as well and return the definition $\Graphs_n(0) = \Graphs_n^2(0) = \K$.
We therefore get a cooperad with the ground field $\K$
as component of arity zero when we take the dual
of these objects.
We prefer to consider the $\La$-cooperad structure equivalent to this ordinary cooperad structure
associated to the dual of the graph operad $\Graphs_n$ (respectively, $\Graphs_n^2$)
in what follows.
We therefore drop the component of arity zero from these objects and we regard $\stG_n$ (respectively, $\stG_n^2$)
as a dg Hopf $\La$-cooperad instead.
We can explicitly depict the corestriction operator $\eta_S$ associated to this cooperad as the addition
of an isolated external vertex labelled by the composition mark $*$
of our operation.

We mentioned in Construction~\ref{const:graph operads} that the dg Lie algebra $\GC_n$ (respectively, $\GC_n^2$) acts on
the graph operad $\Graphs_n$ (respectively, $\Graphs_n^2$)
by operad derivations. We can readily dualize this construction
in order to pass to cooperads.
We actually get that the dg Lie algebra $\GC_n$ (respectively, $\GC_n^2$) acts on our object $\stG_n$ (respectively, $\stG_n^2$)
by biderivations of Hopf $\La$-cooperads.
\end{const}

We can also dualize the definition of the operad morphisms of Proposition~\ref{prop:Poisson to graph operad quasi-iso}.
We easily check that these morphisms preserve the additional commutative algebra structure
which we consider in the previous paragraph.
We therefore have the following statement:

\begin{prop}\label{prop:Poisson to graph cooperad quasi-iso}
The quasi-isomorphisms of Proposition~\ref{prop:Poisson to graph operad quasi-iso}
are dual to quasi-isomorphisms of dg Hopf $\La$-cooperads
\[
\stG_n^2\stackrel{\simeq}{\to}\stG_n\stackrel{\simeq}{\to}\Poiss_n^c,
\]
where $\Poiss_n^c$ denotes the dual cooperad of the Poisson operad as usual.\qed
\end{prop}

\subsection{Hairy graph complexes and coderivation complexes}\label{subsec:graphs:HGC}
We recall the definition of the hairy graph complex in this subsection and we review the relationship
between this object and the deformation complexes of operads.
We devote the next paragraph to these recollections.
We skip most verifications and details. We refer to the articles \cite{LambrechtsTurchin, Turchin1, Turchin2, Turchin3},
where the hairy graph complexes is defined and studied,
for more details on our claims.

We check in a second step that we can dualize the correspondence between the hairy graph complex and the deformation complex of operads
in order to retrieve the coderivation complexes of dg $\La$-cooperads
studied in Section~\ref{sec:coderivation complexes}.
We address this construction in the second part
of this subsection.

\begin{const}[{The hairy graph complexes}]\label{const:HGC}
We first define a hairy analogue of the full graph complex of Construction~\ref{const:GC}.
We then consider the deformation complex of the morphism
\[
\hoPoiss_m\stackrel{*}{\to}\Graphs_n
\]
defined by the composite
\[
\hoPoiss_m\to\Poiss_m\stackrel{*}{\to}\Poiss_n\to\Graphs_n,
\]
where $\Poiss_m\stackrel{*}{\to}\Poiss_n$ is the natural morphism of operads which preserves the representative
of the commutative product operation in our graded Poisson operads
and carries the Lie bracket to zero.
We have a factorization diagram
\[
\begin{tikzcd}
\Poiss_m\ar{rr}{*}\ar{dr}{} && \Poiss_n \\
& \Com\ar{ur} &
\end{tikzcd}
\]
in the category of operads. We may also see that this morphism $\Poiss_m\stackrel{*}{\to}\Poiss_n$
represents the morphism induced by the embedding of the little $m$-discs operad $\lD_m$
into the little $n$-discs operad $\lD_n$
in homology when $n>m>1$.

We still define the deformation complex of this operad morphism $\hoPoiss_m\stackrel{*}{\to}\Graphs_n$
as a degree shift of the corresponding extended derivation complex
\[
\Def_{\dg\Op}(\hoPoiss_m\stackrel{*}{\to}\Graphs_n) = \xDer_{\dg\Sigma}(\hoPoiss_m\stackrel{*}{\to}\Graphs_n)[1],
\]
which is the derivation complex
\[
\xDer_{\dg\Sigma}(\hoPoiss_m\stackrel{*}{\to}\Graphs_n) = \Der_{\dg\Sigma}(\hoPoiss_m^+\stackrel{*}{\to}\Graphs_n)
\]
associated to the extended operad $\hoPoiss_m^+$ (see Section~\ref{subsec:coderivation complexes:coderivations}).
This deformation complex $\DerL = \Def_{\dg\Op}(\hoPoiss_m\stackrel{*}{\to}\Graphs_n)$
forms a dg Lie algebra,
and we may check again that the underlying graded vector space of this object
is isomorphic to the hom-object
\[
\DerL = \Hom_{\gr\Sigma}(\Poiss_m^c\{-m\},\Graphs_n)
\]
in the category of $\Sigma$-collections, where we use the identity $\hoPoiss_m = \Omega(\Poiss_m^c\{-m\})$ (see Section~\ref{subsec:Koszul duality}).
We still adopt the convention that operad derivations are not required to preserve augmentation ideals (as opposed to the coderivations of cooperads)
and we use this convention to retrieve the whole graph operad $\Graphs_n$ (and not the augmentation ideal)
in this expression.

We now define the full hairy graph complex $\fHGC_{m,n}$
as the dg Lie subalgebra
of the deformation complex $\DerL = \Def_{\dg\Op}(\hoPoiss_m\stackrel{*}{\to}\Graphs_n)$
which, in the hom-object $\Hom_{\gr\Sigma}(\Poiss_m^c\{-m\},\Graphs_n)$,
consists of the maps $\alpha\in\Hom_{\gr\Sigma}(\Poiss_m^c\{-m\},\Graphs_n)$
that:
\begin{itemize}
\item factor through the projection $\Poiss_m^c\{-m\}\to\Com^c\{-m\}$,
\item and land in the subcollection $\Graphs_n'\subset\Graphs_n$ consisting of (formal series of) graphs with univalent external vertices
inside the graph operad.
\end{itemize}
We just check that this graded vector subspace is preserved by the dg Lie algebra
structure of the deformation complex, and hence inherits a dg Lie algebra structure.

We use that $\Com^c\{-m\}(r)$ is identified with the trivial representation $\Com^c\{-m\}(r) = \K$
of the symmetric group $\Sigma_r$ when $m$ is even (respectively, with the signature representation
when $m$ is odd) to give an explicit description
of this dg Lie algebra $\fHGC_{m,n}$
as in the case of the Kontsevich graph complex in Construction~\ref{const:GC}.
We get that $\fHGC_{m,n}$ is identified with the graded vector space spanned by formal series of graphs in $\Graphs_n'$
where the external univalent vertices are made indistinguishable.
We can just omit to mark these external vertices in our pictures.
We only keep the incident edges of these vertices, which we call the hairs of the graph.
The differential of the complex $\fHGC_{m,n}$
is then given combinatorially by the same blow-up operation
on the internal vertices
of graphs
as in the case of the Kontsevich graph complex.



We can determine the Lie bracket of two graphs in the full hairy graph complex $\fHGC_{m,n}$
by attaching a hair of one graph to a vertex
of the other in all possible ways
as we already explained in the introduction of this paper (see~\cite{TW}):
\[
\left[
\begin{tikzpicture}[baseline=-.8ex]
\node[draw,circle] (v) at (0,.3) {$\alpha$};
\draw (v) edge +(-.5,-.7) edge +(0,-.7) edge +(.5,-.7);
\end{tikzpicture}
,
\begin{tikzpicture}[baseline=-.65ex]
\node[draw,circle] (v) at (0,.3) {$\beta$};
\draw (v) edge +(-.5,-.7) edge +(0,-.7) edge +(.5,-.7);
\end{tikzpicture}
\right]
=
\sum
\begin{tikzpicture}[baseline=-.8ex]
\node[draw,circle] (v) at (0,1) {$\alpha$};
\node[draw,circle] (w) at (.8,.3) {$\beta$};
\draw (v) edge +(-.5,-.7) edge +(0,-.7) edge (w);
\draw (w) edge +(-.5,-.7) edge +(0,-.7) edge +(.5,-.7);
\end{tikzpicture}
\pm
\sum
\begin{tikzpicture}[baseline=-.8ex]
\node[draw,circle] (v) at (0,1) {$\beta$};
\node[draw,circle] (w) at (.8,.3) {$\alpha$};
\draw (v) edge +(-.5,-.7) edge +(0,-.7) edge (w);
\draw (w) edge +(-.5,-.7) edge +(0,-.7) edge +(.5,-.7);
\end{tikzpicture}\, .
\]

We now define the hairy graph complex $\HGC_{m,n}$ as the subcomplex of the full hairy graph complex $\fHGC_{m,n}$
spanned by the (formal series of) connected graphs.
We just check that this complex $\HGC_{m,n}$ is preserved by the differential and the Lie bracket on $\fHGC_{m,n}$,
which is immediate from our picture.
We moreover have an identity:
\[
\fHGC_{m,n} = \widehat{S}{}^+(\HGC_{m,n}[-m])[m],
\]
where $\widehat{S}^+$ denotes the completed symmetric algebra without constant term (we just identify
an element of the full hairy graph complex with the formal product
of its connected component).
We may note that the natural projection to the connected part in the full graph complex $\fHGC_{m,n}\to\HGC_{m,n}$
is a morphism of dg Lie algebras.

We can also define a variant $\fHGC_{m,n}^2$ of the full hairy graph complex $\fHGC_{m,n}$, together with a variant $\HGC_{m,n}^2$
of the subcomplex of connected hairy graphs $\HGC_{m,n}\subset\fHGC_{m,n}$, by taking the variant $\Graphs_n^2$
of the graph operad $\Graphs_n$
in our constructions and by allowing bivalent (internal) vertices
(while we assume that these vertices are at least trivalent in the complexes $\HGC_{m,n}$ and $\fHGC_{m,n}$).
\end{const}

\begin{rem}\label{rem:HGC degrees}
We can determine the degree of an element $\alpha$ in the hairy graph complex $\HGC_{m,n}$
by taking the assumption that each internal edge contributes to the degree by $n-1$,
each internal vertex contributes by $-n$,
each hair contributes by $n-1$,
each external vertex contributes by $-m$,
and by adding $m$ to the obtained total degree. In the case where each hair is connected
to one external vertex (which is true as soon as our graph
is connected and is not reduced to an isolated
hair with no internal vertices) and we forget about the external vertices,
then we can equivalently assume that each hair contributes by $n-m-1$
to the degree of our graph.
Thus, in the latter case, we have $|\alpha| = (n-1)k - n l + (n-m-1) h + m$,
where $k$ is the number of edges of our graph, $l$ is the number of internal vertices,
and $h$ is the number of hairs,
while we get $|\alpha| = n-m-1$ in the degenerate case $k=l=0$ and $h=1$.
(The terms $(n-1)(k+h) - n l = (n-1) k + (n-1) h - n l$
come from the grading of the graph operad $\Graphs_n$
whereas the terms $- m h + m$ come from the graded vector space $\Com^c\{-m\}(h)$,
as in the case of the Kontsevich graph complexes.)
\end{rem}

We may see that the embedding $\fHGC_{m,n}\hookrightarrow\fHGC_{m,n}^2$ defines a quasi-isomorphism,
just like the operad embedding $\Graphs_{m,n}\hookrightarrow\Graphs_{m,n}^2$,
and we obviously have the same result at the level of the hairy graph complexes $\HGC_{m,n}\stackrel{\simeq}{\to}\HGC_{m,n}^2$.
We also get that the embedding of the full hairy graph complex $\fHGC_{m,n}$
in the operadic deformation complex $\Def_{\dg\Op}(\hoPoiss_m\stackrel{*}{\to}\Graphs_n)$
defines a quasi-isomorphism
of dg Lie algebras $\fHGC_{m,n}\stackrel{\simeq}{\to}\Def_{\dg\Op}(\hoPoiss_m\stackrel{*}{\to}\Graphs_n)$
(see for instance~\cite{Turchin2,Turchin3, TW}).
We use the following counterpart of this statement in the framework of $\La$-cooperads and coderivations:

\begin{thm}[{compare with the statements of \cite{Turchin2,Turchin3,TW}}]\label{thm:HGC to extended coderivations quasi-iso}
We have a quasi-isomorphism of dg Lie algebras
\[
\fHGC_{m,n}\stackrel{\simeq}{\to}\xCoDer_{\dg\La}(\stG_n,B(\Poiss_m\{m\})),
\]
where we use that the operad ${\op P} = \Poiss_m\{m\}$ has a $B\La$ structure (see Example~\ref{ex:BLambda-en-operads})
to provide the bar construction ${\op C} = B(\Poiss_m\{m\})$
with the structure of a $\La$-cooperad,
and we use that the graph cooperad $\stG_n$ forms a good source $\La$-cooperad in the sense of Definition~\ref{defn:good Lambda-cooperads}
to give a sense to the extended coderivation complex
of the right-hand side.
\end{thm}

\begin{proof}
We just take the composite of the morphism $\stG_n\to\Poiss_n^c$ of Proposition~\ref{prop:Poisson to graph cooperad quasi-iso}
with the cooperad morphism $\Poiss_n^c\to\Com^c$
dual to the operad embedding $\Com\hookrightarrow\Poiss_n$
to provide the graph cooperad $\stG_n$
with an augmentation over the commutative cooperad $\Com^c$.
We can identify the kernel $I(\stG_n)$ of this canonical augmentation morphism $\stG_n\to\Poiss_n^c\to\Com^c$
with the subcollection of the graph cooperad $\stG_n$
spanned the graphs that have at least one edge,
whereas the commutative cooperad $\Com^c$ is identified with the subcollection of empty graphs inside $\stG_n$
(the graphs which entirely consists of isolated external vertices).
We moreover have an obvious relation $I(\stG_n) = \La\otimes_{\Sigma}\bS\stG_n$ where we consider the subcollection
of the graph cooperad $\bS(\stG_n)\subset\stG_n$ spanned by the graphs
which have no isolated external vertices.
We accordingly get that $\stG_n$ forms a good source $\La$-cooperad in the sense of Definition~\ref{defn:good Lambda-cooperads}
as asserted in our statement.

We deduce from the observations of Example \ref{ex:BLambda-en-operads} that ${\op C} = B(\Poiss_m\{m\})$
also forms a good target $\La$-cooperad in the sense
of Definition~\ref{defn:good Lambda-cooperads}.
We can therefore apply the result of Proposition~\ref{prop:coderivation lie algebra}
to the extended coderivation complex of our theorem
\[
\widehat{\DerL} = \xCoDer_{\dg\La}(\stG_n,B(\Poiss_m\{m\})).
\]
We explicitly get that this extended coderivation complex forms a dg Lie algebra with an underlying graded vector space
such that:
\[
\widehat{\DerL}\cong\Hom_{\gr\La}(I(\stG_n),\overline{\Poiss}_m\{m\})\subset\Hom_{\gr\Sigma}(I(\stG_n),\overline{\Poiss}_m\{m\}).
\]

Recall that, when we define this complex $\widehat{\DerL}$, we more precisely consider the extended complex of coderivations
of the composite morphism $\stG_n\to\Com^c\to B(\Poiss_m\{m\}$ (with a degree shift).
We have the duality relation
\[
\hoPoiss_m^* = \Omega(\Poiss_m^c\{-m\})^*\cong B(\Poiss_m\{m\})
\]
and we can readily identify this morphism $\stG_n\stackrel{*}{\to}B(\Poiss_m\{m\})$
with the adjoint of the morphism $\hoPoiss_m^* = \Omega(\Poiss_m^c\{-m\})\stackrel{*}{\to}\Graphs_n$
which we consider in the definition of the full graph complex.
We just dualize the definition of the mapping $\fHGC_{m,n}\to\Def_{\dg\Op}(\hoPoiss_m,\Graphs_n)$ to associate an element
of the hom-object $u_{\alpha}\in\Hom_{\gr\Sigma}(I(\stG_n),\overline{\Poiss}_m\{m\})$
to any hairy graph $\alpha$.
In short, for any graph $\gamma\in I(\stG_n)(r)$, we set $u_{\alpha}(\gamma) = 1$ if we have the relation $\alpha = \gamma$
when we forget about the numbering of the external vertices
in the graph $\gamma$, and $u_{\alpha}(\gamma) = 0$
otherwise. We use this pairing to associate an element of the vector space $\K[m(r-1)] = \Com\{m\}(r)\subset\Poiss_m\{m\}(r)$
to any graph $\gamma\in I(\stG_n)(r)$.
We easily check that this map $u_{\alpha}$ preserves the $\La$-diagram structure attached to our objects
when $\alpha$ belongs to the complex $\fHGC_{m,n}$.
We therefore have $u_{\alpha}\in\Hom_{\gr\La}(I(\stG_n),\overline{\Poiss}_m\{m\})$.
We also see that the differential and the Lie bracket correspond, for such maps, to the differential and the Lie bracket
of the hairy graph complex.

This examination gives the mapping of the theorem. We can adapt the arguments of the cited references
to check that this mapping defines a quasi-isomorphism.
In short, we use the spectral sequence associated to the complete descending filtration by the number of edges of graphs
in the graph cooperad $\stG_n$,
and we check that our morphism induces an isomorphism on the $E^1$ page.
\end{proof}

\subsection{The Shoikhet $L_{\infty}$~structure and the complexes of coderivations on the associative cooperad}\label{subsec:graphs:Shoikhet structure}
Recall that we have $\e_m = \Poiss_m$ when $m\geq 2$, whereas $\e_1$ is identified with the associative operad $\Ass$.
Thus, the previous theorem actually gives information about the extended coderivation complex $\xCoDer_{\dg\La}(\stG_n,B(\e_m\{m\}))$
when $m\geq 2$,
but we need to adapt our constructions in order to extend our results in the case $m=1$.

We still have a morphism
\[
\Ass\stackrel{*}{\to}\Graphs_n,
\]
which we define by the composite $\Ass\to\Com\to\Poiss_n\to\Graphs_n$, where we now consider the usual operad morphism
from the associative operad to the commutative operad
instead of the morphism $\Poiss_m\to\Com$
of the previous subsection.
We have not been precise about the range validity of our constructions in the previous subsection.
We may actually see that our statements make sense for $m=1$ and remain also valid in this case,
when we consider the operad $\Poiss = \Poiss_1$
that governs the usual category of (ungraded) Poisson algebras.
We use that the operad $\Ass$ is equipped with a filtration (the Hodge filtration)
such that $\Poiss = \gr\Ass$ (see for instance \cite[section 3]{TW}).
We can observe that the collections $\Ass$ and $\Poiss$, though not isomorphic as operads,
are isomorphic as bimodules over the Lie operad $\Lie$,
while the differential of our complex only depends on this bimodule structure.
We therefore have the following extension
of the quasi-isomorphism claim
of Theorem~\ref{thm:HGC to extended coderivations quasi-iso}:

\begin{prop}[{see again \cite{Turchin2,Turchin3,TW}}]\label{prop:HGC to extended coderivations quasi-iso:Shoikhet case}
We have a quasi-isomorphism of dg vector spaces
\[
\fHGC_{1,n}\stackrel{\simeq}{\to}\xCoDer_{\dg\La}(\stG_n,B(\Ass\{1\})),
\]
where, as in Theorem~\ref{thm:HGC to extended coderivations quasi-iso},
we use that the operad ${\op P} = \Ass\{1\}$ has a $B\La$ structure (see Example~\ref{ex:BLambda-en-operads})
to provide the bar construction ${\op C} = B(\Ass\{1\})$ with the structure of a $\La$-cooperad
and to give a sense to the extended coderivation complex
of the right-hand side.\qed
\end{prop}

We however need to revise our comparison statement in order to deal with the Lie algebra structure of the extended coderivation complex
of this proposition, because this Lie algebra structure depends on the full operad structure
of the associative operad $\Ass$.
We may in principle transport the Lie algebra structure of our extended coderivation complex
into an $L_{\infty}$~algebra structure on the hairy graph complex $\fHGC_{1,n}$
through the quasi-isomorphism of this proposition.
We actually have explicit formulas for this $L_{\infty}$~structure, which are given in \cite{WillDefQ}.
We just record the main outcomes of the construction of this reference.

We use, as a preliminary observation, that the graph cooperad $\stG_n$
is equipped with a weight grading given by the function $\omega(\alpha) = (\text{number of edges}) - (\text{number of internal vertices})$,
for any graph $\alpha\in\stG_n(r)$.
We easily check that this weight grading is preserved by the differential of the graph cooperad (we have the relation $\omega(d\alpha) = \omega(\alpha)$,
for any graph $\alpha\in\stG_n(r)$),
by the commutative product (we have $\omega(\alpha\cdot\beta) = \omega(\alpha) + \omega(\beta)$,
for all $\alpha,\beta\in\stG_n(r)$),
by the action of permutations (we have $\omega(s\alpha) = \omega(\alpha)$, for all $\alpha\in\stG_n(r)$, and for any permutation $s\in\Sigma_r$),
and by the composition coproducts of our cooperad (if $\Delta_*(\alpha) = \sum_{(\alpha)}\alpha'\otimes\alpha''$
denotes the expansion of the composition coproduct of an element $\alpha\in\stG_n(S)$,
then we have $\omega(\alpha) = \omega(\alpha') + \omega(\alpha'')$,
for each term $\alpha'\otimes\alpha''$ of this expansion, which can be assumed to be of homogeneous weight).
We have an induced complete weight grading $\widehat{\DerL} = \prod_{d\geq 1}\widehat{\DerL}_d$ on the dg Lie algebra
of extended coderivations $\widehat{\DerL} = \xCoDer_{\dg\La}(\stG_n,B(\Ass\{1\}))$,
where $\widehat{\DerL}_d$ consists of the coderivations that vanish on graphs
of weight $\omega(\alpha)\not=d$.
We see that the hairy graph complexes $\fHGC_{1,n}$ and $\HGC_{1,n}$ inherit a weight grading too.
We note that we have $\omega(\alpha)>0$ for any element of the augmentation ideal of the graph cooperad $\alpha\in I(\stG_n)(r)$,
because any such graph has at least one edge, and the trivalence condition
of our definition implies that we have more edges
than internal vertices
in a graph.
We therefore get that the weight grading of our extended coderivation complex starts in weight $d=1$,
and we have a similar result for the hairy graph complexes.

\begin{thm}[{see \cite{WillDefQ}}]\label{thm:HGC Shoikhet structure}
There is an $L_{\infty}$~algebra structure on the full hairy graph complex $\fHGC_{1,n}$, called the \emph{Shoikhet $L_{\infty}$~structure},
such that the following properties hold.
\begin{itemize}
\item
The $L_{\infty}$~structure is compatible with the complete weight grading (in the sense of Definition \ref{defn:complete graded Linfty algebras}).
\item
The subcomplex of connected hairy graphs $\HGC_{1,n}\subset\fHGC_{1,n}$ is preserved by this $L_{\infty}$~structure
and forms an $L_{\infty}$~subalgebra of the full hairy graph complex $\fHGC_{1,n}$
therefore.
\item
There is an $L_{\infty}$~quasi-isomorphism
\[
U: \fHGC_{1,n}\stackrel{\simeq}{\to}\xCoDer_{\dg\La}(\stG_n,B(\Ass\{1\}))
\]
compatible with the complete weight grading in the sense of Definition \ref{defn:complete graded Linfty algebras}.
Furthermore, the linear component of this $L_{\infty}$~quasi-isomorphism $U_1$
agrees with the map of Proposition \ref{prop:HGC to extended coderivations quasi-iso:Shoikhet case}.
\item
There is an $L_{\infty}$~morphism $V: \fHGC_{1,n}\to\HGC_{1,n}$ compatible with the weight grading,
whose linear component $V_1$ is the canonical projection
onto the subcomplex $\HGC_{1,n}$.\qed
\end{itemize}
\end{thm}

In the follow-up, we use the notation $\fHGC'_{1,n}$ when we equip the hairy graph complex
with the Shoikhet $L_{\infty}$~structure of this theorem in order to distinguish this object
from the ordinary dg Lie algebra structure
of Construction~\ref{const:HGC}, and we adopt the same convention for the $L_{\infty}$~subalgebra
of connected hairy graphs $\HGC'_{1,n}\subset\fHGC'_{1,n}$.

\begin{rem}\label{rem:HGC Shoikhet structure}
In fact, the $\La$-cooperad structures of our objects are not considered in the article \cite{WillDefQ},
but we may check that the $L_{\infty}$~morphism $U$ constructed in this reference takes values in the (extended) complex
of coderivations of $\La$-cooperads $\xCoDer_{\dg\La}(\stG_n,B(\Ass\{1\}))$,
and induces an $L_{\infty}$~quasi-isomorphism
with values in this object,
as we do in the proof of Theorem~\ref{thm:HGC to extended coderivations quasi-iso}
in the case of the extended coderivation complexes $\xCoDer_{\dg\La}(\stG_n,B(\Poiss_m\{m\}))$
associated to the Poisson operads $\Poiss_m$.

Let us mention that the results of Theorem~\ref{thm:HGC Shoikhet structure}
also hold for the variants $\HGC_{1,n}^2\subset\fHGC_{1,n}^2$
of the graph complexes considered in this statement. We use this observation in the next subsection.
\end{rem}

\subsection{From the Kontsevich graph complexes to hairy graphs}\label{subsec:graphs:GC}
We survey comparison results that relate the Kontsevich graph complexes to the hairy graph complexes
in the case $m = n-1,n$ in this subsection.

We first assume $m=n$. We can check that the line graph
\beq{eq:linegraph}
L = \begin{tikzpicture}[baseline=-.65ex]
\draw (0,0) -- (1,0);
\end{tikzpicture}
\eeq
defines a Maurer-Cartan element in the hairy graph complex $\HGC_{n,n}^2$.
We have a morphism of dg vector spaces
\beq{eq:temp20}
\K L\oplus\GC_n^2[1]\to\HGC_{m,n}^2
\eeq
which is the trivial inclusion on the first summand, and which maps any graph $\gamma\in\GC_n^2$
in the second summand to the sum of graphs $\gamma_1$
that we may obtain by attaching a hair to a vertex of the graph $\gamma$:
\beq{eq:temp20mapping}
\gamma\mapsto\gamma_1 = \sum_v\begin{tikzpicture}[baseline=-.65ex]
\node (v) at (0,.3) {$\gamma$};
\draw (v) edge +(0,-.5);
\end{tikzpicture}
\qquad\text{(attach a hair at vertex $v$)}.
\eeq
We immediately see that the image of this map \eqref{eq:temp20} commutes with the line graph $L$ in $\HGC_{n,n}^2$.
We deduce from this observation that this map induces a morphism with values in the twisted hairy graph complex $(\HGC_{n,n}^2)^L$.
We then have the following statement:

\begin{prop}[{see \cite[Proposition 2.2.9]{FW}}]\label{prop:GC to HGC quasi-iso}
The map \eqref{eq:temp20} (for $m=n$) induces a quasi-isomorphism
\[
\K L\oplus\GC_n^2[1]\stackrel{\simeq}{\to}(\HGC_{n,n}^2)^L,
\]
where we consider the twisted hairy graph complex $(\HGC_{n,n}^2)^L$
associated to the line graph \eqref{eq:linegraph}.\qed
\end{prop}

We can improve this result in order to handle the natural Lie algebra structure associated to the twisted hairy graph complex $(\HGC_{n,n}^2)^L$:

\begin{thm}[{see \cite{WillTriv}}]\label{thm:GC abelian Lie structure}
The quasi-isomorphism $\K L\oplus\GC_n^2[1]\stackrel{\simeq}{\to}(\HGC_{n,n}^2)^L$ of Proposition \ref{prop:GC to HGC quasi-iso}
can be extended to an $L_{\infty}$~quasi-isomorphism,
where we assume that $\K L\oplus\GC_n^2[1]$ is equipped with a trivial $L_{\infty}$~structure.
(In particular, the twisted complex $(\HGC_{n,n}^2)^L$ is formal as an $L_{\infty}$~algebra.)
Furthermore, the components of this $L_{\infty}$~quasi-isomorphism preserves the complete gradings by loop order on both sides.\qed
\end{thm}

We now consider the case $m=n-1$. We review some facts from \cite{TW}.
We can check that the element
\beq{eq:tripodMC}
T := \sum_{k\geq 1}\underbrace{\begin{tikzpicture}[baseline=-.65ex]
\node[int] (v) at (0,0) {};
\draw (v) edge +(-.5,-.5)  edge +(-.3,-.5) edge +(0,-.5) edge +(.3, -.5) edge +(.5,-.5);
\end{tikzpicture}
}_{2k+1}
\in\HGC_{n-1,n}^2.
\eeq
defines a Maurer-Cartan element in $\HGC_{n-1,n}^2$.
We then have a morphism of dg vector spaces
\beq{eq:temp21}
\K T'\oplus\GC_n^2[1]\to(\HGC_{n-1,n}^2)^T
\eeq
which carries the element $T'$ in the first summand
to the sum
\[
T' = \sum_{k\geq 1} (2k+1)\underbrace{\begin{tikzpicture}[baseline=-.65ex]
\node[int] (v) at (0,0) {};
\draw (v) edge +(-.5,-.5)  edge +(-.3,-.5) edge +(0,-.5) edge +(.3, -.5) edge +(.5,-.5);
\end{tikzpicture}}_{2k+1},
\]
and which maps any graph $\gamma\in\GC_n^2$ in the second summand to the following series,
where, for each $k\geq 1$,
we consider the sum of all possible attachments of $2k+1$ hairs
to the graph $\gamma$:
\begin{equation}\label{eq:gamma_img}
\gamma\mapsto\sum_{k\geq 0}\frac{1}{4^k}\sum\underbrace{\begin{tikzpicture}[baseline=-.65ex]
\node (v) at (0,0) {$\gamma$};
\draw (v) edge +(-.5,-.5)  edge +(-.3,-.5) edge +(0,-.5) edge +(.3, -.5) edge +(.5,-.5);
\end{tikzpicture}}_{2k+1}.
\end{equation}

We have the following theorem established in the cited reference:

\begin{thm}[{see \cite{TW}}]\label{thm:cerf lemma quasi-iso}
The map \eqref{eq:gamma_img} defines a quasi-isomorphism of dg vector spaces $\K T'\oplus\GC_n^2[1]\stackrel{\simeq}{\to}(\HGC_{n-1,n}^2)^T$.\qed
\end{thm}

We can again improve this result in order to handle the natural Lie algebra structure associated to the twisted hairy graph complex $(\HGC_{n-1,n}^2)^T$:

\begin{thm}[{see \cite{WillTriv}}]\label{thm:cerf lemma abelian Lie structure}
The quasi-isomorphism $\K T'\oplus\GC_n^2[1]\stackrel{\simeq}{\to}(\HGC_{n-1,n}^2)^T$ of Theorem \ref{thm:cerf lemma quasi-iso}
can be extended to an $L_{\infty}$~quasi-isomorphism,
where we again assume that $\K T'\oplus\GC_n^2[1]$ is equipped with a trivial $L_{\infty}$~structure.
The components of this $L_{\infty}$~quasi-isomorphism also preserves the complete gradings by loop order on both sides.\qed
\end{thm}

We consider the case $m=1$ aside. We may consider the dg vector space $(\HGC_{1,2}^2)^T$
either as a dg Lie algebra with the ``standard'' dg Lie structure
of the hairy graph complexes,
or as an $L_{\infty}$~algebra, with the Shoikhet $L_{\infty}$~structure
of Theorem \ref{thm:HGC Shoikhet structure}.
To avoid ambiguity, we adopt the notation $\HGC'{}_{1,2}^2$ to refer to the dg vector space $\HGC_{1,2}^2$
equipped with the Shoikhet $L_{\infty}$~structure, as we explain at the end of Section~\ref{subsec:graphs:Shoikhet structure}.
We use that the above Maurer-Cartan element $T$ admits a deformation $\tilde{T}\in\HGC'{}_{1,2}^2$
in this $L_{\infty}$~algebra $\HGC'{}_{1,2}^2$ (see \cite{WillTriv}).

We then have the following theorem, which is established in the cited reference:

\begin{thm}[{see \cite{WillTriv}}]\label{thm:cerf lemma abelian Lie structure:Shoikhet case}
We have an $L_{\infty}$~quasi-isomorphism:
\[
\K[1]\oplus\GC_2^2[1]\stackrel{\simeq}{\to}(\HGC'{}_{1,2}^2)^{\tilde{T}},
\]
where we again assume that $\K[1]\oplus\GC_2^2[1]$ is equipped with a trivial $L_{\infty}$~structure.
(In particular, the twisted complex $(\HGC'{}_{1,2}^2)^{\tilde{T}}$ is formal as an $L_{\infty}$~algebra.)
Furthermore, the components of this $L_{\infty}$~quasi-isomorphism
preserves the complete gradings by number of edges minus number of internal vertices on both sides.\qed
\end{thm}

Let us mention that the $L_{\infty}$~morphism of this theorem is a deformation of the $L_{\infty}$~morphism of Theorem \ref{thm:cerf lemma abelian Lie structure},
in the sense that the difference consists of terms which strictly increase the loop order.
The results of Theorems \ref{thm:GC abelian Lie structure}-\ref{thm:cerf lemma abelian Lie structure:Shoikhet case}
remain actually valid if we replace the Maurer-Cartan elements
of these statements
with more complicated Maurer-Cartan elements, as long as the coefficient
of the line graph $L = \begin{tikzpicture}[scale=.5, baseline=-.65ex]
\draw (0,0)--(1,0);
\end{tikzpicture}$ (respectively, of the tripod graph $Y = \begin{tikzpicture}[scale=.3, baseline=-.65ex]
\node[int] (v) at (0,0){};
\draw (v) -- +(90:1) (v) -- ++(210:1) (v) -- ++(-30:1);
\end{tikzpicture}$
in the case $m=n-1$)
is non-zero in the expansion
of these elements (see \cite{WillTriv}). We moreover have the following statement:

\begin{thm}[see \cite{WillTriv}]\label{thm:cerf lemma:Shoikhet case}
For any $\lambda\in\Q$, there exists a Maurer-Cartan element in $\HGC'{}_{1,2}^2$, defined over the field of rationals,
and such that the coefficient of the tripod graph in the expansion $Y = \begin{tikzpicture}[scale=.3, baseline=-.65ex]
\node[int] (v) at (0,0){};
\draw (v) -- +(90:1) (v) -- ++(210:1) (v) -- ++(-30:1);
\end{tikzpicture}$
in the expansion of this element is the scalar $\lambda$.\qed
\end{thm}

\section{From biderivations to hairy graphs, the proof of Theorem \ref{thm:main algebraic statement}
and of Theorems~\ref{thm:main topological statement}-\ref{thm:main topological statement:Shoikhet case}}\label{sec:biderivations to HGC}
The main purpose of this section is to establish the result of Theorem \ref{thm:main algebraic statement}
which gives the connection between the biderivation dg Lie algebras
associated to our dg Hopf cooperad models
of $E_n$-operads in Section~\ref{sec:HGC nerve to mapping spaces}
and the dg Lie algebras of hairy graphs (the Shoikhet $L_{\infty}$~algebra in the case of $E_1$-operads).

To be specific, we go back to the biderivation complex $\DerL = \BiDer_{\dg\La}(C(\stp_n),W(\e_m^c))$
where we take the Chevalley-Eilenberg complex on the dual cooperad of the graded Drinfeld-Kohno Lie algebra operad
as a source object ${\op B} = C(\stp_n)$, for any $n\geq 2$,
and the image of the cooperad $\e_m^c$
under the fibrant resolution functor of Section~\ref{sec:fibrant resolutions}
as a target object ${\op C} = W(\e_m^c)$, for any $m\geq 1$.
Recall that ${\op B} = C(\stp_n)$ forms a good source Hopf $\La$-cooperad
while ${\op C} = W(\e_m^c)$ forms a good target Hopf $\La$-cooperad
so that the results of Section~\ref{subsec:biderivation complexes:biderivations}
apply to this biderivation complex $\DerL = \BiDer_{\dg\La}(C(\stp_n),W(\e_m^c))$.

In a first step, we produce a zigzag of $L_{\infty}$~morphisms between this dg Lie algebra of biderivations
and the complex of hairy graphs $\HGC_{m,n}$ (which we equip with the Shoikhet $L_{\infty}$~algebra
structure in the case $m=1$). We treat the cases $m>1$ and $m=1$ separately, in the first and second subsections
of this section respectively.
In a second step, we prove that these zigzags of $L_{\infty}$~morphisms give the $L_{\infty}$~quasi-isomorphisms
asserted by the claims of Theorem \ref{thm:main algebraic statement}. We establish this result in our third subsection.

To complete our results, we prove that our $L_{\infty}$~quasi-isomorphisms induce an equivalence of simplicial sets
on the nerve of our dg Lie algebras (respectively, between the nerve of our dg Lie algebra of biderivations
and the Shoikhet $L_{\infty}$~algebra in the case $m=1$)
so that we can deduce the result of Theorem~\ref{thm:main topological statement}-\ref{thm:main topological statement:Shoikhet case}
(about the mapping spaces associated to the topological operads of little discs)
from the statement of Theorem \ref{thm:nerve to mapping spaces},
established in the previous section,
and from the algebraic result of Theorem \ref{thm:main algebraic statement}.
We address this question in the fourth subsection of this section.

\subsection{The $L_{\infty}$~morphism in the case $n,m>1$}\label{subsec:biderivations to HGC:morphism}
In this subsection, we explain the definition of an $L_{\infty}$~morphism that relates the dg Lie algebra
of biderivations $\DerL = \BiDer_{\dg\La}(C(\stp_n),W(\e_m^c))$,
to the hairy graph complexes $\HGC_{m,n}$
when $m\geq 2$.

First, by Proposition \ref{prop:biderivations to coderivations forgetful morphism}, we have an $L_{\infty}$~morphism
\[
\BiDer_{\dg\La}(C(\stp_n),W(\e_m^c))\to \CoDer_{\dg\La}(C(\stp_n),W(\e_m^c))
\]
This map is not a quasi-isomorphism, but we will see below that it induces an injection on homology.
Recall that we have an isomorphism $W(\e_m^c)\cong B(\oW(\e_m^c)_{\bo})$ (see Lemma \ref{lemm:W-construction bar structure}).
In Remark~\ref{rem:W-construction bar structure}, we observed that this isomorphism
gives rise to a weak-equivalence $\Omega(\e_m^c)\stackrel{\simeq}{\to}\oW(\e_m^c)_{\bo}$
by the bar-cobar adjunction.
Hence, we have a zigzag of quasi-isomorphisms of dg operads
\[
\oW(\e_m^c)_{\bo}\stackrel{\simeq}{\leftarrow}\Omega(\e_m^c)\stackrel{\simeq}{\rightarrow}\e_m\{m\},
\]
where the morphism on the right-hand side is given by the Koszul duality (see Section~\ref{subsec:Koszul duality}
and Example~\ref{ex:BLambda-en-operads}).
Recall also that this Koszul duality morphism defines a morphism of operads
equipped with a $B\La$-structure (see again Example~\ref{ex:BLambda-en-operads}).
We can also easily check that the quasi-isomorphism on the left-hand side of this zigzag satisfies this property too
when we equip the operad $\oW(\e_m^c)_{\bo}$ with the $B\La$-structure
of Lemma~\ref{lemm:W-construction BLambda-structure}.
We accordingly get that our zigzag of quasi-isomorphisms of dg operads
gives a zigzag of quasi-isomorphisms of dg $\La$-cooperads
when we apply the bar construction:
\[
W(\e_m^c) = B(\oW(\e_m^c)_{\bo})\stackrel{\simeq}{\leftarrow}B\Omega(\e_m^c)\stackrel{\simeq}{\rightarrow}B(\e_m\{m\}),
\]
and, as a consequence, we have a zigzag of quasi-isomorphisms of dg Lie algebras
when we pass to coderivation complexes:
\[
\CoDer_{\dg\La}(C(\stp_n),W(\e_m^c))\stackrel{\simeq}{\leftarrow}\CoDer_{\dg\La}(C(\stp_n),B(\Omega(\e_m^c)))
\stackrel{\simeq}{\rightarrow}\CoDer_{\dg\La}(C(\stp_n),B(\e_m\{m\})).
\]

Recall that $C(\stp_n)$ and $\stG_n$ are both quasi-isomorphic to the $n$-Poisson cooperad $\e_n^c = \Poiss_n^c$
as dg Hopf $\La$-cooperads
when $n\geq 2$.
Hence, we can also form a zigzag of quasi-isomorphisms of dg $\La$-cooperads
\[
C(\stp_n)\stackrel{\simeq}{\rightarrow}\e_n^c\stackrel{\simeq}{\leftarrow}\stG_n
\]
that we insert in our coderivation complexes to produce a zig-zag of quasi-isomorphisms of dg Lie algebras
\[
\CoDer_{\dg\La}(C(\stp_n),B(\e_m\{m\}))\stackrel{\simeq}{\leftarrow}\CoDer_{\dg\La}(\e_n^c , B(\Omega(\e_m^c)))
\stackrel{\simeq}{\rightarrow}\CoDer_{\dg\La}(\stG_n, B(\Omega(\e_m^c)))
\]
since all of these objects $C(\stp_n)$, $\e_n^c = \Poiss_n^c$ and $\stG_n$
form good source $\La$-cooperads (see the proof of Proposition~\ref{prop:biderivations to coderivations forgetful morphism}
for the case of the cooperad ${\op B} = C(\stp_n)$,
the observations of \cite[Paragraphs 0.10-0.11]{FW}
for the case of the $n$-Poisson cooperad ${\op B} = \Poiss_n^c$,
and Construction~\ref{const:graph cooperads}
for the case of the graph cooperad  ${\op B} = \stG_n$).

The dg Lie algebra on the right-hand side of this zigzag
is quasi-isomorphic to its extended version
by Lemma \ref{lemm:extended Lambda-coderivation complex:acyclicity}:
\[
\CoDer_{\dg\La}(\stG_n,B(\Omega(\e_m^c)))\stackrel{\simeq}{\to}\xCoDer_{\dg\La}(\stG_n,B(\Omega(\e_m^c))).
\]
Furthermore, we have a quasi-isomorphism of dg Lie algebras
\[
\fHGC_{m,n}\stackrel{\simeq}{\to}\xCoDer_{\dg\La}(\stG_n,B(\Omega(\e_m^c)))
\]
by Theorem \ref{thm:HGC to extended coderivations quasi-iso}. We just take the canonical projection
\[
\fHGC_{m,n}\to\HGC_{m,n}
\]
which is a morphism of dg Lie algebras as we observed in Section~\ref{subsec:graphs:HGC} (but not a quasi-isomorphism again) to complete this construction
and to get the following lengthy zigzag of $L_{\infty}$~morphisms:
\begin{multline}\label{eq:m2zigzag}
\BiDer_{\dg\La}(C(\stp_n),W(\e_m^c))\to\CoDer_{\dg\La}(C(\stp_n),W(\e_m^c))\\
\begin{aligned}
& \stackrel{\simeq}{\leftarrow}\CoDer_{\dg\La}(C(\stp_n),B(\Omega(\e_m^c)))
\stackrel{\simeq}{\to}\CoDer_{\dg\La}(C(\stp_n),B(\e_m\{m\}))\\
& \stackrel{\simeq}{\leftarrow}\CoDer_{\dg\La}(\e_n^c,B(\Omega(\e_m^c)))
\stackrel{\simeq}{\to}\CoDer_{\dg\La}(\stG_n,B(\e_m\{m\}))
\end{aligned}\\
\stackrel{\simeq}{\to}\xCoDer_{\dg\La}(\stG_n,B(\e_m\{m\}))\stackrel{\simeq}{\leftarrow}\fHGC_{m,n}\to\HGC_{m,n}.
\end{multline}

\subsection{The $L_{\infty}$~morphism in the case $m=1$}\label{subsec:biderivations to HGC:morphism Shoikhet case}
The construction of the first pieces of the zigzag \eqref{eq:m2zigzag} goes through without changes
when we assume $m=1$
up to the appearance of the hairy graph complex $\fHGC_{1,n}$.

Then we have to consider the $L_{\infty}$~quasi-isomorphism
of Theorem \ref{thm:HGC Shoikhet structure}
\[
\fHGC_{1,n}'\stackrel{\simeq}{\to}\xCoDer_{\dg\La}(\stG_n,B(\e_1\{1\})),
\]
where $\fHGC_{1,n}'$ is the full hairy graph complex equipped with the Shoikhet $L_{\infty}$~structure.
By Theorem \ref{thm:HGC Shoikhet structure}, the complex of connected hairy graphs forms an $L_{\infty}$~subalgebra $\HGC_{1,n}'\subset\fHGC_{1,n}'$
of this $L_{\infty}$~algebra $\fHGC_{1,n}'$
and the canonical projection
\[
\fHGC_{1,n}'\to\HGC_{1,n}'
\]
can be extended to an $L_{\infty}$~morphism. (Recall that the projection itself is not compatible with the $L_{\infty}$~operations.)
Thus, our construction now returns a zigzag of $L_{\infty}$~morphisms
of the following form:
\begin{multline}\label{eq:m1zigzag}
\BiDer_{\dg\La}(C(\stp_n),W(\e_1^c))\to\CoDer_{\dg\La}(C(\stp_n),W(\e_1^c))\\
\begin{aligned}
& \stackrel{\simeq}{\leftarrow}\CoDer_{\dg\La}(C(\stp_n),B(\Omega(\e_1^c)))
\stackrel{\simeq}{\to}\CoDer_{\dg\La}(C(\stp_n),B(\e_1\{1\}))\\
& \stackrel{\simeq}{\leftarrow}\CoDer_{\dg\La}(\e_n^c,B(\Omega(\e_1^c)))
\stackrel{\simeq}{\to}\CoDer_{\dg\La}(\stG_n,B(\e_1\{1\}))
\end{aligned}\\
\stackrel{\simeq}{\to}\xCoDer_{\dg\La}(\stG_n,B(\e_1\{1\}))\stackrel{\simeq}{\leftarrow}\fHGC_{1,n}'\to\HGC_{1,n}',
\end{multline}
where we take the Shoikhet $L_{\infty}$~algebra $\HGC_{1,n}'$
rather than the usual dg Lie algebra of hairy graphs
as target object.

\subsection{The quasi-isomorphism property and the proof of Theorem \ref{thm:main algebraic statement}}\label{subsec:biderivations to HGC:main algebraic theorem}
We now check that:

\begin{lemm}\label{lemm:biderivations to HGC quasi-iso}
The zigzag of $L_{\infty}$~morphisms \eqref{eq:m2zigzag} defines an $L_{\infty}$~quasi-isomorphism:
\[
\BiDer_{\dg\La}(C(\stp_n),W(\e_m^c))\stackrel{\simeq}{\to}\HGC_{m,n},
\]
for all $m,n\geq 2$, and the zigzag of $L_{\infty}$~morphisms \eqref{eq:m1zigzag} defines an $L_{\infty}$~quasi-isomorphism
similarly:
\[
\BiDer_{\dg\La}(C(\stp_n),W(\e_1^c))\stackrel{\simeq}{\to}\HGC_{1,n}',
\]
for all $n\geq 2$.
\end{lemm}

\begin{proof}
We are left with proving that the $L_{\infty}$~morphisms defined by our zigzags \eqref{eq:m2zigzag} and \eqref{eq:m1zigzag}
reduce to quasi-isomorphisms of dg vector spaces when we take the linear part of these morphisms.
We can consequently forget about $L_{\infty}$~structures in our first  zigzag \eqref{eq:m2zigzag}
since our claim concerns the zigzag of morphisms of dg vector spaces
underlying our objects.

Recall that we have identities $\fHGC_{1,n}' = \fHGC_{1,n}$ and $\HGC_{1,n}' = \HGC_{1,n}$ in the category of dg vector spaces.
By the construction of Theorem~\ref{thm:HGC Shoikhet structure},
the linear part of the $L_{\infty}$~morphisms $\xCoDer_{\dg\La}(\stG_n,B(\e_1\{1\}))\stackrel{\simeq}{\leftarrow}\fHGC_{1,n}'\rightarrow\HGC_{1,n}'$
which we consider in our second zigzag in the case $m=1$
reduce to the morphisms of dg vector spaces $\xCoDer_{\dg\La}(\stG_n,B(\e_1\{1\}))\stackrel{\simeq}{\leftarrow}\fHGC_{1,n}\rightarrow\HGC_{1,n}$
considered in \eqref{eq:m1zigzag}.
Hence, we can treat both cases of our theorem $m>1$ and $m=1$ together by proving that the zigzag \eqref{eq:m2zigzag},
where we now consider any pair such that $n\geq 2$ and $m\geq 1$,
reduces to a quasi-isomorphism in the category of dg vector spaces.

This zigzag \eqref{eq:m2zigzag} fits into a commutative diagram of dg vector spaces
depicted in Figure~\ref{eq:bigdiagram}.
\begin{figure}[t]
\[
\begin{tikzcd}
\BiDer_{\dg\La}(C(\stp_n),W(\e_m^c))\ar[bend left]{drr}\ar{d}{\cong} && \\
K(\stp_n[-1],\oW(\e_m^c)_{\bo})\ar{r} &
K(IC(\stp_n),\oW(\e_m^c)_{\bo})\ar{r}{\cong} &
\CoDer_{\dg\La}(C(\stp_n),W(\e_m^c)) \\
K(\stp_n[-1],\Omega(\e_m^c))\ar{r}\ar[swap]{u}{\simeq}\ar{d}{\simeq} &
K(IC(\stp_n),\Omega(\e_m^c))\ar{r}{\cong}\ar[dashed,swap]{u}{\simeq}\ar[dashed]{d}{\simeq} &
\CoDer_{\dg\La}(C(\stp_n),B(\Omega(\e_m^c)))\ar[swap]{u}{\simeq}\ar{d}{\simeq} \\
K(\stp_n[-1],\e_m\{m\})\ar{r}\ar{d}{\simeq} &
K(IC(\stp_n),\e_m\{m\})\ar{r}{\cong} &
\CoDer_{\dg\La}(C(\stp_n),B(\e_m\{m\})) \\
K(B_{Harr}(C(\stp_n)),\e_m\{m\})\ar{ur} && \\
K(B_{Harr}(\e_n^c),\e_m\{m\})\ar{d}{\simeq}\ar[swap]{u}{\simeq}\ar{r} &
K(I\e_n^c,\e_m\{m\})\ar{r}{\cong}\ar[dashed]{d}{\simeq}\ar[dashed,swap]{uu}{\simeq} &
\CoDer_{\dg\La}(\e_n^c,B(\e_m\{m\}))\ar{d}{\simeq}\ar[swap]{uu}{\simeq} \\
K(B_{Harr}(\stG_n),\e_m\{m\})\ar{r}\ar{d}{\simeq} &
K(I(\stG_n),\e_m\{m\})\ar{r}{\cong}\ar[dashed]{d}{\simeq} &
\CoDer_{\dg\La}(\stG_n,B(\e_m\{m\}))\ar{d}{\simeq} \\
K(B_{Harr}(\stG_n),\e_m\{m\}_{\bo})\ar{r} &
K(I(\stG_n),\e_m\{m\}_{\bo})\ar{r}{\cong} &
\xCoDer_{\dg\La}(\stG_n,B(\e_m\{m\})) \\
\HGC_{m,n}\ar[swap]{u}{\simeq}\ar{rr} &&
\fHGC_{m,n}\ar[swap]{u}{\simeq}\ar[bend left]{ll}
\end{tikzcd}
\]
\caption{}\label{eq:bigdiagram}
\end{figure}
To be more precise, when we form this diagram, we claim that the triangular, quadrangular, pentagonal, and hexagonal tiles commute,
while the bottom curved morphism defines a retraction of the bottom horizontal arrow $\HGC_{m,n}$.
The zig-zag \eqref{eq:m2zigzag} corresponds to the right outer rim composites
of this diagram.

The notation $B_{Harr}$ in this diagram refers to the Harrison complex of augmented commutative dg algebras.
Briefly recall that this complex $B_{Harr}(A)$ is defined by the (degree shift of the) cofree Lie coalgebra $L^c(IA[1])$
on the dg vector space $IA[1]$, where $IA$ denotes the augmentation ideal
of our commutative dg algebra $A$,
together with a differential which is the unique coderivation $d: L^c(IA[1])\to L^c(IA[1])$ whose projection onto $IA[1]$
reduces to the map $L^c_2(IA[1])\to IA[1]$ induced by the product of our algebra $A$,
where $L^c_2(IA[1])$ denotes the homogeneous component of weight $2$
of the cofree Lie coalgebra $L^c(IA[1])$.
To be precise, we explicitly have $B_{Harr}(A)^{\circ} = L^c(IA^{\circ}[1])[-1]$
when we forget about differential.
In our diagram, we apply this complex to the components of dg Hopf $\La$-cooperads
such as ${\op B} = C(\stp_n),\e_n^c,\stG_n$
arity-wise.
The outcome of this construction inherits the structure of a dg $\La$-bicomodule over the dg $\La$-cooperad
underlying our dg Hopf $\La$-cooperad ${\op B}$ (we refer to~\cite{FW} for details on this structure result),
and we apply the deformation complex construction
of Section \ref{subsec:coderivation complexes:bicomodules}
to this object.

In our constructions, we also use that the canonical morphism $IA\to IA/IA^2$
with values in the indecomposable quotient $IA/IA^2$
of an augmented commutative algebra $A$
admits a factorization $IA\to B_{Harr}(A)\to IA/IA^2$, where the first morphism is given by the identity between $IA$
and the component $L^c_1(IA^{\circ}[1])[-1]$ of the Harrison complex,
whereas the second morphism is given by the identity between $IA/IA^2$
and the cokernel of the top component of Harrison differential $L^c_2(IA[1])\stackrel{d}{\to}L^c_1(IA[1])$.
Recall also that the latter morphism defines a quasi-isomorphism
when $A$ is a cofibrant object of the category of commutative dg algebras.\footnote{We emphasize that this quasi-isomorphism is a quasi-isomorphism of dg vector spaces, it is not necessarily compatible with the Lie coalgebra structures.}
In the case of a dg Hopf cooperad ${\op B}$, we may see that this construction returns a sequence
of morphisms $I{\op B}\to B_{Harr}({\op B})\to I{\op B}/I{\op B}^2$
in the category of dg $\La$-bicomodules
over ${\op B}$.

We now explain the definitions of the arrows of the diagram with full details from the top to the bottom.
The topmost vertical arrow is the isomorphism provided by Lemma \ref{lemm:biderivation bicomodule complex}.
The next two vertical arrows on the left (and on the right) are induced by the zigzag
of operad morphisms
$\oW(\e_m^c)_{\bo}\stackrel{\simeq}{\leftarrow}\Omega(\e_m^c)\stackrel{\simeq}{\rightarrow}\e_m\{m\}$.
The horizontal arrows so far are just the inclusions provided by the construction of Lemma \ref{lemm:biderivation bicomodule complex forgetful morphism}.
For convenience, we have inserted the general isomorphisms $K(I{\op B},{\op P})\cong\CoDer_{\dg\La}({\op B},B({\op P}))$
which occur in the construction of these inclusions Lemma \ref{lemm:biderivation bicomodule complex forgetful morphism},
and we have marked the vertical morphisms
which connect the intermediate objects $K(I{\op B},{\op P})$
in dash.
In fact, we keep taking these isomorphisms $K(I{\op B},{\op P})\cong\CoDer_{\dg\La}({\op B},B({\op P}))$
all along our diagram on the right-hand side,
and the dashed morphisms represents the counterpart, on these complexes $K(I{\op B},{\op P})$,
of the vertical quasi-isomorphisms that form the middle part
of our zig-zag \eqref{eq:m2zigzag}.

To get the next triangle of the diagram, we apply the functor $K(-,\e_m\{m\})$
to the morphisms of dg $\La$-bicomodules $IC(\stp_n)\to B_{Harr}(C(\stp_n))\to IC(\stp_n)/IC(\stp_n)^2$
associated to the dg $\La$-Hopf cooperad ${\op B} = C(\stp_n)$,
for which we have the obvious identity $IC(\stp_n)/IC(\stp_n)^2 = \stp_n$.

The fifth and the sixth vertical arrows on the left-hand side of our diagram are
the morphisms of deformations complexes induced by zigzag
of quasi-isomorphisms
of dg Hopf $\La$-operads $C(\stp_n)\xrightarrow{\simeq}\e_n^c\xleftarrow{\simeq}\stG_n$,
just as the corresponding morphisms which we consider on the right-hand side
of our diagram.
The horizontal arrows that we take in the left-hand side squares at this level of our diagram
are the morphisms of dg $\La$-bicomodules $I{\op B}\to B_{Harr}({\op B})$
associated to the dg $\La$-Hopf cooperads ${\op B} = \e_n^c,C(\stp_n)$,
just as the diagonal arrow which corresponds to the case ${\op B} = C(\stp_n)$
of this construction.

The next vertical arrows of our diagram are given by the obvious inclusion of operads $\e_m\{m\}$ into $\e_m\{m\}_{\bo}$,
and we just transport the previous horizontal morphisms to this level,
by using the (bi)functoriality of the deformation complex $K(-,-)$.

The last vertical quasi-isomorphism on the left-hand side of the diagram has been constructed in \cite{FW}.
The corresponding quasi-isomorphism on the right is the one given by our zigzag constructions.

The commutativity of this diagram implies that the zigzag \eqref{eq:m2zigzag} reduces to the vertical composites of the left-hand side.
The conclusion that this zigzag defines a quasi-isomorphism therefore follows from the observation
that all vertical morphisms on the left are quasi-isomorphisms as indicated in the diagram.
The proof of our lemma is now complete.
\end{proof}

Furthermore, the verification of this lemma finishes the proof of Theorem \ref{thm:main algebraic statement}
when, in the deformation complex $\Def(\E_n^c,\E_m^c) = \BiDer_{\dg\La}(\check{\E}_n^c,\hat{\E}_m^c)$,
we take $\check{\E}_n^c = C(\stp_n)$
for the cofibrant model of the dg Hopf $\La$-cooperad $\E_n^c = R\Omega_{\sharp}{\op E}_n$
and $\hat{\E}_m^c = W(\e_m^c)$,
for the fibrant model of the dg Hopf $\La$-cooperad $\E_m^c = R\Omega_{\sharp}{\op E}_m$
as in the proof of Proposition~\ref{prop:Hopf En-cooperad mapping spaces}.
\hfill\qed
\\

\subsection{Comparison of the nerves and the proof of Theorem \ref{thm:main topological statement}}\label{subsec:biderivations to HGC:main topological theorem}
The goal of this subsection is to establish that the $L_{\infty}$~quasi-isomorphisms constructed in the previous subsections
induce an equivalence of simplicial sets
between the nerve of the dg Lie algebras of biderivations
associated to our models of $E_n$-operads
and the nerve of the dg Lie algebras of hairy graphs
(respectively, the nerve of the Shoikhet $L_{\infty}$~algebra in the case $m=1$).
To be explicit, we check the following lemma:

\begin{lemm}\label{lemm:biderivation to HGC nerve equivalence}
The $L_{\infty}$~quasi-isomorphisms of Lemma~\ref{lemm:biderivations to HGC quasi-iso} induce a weak-equivalence
of simplicial sets
when we pass to the nerve of our dg Lie algebras (respectively, of the Shoikhet $L_{\infty}$~algebra in the case $m=1$).
Thus, we have a weak-equivalence of simplicial sets
of the form
\[
\MC_{\bullet}(\BiDer_{\dg\La}(C(\stp_n),W(\e_m^c)))\stackrel{\simeq}{\to}\MC_{\bullet}(\HGC_{m,n})
\]
in the case $m,n\geq 2$, where $\HGC_{m,n}$ is equipped with the standard Lie algebra structure,
and a weak-equivalence of simplicial sets
of the form
\[
\MC_{\bullet}(\BiDer_{\dg\La}(C(\stp_n),W(\e_1^c)))\stackrel{\simeq}{\to}\MC_{\bullet}(\HGC_{1,n}')
\]
in the case $m=1$, $n\geq 2$, where we consider the dg vector of hairy graphs equipped with the Shoikhet $L_{\infty}$~algebra structure $\HGC_{1,n}'$.
\end{lemm}

\begin{proof}
We check that our dg Lie algebras can be equipped with a weight grading compatible
with the $L_{\infty}$~quasi-isomorphisms
of Lemma~\ref{lemm:biderivations to HGC quasi-iso},
and we rely on the result of Theorem~\ref{thm:GoldmanMillson}, together with the observation
of Remark~\ref{rem:graded quasi-isomorphisms},
to establish this lemma. (We just need to adapt our arguments in the case $m=1$,
when we deal with the Shoikhet $L_{\infty}$~structure,
because we get an $L_{\infty}$~algebra equipped with a filtration
rather than with a weight grading
in this case.)

In Section~\ref{sec:graphs}, we explained that the graph cooperad $\stG_n$
is equipped with a weight grading given by the function
\[
\omega(\alpha) = (\text{number of edges}) - (\text{number of internal vertices}),
\]
for any graph $\alpha\in\stG_n(r)$.
Recall that this weight grading is preserved by the differential of the graph cooperad (we have the relation $\omega(d\alpha) = \omega(\alpha)$,
for any graph $\alpha\in\stG_n(r)$),
by the commutative product (we have $\omega(\alpha\cdot\beta) = \omega(\alpha) + \omega(\beta)$,
for all $\alpha,\beta\in\stG_n(r)$),
by the action of permutations (we have $\omega(s\alpha) = \omega(\alpha)$, for all $\alpha\in\stG_n(r)$, and for any permutation $s\in\Sigma_r$),
and by the composition coproducts of our cooperad (if $\Delta_*(\alpha) = \sum_{(\alpha)}\alpha'\otimes\alpha''$
denotes the expansion of the composition coproduct of an element $\alpha\in\stG_n(S)$,
then we have $\omega(\alpha) = \omega(\alpha') + \omega(\alpha'')$,
for each term $\alpha'\otimes\alpha''$ of this expansion).

We use that the cooperad $\e_n^c = \Poiss_n^c$ is equipped with an analogous weight grading.
For this purpose, we first consider the weight grading of the $n$-Poisson operad $\Poiss_n$
by the number of Lie brackets in the expression of Poisson monomials.
We equivalently determine the weight of Poisson monomials
by the function $\omega: p\mapsto \omega(p)$
such that
\[
\omega(x_1 x_2) = 0,\quad \omega([x_1,x_2]) = 1,\quad\text{and}\quad \omega(p\circ_i q) = \omega(p)+\omega(q),
\]
for any composition product of elements of homogeneous weight $p\in\Poiss_n(k)$, $q\in\Poiss_n(l)$.
We just transport this weight grading to our cooperad $\e_n^c = \Poiss_n^c$ by duality, so that we have the relation $\omega(p^*) = \omega(p)$
for any element $p^*$ of the dual basis of the basis of Poisson monomials in $\Poiss_n(r)$.
We easily check that this weight grading on $\e_n^c$ is preserved by the commutative algebra products,
by the action of permutations, and by the composition coproducts of our cooperad,
like the weight grading of the graph cooperad.
We can equivalently determine the weight grading on $\e_n^c(r)$ from the Arnold presentation, which asserts that the algebra $\e_n^c(r)$
is given by a quotient of a free graded symmetric algebra $S(e_{ij},1\leq i\not=j\leq r)$
by a homogeneous ideal generated by quadratic relations. (Recall that we have $\e_n^c(r) = H^*(\lD_n(r),\K)$, for each $r>0$.
The Arnold presentation concerns the cohomology algebras of the configuration spaces of points
in the Euclidean space $\R^n$, which are the same as the cohomology algebras of the little discs spaces $\lD_n(r)$
by homotopy invariance of the cohomology.)
We actually have $\omega(e_{ij}) = 1$, for each pair $i\not=j$.

We can also provide the graded Drinfeld-Kohno Lie algebra operad ${\alg p}_n$ with a weight grading.
We then use the Drinfeld-Kohno presentation of the Lie algebras ${\alg p}_n(r)$
as a quotient of the free graded Lie algebras $L(t_{ij},1\leq i\not=j\leq r)$
by a homogeneous ideal generated by quadratic relations.
We just assign the weight $\omega(t_{ij}) = 1$ to each generator $t_{ij}\in{\alg p}_n(r)$ in this presentation,
as in the case of the Arnold presentation.
This weight grading is preserved by the Lie algebra structure by construction, by the action of permutations,
and by the additive operadic composition products attached to our object.
We again transport this weight grading to the dual cooperad of this operad $\stp_n$, and we equip the cooperad $C(\stp_n)$
defined by the Chevalley-Eilenberg cochain complex of the graded Lie coalgebras $\stp_n(r)$
with the weight grading such that $\omega(\alpha_1\cdots\alpha_l) = \omega(\alpha_1)+\dots+\omega(\alpha_l)$,
for any product of elements $\alpha_1,\dots,\alpha_l\in\stp_n(r)$
in the Chevalley-Eilenberg cochain complex.
We easily check that this weight grading on $C(\stp_n)$ is preserved by the commutative algebra products,
by the action of permutations, and by the composition coproducts of our cooperad,
in the same sense as in the case of the graph cooperad
and of the $n$-Poisson cooperad.

We provide the dg Lie algebras of coderivations and of biderivations of the diagram of Figure~\ref{eq:bigdiagram}
with the complete weight grading $\widehat{\DerL} = \prod_{d\geq 1}\widehat{\DerL}_d$
such that $\widehat{\DerL}_d$
consists of the coderivations (respectively, biderivations) that vanish on the elements
of weight $\omega(\alpha)\not=d$.
We may still see that we have the relation $\omega(\alpha)>0$
for any element in the augmentation ideal $I{\op B}(r)$
of the algebras ${\op B}(r)$
that form our cooperads ${\op B} = \stG_n,\e_n^c,C(\stp_n)$.
We therefore get that the weight grading of our objects starts in weight $d=1$
as indicated in our decomposition formula.

We use that the hairy graph complexes $\fHGC_{m,n}$ and $\HGC_{m,n}$ inherit a complete weight grading too,
like the hairy graph complexes $\fHGC_{1,n}$ and $\HGC_{1,n}$
which we consider in Section~\ref{subsec:graphs:Shoikhet structure} (we can actually use the same formula as in the case
of the graph cooperad
to define the weight grading of a graph in these complexes $\fHGC_{m,n}$ and $\HGC_{m,n}$).
We assume in particular that the bottom vertical morphisms of the diagram of Figure~\ref{eq:bigdiagram}
preserve the weight grading by construction.
We may also note that the morphisms of dg Hopf $\La$-cooperads $C(\stp_n)\stackrel{\simeq}{\rightarrow}\e_n^c\stackrel{\simeq}{\leftarrow}\stG_n$,
which we use in the definition of the vertical morphism of our diagrams preserve the weight grading too.
We deduce from this observation that the other vertical morphisms
in our diagrams preserve the complete weight grading
of our coderivation
and biderivation complexes too.
We easily check that this is also the case of the corresponding horizontal morphisms,
because these morphisms are given by natural (weight preserving) restriction operations
at the coefficient level.
We conclude that all morphisms in our diagram preserve the complete weight grading that we attach to our objects.

We can pick $L_{\infty}$~quasi-inverses of the upwards $L_{\infty}$~quasi-isomorphisms of (the right hand column of) our diagram
that preserves the complete weight grading (use the standard effective constructions of $L_{\infty}$~quasi-inverses).
We then get an $L_{\infty}$~morphism $\BiDer_{\dg\La}(C(\stp_n),W(\e_m^c))\to\HGC_{m,n}$
(an $L_{\infty}$~quasi-isomorphism by the result of Lemma~\ref{lemm:biderivations to HGC quasi-iso})
which is compatible with the complete weight grading
when we compose these maps.

This verification immediately enables us to apply the result of Theorem~\ref{thm:GoldmanMillson},
together with the observation of Remark~\ref{rem:graded quasi-isomorphisms},
in order to get the conclusion of the lemma in the case $m>1$.
If $m=1$, then we have to replace the complexes $\fHGC_{1,n}$ and $\HGC_{1,n}$
by the Shokheit $L_{\infty}$~algebras $\fHGC_{1,n}'$ and $\HGC_{1,n}'$.
In this case, we just use the assertions of Theorem \ref{thm:HGC Shoikhet structure} to get the weight grading preserving $L_{\infty}$~quasi-isomorphisms
required by our arguments, and the conclusion follows the same way as in the case $m>1$.
\end{proof}

This lemma, together with the statement of Theorem \ref{thm:nerve to mapping spaces}, established in Section \ref{sec:HGC nerve to mapping spaces},
implies that the mapping spaces associated to the topological operads of little discs
are weakly equivalent to the nerve of the dg Lie algebra of hairy graphs (respectively, to the Shoikhet $L_{\infty}$~algebra)
as asserted in Theorem~\ref{thm:main topological statement}-\ref{thm:main topological statement:Shoikhet case},
and hence, completes the proof of these statements,
as expected.
\hfill\qed

\section{The structure of the nerve of \texorpdfstring{$\HGC_{m,n}$}{HGC\_m,n} and the proof of the corollaries of Theorem \ref{thm:main topological statement}}\label{sec:HGC nerve}
Theorem \ref{thm:main topological statement} allows us to deduce homotopical properties of the mapping spaces $\Map(\lD_m,\lD_n^{\Q})$
from the examination of combinatorial properties of the hairy graph complexes $\HGC_{m,n}$.
The purpose of this section is to establish the corollaries of the introduction of the paper by using this combinatorial reduction.
Thus, we mostly forget about the homotopy of the $E_n$-operads in this section, and we focus on the study of the graph complexes themselves.

To summarize the main outcomes of this study, we first check the claims of Corollary \ref{cor:mapping spaces homotopy} and Corollary \ref{cor:mapping spaces homotopy:low degrees},
about the homotopy of our mapping spaces $\Map(\lD_m,\lD_n^{\Q})$
in the case $n-m\geq 2$.
Then we prove the claims of Corollary \ref{cor:codimension zero case}, Corollary \ref{cor:codimension zero case:low degrees} and Corollary \ref{cor:codimension zero null component},
about the homotopy of the connected components of the mapping spaces $\Map(\lD_n,\lD_n^{\Q})$
and their relationship to the Kontsevich graph complex $\GC_n$,
and the parallel claims of Corollary \ref{cor:codimension one case}
and Corollary \ref{cor:codimension one null component},
about the connected components of the mapping spaces $\Map(\lD_{n-1},\lD_n^{\Q})$.
We also check the observation of Remark \ref{rem:1/4},
about the connected component of the canonical map $\lD_{n-1}\to\lD_n$.

We devote the following subsections to these verifications, which we address
in the order of this summary. Throughout this section, we use Berglund's theorem,
which asserts that the homotopy of the nerve of an $L_{\infty}$~algebra $\DerL$
at a base point $\alpha\in\MC(\DerL)$
is given by the homology of the twisted $L_{\infty}$~algebra $\DerL$,
up to a degree shift by one (see Theorem~\ref{thm:berglund}).

\subsection{Proof of Corollary \ref{cor:mapping spaces homotopy}}\label{subsec:HGC nerve:homotopy}
We can check that the complex $\HGC_{m,n}$ is concentrated in degrees $\geq 1$ when $n-m\geq 2$ and $m\geq 1$
by counting the contribution of the internal vertices, of the internal edges, and of the hairs
in the degree of a graph (as in Remark~\ref{rem:HGC degrees}), and by using the assumption that each internal vertex
has at least three incident edges, each internal edge is incident to two vertices,
and each hair is incident to one vertex (see \cite[Proof of Proposition 2.2.7]{FW} for details).
We deduce from this observation that the simplicial set $\MC_{\bullet}(\HGC_{m,n})$
has a unique vertex, given by the zero element in the degree $0$ component
of the hairy graph complex $\HGC_{m,n}$.
Hence, this simplicial set $\MC_{\bullet}(\HGC_{m,n})$ is connected as soon as $n-m\geq 2$, for all $m\geq 1$,
and so is the mapping space $\Map(\lD_m,\lD_n^{\Q})$
by Theorem \ref{thm:main topological statement}-\ref{thm:main topological statement:Shoikhet case}.

Then we can apply the result of Theorem~\ref{thm:berglund} to get the second claim of Corollary \ref{cor:mapping spaces homotopy},
about the computation of the homotopy groups
of the spaces $\Map(\lD_m,\lD_n^{\Q})\simeq\MC_{\bullet}(\HGC_{m,n})$
in terms of the homology of the hairy graph complex $\HGC_{m,n}$.
From the vanishing of the complex $\HGC_{m,n}$ in degree $0$,
we also get the relation $\pi_1\Map(\lD_m,\lD_n^{\Q}) = H_0(\HGC_{m,n}) = 0$,
so that the space $\Map(\lD_m,\lD_n^{\Q})$ is simply connected (and not only connected)
when $n-m\geq 2$, as claimed in Corollary \ref{cor:mapping spaces homotopy},
\hfill\qed

\subsection{Proof of Corollary \ref{cor:mapping spaces homotopy:low degrees}}\label{subsec:HGC nerve:low degree homotopy}
We count the contribution of the internal vertices, of the internal edges,
and of the hairs in the degree of a (connected) hairy graph.
We now assume that our graph $\alpha$ has loop order $g$,
so that we have $k-l = g-1$,
where $k$ is the number of internal edges of our graph and $l$ is the number of internal vertices.
We moreover have $3l\leq 2k+h\Leftrightarrow l\leq 2(g-1)+h$, where $h$ is the number of hairs of our graph,
as soon as we assume that each internal vertex
has at least $3$ incident edges.
We also have $h\geq 1$ by definition of the hairy graph complex.
We then obtain
\begin{align*}
|\alpha| & = (n-1)k - n l + (n-m-1)h + m = (n-1)(g-1) - l + (n-m-1)h + m\\
& \geq (n-1)(g-1) - 2(g-1) + (n-m-2)h + m\\
& \geq (n-3)(g-1) + (n-2).
\end{align*}
We accordingly have $|\alpha|\geq 3n-8$ if $g\geq 3$, so that the components of loop order $g\leq 2$ of the hairy graph complex $\HGC_{m,n}$
capture the whole homology group $H_{i-1}(\HGC_{m,n})$, and hence the whole homotopy group $\pi_i\Map(\lD_m,\lD_n)$
by the result of Corollary \ref{cor:mapping spaces homotopy}, as soon as $i-1<3n-8\Leftrightarrow i\leq 3n-8$. The conclusion of Corollary \ref{cor:mapping spaces homotopy:low degrees} follows.
\hfill\qed

\subsection{Proof of Corollary \ref{cor:codimension zero case}}\label{subsec:HGC nerve:GC homotopy}

We consider the map $F: \MC(\HGC_{m,n})\to\K$ which, to any Maurer-Cartan element $\gamma\in\MC(\HGC_{m,n})$,
assigns the coefficient of the line graph $L = \begin{tikzpicture}[baseline=-.65ex]
\draw (0,0) -- (1,0);
\end{tikzpicture}$
in the expansion of $\gamma$ as a formal series of graphs in the hairy graph complex.
We may check that this mapping corresponds to the map $F: \pi_0\Map(\lD_n,\lD_n)\to\Q$
considered in Corollary \ref{cor:codimension zero case}.

We now consider the element such that $\alpha = \lambda L$ in the hairy graph complex, where we assume $\lambda\in\Q^{\times}$.
We still trivially have $\alpha\in\MC(\HGC_{n,n})$, and by using the translation operation $\gamma\mapsto\alpha+\gamma$,
we can identify the simplicial set $F^{-1}(\lambda)\subset\MC_\bullet(\HGC_{n,n})$
with the nerve of the dg Lie algebra
\[
\HGC_{n,n}^{\sim}\subset\HGC_{n,n}^{\alpha}
\]
formed by the formal series of graphs in which the line graph has a null coefficient.

We may see, on the other hand, that the line graph spans the component of loop order zero of the complex $\HGC_{n,n}$.
Indeed, this subcomplex $\HGC_{n,n}^{0-\text{loop}}$, which consists of the part of loop order zero in the hairy graph complex $\HGC_{n,n}$,
is identified with the complex spanned by trees with the indistinguishable external vertices on the leaves. If such a tree is not the line graph, it necessarily has an (internal) vertex. Since by assumption vertices have valence at least three, there is necessarily a vertex with more than one hair.
But by symmetry such trees vanish (hairs are odd objects), and hence the line graph is the only allowed graph of loop order 0.

Hence $\HGC_{n,n}^\sim$ is the part of loop order $\geq 1$ of the twisted dg Lie algebra $\HGC_{n,n}^{\alpha}$,
\[
\HGC_{n,n}^\sim = (\HGC_{n,n}^{\alpha})^{\geq 1-\text{loop}}.
\]

In our previous statements, we equip the hairy graph complex with the complete weight grading defined by the weight function
such that $\omega(\gamma) = (\text{number of edges}) - (\text{number of internal vertices})$
for any graph $\gamma\in\HGC_{n,n}$.
The Lie bracket with our Maurer-Cartan element $\alpha = \lambda L$
increases this weight by one.
Hence, we get that this weight grading determines a complete descending filtration on the twisted complex $\HGC_{n,n}^{\alpha}$.
To simplify our subsequent verifications, we replace this filtration
by the grading by the loop order, which we can determine by the function
such that $g(\gamma) = (\text{number of internal edges}) - (\text{number of internal vertices}) + 1$,
for any graph $\gamma\in\HGC_{n,n}$.
We have $\omega(\gamma)\geq g(\gamma)$, so that the loop order determines a filtration of our dg Lie algebra
which does not change the nerve by the result of Theorem~\ref{thm:nerve equivalences}.
We moreover see that both the Lie bracket and the differential of the twisted complex $\HGC_{n,n}^{\alpha}$
preserve the loop order.
We accordingly get that the loop order determines a complete weight decomposition (and not only a filtration)
of our dg Lie algebras.
We can then apply the result of Theorem \ref{thm:GoldmanMillson} to conclude that the above quasi-isomorphism of dg Lie algebras
induces a weak-equivalence of simplicial sets
at the nerve level:
\[
\MC_\bullet((\HGC_{n,n}^{\alpha})^{\geq 1-\text{loop}})\stackrel{\simeq}{\to}\MC_{\bullet}(\HGC_{n,n}^\sim) = F^{-1}(\lambda).
\]

We consider the extension of our dg Lie algebra
\[
(\HGC_{n,n}^{\alpha})^{\geq 1-\text{loop}}\subset(\HGC_{n,n}^{2,\alpha})^{\geq 1-\text{loop}}
\]
which we define by allowing graphs with bivalent internal vertices
in the construction of our object.
We get that this inclusion defines a quasi-isomorphism of (complete graded) dg Lie algebras, just like the canonical embedding $\HGC_{n,n}\subset\HGC_{n,n}^2$
when we consider the whole complexes of connected hairy graphs,
and we therefore have the following extra weak-equivalence of simplicial sets
at the nerve level:
\[
\MC_\bullet((\HGC_{n,n}^{\alpha})^{\geq 1-\text{loop}})\stackrel{\simeq}{\to}\MC_\bullet((\HGC_{n,n}^{2,\alpha})^{\geq 1-\text{loop}})
\]
(we use the result of Theorem \ref{thm:GoldmanMillson} again).
We then use the quasi-isomorphism of dg Lie algebras of Theorem~\ref{thm:GC abelian Lie structure}
\[
\K L\oplus\GC_n^2[1]\stackrel{\simeq}{\to}(\HGC_{n,n}^2)^{\alpha},
\]
where we assume that $\K L\oplus\GC_n^2[1]$ represents an abelian dg Lie algebra (equipped with a trivial Lie bracket).
Recall that we need the assumption $\lambda\not=0$ to get this result.
We can obviously provide the graph complex $\GC_n^2$
with the same grading by the loop order as our twisted hairy graph complex $(\HGC_{n,n}^{2,\alpha})^{\geq 1-\text{loop}}$,
and we immediately see that the above quasi-isomorphism induces a quasi-isomorphism
of complete graded dg Lie algebras
\[
\GC_n^2[1]\stackrel{\simeq}{\to}(\HGC_{n,n}^{2,\alpha})^{\geq 1-\text{loop}}
\]
when we drop the part of loop order zero.
We moreover have an obvious formality quasi-isomorphism of dg Lie algebras
\[
H(\GC_n^2)[1]\stackrel{\simeq}{\to}\GC_n^2[1]
\]
since $\GC_n^2[1]$ is equipped with an abelian dg Lie algebra structure (pick representatives of a basis of homology classes
in the graph complex $\GC_n^2[1]$).
We can assume that our formality quasi-isomorphism preserves the loop order, and we can again apply the result of Theorem \ref{thm:GoldmanMillson}
to establish that we have weak-equivalences at the nerve level:
\[
\MC_\bullet(H(\GC_n^2[1]))\stackrel{\simeq}{\to}\MC_\bullet(\GC_n^2[1])\stackrel{\simeq}{\to}\MC_\bullet((\HGC_{n,n}^{2,\alpha})^{\geq 1-\text{loop}}).
\]

By putting all our weak-equivalences together, we conclude that we have a weak-equivalence
of simplicial sets
\[
F^{-1}(\lambda)\simeq\MC_\bullet(H(\GC_n^2)[1]),
\]
for each value of the parameter $\lambda\in\Q^{\times}$.

The list in the statement of our corollary merely spells out the outcome of the homology of degree zero $H_0(\GC_n^2)$
depending on the value of the dimension $n\geq 2$.
In short, we can check that the complex $\GC_n$ with at most trivalent vertices vanishes in degree $<n$ when $n\geq 3$.
Then we deduce from the result of Proposition~\ref{prop:graph complex homology}
that the homology of $\GC_n$
in degree $0$
reduces to the vector space spanned by the loop graph $L_n$
when $n-1\equiv 0\mymod{4}$
and is trivial otherwise.
In the case, we have $H_0(\GC_2^2)\cong\grt_1$ by \cite{Will}, where $\grt_1$ is the part of weight $\geq 1$
of the graded Grothendieck-Teichm\"uller Lie algebra defined by Drinfeld,
and we have an set-theoretic identity between this Lie algebra
and the corresponding group $\GRT_1$. (In fact, this correspondence with the Grothendieck-Teichm\"uller group
is obtained by another method in the reference~\cite{Fr}
cited in the statement of Corollary \ref{cor:codimension zero case}.)

To complete these verifications, let us observe that we have an immediate identity $\pi_k(F^{-1}(\lambda),p_{\lambda})\cong H_k(\GC_n^2)$,
for any choice of base point $p_{\lambda}\in F^{-1}(\lambda)$, since we assume that the dg Lie algebra $H(\GC_n^2)[1]$
is equipped with a trivial dg Lie structure
when form the simplicial sets $F^{-1}(\lambda)\simeq\MC_\bullet(H(\GC_n^2)[1])$.
(Thus, we consider a simple particular case of the result of Theorem \ref{thm:berglund}.)
This observation completes the verification of the claims of Corollary~\ref{cor:codimension zero case}.
\hfill\qed

\subsection{Proof of Corollary \ref{cor:codimension zero case:low degrees}}\label{subsec:HGC nerve:GC low degree homotopy}
We first assume $n=2$. In this case, we can actually deduce our claims from the quasi-isomorphism $\HGC_{2,2}\simeq K(\stp_2^*[-1],\e_2\{2\})$
which occurs in the course of our proof of Theorem~\ref{thm:main algebraic statement}
in Section~\ref{subsec:biderivations to HGC:main algebraic theorem}.
Indeed, the standard Drinfeld-Kohno Lie algebras ${\alg p}(r) = {\alg p}_2(r)$ (and hence, the dual Lie coalgebra $\stp_2(r)$)
are concentrated in degree zero, whereas the components of the suspended operad $\e_2\{2\}(r) = \e_2(r)[2(1-r)]$
are concentrated in degree $\leq 1-r$.
We use a spectral sequence argument to conclude that the twisted complex $\HGC_{2,2}^{\alpha}$
has no homology in degree $>0$, for any Maurer-Cartan element $\alpha\in\MC(\HGC_{2,2})$.
We moreover have the identity $K(\stp_2^*[-1],\e_2\{2\})_0 = \Q$, and we get $H_0(\HGC_{2,2}^{\alpha}) = \Q$.
We can actually take the loop graph $L_1$ to get a representative of a generating class of this homology $H_0(\HGC_{2,2}^{\alpha}) = \Q$
in the complex $\GC_2^2[1]$.
We then use the results established by the proof of Corollary \ref{cor:codimension zero case}
to conclude that the homotopy groups $\pi_i(\Map^h(\lD_2,\lD_2^{\Q}),p_{\lambda})$
vanish in degree $i>1$.

We now assume $n>2$. We adapt the degree counting arguments which the proof of Corollary \ref{cor:mapping spaces homotopy:low degrees}
to the case of the Kontsevich graph complex $\GC_n$.
We then have the degree formula $|\alpha| = (n-1)k - n l + n = (n-1)(g-1) - l + n$,
for each graph $\alpha$ with $k$ edges and $l$ vertices,
and where we use the notation $g$ for the genus,
so that we have $g-1 = k-l$.
We still have $3l\leq 2k\Leftrightarrow l\leq 2(g-1)$, since we assume that each internal vertex
in our graph has at least $3$ incident edges.
We accordingly get:
\begin{align*}
|\alpha| & = (n-1)(g-1) - l + n\\
& \geq (n-1)(g-1) - 2(g-1) + n\\
& \geq (n-3)(g-1) + n = (n-3)g + 3.
\end{align*}
We conclude that the components of loop order $g\leq d$, where $d$ is any fixed bound,
capture the whole homology groups $H_i(\GC_n)$
of the graph complex $\GC_n$
in degree $i\leq (n-3)(d+1) + 2$.
We apply this observation to $d = 8,10$, and we use the result of Proposition~\ref{prop:graph complex homology}
(together with the previously obtained identity between the homology of the graph complex $\GC_n^2$
and the homotopy groups of the spaces $F^{-1}(\lambda)\subset\Map^h(\lD_n,\lD_n^{\Q})$, for $\lambda\not=0$)
to check the claims of Corollary \ref{cor:codimension zero case:low degrees} in the case $n>2$.
\hfill\qed

\subsection{Proof of Corollary \ref{cor:codimension zero null component}}\label{subsec:HGC nerve:zero component homotopy}
The Maurer-Cartan element of the hairy graph complex $\HGC_{n,n}$ that corresponds to the map $*$ is $0$.
Hence, from the weak-equivalence $\Map^h(\lD_n,\lD_n^{\Q})\simeq\MC_{\bullet}(\HGC_{n,n})$ of our main theorem in the case $m=n$,
and from the general statement of Theorem \ref{thm:berglund},
we get the identity $\pi_i(\Map^h(\lD_n,\lD_n^{\Q}),*) = H_{i-1}(\HGC_{n,n})$
asserted in our Corollary \ref{cor:codimension zero null component},
for any $i\geq 1$.

The second assertion of our corollary, about the homotopy groups of the space $\Map^h(\lD_n,\lD_n^{\Q})$
at other base points of the simplicial set $F^{-1}(0)$,
follows from a spectral sequence argument,
by observing that any other Maurer-Cartan element of the hairy graph complex $\HGC_{n,n}$
forms a deformation of the null element.
\hfill\qed

%

\subsection{Proof of Corollary \ref{cor:codimension one case}}\label{subsec:HGC nerve:cerf homotopy}
In a preliminary step, we explain the definition of a function $J: \pi_0\MC_\bullet(\HGC_{n-1,n})\to\Q$
which corresponds to the locally constant function
of our corollary when we pass to the mapping space $\Map(\lD_{n-1},\lD_n)$.
To any Maurer-Cartan element $\gamma\in\MC(\HGC_{n-1,n})$,
we assign the coefficient of the tripod graph $Y = \begin{tikzpicture}[scale=.3, baseline=-.65ex]
\node[int] (v) at (0,0){};
\draw (v) -- +(90:1) (v) -- ++(210:1) (v) -- ++(-30:1);
\end{tikzpicture}$
in the expansion of $\gamma$ as a formal series of graphs in the hairy graph complex.

We aim to check that this mapping $J: \gamma\mapsto J(\gamma)$ induces a well defined map
on the homotopy class set
$$\pi_0\MC_\bullet(\HGC_{n-1,n}) = \MC(\HGC_{n-1,n})/\simeq\, ,$$
where $\simeq$ is the equivalence relation such that $d_0\rho\simeq d_1\rho$
for any $1$-simplex $\gamma\in\MC_1(\HGC_{n-1,n})$.
We may set $\rho = \alpha(t) + \beta(t) dt$, where we assume $\alpha(t)\in\HGC_{n-1,n}\hat{\otimes}\Q[t]$
and $\beta(t)\in\HGC_{n-1,n}\hat{\otimes}\Q[t]$,
with the same definition of the completed tensor product as in the case of the extended Lie algebra $\HGC_{n-1,n}\hat{\otimes}\Omega(\Delta^1)$.
We have $d_0\rho = \alpha(0)$, $d_1\rho = \alpha(1)$, and for this element $\rho\in\HGC_{n-1,n}\hat{\otimes}\Omega(\Delta^1)$
the Maurer-Cartan equation implies $\alpha'(t) = d\beta(t) + [\alpha(t),\beta(t)]$,
where we use the notation $\alpha'(t)$ for the derivative of the function $\alpha: t\mapsto\alpha(t)$,
whereas $d\alpha(t)$ denotes the point-wise the application of the differential
of the graph complex to $\beta(t)$.
We accordingly have the relation
\[
\alpha(1) - \alpha(0) = \int_0^1 (d\beta(t) + [\alpha(t),\beta(t)]) dt.
\]
in the hairy graph complex $\HGC_{n-1,n}$.


It is clear the tripod graph can not arise in the differential of another graph, and neither in the Lie bracket of two other graphs in the dg Lie algebra $\HGC_{n-1,n}$.
More precisely, since the loop order is never reduced by the differential and the bracket, the other graph(s) would need to be of loop order zero. Since by symmetry the line graph is zero in the complex $\HGC_{n-1,n}$, any those graph(s) must have at least one vertex, and hence their differential and bracket must necessarily have at least two vertices, while the tripod $Y$ has only one.

We deduce from this observation
that the expression $d\beta(t) + [\alpha(t),\beta(t)]$ in the above integral
has no tripod graph component.
We therefore get that $d_1\rho = \alpha(1)$ and $d_0\rho = \alpha(1)$
have the same tripod graph component, and hence that our map $J: \gamma\mapsto J(\gamma)$
satisfies the relation $J(d_1\rho) = J(d_0\rho)$,
as required.

The proof of the claims of our corollary is now analogous to the verification of the claims of Corollary \ref{cor:codimension zero case}.
Nevertheless, we have to distinguish the cases $n=2$ and $n\geq 3$,
since we have to consider the Shoikhet $L_{\infty}$~structure (and not the standard dg Lie algebra) on the complex $\HGC_{1,2}$
in the case $n=2$.

We address the generic case $n\geq 3$ first. Let $\lambda\in\Q^{\times}$. We set
\[
\alpha :=  \sum_{k\geq 1} \lambda^k 
\underbrace{\begin{tikzpicture}[baseline=-.65ex]
\node[int] (v) at (0,0) {};
\draw (v) edge +(-.5,-.5)  edge +(-.3,-.5) edge +(0,-.5) edge +(.3, -.5) edge +(.5,-.5);
\end{tikzpicture}
}_{2k+1}
\in\HGC_{n-1,n}^2,
\]
where we consider the obvious rescaling of the element $T$ of \eqref{eq:tripodMC}.
We then have $J(\alpha) = \lambda$, and we can again use the translation operation $\gamma\mapsto\gamma+\alpha$
to identify the simplicial set $J^{-1}(\lambda)\subset\MC_{\bullet}(\HGC_{n-1,n})$
with the nerve of the dg Lie algebra
\[
\HGC_{n-1,n}^{\sim}\subset\HGC_{n-1,n}^{\alpha}
\]
formed by the formal series of graphs in which the tripod graph has a null coefficient.

We can now see that the tripod graph spans the homology of the component of loop order zero
of the complex $H(\HGC_{n-1,n})$.
Concretely, this component is spanned by at least trivalent trees with even hairs, with differential being given by vertex splitting.
The cobar construction of the cocommutative cooperad $\Omega(\Com^c)$ (whose definition is briefly reviewed in Section~\ref{subsec:coderivation complexes:BLambda-structures})
has the same expression,
but we assume that the trees have a root in this case and that the ingoing leaves
are numbered (from $1$ to $r$ when we consider the component of arity $r$
of this operad $\Omega(\Com^c)$).
We have the precise identity $\HGC_{n-1,n}^{0-\text{loop}} = \prod_{r\geq 2}(\Omega(\Com^c)(r)_{\Sigma_{r+1}}$,
where we consider an extension of the natural action of the group $\Sigma_r$ on $\Omega(\Com^c)(r)$
to make the root equivalent to the other undistinguished leaves
in trees. (This extension is the cyclic structure of the operad $\Omega(\Com^c)$
defined by Getzler-Kapranov in \cite{GetzlerKapranov}.)
The operad $\Omega(\Com^c)$ is quasi-isomorphic to the Lie operad $\Lie\{-1\}$
by the Koszul duality of operads, and we therefore have $H(\HGC_{n-1,n}^{0-\text{loop}}) = \bigoplus_{r\geq 2}\Lie\{-1\}(r)_{\Sigma_{r+1}}$.
Now we observe that $\Lie\{-1\}(r)_{\Sigma_r}$ is identified with the vector space of Lie monomials of weight $r$
on one odd variable, which is null when $r\geq 3$.
We conclude that $H(\HGC_{n-1,n}^{0-\text{loop}})$
reduces to the vector space spanned by the tripod graph
in weight $r=2$.
We then use a straightforward spectral sequence argument (by using the filtration by the number of hairs)
to get the identity $H(\HGC_{n-1,n}^{\alpha})^{0-\text{loop}} \cong \K Y$
when we consider the twisted complex.

Then we deduce from this identity that the embedding
\[
(\HGC_{n-1,n}^{\alpha})^{\geq 1-\text{loop}}\subset\HGC_{n-1,n}^\sim,
\]
where we consider the part of loop order $\geq 1$ of the twisted dg Lie algebra $\HGC_{n-1,n}^{\alpha}$,
defines a quasi-isomorphism of dg Lie algebras.
We equip our twisted complexes with the grading by the loop order again,
rather than the weight grading $\omega(\gamma) = (\text{number of edges}) - (\text{number of internal vertices})$,
which we have considered so far. We just observe, as in our proof of Corollary \ref{cor:codimension zero case}, that this change of filtration
does not change the nerve of our dg Lie algebras.
Recall that we determine this loop order grading by the function
such that $g(\gamma) = (\text{number of internal edges}) - (\text{number of internal vertices}) + 1$,
for any graph $\gamma\in\HGC_{n-1,n}$.
We use that the differential and the Lie bracket of our twisted Lie algebras are homogeneous with respect to the loop order.
We can therefore apply the result of Theorem \ref{thm:GoldmanMillson} to the above quasi-isomorphism
to conclude that we have a weak-equivalence of simplicial sets
at the nerve level:
\[
\MC_\bullet((\HGC_{n-1,n}^{\alpha})^{\geq 1-\text{loop}})\stackrel{\simeq}{\to}\MC_{\bullet}(\HGC_{n-1,n}^\sim) = J^{-1}(\lambda).
\]

We again consider the extension of this dg Lie algebra
\[
(\HGC_{n-1,n}^{\alpha})^{\geq 1-\text{loop}}\subset(\HGC_{n-1,n}^{2,\alpha})^{\geq 1-\text{loop}}
\]
which we define by allowing graphs with bivalent internal vertices
in the construction of our object.
We still get that this inclusion defines a quasi-isomorphism of (complete graded) dg Lie algebras, just like the canonical embedding $\HGC_{n-1,n}\subset\HGC_{n-1,n}^2$
when we consider the whole complexes of connected hairy graphs,
and we therefore have the following extra weak-equivalence of simplicial sets
at the nerve level:
\[
\MC_\bullet((\HGC_{n-1,n}^{\alpha})^{\geq 1-\text{loop}})\stackrel{\simeq}{\to}\MC_\bullet((\HGC_{n-1,n}^{2,\alpha})^{\geq 1-\text{loop}})
\]
(by the result of Theorem \ref{thm:GoldmanMillson} again).
We then use the quasi-isomorphism of dg Lie algebras of Theorem~\ref{thm:cerf lemma abelian Lie structure}
\[
\K T'\oplus\GC_n^2[1]\stackrel{\simeq}{\to}(\HGC_{n-1,n}^2)^{\alpha},
\]
where we assume that $\K T'\oplus\GC_n^2[1]$ represents an abelian dg Lie algebra (equipped with a trivial Lie bracket).
We assume that the graph complex $\GC_n^2$ is equipped with the loop order as in the proof of Corollary \ref{cor:codimension zero case},
and we again immediately see that the above quasi-isomorphism induces a quasi-isomorphism
of complete graded dg Lie algebras
\[
\GC_n^2[1]\stackrel{\simeq}{\to}(\HGC_{n-1,n}^{2,\alpha})^{\geq 1-\text{loop}}
\]
when we drop the part of loop order zero of our complexes.
We consider the formality quasi-isomorphism of dg Lie algebras
\[
H(\GC_n^2)[1]\stackrel{\simeq}{\to}\GC_n^2[1]
\]
again, and we apply the result of Theorem \ref{thm:GoldmanMillson} to establish these quasi-isomorphisms
give weak-equivalences at the nerve level:
\[
\MC_\bullet(H(\GC_n^2[1]))\stackrel{\simeq}{\to}\MC_\bullet(\GC_n^2[1])\stackrel{\simeq}{\to}\MC_\bullet((\HGC_{n-1,n}^{2,\alpha})^{\geq 1-\text{loop}}).
\]

By putting all our weak-equivalences together, we conclude that we have a weak-equivalence
of simplicial sets
\[
J^{-1}(\lambda)\simeq\MC_\bullet(H(\GC_n^2)[1]),
\]
for each value of the parameter $\lambda\in\Q^{\times}$.

We now turn to the case $n=2$. We then use the result of Theorem \ref{thm:cerf lemma:Shoikhet case}
to get the existence of a Maurer-Cartan element $\alpha$
such that $J(\alpha)=\lambda$.
We then use the same arguments as in the case $n\geq 3$
to establish the existence of a weak-equivalence of simplicial sets
\[
J^{-1}(\lambda)\simeq\MC_\bullet((\HGC'{}_{1,2}^{2,\alpha})^{\geq 1-\text{loop}}),
\]
where we again use the notation $\HGC'{}_{1,2}^2$ for the complex $\HGC_{1,2}^2$
equipped with the Shoikhet $L_{\infty}$~structure.
We then consider the $L_{\infty}$~quasi-isomorphism of Theorem \ref{thm:cerf lemma abelian Lie structure:Shoikhet case}
\[
\K[1]\oplus\GC_2^2[1]\stackrel{\simeq}{\to}(\HGC'{}_{1,2}^2)^{\alpha},
\]
where we still assume that $\K[1]\oplus\GC_n^2[1]$ represents an abelian dg Lie algebra (equipped with a trivial Lie bracket).
We again obtain an $L_{\infty}$~quasi-isomorphism
\[
\GC_2^2[1]\stackrel{\simeq}{\to}(\HGC'{}_{1,2}^{2,\alpha})^{\geq 1-\text{loop}}
\]
when we withdraw the components of loop order zero.
We use that this quasi-isomorphisms agrees with the quasi-isomorphisms of Theorem \ref{thm:cerf lemma abelian Lie structure}
up to terms that strictly increase the loop order (see our observations following the statement
of Theorem \ref{thm:cerf lemma abelian Lie structure:Shoikhet case}).
We use this observation to apply Theorem \ref{thm:GoldmanMillson} to our quasi-isomorphisms.
We eventually conclude that we have a weak-equivalence of simplicial sets
\[
J^{-1}(\lambda)\simeq\MC_\bullet(H(\GC_2^2)[1])
\]
in this case $n=2$ as well, as we expect.\hfill\qed

\subsection{Proof of the observation of Remark \ref{rem:1/4}}\label{subsec:HGC nerve:canonical embedding component}
The canonical embedding $\lD_{n-1}\hookrightarrow\lD_n$ corresponds to a morphism of dg Hopf $\La$-cooperads
in our model, and this morphism is represented by a Maurer-Cartan element $m$
in the dg Lie algebra $\BiDer_{\dg\La}(C(\stp_n),W(\e_{n-1}^c))$.
We can also forget about the Hopf structure and consider the morphism of dg $\La$-cooperads
underlying our model.
By Proposition \ref{prop:biderivations to coderivations forgetful morphism}, this forgetful operation is realized by the morphism
of filtered complete dg Lie algebras (of the same proposition)
\[
\BiDer_{\dg\La}(C(\stp_n),W(\e_{n-1}^c))\to\CoDer_{\dg\La}(C(\stp_n),W(\e_{n-1}^c))
\]
when we pass to the complex of biderivations. This morphism sends our Maurer-Cartan element $m$
to some Maurer-Cartan element $m'$ on the right.

We go back to the proof of Lemma~\ref{lemm:biderivations to HGC quasi-iso},
where we check that the dg Lie algebra of biderivations
on the source of this forgetful map
is $L_{\infty}$~quasi-isomorphic to $\HGC_{n-1,n}$,
while the dg Lie algebra of coderivations on the target is quasi-isomorphic to $\fHGC_{n-1,n}$.
We just need to assume that hairy graph complexes are equipped with the Shoikhet $L_{\infty}$~structure
in the case $n=2$.
The $L_{\infty}$~morphisms in our construction preserve the complete filtrations which we attach to our objects,
and hence, can be used to transfer Maurer-Cartan elements.
We get that $m$ corresponds to some Maurer-Cartan element $\hat{m}$ in the complex of connected hairy graphs $\HGC_{n-1,n}$,
while $m'$ corresponds to some Maurer-Cartan element $\hat{m}'$ in the full hairy graph complex $\fHGC_{n-1,n}$.

We use that the composite of the above forgetful map with our $L_{\infty}$~quasi-isomorphism
determines an $L_{\infty}$~morphism from $\HGC_{n-1,n}$ to $\fHGC_{n-1,n}$.
We deduce from an immediate inspection of our construction in the proof of Lemma~\ref{lemm:biderivations to HGC quasi-iso}
that the linear component of this $L_{\infty}$~morphism
is the canonical inclusion $\HGC_{n-1,n}\hookrightarrow\fHGC_{n-1,n}$.

Recall that the complete filtration which we associate to our graph complexes is determined by the weight function
such that $\omega(\gamma) = (\text{number of edges}) - (\text{number of internal vertices})$,
for any graph $\gamma\in\fHGC_{n-1,n}$.
In the previous proof, we observed that this weight function satisfies $\omega(\gamma)\geq 2$ for any graph,
and that this lower bound is reached by the tripod graph $Y$.
These observations imply that the components of order $>1$ of our $L_{\infty}$~morphism
from $\HGC_{n-1,n}$ to $\fHGC_{n-1,n}$
can not touch the tripod graph component.
Hence, the coefficient of the tripod graph in the expansion of the Maurer-Cartan element $\hat{m}\in\HGC_{n-1,n}$
is the same as the coefficient of the tripod graph
in the expansion of $\hat{m}'$.

By \cite[Lemma 8]{TW}, the coefficient of the tripod graph in the latter Maurer-Cartan element $\hat{m}'$
is $1/4$.\footnote{In fact, the statement in loc. cit. is about the real version of $\hat{m}'$,
obtained by changing the ground field to $\R$.
However, since the coefficient of the tripod graph cannot be altered by gauge transformations,
it follows that the same result has to be true over $\Q$.}
We therefore have $\hat{m}\in J^{-1}(1/4)$, and this result completes the proof of the claim of Remark \ref{rem:1/4}.
\hfill\qed

\subsection{Proof of Corollary \ref{cor:codimension one null component}}\label{subsec:HGC nerve:cerf zero component homotopy}
The proof of the assertion of Corollary \ref{cor:codimension one null component} about the homotopy groups of the mapping space $\Map^h(\lD_{n-1},\lD_n)$
at the trivial map $*$
as base point is strictly parallel to the proof of the analogous statement of Corollary \ref{cor:codimension zero null component}
for the homotopy groups of the mapping space $\Map^h(\lD_n,\lD_n)$
at the trivial map $*$. We use that the Maurer-Cartan element of the hairy graph complex $\HGC_{n-1,n}$
that corresponds to this map $*$ is $0$.
We then use the weak-equivalence $\Map^h(\lD_{n-1},\lD_n^{\Q})\simeq\MC_{\bullet}(\HGC_{n-1,n})$
given by the result of our main theorem
and the general statement of Theorem \ref{thm:berglund}
to get the identity $\pi_i(\Map^h(\lD_{n-1},\lD_n^{\Q}),*) = H_{i-1}(\HGC_{n-1,n})$
asserted in our corollary,
for any $i\geq 1$.

We can similarly use the same arguments as in the proof of Corollary \ref{cor:codimension zero null component}
to establish our assertions about the homotopy groups
of the other connected components
of the space $J^{-1}(0)$.
\hfill\qed

\part{Mapping spaces of truncated operads in connection with the embedding calculus}\label{part:rationalization}
The Goodwillie-Weiss calculus can be used to relate the mapping spaces associated to the little discs operads $\Map^h(\lD_m,\lD_n)$
to the embedding spaces $\Emb_{\partial}(\D^m,\D^n)$ as we explain in the introduction of this article.
The main purpose of this part is to prove that our result about the mapping spaces $\Map^h(\lD_m,\lD_n^{\Q})$,
where we take the rationalization of the little $n$-discs operad as a target $\lD_n^{\Q}$,
can be used to determine the rationalization of the spaces $\Map^h(\lD_m,\lD_n)$
in a suitable range of dimension,
and hence, to give information about the rational homotopy of the embedding spaces $\Emb_{\partial}(\D^m,\D^n)$.

To be more precise, recall that the Goodwillie-Weiss calculus expresses the homotopy type of a space
of embeddings modulo immersions $\Embbar_{\partial}(\D^m,\D^n)$
in terms of the limit of a tower of polynomial approximations $T_k\Embbar_{\partial}(\D^m,\D^n)$, $k\geq 0$.
In fact, we study a parallel tower decomposition of our operadic mapping spaces $\Map^h_{\leq k}(\lD_m|_{\leq k},\lD_n|_{\leq k})$, $k\geq 0$,
by using truncation functors on the category of operads ${\op P}\mapsto{\op P}|_{\leq k}$.
We establish our comparison result for the rationalization of this tower of mapping spaces.
We devote most of this part to the analysis of this tower of operadic mapping spaces therefore.
We only address the consequences of our results for the Goodwillie-Weiss calculus of embedding spaces
as a conclusion of this study. We refer to the introduction of this paper for a more detailed
survey of the connections between our results
and the Goodwillie-Weiss calculus.

\section{Truncated operads and the rationalization of mapping spaces}\label{sec:rationalization}
The truncated operads ${\op P}|_{\leq k}$, which we consider to define our tower of mapping spaces,
are defined by forgetting about the components ${\op P}(r)$
of arity $r>k$ in our objects.
We can consider an analogous truncation operation for $\La$-operads, dg $\La$-cooperads, dg $\La$-cooperads and collections.
We briefly review the applications of homotopy theory methods
to truncated operads in a preliminary subsection.
We establish our comparison result for the rationalization of the mapping spaces of the truncated operads of little discs afterwards,
and we state the truncated analogues of the results obtained in the previous part of this paper to conclude
this study.

We mostly consider topological operads rather than simplicial operads in this part. Therefore we emphasize
the applications of our constructions for topological operads. Recall simply that the model category of $\La$-operads
in topological spaces is Quillen equivalent to the model category of $\La$-operads
in simplicial sets, so that the mapping spaces which we associate to operads
in topological spaces are weakly-equivalent
to the mapping which we may associate
to the corresponding objects in the category of operads in simplicial sets.

\subsection{Truncated operads and mapping spaces}\label{subsec:rationalization:truncated operads}
We explicitly define a $k$-truncated $\La$-operad in topological spaces as the structure formed by a finite collection ${\op P}(1),\dots,{\op P}(k)$,
where each object ${\op P}(r)$ is a topological space equipped with an action of the symmetric group $\Sigma_r$, for $r = 1,\dots,k$,
together with restriction operators $u^*: {\op P}(n)\to{\op P}(m)$
and composition products $\circ_i: {\op P}(m)\times{\op P}(n)\to{\op P}(m+n-1)$
which we may form in the range of arity $r\leq k$
where our collection ${\op P}$
is defined.
We also assume that a $k$-truncated $\La$-operad is equipped with an operadic unit $1\in{\op P}(1)$,
and that the equivariance, unit and associativity axioms of $\La$-operads
hold for $k$-truncated $\La$-operads in the range
of definition of our objects.
We can also regard the underlying collection of a $k$-truncated operad as a contravariant diagram
over the category $\Lambda_{\leq k}$
which has the ordinals $\underline{r} = \{1,\dots,r\}$ such that $1\leq r\leq k$
as objects and the injective maps between these ordinals
as morphisms.
We can still define the structure of a $k$-truncated plain (symmetric) operad by forgetting the action
of the restriction operators $u^*: {\op P}(n)\to{\op P}(m)$
in this definition.
We use the same definitions in the context of simplicial sets.

We adopt the notation $\TopCat\La\Op_{\leq k}$ (respectively, $\TopCat\Op_{\leq k}$)
for the category of $k$-truncated $\La$-operads (respectively, of $k$-truncated plain symmetric operads) in topological spaces,
for any $k\geq 1$.
We similarly use the notation $s\La\Op_{\leq k} = s\SetCat\La\Op_{\leq k}$ (respectively, $s\Op_{\leq k} = s\SetCat\Op_{\leq k}$)
for the category of $k$-truncated $\La$-operads (respectively, of $k$-truncated plain symmetric operads) in simplicial sets. Recall also from Section~\ref{sec:rational homotopy} that the categories of $\La$-operads
in topological spaces and in simplicial sets are denoted   by $\TopCat\La\Op_\varnothing$
and $s\La\Op_\varnothing = s\SetCat\La\Op_\varnothing$, respectively.

We can adapt the definition of the model structure of $\La$-operads in \cite[Section II.8.4]{Fr}
in the context of $k$-truncated $\La$-operads.
In short, we provide the category of $k$-truncated $\La$-operads in topological spaces with the model structure
where the weak-equivalences are the morphisms of $k$-truncated $\La$-operads
which define a weak-equivalence in the category of topological spaces arity-wise,
whereas the fibrations are the morphisms that define a fibration
in the Reedy model category of $\Lambda_{\leq k}$-diagrams
in topological spaces.
We then consider a natural generalization of the notion of a Reedy model structure
in the context of categories of diagrams over indexing categories
with non-trivial isomorphism sets.
We characterize the cofibrations of our model category of $k$-truncated $\La$-operads
by the left lifting property with respect to the class
of acyclic fibrations.
We still have a model structure on the category of $k$-truncated plain symmetric operads $\TopCat\Op_{\leq k}$ (see~\cite[Section 2]{FNote}).
We just forget about the restriction operators in this case and we assume that a morphism defines a fibration
in $\TopCat\Op_{\leq k}$ if this morphism forms a fibration
in the projective model category of $k$-truncated $\Sigma$-collections
which underlies this model category of operads.

Let ${\op P}^{\varnothing}$ be the $k$-truncated symmetric operad which we obtain by the forgetting about the action of the restriction
operators in a $k$-truncated $\La$-operad and by taking ${\op P}(0) = \varnothing$ in arity zero.

We have the following observation:

\begin{prop}\label{prop:cofibrant truncated operads}
A $k$-truncated operad ${\op P}$ is cofibrant in $\TopCat\La\Op_{\leq k}$ if and only if the $k$-truncated operad ${\op P}^{\varnothing}$
is cofibrant in $\TopCat\Op_{\leq k}$,
and a morphism of $k$-truncated operads $f: {\op P}\to{\op Q}$ similarly defines a cofibration in $\TopCat\La\Op_{\leq k}$
if and only if the associated morphism of $k$-truncated operads $f: {\op P}^{\varnothing}\to{\op Q}^{\varnothing}$
forms a cofibration in $\TopCat\Op_{\leq k}$.
\end{prop}

\begin{proof}
We can easily adapt the proof of an analogue of this statement in \cite[Theorem~II.8.4.12]{Fr}
to the case of truncated operads.
\end{proof}

We now consider the obvious truncation functor $(-)|_{\leq k}: \TopCat\La\Op_{\varnothing}\to\TopCat\La\Op_{\leq k}$
which we get by forgetting about the components ${\op P}(r)$ of arity $r>k$
in the structure of a standard $\La$-operad ${\op P}\in\TopCat\La\Op_{\varnothing}$.
We can also consider the restriction of this truncation functor to the category of $(k+1)$-truncated $\La$-operads
$(-)|_{\leq k}: \TopCat\La\Op_{\leq k+1}\to\TopCat\La\Op_{\leq k}$.
We have the following statement, whose verification follows from straightforward categorical arguments:

\begin{prop}\label{prop:truncation Quillen adjunction}
The truncation functor $(-)|_{\leq k}: \TopCat\La\Op_{\varnothing}\to\TopCat\La\Op_{\leq k}$
admits a left adjoint $F: \TopCat\La\Op_{\leq k}\to\TopCat\La\Op_{\varnothing}$
and we have a Quillen adjunction
\[
F: \TopCat\La\Op_{\leq k}\rightleftarrows\TopCat\La\Op_{\varnothing} :(-)|_{\leq k}
\]
between the model category of $k$-truncated $\La$-operads $\TopCat\La\Op_{\leq k}$
and the model category of standard $\La$-operads $\TopCat\La\Op_{\varnothing}$.
The restriction of our truncation functor $(-)|_{\leq k}$ to the category of $(k+1)$-truncated $\La$-operads $\TopCat\La\Op_{\leq k+1}$
similarly admits a left adjoint $F_{k+1}: \TopCat\La\Op_{\leq k}\to\TopCat\La\Op_{\leq k+1}$
so that we have a Quillen adjunction between the model categories of truncated $\La$-operads of different levels:
\[
F_{k+1}: \TopCat\La\Op_{\leq k}\rightleftarrows\TopCat\La\Op_{\leq k+1} :(-)|_{\leq k},
\]
for all $k\geq 1$.\qed
\end{prop}

We deduce from this result that the truncation ${\op P}|_{\leq k}$ of a fibrant object of our model category of $\La$-operads ${\op P}$
is still fibrant in the model category of $k$-truncated operads.
We also have the following additional properties:

\begin{prop}\label{prop:truncation cofibrant operads}
If ${\op P}$ is a cofibrant object of the model category of $\La$-operads $\TopCat\La\Op$, then its truncation ${\op P}|_{\leq k}$
still forms a cofibrant object in the model category of $k$-truncated $\La$-operads $\TopCat\La\Op_{\leq k}$.
Furthermore, the morphism $F({\op P}|_{\leq k})\to{\op P}$
which is given by the unit of our Quillen adjunction in Proposition~\ref{prop:truncation Quillen adjunction}
defines a cofibration in the model category of $\La$-operads,
and so does the morphism $F_{k+1}({\op P}|_{\leq k})\to{\op P}_{\leq k+1}$
which we get in the category of $k$-truncated $\La$-operads,
for every $k\geq 1$.
\end{prop}

\begin{proof}
This result can be deduced from the characterization of the cofibrations given in Proposition~\ref{prop:cofibrant truncated operads}.
\end{proof}

We may actually identify the image $F({\op P}|_{\leq k})$ of a $\La$-operad ${\op P}\in\TopCat\La\Op_{\varnothing}$
under the composite of the adjoint functors of Proposition~\ref{prop:truncation Quillen adjunction}
with the image of our operad under the operadic enhancements of the arity-wise filtration
functors $\ar^{\sharp}_k: {\op P}\mapsto\ar^{\sharp}_k{\op P}$
defined in~\cite[Proof of Theorem 8.4.12]{Fr}. The result of the above proposition can therefore be regarded
as a variant of the statements given in the verifications of this reference,
which imply that the morphism $\ar^{\sharp}_k{\op P}\to{\op P}$
is a cofibration when ${\op P}$ is a cofibrant object of the model category of $\La$-operads. Let us also recall from loc. cit. that
\beq{equ:limFres}
 \op P = \colim_k \ar^{\sharp}_k{\op P} = \colim_k F({\op P}|_{\leq k}).
\eeq
In fact we have $F({\op P}|_{\leq k})(r) = \op P(r)$ for $k>r$ if $\op P(0)$ is the empty set.

\begin{const}\label{const:truncated mapping spaces}
Let ${\op P}$ and ${\op Q}$ be $\La$-operads. We set:
\[
\Map_{\leq k}^h({\op P},{\op Q}) = \Map^h({\op P}|_{\leq k},{\op Q}|_{\leq k}),
\]
for each $k\geq 1$, where we consider the (derived) mapping space of the objects ${\op P}|_{\leq k}$ and ${\op Q}|_{\leq k}$
in the model category of $k$-truncated $\La$-operads $\TopCat\La\Op_{\leq k}$.
We can use the Quillen adjunctions of the previous definition
to get a tower of maps of mapping spaces
\[
\begin{tikzcd}
\Map_{\leq 1}^h({\op P},{\op Q}) & \Map_{\leq 2}^h({\op P},{\op Q})\ar{l} & \cdots\ar{l} & \Map_{\leq k}^h({\op P},{\op Q})\ar{l} & \cdots\ar{l} \\
& \Map^h({\op P},{\op Q})\ar{lu}\ar{u}\ar{rru} &&&
\end{tikzcd}
\]
and we moreover have:
\[
\Map^h({\op P},{\op Q}) \simeq \holim_k\Map_{\leq k}^h({\op P},{\op Q}).
\]

To be more explicit, recall that we have by definition $\Map^h({\op P},{\op Q}) = \Map(\check{\op P},\hat{\op Q})$,
where we consider the mapping space (in the ordinary sense)
associated to a cofibrant replacement $\check{\op P}$
of the object ${\op P}$
in the model category of $\La$-operads,
and to a fibrant replacement $\hat{\op Q}$
of the object ${\op Q}$.
The result of Proposition~\ref{prop:truncation cofibrant operads} implies that $\check{\op P}|_{\leq k}$
still forms a cofibrant object in the category of $k$-truncated operads,
for any $k\geq 1$, while we get that $\hat{\op Q}|_{\leq k}$
is fibrant by the general properties of Quillen adjunctions.
Thus, we also have the relation $\Map_{\leq k}^h({\op P},{\op Q}) = \Map_{\leq k}(\check{\op P},\hat{\op Q})$
at each stage or our tower, where we again set $\Map_{\leq k}(\check{\op P},\hat{\op Q}) = \Map(\check{\op P}|_{\leq k},\hat{\op Q}|_{\leq k})$.
The horizontal maps connecting these mapping spaces
are given by the composites
\[
\Map_{\leq k+1}(\check{\op P},\hat{\op Q})\to
\Map_{\leq k+1}(F_k(\check{\op P}|_{\leq k}),\hat{\op Q})
\cong
\Map_{\leq k}(\check{\op P},\hat{\op Q})\, ,
\]
where the isomorphism follows from the adjunction, while the first map is induced by the counit
of our adjunction
\[
F_k(\check{\op P}|_{\leq k})\to\check{\op P}_{\leq k+1}.
\]
Recall that this morphism is also a cofibration by the result of Proposition~\ref{prop:truncation cofibrant operads}.
This result and the general properties of mapping spaces in model categories
imply that the above map of mapping spaces
defines a fibration in the category of simplicial sets.
We therefore have the identity
$$
\Map^h({\op P},{\op Q})
= \Map(\check{\op P},\hat{\op Q})
\stackrel{\eqref{equ:limFres}} =
 \Map(\colim_k F(\check{\op P}|_{\leq k}),\hat{\op Q})
= \lim_k\Map_{\leq k}(\check{\op P},\hat{\op Q})
\simeq \holim_k\Map_{\leq k}^h({\op P},{\op Q})
$$
when we take the mapping space associated to our cofibrant and fibrant replacements
of the $\La$-operads ${\op P}$ and ${\op Q}$.
\end{const}

To a $k$-truncated $\La$-operad ${\op P}$, we may associate a $k$-truncated symmetric operad in the ordinary sense ${\op P}^+$
with the one point set as component of arity zero ${\op P}^+(0) = *$.
We now regard this $k$-truncated operad ${\op P}^+$ as an object of the category of $k$-truncated symmetric operads
with an arbitrary object in arity zero.
We consider the mapping space $\Map_{\leq k}^h({\op P}^+,{\op Q}^+)$
associated to such objects in this category of $k$-truncated symmetric operads.
We have the following observation

\begin{prop}[{see \cite[Theorem 2']{FNote}}]\label{prop:mapping spaces of unitary operads}
For any ${\op P},{\op Q}\in\TopCat\La\Op_{\leq k}$, we have a weak-equivalence of mapping spaces
\[
\Map_{\leq k}^h({\op P},{\op Q})\simeq\Map^h_{\leq k}({\op P}^+,{\op Q}^+),
\]
where the mapping space on the left-hand side is formed in the category of $k$-truncated $\La$-operads $\TopCat\La\Op_{\leq k}$
whereas the mapping space on the right-hand side is formed in the category of $k$-truncated symmetric operads
in topological spaces with an arbitrary component in arity zero.\qed
\end{prop}

We need this proposition to relate our statements to the results obtained for the spaces of embeddings,
because the claims of Theorem~\ref{thm:tower delooping} are established
for the categories of $k$-truncated operads with an arbitrary component of arity zero.

\subsection{The mapping spaces of truncated operads and the proof of Theorem~\ref{thm:rational mapping space towers}}\label{subsec:rationalization:mapping spaces}
We now tackle the proof of Theorem~\ref{thm:rational mapping space towers},
about the rationalization of the mapping spaces $\Map^h(\lD_m,\lD_n)$
associated to the little discs operads.
We use the tower of mapping spaces of truncated operads defined in the proof subsection.
The idea is to establish our comparison statement level-wise, by studying the homotopy of the fibers of the maps
in our tower in order to extend the result obtained for one level to the next level.

We assume $m\geq 1$ and $n-m\geq 2$.
In a preliminary step, we have to pick a cofibrant replacement $\check{\lD}_m$
of the operad on the source of our mapping space $\lD_m$,
and a fibrant replacement $\hat{\lD}_n$
of the operad on the target $\lD_n$. Besides, we need to fix a rationalization $\hat{\lD}_n^{\Q}$
of the operad $\hat{\lD}_n\simeq\lD_n$.
We use specific choices which we explain in the next paragraph.
We tackle the proof of our theorem afterwards.

\begin{const}[The Fulton-MacPherson operad and the cofibrant and fibrant replacements of the little discs operads]\label{const:Fulton-MacPherson operads}
We actually take
\[
\check{\lD}_m = \FM_m,
\]
where $\FM_m$ is the Fulton-MacPherson operad. This operad is weakly-equivalent to the little $m$-discs operad,
and is cofibrant as a topological $\La$-operad by the characterization of the cofibrant objects
of the category of $\La$-operads \cite[Theorem~II.8.4.12]{Fr} (see also \ref{prop:cofibrant truncated operads}),
and because the operad $\FM_m^{\varnothing}$, where we take $\FM_m^{\varnothing}(0) = \varnothing$,
forms a cofibrant object in the category of ordinary symmetric operads (see~\cite{Salvatore}).

In what follows, we mainly use that the spaces $\FM_m(r)$ underlying this operad $\FM_m$
are manifolds with corners of dimension $m(r-1)-1$, and that the subspace of composite operations in $\FM_m(r)$
is identified with the boundary $\partial\FM_m(r)$
of this manifold structure.
Recall also that we have a natural operad embedding $\FM_{m}\hookrightarrow\FM_{m+1}$,
which is equivalent to the canonical embedding of the little discs operads $\lD_{m}\hookrightarrow\lD_{m+1}$
in the homotopy category of operads, for each $m\geq 1$.

In principle, the rationalization of the operad of little $n$-discs $\lD_n^{\Q}$, which we form in the category of simplicial sets first
by using the general construction of \cite[Sections II.10.2, II.12.2]{Fr}
and which we transport to the category of topological spaces afterwards (by using the geometric realization functor),
is only related to the operad of little $n$-discs by a zigzag $\lD_n\stackrel{\simeq}{\leftarrow}\cdot\rightarrow\lD_n^{\Q}$.
For the cofibrant $\La$-operad $\check{\lD}_n = \FM_n$, we can get a direct morphism $\FM_n\to\lD_n^{\Q}$
with values in this operad $\lD_n^{\Q}$.
We pick a factorization $\FM_n\rightarrowtail\FM_n^{\Q}\stackrel{\simeq}{\rightarrow}\lD_n^{\Q}$,
where the first morphism is a cofibration in the category of $\La$-operads.

Then we consider the Kontsevich-Sinha operad $\KSi_n$, which is a variant of the Fulton-MacPherson operad $\FM_n$,
such that we have an embedding $\KSi_{n-1}\hookrightarrow\KSi_n$, for each $n\geq 2$,
and where $\KSi_1 = \Ass$~\cite{Sinha}.
We have homotopy equivalences $\FM_n\stackrel{\simeq}{\to}\KSi_n$, for all $n\geq 1$, which commute with the embeddings
that link the sequence of the Fulton-MacPherson operads together and the parallel sequence
of the Kontsevich-Sinha operads.
Let $\KSi_n^{\Q}$ be the operad which we obtain by taking the pushout of our rationalization morphism $\FM_n\rightarrowtail\FM_n^{\Q}$
along the weak-equivalence $\FM_n\stackrel{\simeq}{\to}\KSi_n$
in the diagram:
\[
\begin{tikzcd}
\FM_m\ar{r}\ar{d}{\simeq} & \FM_n\ar{r}\ar{d}{\simeq} & \FM_n^{\Q}\ar[dashed]{d}{\simeq} \\
\KSi_m\ar{r} & \KSi_n\ar[dashed]{r} & \KSi_n^{\Q}
\end{tikzcd}.
\]
The morphism $\FM_n^{\Q}\to\KSi_n^{\Q}$ which we get in this pushout construction
is a weak-equivalence because the model category of $\La$-operads
in topological spaces satisfies good left properness properties,
like the model category of ordinary symmetric operads (see~\cite{BM}).

Finally, we can use general model category constructions to get a fibrant object $\hat{\KSi}_n^{\Q}$
together a weak-equivalence $\KSi_n^{\Q}\stackrel{\simeq}{\to}\hat{\KSi}_n^{\Q}$
in the model category of topological $\La$-operads,
and we can pick a factorization $\KSi_n\stackrel{\simeq}{\to}\hat{\KSi}_n\twoheadrightarrow\hat{\KSi}_n^{\Q}$
of the composite morphism $\KSi_n\to\KSi_n^{\Q}\stackrel{\simeq}{\to}\hat{\KSi}_n^{\Q}$
such that the morphism $\hat{\KSi}_n\to\hat{\KSi}_n^{\Q}$ is a fibration.
The object $\hat{\KSi}_n$, which we obtain by this construction,
is clearly fibrant.
We actually take:
\begin{align*}
\hat{\lD}_n & = \hat{\KSi}_n
\intertext{to get our fibrant replacement of the operad of little $n$-discs $\lD_n$, and we similarly take:}
\hat{\lD}_n^{\Q} & = \hat{\KSi}_n^{\Q},
\end{align*}
to define our fibrant model of the rationalization.
\end{const}

Recall that $\KSi_1 = \Ass$. Thus, we can take an obvious prolongment of the inclusion $\Ass = \KSi_1\hookrightarrow\KSi_n$
to our fibrant models of the little $n$-discs operad $\hat{\lD}_n = \hat{\KSi}_n$
and of the rational operad $\hat{\lD}_n^{\Q} = \hat{\KSi}_n^{\Q}$
to provide these objects with a coaugmentation over the associative operad $\Ass$:
\[
\Ass = \KSi_1\to\hat{\KSi}_n\to\hat{\KSi}_n^{\Q}.
\]
In general, we say that a $\La$-operad ${\op P}$ is multiplicative when we have such a morphism $\Ass\to{\op P}$
from the associative operad $\Ass$ to ${\op P}$.
We use this property in our proof of the claims of Theorem~\ref{thm:rational mapping space towers}.

To be specific, we use that the existence of a morphism of $\La$-operads $\Ass\to{\op P}$
implies that the collections of spaces ${\op P}(r)$ underlying our operads ${\op P} = \hat{\KSi}_n,\hat{\KSi}_n^{\Q}$
inherit the structure of a cosimplicial object in the category of topological spaces.
The component ${\op P}(k)$ represents the term of cosimplicial dimension $k$
of this cosimplicial object
when $k\geq 1$, and we just take the base point $*$
in cosimplicial dimension zero.
The codegeneracies $s^j: {\op P}(k+1)\to{\op P}(k)$, $j = 0,\dots,k$,
are defined by the corestriction operators
given with our $\La$-structure,
while the cofaces $d^i: {\op P}(k-1)\to{\op P}(k)$
are given by the composition operations
$d^0(p) = \mu\circ_1 p$, $d^i(p) = p\circ_i\mu$, for $i = 1,\dots,n-1$,
and $d^k(p) = \mu\circ_2 p$, for any $p\in{\op P}(k)$,
where $\mu\in{\op P}(2)$ represents the image of the generating operation
of the associative operad $\Ass$
in our object. We just take $d^0(*) = d^1(*) = 1$, where $1$ is the unit of our operad, to extend these coface operations
to the base point in cosimplicial dimension zero.

To go further, we need to review the explicit definition of the notion of a fibrant object in the category of $\La$-operads.
To each $\La$-operad ${\op O}$, we associate a sequence of matching objects $M{\op O}(r)$,
which we define by the limits of the subcubical diagrams
\[
M{\op O}(r) = \lim_{S\subsetneq\{1,\ldots,r\}}{\op O}(S),
\]
whose arrows ${\op O}(S')\to{\op O}(S)$, which are defined for all pairs of nested subsets $S\subset S'\subsetneq\{1,\ldots,r\}$,
are given by the restriction operators of the $\La$-structure attached to our object ${\op O}$.
In this expression, we are not precise about the term associated to the empty set $S = \emptyset$ in our diagram.
By convention, we can assume that this term is the one-point set $*$. (Thus, we actually take the extension ${\op O}^+$
of our operad ${\op O}$ when we form this diagram.)
Equivalently, we may forget about this term and shape our limit on the diagram spanned by the non-empty sets
such that $\emptyset\subsetneq S\subsetneq\{1,\ldots,r\}$.
We have a canonical matching map
\[
m: {\op O}(r)\to M{\op O}(r),
\]
for each $r>0$, which is given by the restriction operator ${\op O}(r)\to{\op O}(S)$ associated to the embedding $S\hookrightarrow\{1,\ldots,r\}$
on each term of our limit.
The $\La$-operad ${\op O}$ is fibrant precisely when these matching maps define fibrations
in the model category of topological spaces, for every $r>0$.

In our case ${\op O} = \hat{\KSi}_n,\hat{\KSi}_n^{\Q}$, we can identify the above matching objects $M{\op O}(r)$
with the matching spaces of the cosimplicial structure which we attach to our object
in the sense of \cite[Paragraph X.4.5]{BK}.
Furthermore, each space ${\op O}(r)$ can be equipped with a canonical base point.
Indeed, we consider the element of the associative operad $\mu_r\in\Ass(r)$
that corresponds to the $r$-fold product operation $\mu_r(x_1,\dots,x_r) = x_1\cdot\ldots\cdot x_r$
in the structure of a monoid. We just take the image of this operation under the morphism $\Ass\to{\op O}$
which we attach to our object ${\op O} = \hat{\KSi}_n,\hat{\KSi}_n^{\Q}$
in order to get our base point in the space ${\op O}(r)$,
for each $r>0$.
Let us note that the spaces ${\op O}(r) = \hat{\KSi}_n(r),\hat{\KSi}_n^{\Q}(r)$ are simply connected since we assume $n\geq m+2\geq 3$.

These base points are clearly preserved by the restriction operators and, as a consequence, determine a base point in the matching object.
Then we consider the fiber of the matching map $m: {\op O}(r)\to M{\op O}(r)$ at this base point
and we use the notation $\bN{\op O}(r)$ for this object.
The fibration condition implies that we have the relation:
\[
\pi_*\bN{\op O}(r) = \bN\pi_*{\op O}(r),
\]
for each $r>0$, where on the right-hand side, we take the component of degree $r$ of the conormalized complex
of the cosimplicial group $\pi_*{\op O}(-)$ (see \cite[Proposition X.6.3]{BK}). We have the following lemma:

\begin{lemm}\label{lemm:matching map fiber connectedness}
The homotopy groups $\pi_*\bN{\op O}(r) = \bN\pi_*{\op O}(r)$ of the fiber $\bN{\op O}(r)$ of the matching map $m: {\op O}(r)\to M{\op O}(r)$,
for ${\op O} = \hat{\KSi}_n,\hat{\KSi}_n^{\Q}$, are trivial in degrees $*\leq(n-2)(r-1)$.
\end{lemm}

\begin{proof}
We have
\begin{equation}\label{eq:tensor_Q}
\pi_*\hat{\KSi}_n^{\Q}(r) = \pi_*\hat{\KSi}_n(r)\otimes\Q\Rightarrow\bN\pi_*\hat{\KSi}_n^{\Q}(r) = \bN\pi_*\hat{\KSi}_n(r)\otimes \Q
\end{equation}
by definition of our rationalization functor on operads. We therefore can only look at  the case $\mO = \hat{\KSi}_n$.

However, for simplicity let us consider first  the rational case. Up to a shift in degree by one,
$\pi_*\hat{\KSi}_n(r)\otimes\Q$ is the Drinfeld-Kohno Lie algebra $\mathfrak{p}_n(r)$ described
by Theorem~\ref{thm:Drinfeld Kohno model}. Its normalized part $\bN\mathfrak{p}_n(r)$ is the intersection
of the kernels of the restriction maps $s_i\colon \mathfrak{p}_n(r)\to \mathfrak{p}_n(r-1)$, $i=1\ldots r$, induced by
the $\La$-structure. It is easy to see that the kernel of $s_1$ is the free Lie algebra generated by $t_{1j}$,
$j=1\ldots r$. The part $\bN\mathfrak{p}_n(r)$ can be described as the subspace of this free Lie algebra spanned by
brackets in which every generator $t_{1j}$, $j=1\ldots r$, appears at least once. The smallest possible
degree in which this space is non-zero corresponds to the brackets in which every generator appears exactly once, which implies the result.

Let us now focus on the case $\mO = \hat{\KSi}_n$.
We use that we have an identity:
\begin{equation}\label{eq:norm_explicit}
\bN\pi_*\hat{\KSi}_n(r) = \bigoplus_{w\in B_{r-1}}\pi_*(S^{|w|(n-2)+1}),
\end{equation}
where $B_{r-1}$ is the subset of a monomial basis set of a free Lie algebra with $r-1$ generators that consists only of Lie words $w$ in which every generator appears at least once.  By $|w|$
one denotes the length of $w$.
We refer to \cite[Proposition~7.2]{BCKS} for a proof of this relation in the case $n=3$.
We can use the same proof for every $n\geq 3$.
In short, one first observes that the fiber of a single restriction operator $\hat{\KSi}_n(r)\to\hat{\KSi}_n(r-1)$
is equivalent to a wedge of spheres $\vee_{r-1} S^{n-1}$
in the homotopy category of spaces.
Then we can use Hilton's Theorem~\cite{Hilton} to express the homotopy groups
of this wedge as a direct sum over a monomial basis
of the free Lie algebra.
If we take the intersection of the kernels of all restriction operators,
then we just keep the terms of this decomposition
which we associate to the basis elements
in which every generator  appears.

We immediately get the conclusion of the lemma from this relation \eref{eq:norm_explicit}, since the homotopy groups of a sphere
vanish below the dimension.
\end{proof}

We use this lemma in the proof of the following proposition:

\begin{prop}\label{prop:truncated mapping spaces connectedness}
We still assume $m\geq 1$ and $n-m\geq 2$. We then have $\Map_{\leq 1}^h(\lD_m,\lD_n)\simeq\Map_{\leq 1}^h(\lD_m,\lD_n^{\Q})\simeq *$,
while the natural restriction maps
\begin{align*}
\Map_{\leq k}^h(\lD_m,\lD_n) & \to\Map_{\leq k-1}^h(\lD_m,\lD_n) \\
\text{and}\quad\Map_{\leq k}^h(\lD_m,\lD_n^{\Q}) & \to\Map_{\leq k-1}^h(\lD_m,\lD_n^{\Q})
\end{align*}
are $((n-m-2)(k-1)+1)$-connected when $k\geq 2$.
\end{prop}

\begin{proof}
We still take the Fulton-MacPherson operad $\FM_m$ as a cofibrant replacement $\check{\lD}_m = \FM_m$ of the little $m$-discs operad $\lD_m$,
and the operad $\hat{\lD}_n = \hat{\KSi}_n$ (respectively, $\hat{\lD}_n^{\Q} = \hat{\KSi}_n^{\Q}$),
which we define from the Kontsevich-Sinha operad in Construction~\ref{const:Fulton-MacPherson operads},
as a fibrant replacement of the little $n$-discs operad $\lD_n$ (respectively, of the rationalization
of the little $n$-discs operad $\lD_n^{\Q}$).
We then have $\Map_{\leq k}^h(\lD_m,\lD_n) = \Map_{\leq k}(\FM_m,\hat{\KSi}_n)$
and $\Map_{\leq k}^h(\lD_m,\lD_n^{\Q}) = \Map_{\leq k}(\FM_m,\hat{\KSi}_n^{\Q})$,
for any $k\geq 1$.
We treat the case of both mapping spaces in parallel and we set ${\op O} = \hat{\KSi}_n,\hat{\KSi}_n^{\Q}$.

The relation $\Map_{\leq 1}(\FM_m,{\op O})\simeq *$ is immediate, since the truncated operad $\FM_m|_{\leq 1}$
has only one component, defined by the one-point set
of the operadic unit in arity one.
Thus, we now check that the fiber of the map
\begin{equation}\label{eq:restrict}
\Map_{\leq k}(\FM_m,{\op O})\twoheadrightarrow\Map_{\leq k-1}(\FM_m,{\op O})
\end{equation}
is $((n-m-2)(k-1)+1)$-connected, for any $k\geq 2$, as asserted in our proposition.

Recall that this map is a fibration (see Construction~\ref{const:truncated mapping spaces}).
Moreover, we can assume that the target of this map is connected (actually simply connected)
by induction (recall that we assume $n-m\geq 2$).
We fix a base point, represented by a map
\[
\FM_m|_{\leq k-1}\stackrel{\alpha}{\to}{\op O}|_{\leq k-1},
\]
in this mapping space. We can identify the fiber of our map~\eqref{eq:restrict} over this base point with the space
formed by the $\Sigma_k$ equivariant maps of topological spaces $\rho: \FM_m(k)\to{\op O}(k)$
that make the diagram
\[
\begin{tikzcd}
\partial\FM_m(k)\ar[hookrightarrow]{d}\ar{r} & {\op O}(k)\ar{d}{m} \\
\FM_m(k)\ar[dashed]{ru}{\rho}\ar{r}{f} & M{\op O}(k)
\end{tikzcd}
\]
commute. Recall that $\FM_m(k)$ is a manifold with corners and that $\partial\FM_m(k)$
represents the subspace of decomposable operations
in this component $\FM_m(k)$
of our operad.
The factors of such composites operations in $\FM_m(k)$ necessarily belong to components $\FM_m(l)$ of arity $l<k$
in our operad, because of the arity grading of the composition products of an operad,
and because $\FM_m(1)$ is reduced to the one point set
formed by the operadic unit.
The upper horizontal arrow in our diagram is obtained by applying our map $\alpha$
on each factor of these composite operations
in our operad.
The lower horizontal arrow $f$ is given by the composite
\[
\FM_m(k)\stackrel{m}{\to}M\FM_m(k)\stackrel{M\alpha}{\to}M{\op O}(k),
\]
where we take the matching map associated to the object $\FM_m(k)$ followed by the map induced by our morphism $\alpha: \FM_m(S)\to{\op O}(S)$
on the terms of the matching limit, for $S\subsetneq\{1,\dots,k\}$.
We equivalently get that the fiber of our map~\eqref{eq:restrict}
can be described as the space of $\Sigma_k$ equivariant sections
  of the fibration $f^*m\colon f^*{\op O}(k)\to\FM_m(k)$ that are fixed on $\partial\FM_m(k)$.
  The fiber of $f^*m$ is the same as the fiber of $m$ and is equivalent to $\bN{\op O}(k)$. Notice that the symmetric group action on $\FM_m(k)$  is free. Since being a Serre fibration is a local property, the
  map $f^*{\op O}(k)/\Sigma_k\to\FM_m(k)/\Sigma_k$ is also a fibration. We obtain that our space is  the space of sections of this fiber bundle over $\FM_m(k)/\Sigma_k$ with the same fibers
  (equivalent to $\bN{\op O}(k)$) and  fixed on
  $\partial\FM_m(k)/\Sigma_k$. We have that the dimension of $\FM_m(k)$ is $m(k-1)-1$. Applying Lemma~\ref{lemm:matching map fiber connectedness}  the space of sections in question has connectivity $(n-m-2)(k-1)+1$, as asserted in our proposition.

%
\end{proof}

\begin{rem}\label{rem:finiteness of pi}
In the proof of Proposition~\ref{prop:truncated mapping spaces connectedness}, we expressed the fiber $F_k$ of
each map $\Map_{\leq k}^h(\lD_m,\lD_n) \to\Map_{\leq k-1}^h(\lD_m,\lD_n)$ as a certain relative  space of
sections of a fibration over a finite complex, such that the connectivity of the fibers is higher than the dimension of the source. The fiber is $\bN\lD_n(k)$ whose homotopy groups are finitely generated. By standard homotopy techniques,
the homotopy groups $\pi_*F_k$ also have the same property. As a consequence the groups $\pi_* \Map_{\leq k}^h(\lD_m,\lD_n)$, $n-m\geq 2$, and $\pi_*\Map^h(\lD_m,\lD_n)$, $n-m > 2$, are all finitely generated (and abelian).
\end{rem}

\begin{prop}\label{prop:mapping spaces rationalization}
In the context of Proposition~\ref{prop:truncated mapping spaces connectedness},
we moreover get that the operadic rationalization morphism $\lD_n\to\lD_n^{\Q}$
induces a rational equivalence of mapping spaces
\[
\Map_{\leq k}^h(\lD_m,\lD_n)\stackrel{\simeq_{\Q}}{\to}\Map_{\leq k}^h(\lD_m,\lD_n^{\Q}),
\]
for each $k\geq 1$.
\end{prop}

\begin{proof}
We keep the conventions of the proof of the previous proposition. We prove our claim by induction on $k\geq 1$.
We already observed, in the proof of Proposition~\ref{prop:truncated mapping spaces connectedness},
that our results
imply that the spaces $\Map_{\leq k-1}(\FM_m,{\op O})$, where ${\op O} = \hat{\KSi}_n,\hat{\KSi}_n^{\Q}$,
are simply connected.
We then form the commutative square
\[
\begin{tikzcd}
\Map_{\leq k}(\FM_m,\hat{\KSi}_n)\ar{d}\ar{r} & \Map_{\leq k}(\FM_m,\hat{\KSi}_n^{\Q})\ar{d} \\
\Map_{\leq k-1}(\FM_m,\hat{\KSi}_n)\ar{r}{\simeq_\Q} & \Map_{\leq k-1}(\FM_m,\hat{\KSi}_n^{\Q})
\end{tikzcd},
\]
where we assume by induction that the operadic rationalization morphism $\hat{\KSi}_n\to\hat{\KSi}_n^{\Q}$
induces a rational equivalence at the $(k-1)$st level of our tower of mapping
spaces.
  The vertical arrows are fibrations. Let $F_k$ and $F_k^{\Q}$  denote the fibers of the left and right vertical arrows, respectively.  
   By the proof of Proposition~\ref{prop:truncated mapping spaces connectedness}, these fibers are described as certain spaces of sections over $\FM_m(k)/\Sigma_k$. By Lemma~\ref{lemm:matching map fiber connectedness} and equation~\eqref{eq:tensor_Q},  the induced morphism of the corresponding fibrations over $\FM_m(k)/\Sigma_k$ induces a rational equivalence of fibers which are higher connected than the dimension of the base
  $\FM_m(k)/\Sigma_k$. This implies that $F_k\to F_k^{\Q}$ is a rational equivalence.
  Applying induction, the upper horizontal arrow of the square above is also a rational equivalence.
%
\end{proof}

\begin{proof}[Proof of Theorem~\ref{thm:rational mapping space towers}]
We are left with showing the last statement of the theorem.
Recall that we have $\Map(\check{\op P},\hat{\op Q}) = \lim_k\Map_{\leq k}(\check{\op P},\hat{\op Q})$,
for any pair $(\check{\op P},\hat{\op Q})$, where $\check{\op P}$ is a cofibrant $\La$-operad
and $\hat{\op Q}$ is a fibrant $\La$-operad (see Construction~\ref{const:truncated mapping spaces}).
In the context of the little discs operads, the results of Proposition~\ref{prop:truncated mapping spaces connectedness}
imply that the map $\Map^h(\lD_m,{\op O})\to\Map^h_{\leq k}(\lD_m,{\op O})$,
where ${\op O} = \lD_n,\lD_n^{\Q}$, induces an isomorphism on the homotopy groups
when $k$ is large enough with respect to the degree,
at least when we assume $n-m>2$.
Thus, we deduce from the result of Proposition~\ref{prop:mapping spaces rationalization}
that the operadic rationalization morphism $\lD_n\to\lD_n^{\Q}$
induces a rational equivalence at the limit of our mapping spaces
\[
\Map^h(\lD_m,\lD_n)\xrightarrow{\simeq_{\Q}}\Map^h(\lD_m,\lD_n^{\Q}),
\]
when $n-m>2$.

In the case $n-m=2$, we might not get a rational equivalence at this level
because the rationalization functor does not commute with limits
on the category of abelian groups
in general.
Nevertheless, the result of Proposition~\ref{prop:mapping spaces rationalization} implies that we do have the levelwise rational equivalences
asserted by the claims of Theorem~\ref{thm:rational mapping space towers}.
\end{proof}

\subsection{The biderivation complexes of truncated cooperads, the proof of Theorem~\ref{thm:nerve to rational mapping spaces}, and of Theorem~\ref{thm:nerve to truncated mapping spaces1}}\label{subsec:rationalization:truncated biderivation complexes}
We now draw the consequences of the results of the previous sections for the description
of the rational homotopy type
of the operadic mapping spaces $\Map^h(\lD_m,\lD_n)$.
We can readily complete the

\begin{proof}[Proof of Theorem~\ref{thm:nerve to rational mapping spaces}]
If we assume $n-m\geq 3$, then we have a chain of (rational) equivalences
\[
\Map^h(\lD_m,\lD_n)\simeq_{\Q}\Map^h(\lD_m,\lD_n^{\Q})\simeq\MC_{\bullet}(\Def(\E_n^c,\E_m^c))\simeq\MC_{\bullet}(\HGC_{m,n})
\]
which are respectively given by the results of Theorem~\ref{thm:rational mapping space towers},
of Theorem~\ref{thm:nerve to mapping spaces}, and of Theorem~\ref{thm:main algebraic statement},
and we just compose these (rational) equivalence to get the assertion
of Theorem~\ref{thm:nerve to rational mapping spaces}.
\end{proof}

We briefly examine the truncated counterparts of our main statements to conclude this study. We consider an obvious
analogue for dg Hopf $\La$-cooperads of the truncation functors $(-)|_{\leq k}$ which we associate
to the category of $\La$-operads.
We can still provide the category of $k$-truncated cochain dg Hopf $\La$-cooperads $\dg^*\Hopf\La\Op_{\leq k}^c$
with a model structure
so that the functor $(-)|_{\leq k}$ forms the left adjoint
of a Quillen adjunction
\[
(-)_{\leq k}: \dg^*\Hopf\La\Op_{01}^c\rightleftarrows\dg^*\Hopf\La\Op_{\leq k}^c :G,
\]
for any $k\geq 1$. We moreover have an analogue of the notion of biderivation for $k$-truncated dg Hopf $\La$-cooperads
and we can associate a truncated deformation complex $\Def_{\leq k}({\op B},{\op C})$
to any pair of cochain dg Hopf $\La$-cooperads linked by a morphism $*: {\op B}\to{\op C}$.
We then have the following truncated analogue of the result
of Theorem~\ref{thm:nerve to mapping spaces}:

\begin{thm}\label{thm:nerve to truncated mapping spaces}
We have a weak equivalence of simplicial sets
\[
\Map^h_{\leq k}(\lD_m,\lD_n^{\Q})\simeq\MC_\bullet(\Def_{\leq k}(\E_n^c,\E_m^c)),
\]
for each $k\geq 1$.
\end{thm}

\begin{proof}
We deduce this result from truncated versions of the results of Proposition~\ref{prop:Lambda-biderivation Lie algebra},
Proposition~\ref{prop:W-construction Lambda-structure},
and Proposition~\ref{prop:W-construction framing Lambda-structure}.
We just check by a careful inspection that our constructions
remain valid in the truncated setting.
We get in particular that, for a cochain dg Hopf $\La$-cooperad ${\op C}$, the $k$-truncation of the object $W{\op C}$
forms a fibrant replacement of the $k$-truncated dg Hopf $\La$-cooperad ${\op C}|_{\leq k}$.
\end{proof}

Theorem~\ref{thm:nerve to truncated mapping spaces1} occurs as a corollary of this statement.
\hfill\qed

We can also form a truncated version of the diagram of Figure~\ref{eq:bigdiagram},
except that we do not have any truncated version
of the hairy graph complexes,
and hence of the last row of this diagram.
However, we still have the following statement:

\begin{thm}\label{thm:truncated deformation complex}
For $n\geq m\geq 1$ and $n\geq 2$, the dg Lie algebra $\Def_{\leq k}(\E_n^c,\E_m^c)$ is $L_{\infty}$~quasi-isomorphic
to the biderivation dg Lie algebra $\BiDer_{\dg\La}(C(\stp_n)|_{\leq k},W(\e_m^c)|_{\leq k})$.
If we forget about dg Lie algebra structures, then we also get that $\Def_{\leq k}(\E_n^c,\E_m^c)$
is quasi-isomorphic to the truncated version $K(\stp_n|_{\leq k},\e_m\{m\}|_{\leq k})$
of the complex $K(\stp_n,\e_m\{m\})$
that occurs in the diagram of Figure~\ref{eq:bigdiagram}
 \end{thm}

\begin{proof}[Explanations]
The first statement of this theorem follows from the (truncated) formality of the little discs operads
and from the truncated version the result
of Proposition~\ref{prop:W-construction Lambda-structure}.

Recall that we have $K(\stp_n,\e_m\{m\}) = (\Hom_{\dg\La}(\stp_n,\e_m\{m\}),\partial)$ by construction.
By duality, we also have the relation:
\[
\Hom_{\dg\La}(\stp_n,\e_m\{m\})\cong\int_{\underline{r}\in\Lambda}\hat{\alg p}_n(r)\otimes\e_m\{m\}(r)
\cong\prod_{1\leq r}\bS\hat{\alg p}_n(r)\hat{\otimes}_{\Sigma_r}\e_m\{m\}(r),
\]
where we again consider the generating $\Sigma$-collection $\bS{\alg p}_n$
of the $\La$-collection underlying the Drinfeld-Kohno Lie algebra operad ${\alg p}_n$.
To define the truncated complex $K(\stp_n|_{\leq k},\e_m\{m\}|_{\leq k})$
of the proposition,
we just keep the factors of arity $r\leq k$
of this cartesian product
and we forget about the other factors.

The second statement is obtained by using a truncated version of the diagram of Figure~\ref{eq:bigdiagram}. We explicitly get a zigzag of quasi-isomorphisms
\[
\BiDer_{\dg\La}(C(\stp_n)|_{\leq k},W(\e_m^c)|_{\leq k}) = K(\stp_n|_{\leq k},\oW(\e_m^c)|_{\leq k})
\stackrel{\simeq}{\leftarrow} K(\stp_n|_{\leq k},\Omega(\e_m^c)|_{\leq k})
\stackrel{\simeq}{\rightarrow} K(\stp_n|_{\leq k},\e_m\{m\}|_{\leq k})
\]
from which the relation of our statement follows.
\end{proof}

The complex $K(\stp_n,\e_m\{m\})$, $m\geq 2$, was actually defined in the form given in the proof of this definition in~\cite{Turchin3}
under the name ${\mathcal K}^{m,n}_\pi$ (see~\cite[Section~5]{Turchin3}).
In the case $m=1$, this complex is identified with the $k$-totalization of the Scannell-Sinha cosimplicial complex~\cite{ScannellSinha}.

The results of Theorem~\ref{thm:tower delooping}, Theorem~\ref{thm:nerve to truncated mapping spaces},
and Theorem~\ref{thm:truncated deformation complex}
imply:

\begin{cor}\label{cor:embedding spaces tower}
For $n-m\geq 2$, the rational homotopy groups $\pi_*(T_k\Embbar_{\partial}(\D^m,\D^n))\otimes\Q$
of the Goodwillie-Weiss Taylor tower of the space $\Embbar_{\partial}(\D^m,\D^n)$
can be determined by:
\[
\pi_*(T_k\Embbar_{\partial}(\D^m,\D^n))\otimes\Q\simeq H_{*+m}(K(\stp_n|_{\leq k},\e_m\{m\}|_{\leq k})),
\]
where we consider the shifted homology of the truncated deformation complex $K(\stp_n|_{\leq k},\e_m\{m\}|_{\leq k})$
on the right hand side.\qed
 \end{cor}

%
%
%

\part*{Appendices}

\begin{appendix}

\section{Filtered complete $L_{\infty}$~algebras and Maurer--Cartan elements}\label{sec:Linfty algebras}
In this appendix, we recall the definition of the set of solutions of the Maurer--Cartan equation in an $L_{\infty}$~algebra
and of the nerve of an $L_{\infty}$~algebra.
In passing, we review the homotopy invariance properties of the nerve of $L_{\infty}$~algebras
and we recall the applications of twisted $L_{\infty}$~algebra structures
to the computation of the homotopy of these simplicial sets.

In Part~\ref{part:biderivation complexes}, we use that the structure of an $L_{\infty}$~algebra on a graded vector space $\DerL$
is determined by a coderivation
\[
D: S^+(\DerL[1])\to S^+(\DerL[1]),
\]
such that $\deg(D) = -1$ and $D^2 = 0$,
where we consider the non-counital coalgebra of symmetric tensors $S^+(\DerL[1])$
on the graded vector space $V = \DerL[1]$.
For simplicity, we can also regard $D$ as a coderivation defined on the whole symmetric coalgebra $D: S(\DerL[1])\to S(\DerL[1])$,
with the assumption that $D$ vanishes over the unit tensor $1\in S(\DerL[1])$ (see Section \ref{subsec:geometric Linfty:structures}).
The conditions $\deg(D) = -1$ and $D^2 = 0$ amounts to the requirement that $D$ defines a differential on the symmetric coalgebra $S(V)$
cogenerated by $V = \DerL[1]$.

Let $S_r(V)$ denote the component of weight $r$ of the symmetric algebra $S(V)$ on any graded vector space $V$.
The coderivation property implies that $D$ is fully determined by a map $\mu: S(\DerL[1])\to\DerL[1]$
such that $\mu_* = \pi D$, where $\pi: S(\DerL[1])\to\DerL[1]$ denotes the canonical projection,
and equivalently, by a collection of maps $\mu_r: S_r(\DerL[1])\to\DerL[1]$
such that
\[
\mu_r = \pi D|_{S_r(\DerL[1])},
\]
for $r\geq 1$ (with the convention $D(1) = 0\Leftrightarrow\mu_0 = 0$ for $r=0$)
The relation $D^2 = 0$ can be expressed in terms of the operations $\mu_r$ and is equivalent to the assumptions that the operation $\mu_1: \DerL[1]\to\DerL[1]$
defines a differential on $\DerL[1]$, the operation $\mu_2: S_2(\DerL[1])\to\DerL[1]$ commute with this differential $d = \mu_1$
on $\DerL[1]$,
and the higher operations $\mu_r$, $r\geq 3$,
define chain-homotopies for higher versions of the Jacobi identity in $\DerL[1]$.

Recall also that an $L_{\infty}$~morphism of $L_{\infty}$~algebras is defined by a morphism
of coaugmented dg coalgebras
\[
\phi: (S(\DerL[1]),D)\to(S(\DerL'[1]),D'),
\]
where $D: S(\DerL[1])\to S(\DerL[1])$ (respectively, $D': S(\DerL'[1])\to S(\DerL'[1])$)
is the coderivation that determines the $L_{\infty}$~structure on our source object $\DerL$ (respectively, on our target object $\DerL'$).

The cofree coalgebra property of the symmetric coalgebra $S(\DerL'[1])$
implies that such a morphism $\phi$
is fully determined by a map $U_*: S^+(\DerL[1])\to\DerL'[1]$
such that $U_* = \pi\phi$, where we consider the canonical projection $\pi: S(\DerL'[1])\to\DerL'[1]$ again.
Equivalently, we get that $\phi$ is determined by a collection of maps $U_r: S_r(\DerL[1])\to\DerL'[1]$
such that
\[
U_r = \pi\phi|_{S_r(\DerL[1])},
\]
for $r\geq 1$ (with the convention $\phi(1) = 1\Leftrightarrow U_0 = 1$ for $r=0$).
The relation $D'\phi = \phi D$, which expresses the preservation of differentials by this morphism $\phi$,
can be expressed in terms of these maps $U_r$,
and of the operations $\mu_r = \pi D|_{S_r(\DerL[1])}$ (respectively, $\mu'_r = \pi D'|_{S_r(\DerL'[1])}$)
that determine the $L_{\infty}$ algebra structure of our source object $\DerL$
(respectively, of our target object $\DerL'$).
In short, we get that $U_1: \DerL[1]\to\DerL'[1]$ carries the differential $\mu_1$ on the graded vector space $\DerL[1]$
to the differential $\mu'_1$ on the graded vector space $\DerL'[1]$,
while the higher maps $U_r: S_r(\DerL[1])\rightarrow\DerL'$, $r\geq 2$, define chain-homotopies
which formalize a preservation of Lie brackets
in the homotopy sense.

In this appendix, we use the representation of the structure of an $L_{\infty}$~algebra in terms of a collection
of operations $\mu_r: S_r(\DerL[1])\to\DerL[1]$, $r\geq 1$,
rather that the definition in terms of a differential on the symmetric algebra $D: S(\DerL[1])\to S(\DerL[1])$.
Thus, in what follows, we generally specify the structure of an $L_{\infty}$~algebra
by a pair
\[
(\DerL;\mu_1,\mu_2,\ldots),
\]
where $\DerL$ is our graded vector space, and $\mu_1,\mu_2,\ldots$ is the corresponding collection
of $\mu_r: S_r(\DerL[1])\to\DerL[1]$, $r\geq 1$.
Besides, we use the representation of an $L_{\infty}$~morphism in terms of a collection of maps $U_r: S_r(\DerL[1])\to\DerL'[1]$, $r\geq 1$,
rather that the definition in terms of a morphism of coaugmented dg coalgebras $\phi: (S(\DerL[1]),D)\to(S(\DerL'[1]),D')$
and we therefore use an expression of the form
\[
U: \DerL\to\DerL'
\]
to specify an $L_{\infty}$~morphism.

Recall that the operation $\mu_1$ in an $L_{\infty}$~algebra $\DerL$ defines a differential on $\DerL[1]$.
By desuspension, we get that the object $\DerL$ inherits a dg vector space structure with the differential $d: \DerL\to\DerL$
given by the (degree shift of this) operation
\[
d = \mu_1.
\]
In general, we refer to this dg vector space as the underlying dg vector space
of our $L_{\infty}$~algebra $\DerL$. We similarly get that the map $U_1$ in the definition of an $L_{\infty}$~morphism $U: \DerL\to\DerL'$
defines a morphism of dg vector spaces
\[
U_1: \DerL\to\DerL'
\]
between the underlying dg vector spaces
of our $L_{\infty}$~algebras $(\DerL,\mu_1)$ and $(\DerL',\mu'_1)$. We refer to this morphism of dg vector spaces $U_1$
as the linear part of our $L_{\infty}$~morphism of $L_{\infty}$~algebras $U: \DerL\to\DerL'$.

Under these conventions, we get that the structure of a dg Lie algebra $\DerL$ is identified with an $L_{\infty}$~structure
such that $\mu_1 = d$ is given by the (degree shift of the) differential $d: \DerL\to\DerL$
on our object $\DerL$, while $\mu_2$ is the Lie bracket
and $\mu_r = 0$ for $r\geq 3$. Similarly, a morphism of dg Lie algebras in the ordinary sense
is identified with an $L_{\infty}$~morphism
such that $U_r = 0$ for $r\geq 2$.
Note that we may deal with $L_{\infty}$~morphisms of dg Lie algebras, which we define as $L_{\infty}$~morphisms $U: \DerL\to\DerL'$
between the $L_{\infty}$~algebras
associated to the dg Lie algebras $\DerL$ and $\DerL'$.
In an orthogonal direction, we may also consider strict morphisms of $L_{\infty}$~algebras (as opposed to general $L_{\infty}$~morphisms),
which are identified with $L_{\infty}$~morphisms of $L_{\infty}$~algebras $U: \DerL\to\DerL'$
satisfying $U_r = 0$ for $r\geq 2$.

Recall also that, by convention, an $L_{\infty}$~morphism of $L_{\infty}$~algebras $U: \DerL\to\DerL'$ is called an $L_{\infty}$~quasi-isomorphism
when the linear part of this $L_{\infty}$~morphism is a quasi-isomorphism
of dg vector spaces $U_1: \DerL\stackrel{\simeq}{\to}\DerL'$.
In general, we use an `$\simeq$' mark in the expression of a $L_{\infty}$~morphism
to specify this quasi-isomorphism property:
\[
U: \DerL\stackrel{\simeq}{\to}\DerL'.
\]
In our constructions, we use that the $L_{\infty}$~quasi-isomorphisms are invertible, as opposed to the strict quasi-isomorphisms of $L_{\infty}$~algebras,
which are only invertible as $L_{\infty}$~quasi-isomorphisms.

\medskip
We generally deal with $L_{\infty}$~algebras equipped with a complete filtered structure when we use the nerve construction.
We make the definition of this notion of a filtered complete $L_{\infty}$~algebra
precise first.




\begin{defn}\label{defn:complete filtered Linfty algebras}
We say that an $L_{\infty}$~algebra $(\DerL;\mu_1,\mu_2,\dots)$ is a filtered complete $L_{\infty}$~algebra
when the dg vector space $\DerL$
is equipped with a descending filtration $\DerL = \mF^1\DerL\supset\mF^2\DerL\supset\cdots$
such that $\DerL = \lim_k \DerL/\mF^k\DerL$
and the $L_{\infty}$~operations $\mu_r: S(\DerL[1])\rightarrow\DerL$, $r\geq 1$,
preserve this filtration
in the sense that we have the relation
\[
\mu_r(\mF^{k_1}\DerL,\dots,\mF^{k_r}\DerL)\subset\mF^{k_1+\dots+k_r}\DerL,
\]
for every $k_1,\dots,k_r\geq 1$.
We also say that an $L_{\infty}$~morphism $U: \DerL\to \DerL'$ between filtered complete $L_{\infty}$~algebras $\DerL$ and $\DerL'$
is an $L_{\infty}$~morphism of filtered complete $L_{\infty}$~algebras
when the components of this morphism $U_r: S(\DerL[1])\rightarrow\DerL'$, $r\geq 1$,
preserve the filtration associated to our objects
in the sense that we have the relation
\[
U_r(\mF^{k_1}\DerL,\dots,\mF^{k_r}\DerL)\subset\mF^{k_1+\dots+k_r}\DerL',
\]
for every $k_1,\dots,k_r\geq 1$.
\end{defn}

Recall now that a Maurer--Cartan element in a filtered complete $L_{\infty}$~algebra $\DerL$
is a degree $-1$ element $\alpha\in\DerL$
that satisfies the following $L_{\infty}$~version of the Maurer--Cartan equation:
\[
\sum_{n=1}^\infty\frac{1}{n!}\mu_n(\alpha,\dots,\alpha)=0
\]
Note that the infinite sum in this equation converges by the completeness and the compatibility of the filtration
with the $L_{\infty}$~structure
in our Definition \ref{defn:complete filtered Linfty algebras}. We denote the set of Maurer--Cartan elements by $\MC(\DerL)$.

We readily see that an $L_{\infty}$~morphism of filtered complete $L_{\infty}$~algebras $U:\DerL\to \DerL'$
induces a map on the sets of Maurer--Cartan elements $U_*: \MC(\DerL)\to\MC(\DerL')$
such that
\[
U_*(\alpha) = \sum_{n=1}^\infty \frac{1}{n!} U_n(\alpha,\dots,\alpha),
\]
for $\alpha\in\MC(\DerL)$.
Note again that the infinite sum in this formula converges by completeness of the filtration on $\DerL'$
and by compatibility of $U$ with the filtrations
attached to our objects.

If we have a Maurer-Cartan element $\alpha\in\MC(\DerL)$, then we can build a filtered complete $L_{\infty}$~algebra $\DerL^{\alpha}$
by \emph{twisting} the $L_{\infty}$~operations on $\DerL$.
To be explicit, we define this filtered complete $L_{\infty}$~algebra by providing the graded vector space $\DerL^{\alpha}=\DerL$
with the $L_{\infty}$~operations $\mu_1',\mu_2',\dots$
such that:
\[
\mu_r'(-,\dots, -) = \sum_{n\geq 0} \frac{1}{n!} \mu_{r+n}(-,\dots,-,\underbrace{\alpha,\dots,\alpha}_{n\times}).
\]


Let $A$ be differential graded commutative algebra and let $\DerL$ be a filtered complete $L_{\infty}$~algebra.
Then the completed tensor product
\[
\DerL\hat\otimes A := \lim_n \left(\left(\DerL/\mF^n\DerL\right)\otimes A\right)
\]
carries a natural filtered complete $L_{\infty}$~structure, by $A$-linear extension of the $L_{\infty}$~structure on $\DerL$.
Now we define the \emph{nerve} of a filtered complete $L_{\infty}$~algebra $\DerL$
as the simplicial set
\beq{equ:MCbulletdef}
\MC_{\bullet}(\DerL) = \MC(\DerL\hat\otimes\Omega(\Delta^{\bullet}))
\eeq
whose $n$-simplices are the Maurer-Cartan elements in the filtered complete $L_{\infty}$~algebra $\DerL\hat\otimes\Omega(\Delta^n)$,
where $\Omega(\Delta^n)$ is the commutative dg algebra of polynomial differential forms on the $n$-simplex $\Delta^n$.
The simplicial structure of this object is defined by pulling back the cosimplicial structure
on the simplices $\Delta^{\bullet}$.

The important result for us is the following filtered version of the Goldman-Millson Theorem,
which we borrow from \cite{DolRog,Yalin}:

\begin{thm}[Filtered Goldman-Millson Theorem, \cite{DolRog}, \cite{Yalin}]\label{thm:GoldmanMillson}
Let $U: \DerL \to \DerL'$ be an $L_{\infty}$~morphism of filtered complete $L_{\infty}$~algebras.
Suppose that the restrictions $U_1 : \mF^n\DerL\to\mF^n\DerL'$
are quasi-isomorphisms for all $n=1,2,\dots$.
Then the map $\MC_\bullet(\DerL) \to \MC_\bullet(\DerL')$
induced by $U$ is a weak equivalence of simplicial sets.
\end{thm}

In fact, most filtered complete $L_{\infty}$~algebras considered in this paper have the following simple form:

\begin{defn}\label{defn:complete graded Linfty algebras}
We say that an $L_{\infty}$~algebra $(\DerL; \mu_1,\mu_2,\dots)$ is a complete graded $L_{\infty}$~algebra
when the graded vector space $\DerL$
is equipped with a product decomposition $\DerL =\prod_{k\geq 1} \DerL_k$
which is preserved by the $L_{\infty}$~operations $\mu_r: S(\DerL[1])\rightarrow\DerL$, $r\geq 1$,
in the sense that we have the relation
\[
\mu_r(\DerL_{k_1},\dots,\DerL_{k_r})\subset\DerL_{k_1+\dots+k_r},
\]
for every $k_1,\dots,k_r\geq 1$.
We also say that an $L_{\infty}$~morphism $U: \DerL\to \DerL'$ between graded complete $L_{\infty}$~algebras $\DerL$ and $\DerL'$
is an $L_{\infty}$~morphism of graded complete $L_{\infty}$~algebras
when the components of this morphism $U_r: S(\DerL[1])\rightarrow\DerL'$, $r\geq 1$,
preserve the grading attached to our objects
in the sense that we have the relation
\[
U_r(\DerL_{k_1},\dots,\DerL_{k_r})\subset\DerL'_{k_1+\dots+k_r},
\]
for every $k_1,\dots,k_r\geq 1$.

We can obviously use these definitions in the context of dg Lie algebras. We just say that $\DerL$ is a complete graded dg Lie algebra
when we have the relations $d(\DerL_k)\subset\DerL_k$ and $[\DerL_k,\DerL_l]\subset\DerL_{k+l}$
for the differential $\mu_1 = d$ and the Lie bracket operation $\mu_2 = [-,-]$
attached to our object so that $\DerL$ forms a complete graded $L_{\infty}$-algebra
when we regard the dg Lie algebras as particular examples of $L_{\infty}$-algebra structures.
\end{defn}

In some cases, we use the phrase ``(complete) weight grading'' to distinguish a complete grading
from the grading of the internal dg vector space structure
of our objects.

\begin{rem}\label{rem:graded quasi-isomorphisms}
We immediately see that a graded complete $L_{\infty}$~algebra $\DerL$
inherits a filtered complete $L_{\infty}$~algebra structure
with the filtration defined by $\mF^k\DerL=\prod_{l\geq k}\DerL_l$,
for all $k\geq 1$.
Note that the condition of Theorem \ref{thm:GoldmanMillson} is satisfied as soon as $U_1$ is a quasi-isomorphism of dg vector spaces
when we assume that $U$ is a morphism of graded complete $L_{\infty}$~algebras.
\end{rem}

The set of vertices of the nerve of a filtered complete $L_{\infty}$~algebra $\MC_0(\DerL)$
is identified with the set of Maurer-Cartan elements in $\DerL$
since we have $\Omega(\Delta^0) = \K\Rightarrow\MC_0(\DerL) = \MC(\DerL\hat\otimes\Omega(\Delta^0)) = \MC(\DerL)$.
We can use the following result to compute the homotopy groups of the nerve of a filtered complete $L_{\infty}$~algebra.

\begin{thm}[A. Berglund \cite{Be}]\label{thm:berglund}
Let $\DerL$ be any filtered complete $L_{\infty}$~algebra.
The homotopy groups of $\MC_\bullet(\DerL)$ at any basepoint $\tau\in\MC(\DerL)$
satisfy:
\[
\pi_k(\MC_\bullet(\DerL),\tau)\simeq H_{k-1}(\DerL^\tau),
\]
for $k = 1,2,\dots$, where we consider the twisted $L_{\infty}$~algebra $\DerL^\tau$.\footnote{Note that we do not require that the filtered complete $L_{\infty}$~algebra $\DerL$
satisfies any finiteness nor boundedness assumptions at this point.}
Furthermore, the group law of the homotopy group $\pi_k(\MC_\bullet(\DerL)$
corresponds to the addition in the homology $H_{k-1}(\DerL^\tau)$
when we assume $k\geq 2$.
In the case $k=1$, this group law is given by the Baker-Campbell-Hausdorff formula
in the homology group $H_0(\DerL^\tau)$.

If we also assume that $\DerL$ forms a dg vector space of finite type, then the Chevalley-Eilenberg complex $C(\DerL^\tau_{\geq 0})$,
where $\DerL^\tau_{\geq 0}$ is the truncation of the dg Lie algebra $\DerL^\tau$ in degree $*\geq 0$,
is a Sullivan model for the connected component of the space $\MC_\bullet(\DerL)$
at the basepoint $\tau\in\MC(\DerL)$.
\end{thm}

Let us observe that our definition of the nerve $\MC_\bullet(\DerL)$ depends on the filtration used to equip $\DerL$
with a complete filtered $L_{\infty}$~structure,
because this filtration occurs in the completed tensor product in \eqref{equ:MCbulletdef}.
However, the dependence on the choice of filtration is inessential, as the following result shows. We learned the result from V. Dolgushev.

\begin{thm}[Dolgushev (private communication)]\label{thm:nerve equivalences}
Let $\DerL$ be an $L_{\infty}$~algebra. We assume that $\DerL$ is equipped with comparable filtrations $\mF^n\DerL\subset\mG^n\DerL$
such that we have the relation $\DerL\cong\lim_n\DerL/\mF^n\DerL\cong\lim_n\DerL/\mG^n\DerL$
for both filtrations.
Then we have a weak-equivalence
\[
\MC_{\bullet}^{\mF}(\DerL)\stackrel{\simeq}{\to}\MC_{\bullet}^{\mG}(\DerL)
\]
between the nerves $\MC_{\bullet}^{\mF}(\DerL)$ and $\MC_{\bullet}^{\mG}(\DerL)$ determined by these complete filtered $L_{\infty}$~algebra structures
on $\DerL$.
\end{thm}

\begin{proof}
The condition $\mF^n\DerL\subset\mG^n\DerL$ implies that the identity map on $\DerL$
induces a morphism
of complete filtered $L_{\infty}$-algebras
$\lim_n\DerL/\mF^n\DerL\otimes\Omega(\Delta^{\bullet})\to\lim_n\DerL/\mG^n\DerL\otimes\Omega(\Delta^{\bullet})$.
We take the mapping induced by this morphism on Maurer-Cartan elements.
We just get the identity map of the set $\MC(\DerL)$
at the vertex level.

We now consider a $1$-simplex $\gamma\in\MC_1^{\mG}(\DerL)$.
We may write this element as
\[
\gamma = \alpha(t) + \beta(t) dt,
\]
where we assume $\alpha(t),\beta(t)\in\lim_n\DerL/\mG^n\DerL\otimes\K[[t]]$.
We can also express $\alpha(t)$ as a formal power series
\[
\alpha(t) = \sum_k\alpha_k t^k
\]
with the condition that we have $\alpha_k\in\mG^{n_k}\DerL$
with $n_k\to\infty$ when $k\to\infty$.
We get a similar representation for $\beta(t)$.
We then get that the Maurer-Cartan equation
is equivalent to the system
of equations
\[
\sum_{n=1}^\infty\frac{1}{n!}\mu_n(\alpha(t),\dots,\alpha(t)) = 0,
\quad\alpha'(t) = \sum_{n=1}^{\infty}\frac{1}{(n-1)!}\mu_n(\alpha(t),\dots,\alpha(t),\beta(t)),
\]
where we use the notation $\mu_r(-,\dots,-)$ for the maps induced by the internal operations of our $L_{\infty}$-algebra $\DerL$,
whereas $\alpha'(t)$ denote the derivative of the power series $\alpha(t)$
with respect to the $t$ variable.
We can assume that $\beta(t)$ is a constant power series $\beta(t)\equiv\beta_0$ by \cite[Lemma B.2]{DolRog}.
We then use the relation
\[
(k+1)\alpha_{k+1} = \sum_{r=1}^{\infty}\sum_{k_1+\dots+k_r=k} \frac{1}{r!}\mu_{r+1}(\alpha_{k_1},\dots,\alpha_{k_r},\beta_0),
\]
which we deduce from the second equation of our system for $k\geq 0$,
and the relation $\mu_{r+1}(\mF^{k_1}\DerL,\dots,\mF^{k_r}\DerL,\mF^1\DerL)\subset\mF^{k_1+\dots+k_r+1}\DerL$,
which we deduce from the compatility of the $L_{\infty}$-operations with respect to the $\mF$-filtration,
to establish by induction that we have $\alpha_k\in\mF^k\DerL$,
for each $k\geq 1$. These relations imply that our $1$-simplex $\gamma = \alpha(t) + \beta(t) dt\in\MC_1^{\mF}(\DerL)$
forms a well-defined element of the set $\MC_1^{\mG}(\DerL)$
as well.
We conclude from this argument line that our map induces a bijection
at the $\pi_0$-level:
\[
\pi_0(\MC_{\bullet}^{\mF}(\DerL))\stackrel{\cong}{\to}\pi_0(\MC_{\bullet}^{\mG}(\DerL)).
\]

We now fix a base point $\alpha\in\MC(\DerL)$. We then have the relations
\[
\pi_k(\MC_{\bullet}^{\mF}(\DerL),\alpha)\cong H_{k-1}(\DerL^{\alpha})\cong\pi_k(\MC_{\bullet}^{\mG}(\DerL)),
\]
for any $k\geq 1$, by Theorem \ref{thm:berglund}.
We accordingly get that our map induces a bijection at the homotopy group level
\[
\pi_k(\MC_{\bullet}^{\mF}(\DerL))\stackrel{\cong}{\to}\pi_k(\MC_{\bullet}^{\mG}(\DerL))
\]
in all degrees $k\geq 1$, and for any choice of base point $\alpha\in\MC(\DerL)$.
We conclude that our map defines a weak-equivalence as asserted in our statement.
\end{proof}

\end{appendix}

\end{document}